\documentclass[reqno,10pt]{amsart} \usepackage{graphicx,amsfonts,amssymb,amsmath,amsthm,url,amscd} \usepackage{enumerate} \usepackage{color} \usepackage[usenames,dvipsnames]{xcolor} \usepackage[normalem]{ulem} \usepackage{graphicx, amsmath, amssymb, amsfonts, amsthm, tikz, amscd} \usetikzlibrary{matrix,arrows,decorations.pathmorphing}

\usepackage{tikz,pgfplots} \pgfplotsset{compat=1.13}

\providecommand{\Abs}[1]{\left\lvert#1\right\rvert}

\usepackage{tikz} \usetikzlibrary{calc,fadings,decorations.pathreplacing} \usepackage{verbatim}

\newcommand\pgfmathsinandcos[3]{%
  \pgfmathsetmacro#1{sin(#3)}%
  \pgfmathsetmacro#2{cos(#3)}%
} \newcommand\LongitudePlane[3][current plane]{%
  \pgfmathsinandcos\sinEl\cosEl{#2} 
  \pgfmathsinandcos\sint\cost{#3} 
  \tikzset{#1/.style={cm={\cost,\sint*\sinEl,0,\cosEl,(0,0)}}} }

\newcommand\LatitudePlane[3][current plane]{%
  \pgfmathsinandcos\sinEl\cosEl{#2} 
  \pgfmathsinandcos\sint\cost{#3} 
  \pgfmathsetmacro\yshift{\cosEl*\sint} \tikzset{#1/.style={cm={\cost,0,0,\cost*\sinEl,(0,\yshift)}}} %
} \newcommand\NewLatitudePlane[4][current plane]{%
  \pgfmathsinandcos\sinEl\cosEl{#3} 
  \pgfmathsinandcos\sint\cost{#4} 
  \pgfmathsetmacro\yshift{#2*\cosEl*\sint} \tikzset{#1/.style={cm={\cost,0,0,\cost*\sinEl,(0,\yshift)}}} %
} %
  \newcommand\DrawLatitudeCircle[2][1]{ \LatitudePlane{\angEl}{#2} \tikzset{current plane/.prefix style={scale=#1}} \pgfmathsetmacro\sinVis{sin(#2)/cos(#2)*sin(\angEl)/cos(\angEl)} \pgfmathsetmacro\angVis{asin(min(1,max(\sinVis,-1)))} \draw[current plane] (\angVis:1) arc (\angVis:-\angVis-180:1); \draw[current plane,dashed] (180-\angVis:1) arc (180-\angVis:\angVis:1); }

\tikzset{%
  >=latex, 
  inner sep=0pt,%
  outer sep=2pt,%
  mark coordinate/.style={inner sep=0pt,outer sep=0pt,minimum size=3pt, fill=black,circle}%
}

\usepackage{graphicx, amsmath, amssymb, amsfonts, amsthm, stmaryrd, tikz, amscd} \usepackage[all]{xy} \usetikzlibrary{matrix,arrows,decorations.pathmorphing}

\usepackage[hyperfootnotes=false, colorlinks, citecolor= RoyalBlue, urlcolor=blue, linkcolor=blue ]{hyperref}

\theoremstyle{plain} 

\newtheorem{theorem} {Theorem}

\theoremstyle{remark}

\newtheorem{definition} [theorem] {Definition}

\newtheorem{example} [theorem] {Example} \newtheorem{remark} [theorem] {Remark}

\newtheoremstyle{itplain} 
{6pt} 
{5pt\topsep} 
{\itshape} 
{} 
{\itshape} 
{.}  
{5pt plus 1pt minus 1pt} 
{} 

\theoremstyle{itplain} 

\newtheorem{lemma}[theorem]{Lemma}

\newtheorem{proposition} [theorem]{Proposition}

\newtheorem{corollary} [theorem] {Corollary}

\usepackage{etoolbox} \patchcmd{\section}{\scshape}{\bfseries}{}{}

\usepackage{multicol} \makeindex

  

\makeatletter \renewcommand{\@secnumfont}{\bfseries} \makeatother

\renewcommand{\Re}{\mathrm{Re}} \renewcommand{\Im}{\mathrm{Im}}

\renewcommand{\geq}{\geqslant} \renewcommand{\leq}{\leqslant}

\def\h{\operatorname{h}} \numberwithin{equation}{section} \numberwithin{theorem}{section} \DeclareMathOperator{\sgn}{sgn} \DeclareMathOperator{\SL}{SL}  \DeclareMathOperator{\GL}{GL}     \def\Gr{\operatorname{Gr}}      \DeclareMathOperator{\SO}{SO}  \DeclareMathOperator{\ad}{ad} \DeclareMathOperator{\Ad}{Ad}     \DeclareMathOperator{\Hom}{Hom}      \DeclareMathOperator{\End}{End}  \DeclareMathOperator{\image}{image}   \def\eps{\varepsilon}     \DeclareMathOperator{\glLie}{\mathfrak{g}\mathfrak{l}} \def\Haar{\operatorname{Haar}} \def\PGL{\operatorname{PGL}}               \DeclareMathOperator{\id}{id}     \DeclareMathOperator{\U}{U}   \DeclareMathOperator{\trace}{trace}                        \DeclareMathOperator{\jac}{jac}  \DeclareMathOperator{\Aut}{Aut}       \DeclareMathOperator{\dist}{dist} \DeclareMathOperator{\RS}{RS}  \DeclareMathOperator{\Gal}{Gal}   \DeclareMathOperator{\diag}{diag} \DeclareMathOperator{\Sym}{Sym}    \DeclareMathOperator{\Lie}{Lie}   \DeclareMathOperator{\stab}{stab}     \def\O{\operatorname{O}}   \DeclareMathOperator{\Opp}{Op}    \DeclareMathOperator{\sym}{sym}                   \DeclareMathOperator{\rank}{rank} \DeclareMathOperator{\pr}{pr}    \DeclareMathOperator{\reg}{reg}    \DeclareMathOperator{\Ind}{Ind}      \DeclareMathOperator{\res}{res}       \newcommand{\ev}{\mathrm{ev}}          \DeclareMathOperator{\vol}{vol}     \DeclareMathOperator{\supp}{supp}

\newcommand{\sm}{\left(\begin{smallmatrix}} \newcommand{\esm}{\end{smallmatrix}\right)} \newcommand{\bpm}{\begin{pmatrix}} \newcommand{\ebpm}{\end{pmatrix}}

\usepackage{xparse}

\ExplSyntaxOn \NewDocumentCommand{\giit}{O{}} { \str_case:nn { #1 } { {}{/\mkern-6mu/} {\big}{\big/\mkern-7mu\big/} {\Big}{\Big/\mkern-10mu\Big/} {\bigg}{\bigg/\mkern-14mu\bigg/} {\Bigg}{\Bigg/\mkern-18mu\Bigg/} } } \ExplSyntaxOff

\author{Paul D. Nelson} \address{Aarhus University, Denmark}
 \email{nelson.paul.david@gmail.com} \subjclass[2010]{Primary 11F67; Secondary 11F70, 11M99, 35A27, 53D50}

\title{Spectral aspect subconvex bounds for $\U_{n+1} \times \U_n$} \hypersetup{ pdfkeywords={}, pdfsubject={}, pdfcreator={Emacs 24.5.1 (Org mode 8.2.10)}}
\begin{document}

\begin{abstract}
  Let $(\pi,\sigma)$ traverse a sequence of pairs of cuspidal automorphic representations of a unitary Gan--Gross--Prasad pair $(\U_{n+1},\U_n)$ over a number field, with $\U_n$ anisotropic.  We assume that at some distinguished archimedean place, the pair stays away from the conductor dropping locus, while at every other place, the pair has bounded ramification and satisfies certain local conditions (in particular, temperedness).  We prove that the subconvex bound
  \[
    L(\pi \times \sigma,1/2) \ll C(\pi \times \sigma)^{1/4 - \delta}
  \]
  holds for any fixed
  \[
    \delta < \frac{1}{8 n^5 + 28 n^4 + 42 n^3 + 36 n^2 + 14 n}.
  \]
  Among other ingredients, the proof employs a refinement of the microlocal calculus for Lie group representations developed with A.\ Venkatesh and an observation of S.\ Marshall concerning the geometric side of the relative trace formula.
\end{abstract}

\maketitle

\setcounter{tocdepth}{1} \tableofcontents

\maketitle

\section{Introduction}\label{sec:introduction}

\subsection{The refined GGP conjecture
for unitary groups}
Let $F$ be a number field with adele ring $\mathbb{A}$.  Let $E/F$ be a quadratic extension, let $V$ be an $(n+1)$-dimension hermitian space over $E$, and let $W$ be an $n$-dimensional nondegenerate subspace of $V$.  The pair of unitary groups $(G,H) := (\U(V),\U(W))$ will be referred to as a unitary Gan--Gross--Prasad (GGP) pair over $F$ (see \S\ref{sec:unitary-groups-ggp} for details).

Let $S$ be a finite set of places of $F$, containing every archimedean place.  We assume that $S$ is sufficiently large in various technical senses enumerated in \S\ref{sec:assumpt-conc-s}.

Let $(\pi,\sigma)$ be a pair of cuspidal automorphic representations $\pi$ of $G(\mathbb{A})$ and $\sigma$ of $H(\mathbb{A})$, respectively, that are unramified outside $S$ and tempered inside $S$.  One may attach to this pair a branching coefficient $\mathcal{L}(\pi,\sigma)$, quantifying how automorphic forms in $\pi$ correlate with those in $\sigma$ (see \S\ref{sec:period-formulas-1} and \S\ref{sec:branch-coeff}).

Assume for the moment that $\pi$ and $\sigma$ are everywhere tempered, i.e., that each of their local components is tempered.  As explained in \cite[\S1.1.6]{2020arXiv200705601B} $\pi$ and $\sigma$ are known to admit base change lifts $\pi_E$ and $\sigma_E$ to $\GL_{n+1}(\mathbb{A}_E)$ and $\GL_{n}(\mathbb{A}_E)$, respectively.  We write
\[
  L(\pi, \sigma, s) := L(\pi_E \otimes \sigma_E^\vee, s),
\]
where a superscripted $\vee$ denotes the contragredient representation and the RHS is the finite part of the Rankin--Selberg $L$-function, given for $s$ of large real part by a degree $2 n(n+1)$ Euler product over the finite primes of $F$.  It is known then that the central $L$-value $L(\pi, \sigma, 1/2)$ coincides up to mild factors with the branching coefficient $\mathcal{L}(\pi,\sigma)$ (see \S\ref{sec:period-formulas} for a more precise statement).  This result, an affirmation of conjectures of Ichino--Ikeda \cite{MR2585578} and N.\ Harris \cite{MR3159075}, was established in the stated generality recently by Beuzart-Plessis--Chaudouard--Zydor \cite[\S1.1.6]{2020arXiv200705601B}, generalizing earlier results of Wei Zhang \cite{MR3164988}, Beuzart-Plessis \cite{2016arXiv160206538B,2018arXiv181200047B} and Beuzart-Plessis--Liu--Zhang--Zhu \cite{2019arXiv191207169B}.

\subsection{Conductor dropping}
We fix an archimedean place $\mathfrak{q} \in S$, which plays a privileged role in our analysis.  According as the pair $(F_\mathfrak{q},E_\mathfrak{q})$ of local components of our fields is $(\mathbb{R},\mathbb{R} \times \mathbb{R})$ or $(\mathbb{R},\mathbb{C})$ or $(\mathbb{C},\mathbb{C} \times \mathbb{C})$, the pair $(G(F_\mathfrak{q}), H(F_\mathfrak{q}))$ of local components of our GGP pair is
\[
  (\GL_{n+1}(\mathbb{R}), \GL_n(\mathbb{R})), \quad (\U(p+1,q), \U(p,q)) \text{ with } p+q=n, \quad (\GL_{n+1}(\mathbb{C}),\GL_n(\mathbb{C})).
\]

As explained in \cite[\S15]{nelson-venkatesh-1}, the local factor of $L(\pi, \sigma, s)$ at the distinguished archimedean place $\mathfrak{q}$ is given in terms of the ``archimedean Satake parameters'' (see \S\ref{sec:satake-param-arch}) $\lambda_{\pi,1},\dotsc,\lambda_{\pi,n+1}$ of $\pi_\mathfrak{q}$ and $\lambda_{\sigma,1},\dotsc,\lambda_{\sigma,n}$ of $\sigma_\mathfrak{q}$ by the formula
\begin{equation}\label{eq:gamma-factors-intro}
  L_\mathfrak{q}(\pi, \sigma, s)
  =
  \begin{cases}
    \prod _{i,j,\pm} \Gamma_\mathbb{R}(s + (\pm (\lambda_{\pi,i} - \lambda_{\sigma,j}))^+ + a_{i j}) & \text{ if } F = \mathbb{R}, \\
    \prod _{i,j,\pm} \Gamma_\mathbb{C}(s + (\pm (\lambda_{\pi,i} - \lambda_{\sigma,j}))^+) & \text{ if } F = \mathbb{C}, \\
  \end{cases}
\end{equation}
for some $a_{i j} \in \{0, 1\}$, where $(x + i y)^+ := |x| + i y$, $\Gamma_\mathbb{R}(s) := \pi^{-s/2} \Gamma(s/2)$ and $ \Gamma_\mathbb{C}(s) := 2 (2 \pi )^{-s} \Gamma(s)$.  The local analytic conductor at $\mathfrak{q}$ is thus
\begin{equation}\label{eqn:conductor-at-q}
  C_\mathfrak{q}(\pi, \sigma)
  \asymp
  \prod_{i,j}
  (1 + |\lambda_{\pi,i} - \lambda_{\sigma,j}|_{F_\mathfrak{q}})^{2},
\end{equation}
where $|x|_{\mathbb{R}} = |x|$ and $|x|_{\mathbb{C}} = |x|^2$.

Informally, we say that $(\pi, \sigma)$ experiences \emph{conductor dropping} at $\mathfrak{q}$ if the factors in \eqref{eqn:conductor-at-q} are not all of comparable size.

\subsection{Main results}\label{sec:main-results}
Our results address the subconvexity problem, a fundamental problem in the analytic theory of $L$-functions (see for instance \cite[\S2]{MR1826269}, \cite[\S4, \S5]{MR2331346}, \cite{MichelVenkateshICM} and \cite{michel-2009}).  Informally, we show that if there is no conductor dropping at $\mathfrak{q}$, then the branching coefficient $\mathcal{L}(\pi,\sigma)$ enjoys a subconvex bound in the archimedean depth aspect at $\mathfrak{q}$.  In the everywhere tempered case, this yields a subconvex bound for $L(\pi, \sigma, 1/2)$.

We now formulate our results more precisely.

The number field $F$ and the GGP pair $(G,H)$ are regarded as fixed, as is the finite set of places $S$ taken large enough in the sense of \S\ref{sec:assumpt-conc-s}.

We assume that $H$ is anisotropic (so that $H(F) \backslash H(\mathbb{A})$ is compact) and that $G(F_\mathfrak{p})$ and $H(F_\mathfrak{p})$ are compact for all archimedean places $\mathfrak{p} \neq \mathfrak{q}$.  The anisotropy of $H$ excludes in particular the split case $E = F \times F$, corresponding to $(G,H) = (\GL_{n+1}(F), \GL_n(F))$.

For each ``auxiliary'' place $\mathfrak{p} \in S - \{\mathfrak{q} \}$, we fix ``bounded subsets'' $\Pi_{G,\mathfrak{p}}$ and $\Sigma_{H,\mathfrak{p}}$ of the respective unitary duals of $G(F_\mathfrak{p})$ and $H(F_\mathfrak{p})$.  More precisely:
\begin{itemize}
\item For archimedean $\mathfrak{p} \neq \mathfrak{q}$ (so that $G(F_\mathfrak{p})$ and $H(F_\mathfrak{p})$ are compact), the sets $\Pi_{G,\mathfrak{p}}$ and $\Sigma_{H,\mathfrak{p}}$ are assumed to be finite.
\item For non-archimedean $\mathfrak{p} \in S$, we require that $\Pi_{G,\mathfrak{p}}$ and $\Sigma_{H,\mathfrak{p}}$ have ``bounded depth'' in the sense that there is a compact open subgroup of $G(F_\mathfrak{p})$ -- call it $J$, for the moment -- such that every element of $\Pi_{G,\mathfrak{p}}$ has a nonzero $J$-invariant vector, and similarly for $\Sigma_{H,\mathfrak{p}}$.
\end{itemize}

We fix a positive constant $c > 0$, which we use to quantify conductor dropping at the place $\mathfrak{q}$.

We denote by $\mathcal{F}$ the set of pairs $(\pi,\sigma)$ of cuspidal automorphic representations of $G(\mathbb{A})$ and $H(\mathbb{A})$ with the following properties:\index{representations!families $\mathcal{F}, \mathcal{F}_T$}

\begin{enumerate}[(i)]
\item $\pi$ and $\sigma$ have unitary central characters.
\item $(\pi,\sigma)$ is locally distinguished: there is a nonzero $H(\mathbb{A})$-invariant functional $\pi \rightarrow \sigma$.  Moreover, $(\pi_\mathfrak{q}, \sigma_\mathfrak{q})$ is orbit-distinguished (see \S\ref{sec:relat-coadj-orbits}).
\item For $\mathfrak{p} \notin S$, the local components $\pi_\mathfrak{p}$ and $\sigma_\mathfrak{p}$ are unramified, i.e., contain nonzero vectors invariant by $G(\mathbb{Z}_\mathfrak{p})$ and $H(\mathbb{Z}_\mathfrak{p})$ (see \S\ref{sec:assumpt-conc-s}).
\item For $\mathfrak{p} \in S$, the local components $\pi_\mathfrak{p}$ and $\sigma_\mathfrak{p}$ are tempered.
\item For $\mathfrak{p} \in S - \{\mathfrak{q} \}$, the local components $\pi_\mathfrak{p}$ and $\sigma_\mathfrak{p}$ belong to the bounded subsets $\Pi_{G,\mathfrak{p}}$ and $\Sigma_{H,\mathfrak{p}}$ of the respective unitary duals fixed above.
\item Let \[T := \max ( \{|\lambda_{\pi,i}|_{F_\mathfrak{q}}\} \cup \{|\lambda_{\sigma,j}|_{F_\mathfrak{q}}\})\] denote the size of the largest archimedean Satake parameter of either $\pi_\mathfrak{q}$ or $\sigma_\mathfrak{q}$.  Then for all $i$ and $j$, we have
  \begin{equation}\label{eqn:no-conductor-drop-assumption}
    |\lambda_{\pi,i} - \lambda_{\sigma,j}|_{F_\mathfrak{q}} \geq c T.
  \end{equation}
\end{enumerate}
We write $\mathcal{F}_T \subseteq \mathcal{F}$ for the subset corresponding to a given value of the parameter $T$ defined above.

For $T \geq 1$ and $(\pi,\sigma) \in \mathcal{F}_T$, the ramification of $\pi$ and $\sigma$ at places $\mathfrak{p} \neq \mathfrak{q}$ is uniformly bounded, so the global analytic conductor $C(\pi, \sigma)$ is comparable to the local analytic conductor $C_\mathfrak{q}(\pi, \sigma)$.  The content of the assumption \eqref{eqn:no-conductor-drop-assumption} is thus that
\begin{equation*}
  C(\pi, \sigma)
  \asymp
  T^{2n(n+1)}.
\end{equation*}

We refer to \S\ref{sec:asymptotic-notation} for our conventions on asymptotic notation and to \S\ref{sec:prel-hecke-algebra} for the meaning of ``$\vartheta$-tempered;'' we note for now only that
\begin{itemize}
\item ``$0$-tempered'' has the same meaning as ``tempered,'' and
\item if the lift of $\sigma$ (as above) to $\GL_n/E$ is cuspidal, then $\sigma$ is known by work of Luo--Rudnick--Sarnak \cite{MR1703764} to be $\vartheta_n$-tempered, where $\vartheta_n := 1/2 - 1/(n^2+1) < 1/2$.
\end{itemize}

\begin{theorem}\label{thm:main}
  Let $T$ be sufficiently large.  Let $(\pi,\sigma) \in \mathcal{F}_T$.  Suppose for some fixed $\vartheta \in [0,1/2)$ that $\sigma$ is $\vartheta$-tempered at every finite place $\mathfrak{p} \notin S$ that splits in $E$.  Then the subconvex bound
  \begin{equation*}
    \mathcal{L}(\pi, \sigma)
    \ll
    T^{n(n+1)/2 - \delta}
  \end{equation*}
  holds for each fixed
  \begin{equation*}
    \delta < 
    \frac{1 - 2 \vartheta }{2 (A + 1 - 2 \vartheta)   },
    \quad
    A :=
    (2 (n+1)^2 - n + 1/2)(n+1).
  \end{equation*}
\end{theorem}
For example, if $\vartheta =0$ and $n=1,2,\dotsc,10$, then the above bound may be written $\mathcal{L}(\pi,\sigma) \ll_\eps T^{n(n+1)/2 - \delta+\eps}$, where $\delta$ is the reciprocal of
\[32, \, 101, \, 238, \, 467, \, 812, \, 1297, \, 1946, \, 2783, \, 3832, \, 5117.
\]
\begin{corollary}\label{cor:main}
  Let $(\pi,\sigma) \in \mathcal{F}$.  Suppose that $\pi$ and $\sigma$ are everywhere tempered.  Then the subconvex bound
  \begin{equation*}
    L(\pi, \sigma,1/2)
    \ll
    C(\pi, \sigma)^{1/4 - \delta}
  \end{equation*}
  holds for each fixed
  \begin{equation*}
    \delta <
    \frac{1}{4  n(n+1) ((2 (n+1)^2 - n + 1/2)(n+1) + 1)   },
  \end{equation*}
  the RHS of which simplifies to the fraction indicated in the abstract.
\end{corollary}
For example, for $n=1,2,\dotsc,10$, we obtain $L(\pi,\sigma,1/2) \ll_\eps C(\pi,\sigma)^{1/4-\delta+\eps}$, with $\delta$ the reciprocal of
\[
  128, \, 1212, \, 5712, \, 18680, \, 48720, \, 108948, \, 217952, \, 400752, \, 689760, \, 1125740.
\]
\begin{proof}[Proof of Corollary \ref{cor:main}]
  We apply Theorem \ref{thm:main} with $\vartheta = 0$.  By \cite[\S1.1.6]{2020arXiv200705601B}, we have $L(\pi, \sigma, 1/2) \asymp B_1 B_2 \mathcal{L}(\pi,\sigma)$, where $B_1$ denotes the value at $s=1$ of a certain adjoint $L$-function and $B_2 := \prod_{\mathfrak{p} \in S} L_\mathfrak{p}(\pi, \sigma, 1/2)$.  We bound $B_1 \ll T^{\eps}$ cheaply by majorizing its Euler product on the line $\Re(s) = 1 + \eps$, invoking the temperedness assumption, and applying Phragmen--Lindel{\"o}f.  We bound $B_2 \ll 1$ using the temperedness assumption.
\end{proof}

\begin{remark}\label{rmk:provable-examples}
  There are many infinite families to which Corollary \ref{cor:main} provably applies.  For instance, taking for $E/F$ an imaginary quadratic extension of $\mathbb{Q}$ and starting with one pair $(\pi,\sigma) \in \mathcal{F}$ for which $\pi$ and $\sigma$ are known to be everywhere tempered (e.g., via the cohomology of Shimura varieties), we obtain an infinite family of such pairs by twisting either $\pi$ or $\sigma$ by any sequence of characters having fixed finite conductor and increasing analytic conductor at the archimedean place.  Corollary \ref{cor:main} gives in such cases what one might call a ``twisted $t$-aspect'' subconvex bound on $\U_{n+1} \times \U_n$.
\end{remark}

\begin{remark}\label{rmk:unoptimized}
  We have not optimized the numerical quality of our results.  We have aimed instead for the simplest argument that yields an explicit exponent.  We indicate throughout the text several places where, with additional effort, one should be able to improve that exponent (see Remarks \ref{rmk:cancellation-in-integral}, \ref{rmk:opt-1}, \ref{rmk:opt-2}, \ref{rmk:refine-via-microlocalization-of-Psi}).  To indicate the potential for improvement, suppose $n=1$.  Corollary \ref{cor:main} then reads
  \[
    L(\pi, \sigma, 1/2) \ll_{\eps} C(\pi, \sigma)^{1/4 - 1/128 + \eps},
  \]
  but an optimized version of the proof would be equivalent to the original argument of Iwaniec--Sarnak \cite{iwan-sar} (see Remark \ref{rmk:compare-iwaniec-sarnak}), which gives (at least for $\sigma$ trivial) the much stronger bound
  \[
    L(\pi, \sigma, 1/2) \ll_{\eps} C(\pi, \sigma)^{1/4 - 1/24 + \eps}.
  \]
  We expect similar potential for improvement in higher rank.  We hope the pursuit of such improvement may serve as an interesting challenge. 
\end{remark}

\begin{remark}
  In the special case of Theorem \ref{thm:main} in which $\sigma$ is fixed, one can likely use Rankin--Selberg estimates to remove the dependence upon $\vartheta$.
\end{remark}

\begin{remark}
  The interested reader will see that many of our local calculations apply over any local field.  The restriction to the archimedean aspect comes primarily from invocation of the Kirillov formula, which is most readily available in the desired generality over archimedean local fields.
\end{remark}

\begin{remark}
  It would be an interesting challenge to adapt the proof to the split case $(G,H) = (\GL_{n+1}, \GL_n)$, possibly with $\pi$ or $\sigma$ an Eisenstein series.  The non-compactness of the corresponding quotients presents significant difficulties.\footnote{These difficulties have since been addressed in the preprint \cite{2021arXiv210915230N}.}
\end{remark}

\subsection{Related results}
There have been relatively few subconvex bounds in higher rank (i.e., for groups with a simple factor of rank $\geq 2$, such as $\GL_3$).  The first were X.\ Li's bounds \cite{MR2753605} for $\GL_3 \times \GL_1$ and $\GL_3 \times \GL_2$.  In her setup, the $\GL_3$ form is fixed (and self-dual) while the $\GL_2/\GL_1$ form varies.  Much recent work on the problem is similar in spirit (see for instance \cite{MR3334086, MR3369905, MR3418527,RHPNtwists,2019arXiv190309638A, 2018arXiv181000539M,2020arXiv201010153S,2019arXiv191209473L}).

The first subconvex bounds involving genuine variation of a form on a higher rank group were achieved in a seminal paper of Blomer--Buttcane \cite{MR4203038}, following their development of the $\GL_3$ Kuznetsov formula and a deep study of the associated integral transforms \cite{MR3127065,MR3130656,MR3461048}.  They proved that $L(\pi,1/2) \ll T^{3/4 - 1/120000}$ for full-level spherical Maass forms $\pi$ on $\PGL_3$ that are tempered at $\infty$ and whose parameter $\mu \in (i \mathbb{R})^3$ satisfies, with $T := \max(|\mu_1|, |\mu_2|,|\mu_3|)$, the following conditions for some fixed $c > 0$:
\begin{equation}\label{eqn:BlBu-cond-drop}
  |\mu_k| \geq c T
  \quad (1 \leq k \leq 3)
  \quad
  \text{``no conductor dropping,''}
\end{equation}
\begin{equation}\label{eqn:BlBu-Weyl-wall}
  |\mu_i - \mu_j| \geq c T
  \quad (1 \leq i < j \leq 3)
  \quad
  \text{``avoidance of Weyl chamber walls.''}
\end{equation}
Our assumption \eqref{eqn:no-conductor-drop-assumption} is analogous to \eqref{eqn:BlBu-cond-drop} (with the differences $\lambda_{\pi,i} - \lambda_{\sigma, j}$ playing the role of the $\mu_k$), but we do not require any assumption like \eqref{eqn:BlBu-Weyl-wall}.  Indeed, our method applies in the $t$-aspect, where $\mu_i - \mu_j \ll 1$.

Blomer--Buttcane generalized their result to the family of generalized principal series on $\GL_3$ \cite{MR4039487}.

Simon Marshall (March 2018 informal IAS seminar) tentatively announced a subconvex bound on GGP pairs in the $p$-adic depth aspect, for principal series representations with parameters satisfying the analogues of \eqref{eqn:BlBu-cond-drop} and \eqref{eqn:BlBu-Weyl-wall}, introducing important and fundamental ideas (see \S\ref{sec:relat-trace-form-1}).

Kumar--Mallesham--Singh \cite{2020arXiv200607819K}, by a different method than that of Blomer--Buttcane, recently established subconvex bounds on $\PGL_3$ (with any fixed $\GL_2$ twist) assuming \eqref{eqn:BlBu-cond-drop} and a modified form of \eqref{eqn:BlBu-Weyl-wall}.  P.\ Sharma \cite{2020arXiv201010153S} recently generalized Blomer--Buttcane's results to the case of fixed cuspidal $\GL_2$ twists.

\subsection{Proof sketch}\label{sec:proof-sketch}
We record here a very high-level overview of the proof, intended for experts.  We give in \S\ref{sec:overview-proof} a more leisurely introduction to the main ideas of the paper.

\subsubsection{Setup and notation}\label{sec:proof-sketch-notation}
For simplicity of presentation, we pretend that
\[
  F = \mathbb{Q}, \quad S = \{\infty\}, \quad \mathfrak{q} = \infty
\]
so that $F_\mathfrak{q} = \mathbb{R}$.  We suppose further that $E$ splits at $\infty$, so that
\begin{equation*}
  G(F_\mathfrak{q}) = G(\mathbb{R}) = \GL_{n+1}(\mathbb{R}), \quad H(F_\mathfrak{q}) = H(\mathbb{R}) = \GL_n(\mathbb{R}).
\end{equation*}
We replace the adelic quotients $G(F) \backslash G(\mathbb{A})$ and $H(F) \backslash H(\mathbb{A})$ with real quotients, and simplify our notation a bit:
\[
  G := \GL_{n+1}(\mathbb{R}), \quad H := \GL_{n}(\mathbb{R}), \quad [G] := \Gamma \backslash G, \quad [H] := \Gamma_H \backslash H.
\]
Here $\Gamma$ and $\Gamma_H$ are lattices satisfying $\Gamma_H = \Gamma \cap H$, with $\Gamma_H$ cocompact.  Finally, for $(\pi,\sigma) \in \mathcal{F}_T$, we will confuse $\pi$ and $\sigma$ with their local components at $\infty$.

We write $\mathfrak{g}, \mathfrak{h}$ for the Lie algebras and $\mathfrak{g}^\wedge, \mathfrak{h}^\wedge$ for their imaginary duals.  We write $Z \leq G$ and $Z_H \leq H$ for the centers and $\mathfrak{z}, \mathfrak{z}_H$ for their Lie algebras.

\subsubsection{Analytic test vectors}\label{sec:analyt-test-vect-1}
Let $(\pi,\sigma) \in \mathcal{F}_T$ with $T$ large.  We construct test vectors in the manner described in \cite[\S1.10]{nelson-venkatesh-1}:

The distinction assumption implies that we may choose an element $\tau \in \mathfrak{g}^\wedge$, with restriction $\tau_H \in \mathfrak{h}^\wedge$, so that $\tau$ (resp. $\tau_H$) lies in the coadjoint orbit $\mathcal{O}_\pi$ of $\pi$ (resp.  $\mathcal{O}_\sigma$ of $\sigma$) and $|\tau| \asymp T$.

We define a test function $f \in C_c^\infty(G)$, as follows.  We take $f$ supported near the identity.  We may describe it by its pullback to the Lie algebra $\mathfrak{g}$, hence by the Fourier transform $a : \mathfrak{g}^\wedge \rightarrow \mathbb{C}$ of that pullback.  We specify that $a$ is a smooth bump concentrated on
\begin{equation}\label{eq:lefttau-+-xi}
  \left\{\tau + \xi : \xi ' \ll T^{1/2+\eps}, \xi '' \ll T^{\eps} \right\},
\end{equation}
where $\xi = (\xi ', \xi '')$ denotes a coordinate system for which the $\xi '$-axis consists of all directions tangent to $\mathcal{O}_\pi$ at $\tau$ while the $\xi ''$-axis consists of the orthogonal directions, as in Figure \ref{fig:tau-coordinates-intro-0}.  Then $f$ is essentially supported on
\begin{equation}\label{eqn:support-f-intro}
  \left\{ g \in G : g = 1 + \O(T^{-\eps}),
    \, \,
    \Ad^*(g) \tau = \tau + \O(T^{1/2-\eps})
  \right\}.
\end{equation}
The corresponding operator $\pi(f)$ is approximately a rank one projector with range spanned by a unit vector ``microlocalized at $\tau$'' in the sense of \cite{nelson-venkatesh-1}.  This feature, which we discuss further in \S\ref{sec:micr-vect}, motivates the shape of \eqref{eq:lefttau-+-xi}.

\setlength{\unitlength}{1.5cm}
\begin{figure}
  \begin{picture}(4,3)

    \put(-1,0){\vector(1,0){6}}
    \put(-1,0){\vector(0,1){1.5}}

    {%
      \thicklines
      \color{black}%
    }

    {%
      \thicklines
      \color{black}%

      {%
        \thicklines
        \color{black}%
        \qbezier(0,1)(2,0)(4,1)
        \put(-0.6,1.1){$\mathcal{O}_\pi$}
      }

      \color{black}
      \put(-1.4,1.4){$\xi''$}
      \put(4.9,-0.25){$\xi'$}
      \put(2,0.35){$\tau$}
      \put(1.95,0.5){\circle*{0.1}}
    }

    {%
      \thicklines
      \color{black}%
      \multiput(0.9,0.3)(0,0.1){4}{\line(0,1){0.05}}
      \multiput(2.8,0.3)(0,0.1){4}{\line(0,1){0.05}}
    }

    {%
      \thicklines
      \color{black}%
      \multiput(0.8,0.3)(0.1,0){20}{\line(1,0){0.05}}
      \multiput(0.8,0.7)(0.1,0){20}{\line(1,0){0.05}}
      \put(2.9, 0.4){$T^{\eps}$}
      \put(1.5,0.8){$T^{1/2+\eps}$}
    }

  \end{picture}
  \caption{ The coadjoint orbit $\mathcal{O}_\pi$ near $\tau$.  The dotted rectangle indicates the support of $a$.  }
  \label{fig:tau-coordinates-intro-0}
\end{figure}
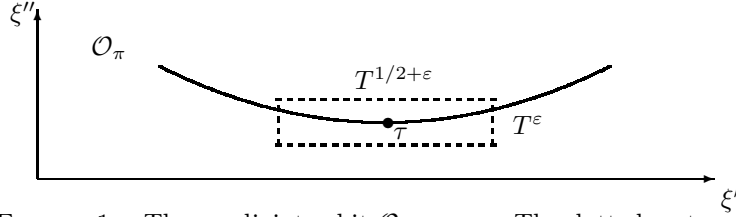

We construct a unit vector $\Psi \in \sigma$, microlocalized at $\tau_H$.

The definition of $\mathcal{L}(\pi,\sigma)$, combined with the local matrix coefficient integral asymptotics of \cite[\S19]{nelson-venkatesh-1}, give the integral representation
\begin{equation}\label{eqn:high-level-sketch-integral-rep}
  \mathcal{L}(\pi,\sigma)
  \approx
  T^{n^2/2}
  \sum _{\varphi \in \mathcal{B}(\pi)}
  \left\lvert
    \int _{[H]}
    \pi(f) \varphi
    \cdot \bar{\Psi}
  \right\rvert^2.
\end{equation}

A key ingredient in these constructions and proofs is a refinement of the microlocal calculus developed in \cite{nelson-venkatesh-1} (see \S\ref{sec:microlocal-calculus}).

\subsubsection{Relative trace formula}\label{sec:relat-trace-form-1}
The main point is to show that
\begin{equation}\label{eqn:basic-pretrace-intro-0}
  T^{-n/2}
  \int_{x,y \in [H]}
  \sum _{\gamma \in \Gamma }
  \bar{\Psi}(x)
  \Psi(y)
  f(x^{-1} \gamma y) \, d x \, d y
  = \text{(main term)}  + \O(T^{-\delta}).
\end{equation}
Indeed, by combining \eqref{eqn:high-level-sketch-integral-rep} with the pretrace formula for $[G]$ applied to $f$, we may understand the LHS of \eqref{eqn:basic-pretrace-intro-0} as an average of $\mathcal{L}(\pi,\sigma)$ over $\pi$ in a family of size $T^{n(n+1)/2}$.  In view of our ``no conductor dropping'' assumption, the family size is approximately $C(\pi,\sigma)^{1/4}$.  Strictly speaking, we should replace $f$ here by a convolution $f \ast f^*$, but that convolution has the same shape as $f$.  Given a sufficiently robust proof of \eqref{eqn:basic-pretrace-intro-0}, the amplification method will thus deliver a subconvex bound.  We refer to \S\ref{sec:relat-trace-form} for further discussion concerning the reduction to \eqref{eqn:basic-pretrace-intro-0}.

An apparent main term in \eqref{eqn:basic-pretrace-intro-0} of size $\O(1)$ comes from those $\gamma$ lying in $H Z$, so it suffices to show that the remaining terms contribute $\O(T^{-\delta})$.  Since $f$ is supported in a fixed compact set, the number of $\gamma$ is $\O(1)$ (although after amplification, that number will instead be a small positive power of $T$).  Fixing a fundamental domain $\mathcal{H}$ for $[H]$, our task is thus to show for each essentially fixed $\gamma \in \Gamma - H Z$ that
\begin{equation}\label{eqn:off-diagonal-intro}
  T^{-n/2} \int _{x,y \in \mathcal{H}}
  \bar{\Psi}(x) \Psi(y)
  f(x^{-1} \gamma y) \, d x \, d y
  \ll T^{-\delta}.
\end{equation}

Such integrals were first considered by S.\ Marshall in a related $p$-adic depth aspect.  He proposed that \eqref{eqn:off-diagonal-intro} should follow from consideration of the $L^2$-normalization of $\Psi$, the (approximate) equivariance of $\Psi$ with respect to the centralizer $H_{\tau_H}$ of $\tau_H$, and the size and support properties of $f$.  More precisely, he proposed the nontrivial volume bound
\begin{equation}\label{eqn:volume-to-bound-intro}
  \vol \left\{z \in H_{\tau_H} \middle\vert
    \begin{array}{c}
      z = \O(1),
      \\
      f(x^{-1} \gamma y z) \neq 0
      \text{ for some } x \in \mathcal{H}
    \end{array}
  \right\} \ll T^{-\delta},
\end{equation}
which turns out to yield \eqref{eqn:off-diagonal-intro} (with a different value of $\delta$) after using the approximate equivariance to replace $\Psi(y)$ by an average of $\Psi(y z)$ over small elements $z \in H_{\tau_H}$ and appealing to Cauchy--Schwarz.  We emphasize here that $\gamma \in \Gamma - H Z$ is essentially fixed.  

\subsubsection{Volume bounds}\label{sec:volume-bounds}
In the setting of Theorem \ref{thm:main}, we establish \eqref{eqn:volume-to-bound-intro} indirectly, only as a consequence of the stronger estimate
\begin{equation}\label{eqn:volume-to-bound-intro-2}
  \vol \left\{z \in Z_H
    \middle\vert
    \begin{array}{c}
      z = \O(1) ,
      \\
      f(x^{-1} \gamma y z) \neq 0
      \text{ for some } x \in \mathcal{H}
    \end{array}
  \right\} \ll T^{-\delta}
\end{equation}
obtained by specializing from the $n$-dimensional centralizer $H_{\tau_H}$ of $\tau_H$ to the $1$-dimensional center $Z_H$ of $H$.

That central directions suffice for establishing \eqref{eqn:volume-to-bound-intro} was first observed empirically after extensive numerical experimentation using CoCalc \cite{cocalc}

\subsubsection{Coadjoint reformulation}
The crucial case in proving \eqref{eqn:volume-to-bound-intro-2} (to which others ultimately reduce) is when $y = 1$ and $\Ad^*(\gamma) \tau = \tau$.  In view of the support condition \eqref{eqn:support-f-intro} for $f$, we reduce to verifying that
\begin{equation}\label{eqn:coadjoint-reformulate-volume-intro}
  \vol \left\{ z \in Z_H
    \middle\vert
    \begin{array}{c}
      z = \O(1), \\
      \Ad^*(\gamma z) \tau \in \Ad^*(H) \tau +
      \O(T^{1/2})
    \end{array}
  \right\}
  \ll T^{-\delta}.
\end{equation}

\subsubsection{Reduction to the Lie algebra}\label{sec:reduct-lie-algebra}
The estimate \eqref{eqn:coadjoint-reformulate-volume-intro} would fail most spectacularly if $\Ad^*(\gamma Z_H)\tau \subseteq \Ad^*( H ) \tau$.  Using that $\Ad^*(\gamma z ) \tau = \Ad^*(\gamma z \gamma^{-1}) \tau$ and taking Lie algebras, we would then have $\ad^*(\Ad(\gamma) \mathfrak{z}_H) \tau \subseteq \ad^*(\mathfrak{h}) \tau$.  If this last containment were to hold for all $\gamma$ in some one-parameter subgroup, we would obtain
\begin{equation}\label{eqn:intro-reduction-after-implicit-func-thm}
  \ad^*(\ad(x) \mathfrak{z}_H) \tau \subseteq
  \ad^*(\mathfrak{h}) \tau
  \text{ for some $x \in \mathfrak{g}_\tau - \mathfrak{z}$,}
\end{equation}
where $\mathfrak{g}_\tau := \{x \in \mathfrak{g} : \ad^*(x) \tau = 0\}$.

In \S\ref{sec:volume-estimates}, we use the implicit function theorem and related ideas to reverse the above reasoning: we show that to establish the required volume bound \eqref{eqn:volume-to-bound-intro-2}, it suffices to exclude the possibility of the apparent ``worst-case scenario'' \eqref{eqn:intro-reduction-after-implicit-func-thm} at the level of the Lie algebra.

\subsubsection{Endgame}\label{sec:endgame}
The exclusion of \eqref{eqn:intro-reduction-after-implicit-func-thm} is the content of the following linear algebra result, proved in \S\ref{sec:some-invar-theory}.  Here the eigenvalue condition on $\tau$ corresponds precisely to the ``no conductor dropping'' assumption on $(\pi, \sigma)$.

\begin{theorem}[Theorem \ref{thm:stability-consequence-for-1-H}, formulated explicitly]\label{thm:endgame}
  Let $M_{n}$ denote the space of complex $n \times n$ matrices, included in the space $M_{n+1}$ of $(n+1) \times (n+1)$ matrices as the upper-left block, e.g., for $n = 2$, as
  \[
    \begin{pmatrix}
      \ast & \ast & 0 \\
      \ast & \ast & 0 \\
      0 & 0 & 0
    \end{pmatrix}.
  \]
  Let $\tau$ be an element of $M_{n+1}$ with the property that no eigenvalue of $\tau$ is also an eigenvalue of the upper-left $n \times n$ submatrix $\tau_H$ of $\tau$.  Let $x \in M_{n+1}$ with $[x,\tau] = 0$, where $[a,b] := a b -b a$.  Let $z$ denote the image in $M_{n+1}$ of the identity element of $M_n$, thus $z = \diag(1,\dotsc,1,0)$ with $n$ ones.  Suppose that
  \[
    [x,[z,\tau]] = [y, \tau]
  \]
  for some $y \in M_n$.  Then $x$ is a scalar matrix.
\end{theorem}

\subsection{Organization and reading suggestions}

The paper is intended to be read linearly, but some general comments on its organization may be useful.

\S\ref{sec:gener-prel-notat} introduces general notation and conventions, in effect throughout the paper.

The paper is organized around a main local result, Theorem \ref{thm:construct-test-function}.  We encourage the interested reader to focus first on understanding the statement of that result.  In the language of \S\ref{sec:overview-proof}, it describes the essential properties of the test function $f$ and the vector $\Psi$.

The remainder of the paper divides roughly into four parts, each of which may be read independently:
\begin{itemize}
\item In \S\ref{sec:prel-proof-reduct}--\S\ref{sec:reduction-proof}, we assume our main local result and deduce our main global result, Theorem \ref{thm:main}.  These sections contain the expected global ingredients: amplification method, sums over rational points, and so on.  The remainder of the paper is devoted to the proof of the main local result.
\item \S\ref{sec:repr-theor-prel}--\S\ref{sec:analyt-test-vect} develop the harmonic analysis needed to reduce the proof of the main local result to that of bilinear forms estimates like \eqref{eqn:off-diagonal-intro}.  The techniques developed here may be more broadly useful.
\item \S\ref{sec:bilinear-forms-estimates} applies the Cauchy--Schwarz inequality and topological arguments to reduce such bilinear forms estimates to volume bounds like \eqref{eqn:volume-to-bound-intro-2}.  We show moreover that it suffices to treat the case that the group element $\gamma$ is arbitrarily close to the identity.
\item \S\ref{sec:volume-estimates} applies the implicit function theorem and related ideas to reduce the proof of the volume bounds to a problem in linear algebra.
\item \S\ref{sec:some-invar-theory} solves the linear algebra problem.
\end{itemize}

We have attempted to make each section self-contained.  We have included a short notational index at the end of the paper.

\subsection*{Acknowledgments}
We would like to thank Valentin Blomer, Simon Marshall, Philippe Michel, Abhishek Saha, Peter Sarnak, Akshay Venkatesh, Liyang Yang and Wei Zhang for their helpful feedback on an earlier draft.  We are also grateful to Liyang Yang for suggesting an improvement to Lemma \ref{prop:scratch-research:let-f-be} and to the anonymous referee for a very detailed reading and many helpful comments and corrections.

\section{Overview}\label{sec:overview-proof}
Expanding upon the proof sketch of \S\ref{sec:proof-sketch}, we now explain the main ideas of the paper in more detail, but still informally.  The reader interested only in actual proofs may proceed directly to \S\ref{sec:gener-prel-notat}.

We retain throughout this section the setup and notation of \S\ref{sec:proof-sketch-notation}.

\subsection{Amplified moments}
Our basic strategy may be phrased in a familiar way in terms of the amplification method applied to moments of families of $L$-functions.  The implementation of that strategy relies heavily upon the method developed in \cite{nelson-venkatesh-1} for analyzing families of automorphic forms using coadjoint orbits and microlocalized vectors.  We outline below how a refinement of that method applies in our setting and reduces our task to elementary problems in calculus and linear algebra.

We note the arguments in the body of this paper do not explicitly refer to families or microlocalized vectors, but instead work with them implicitly through their approximate projectors.  We hope that by phrasing this overview in such terms, it may serve as a useful guide to those arguments.

Let $(\pi_0,\sigma) \in \mathcal{F}_T$ with $T$ large.  The branching coefficients $\mathcal{L}(\pi,\sigma)$ are nonnegative (see \S\ref{sec:branch-coeff}), so we may bound $\mathcal{L}(\pi_0,\sigma)$ by the first moment $\sum_{\pi \in \Pi} \mathcal{L}(\pi,\sigma)$ taken over any family $\Pi$ containing $\pi_0$.  The Lindel{\"o}f hypothesis predicts that each $\mathcal{L}(\pi,\sigma)$ has size $\O(T^\eps)$.  We will chose $\Pi$ to have cardinality
\begin{equation}\label{eqn:Pi-approx-cardinality}
  |\Pi| \approx T^{n(n+1)/2}.
\end{equation}
In view of our ``no conductor dropping'' assumption, Lindel{\"of} on average for the moment then recovers the convexity bound for the branching coefficient of interest.  Experience suggests that if we can prove an asymptotic formula for the moment with enough room to spare, then the amplification method will deliver a subconvex bound.  Our main task is thus to prove (in a sufficiently robust way) that
\begin{equation}\label{eqn:moment-of-branching-coefs}
  |\Pi|^{-1}
  \sum _{\pi \in \Pi } \mathcal{L}(\pi,\sigma)
  = \text{(constant)}  + \O(T^{-\delta}).
\end{equation}

\begin{remark}
  If we were in the excluded case that $(G,H) = (\GL_3(\mathbb{R}), \GL_2(\mathbb{R}))$, $(\Gamma, \Gamma_H) = (\GL_3(\mathbb{Z}), \GL_2(\mathbb{Z}))$ and $\sigma$ were the $1 \boxplus 1$ Eisenstein series, then $\mathcal{L}(\pi,\sigma)$ would be essentially the fourth power of the standard $L$-function $L(\pi,1/2)$, and so the moment problem \eqref{eqn:moment-of-branching-coefs} would be of the same shape as that in \cite{MR4203038}.
\end{remark}

\subsection{Families and coadjoint orbits}\label{sec:famil-coadj-orbits}
Which family $\Pi$ should we take?  We explain our choice here, and then describe how the orbit method as advocated in \cite{nelson-venkatesh-1} provides a natural tool for its study.

The general shape of the family is perhaps unsurprising: we take those representations whose local components at $\infty$ belong to some carefully-chosen subset of the tempered dual of $G$, which we define in turn using the parameters $\lambda_{\pi,i}$ appearing in the formula \eqref{eq:gamma-factors-intro} for the local $L$-factor at $\infty$.  Roughly speaking, we take our family to consist of $\pi$ for which each coefficient of the polynomial
\[
  \prod_{i} \left(X - \frac{\lambda_{\pi,i}}{T} \right) \in \mathbb{C}[X]
\]
differs from the corresponding coefficient for $\pi_0$ by at most $\O(1/T)$.  Strictly speaking, we relax this condition by a factor of $T^\eps$, but to simplify exposition do not display such factors here.

It is natural to consider the coefficients of such polynomials.  One reason is that they are asymptotic to the eigenvalues of $\pi$ under a standard system of generators for the center of the universal enveloping algebra of
\[
  \mathfrak{g} = \Lie(G) = \glLie_{n+1}(\mathbb{R}).
\]
Using that Plancherel measure for $G$ is asymptotic to Lebesgue measure with respect to such coefficients (see for instance \cite[\S17.5]{nelson-venkatesh-1}) and assuming the Weyl law, one can check that $\Pi$ has the cardinality promised in \eqref{eqn:Pi-approx-cardinality}.  (Indeed, writing $c_j(\pi)$ for the coefficient of $X^{n+1-j}$ in $\prod_i (X - \lambda_{\pi,i})$, our family is defined by the conditions $c_j(\pi_0) = c_j(\pi) + \O(T^{j-1})$, and we have $\prod_{j=1}^{n+1} T^{j-1} = T^{n(n+1)/2}$.)

Another reason is that such coefficients are intimately related to harmonic analysis on the representation $\pi$.  The relationship is encoded by the Kirillov formula (\S\ref{sec:kirillov-formula}) for the distributional character $\chi_{\pi} : G \rightarrow \mathbb{C}$, which we now describe in the present example (simplifying a bit using special features of that example).  Recall that $\mathfrak{g}^\wedge = \Hom(\mathfrak{g}, i \mathbb{R})$ denotes the imaginary dual of $\mathfrak{g}$.  We identify $\mathfrak{g}^\wedge$ with the space $i \glLie_{n+1}(\mathbb{R})$ of imaginary matrices using the trace pairing.  A standard theorem from linear algebra is that the set of matrices
\[
  \mathcal{O}_\pi := \left\{\xi \in \mathfrak{g}^\wedge \text{ with minimal polynomial } \prod_i (X - \lambda_{\pi,i}) \right\}
\]
forms a conjugacy class (one can check that the unitarity of $\pi$ implies that $\mathcal{O}_\pi$ is nonempty).  In other words, $\mathcal{O}_\pi$ is a \emph{coadjoint orbit}, i.e., an orbit for the group $G$ under its coadjoint action on $\mathfrak{g}^\wedge$.  The Kirillov formula asserts roughly that for small enough $x \in \mathfrak{g}$,
\begin{equation}\label{eqn:kirillov-intro}
  \chi_{\pi}(\exp(x))
  \approx \int_{\xi \in \mathcal{O}_\pi}
  e^{\langle x, \xi  \rangle}
  \, d \omega(\xi),
\end{equation}
where $\omega$ denotes the normalized symplectic measure on a coadjoint orbit.  The measures $\omega$ are defined most directly using the Lie bracket on $\mathfrak{g}$, but may be characterized in the relevant examples as follows (see for instance \cite[(17.2)]{nelson-venkatesh-1} for details): one may compute an integral over $\mathfrak{g}^\wedge$ with respect to Lebesgue measure by first integrating over the space of polynomials (with respect to Lebesgue measure in each coefficient) and then integrating over each corresponding (regular) coadjoint orbit with respect to its symplectic measure $\omega$.  The Kirillov formula and the geometry of coadjoint orbits form the cornerstone of our approach to harmonic analysis on $\pi$, so it is natural that we parametrize our families in terms of associated quantities.

\begin{remark}
  If the parameters of $\pi_0$ are dyadically spaced in the sense that
  \begin{equation}\label{eqn:well-spaced-parameters}
    \lambda_{\pi_0,i} - \lambda_{\pi_0,j} \gg T
    \quad
    \text{ for } i \neq j,
  \end{equation}
  then our family $\Pi$ consists of $\pi$ for which
  \begin{equation}\label{eqn:small-family-usual}
    \lambda_{\pi,i} = \lambda_{\pi_0,i} + \O(1),
  \end{equation}
  as in \cite{MR4203038}.  On the other hand, if we suppose instead that, for instance, \eqref{eqn:well-spaced-parameters} holds for $(i,j) \neq (1,2)$ but $\lambda_{\pi_0,1} = \lambda_{\pi_0,2}$, then the condition \eqref{eqn:small-family-usual} is unchanged for $i \geq 3$ but must be modified for $i=1,2$ to the pair of conditions
  \[
    \lambda_{\pi,1} + \lambda_{\pi,2} = \lambda_{\pi_0,1} + \lambda_{\pi_0,2} + \O(1),
  \]
  \[
    \lambda_{\pi,1} - \lambda_{\pi,2} = \lambda_{\pi_0,1} - \lambda_{\pi_0,2} + \O(T^{1/2}).
  \]
  As this example already suggests, the description of our family in terms of individual parameters $\lambda_{\pi,i}$ becomes more complicated the more those parameters ``collide.''  We do not make use of any such description in this paper.
\end{remark}

\subsection{Weyl's law}\label{sec:weyls-law}
We pause to indicate one way to prove the Weyl law for the family $\Pi$ (smoothened and $T^\eps$-enlarged), i.e., for the moment as in \eqref{eqn:moment-of-branching-coefs}, but with each branching coefficient $\mathcal{L}(\pi,\sigma)$ replaced by the simpler quantity $1$.  This problem is much simpler than the moment problem of interest, but discussing it allows us to describe some techniques that apply also to the latter problem.

By the trace formula for the compact quotient $[G]$, our main task is to produce a test function $f$ on $G$, supported arbitrarily close to the identity element, for which $\trace(\pi(f))$ approximates the indicator function of $\Pi$.  Such a function $f$ may be described in terms of its pullback to the Lie algebra $\mathfrak{g}$ and then, by Fourier transform, in terms of a Schwartz function $a$ on $\mathfrak{g}^\wedge$.  By the Kirillov formula \eqref{eqn:kirillov-intro} and the Parseval relation, our task reduces to arranging that
\begin{itemize}
\item the inverse Fourier transform $a^\vee$ is supported close to the origin in $\mathfrak{g}$ (so that $f$ is supported close to the identity in $G$), and
\item the map $\pi \mapsto \int_{\xi \in \mathcal{O}_\pi} a(\xi) \, d \omega(\xi)$ approximates the indicator function of $\Pi$.
\end{itemize}
The support condition is valid provided that $a$ is essentially constant on balls of size $\O(1)$, so the main point is to achieve the required approximation property.

To that end, we choose some element $\pi$ of our family $\Pi$ (e.g., $\pi = \pi_0$, although the choice won't matter) and a regular element $\tau \in \mathfrak{g}^\wedge$ of Euclidean norm $|\tau| \asymp T$ which belongs to the corresponding coadjoint orbit $\mathcal{O}_\pi$, so that $\mathcal{O}_{\pi} = \Ad^*(G) \tau$.  For concreteness, we might construct $\tau$ using rational canonical form, e.g., for $G = \GL_3(\mathbb{R})$ by
\begin{equation}\label{eqn:tau-defn-intro}
  \tau =
  \sqrt{-1}
  \begin{pmatrix}
    0 & 0 & c_3 \\
    T & 0 &   c_2 \\
    0 & T & c_1
  \end{pmatrix}
  \quad 
  \text{ if }
  \prod_{j=1}^3
  \left(X - \frac{\lambda_{\pi,j}}{\sqrt{-1} T} \right)
  =
  X^3 - c_1 X^2 - c_2 X - c_3.
\end{equation}
Near such (regular) elements $\tau$, the map $\xi \mapsto \det(X - \xi)$ equips the space $\mathfrak{g}^\wedge$ with the structure of a fibered manifold, with fibers the coadjoint orbits (lemma \ref{lem:easy-properties-regular-elements}).  It is convenient to introduce coordinates $\xi = (\xi ', \xi '')$ on $\mathfrak{g}^\wedge$ as in \S\ref{sec:analyt-test-vect-1} (see \S\ref{sec:coord-tail-regul} for details), so that the $\xi '$-axis consists of all directions tangent to $\mathcal{O}_{\pi}$ at $\tau$ while the $\xi ''$-axis consists of those transverse to the fibers.  We then take
\begin{equation}\label{eqn:a-smooth-bump-on}
  a = \text{ smooth bump on }
  \left\{ \tau + \xi :
    |\xi '| \ll T^{1/2+\eps },
    |\xi ''| \ll T^\eps 
  \right\}.
\end{equation}
As depicted in Figure \ref{fig:tau-coordinates-intro-0} (see Lemma \ref{lem:parabola-y-x-squared} for some justification of this depiction), the support of $a$ is thus a thin ``coin-shaped'' neighborhood of $\tau$, concentrated near the coadjoint orbit $\mathcal{O}_{\pi}$.  One checks that the coadjoint orbits intersecting this neighborhood are those arising from our family $\Pi$, while the symplectic measures of the intersections are approximately one (see, e.g., the proof of lemma \ref{lem:pass-to-individual-u}).  The required approximation property follows.

The number of $\xi '$ (resp. $\xi ''$) directions is $n(n+1)$ (resp. $n+1$), so the total volume of the support of $a$ is approximately the promised cardinality $T^{n(n+1)/2}$ of the family $\Pi$ (as we could have predicted using the characterization of symplectic measures noted in \S\ref{sec:famil-coadj-orbits}).  In particular,
\begin{equation}\label{eqn:f-size-bound-intro}
  \|f\|_{L^\infty(G)}
  \approx f(1) \approx T^{n(n+1)/2}
\end{equation}

We note that the strategy suggested above is flexible with respect to the choice of $\tau$: the same argument works if we replace $\tau$ by any element of $\mathcal{O}_{\pi}$ with similar properties, e.g., by $\Ad^*(g) \tau$ for any group element $g \in G$ close to the identity.

\subsection{Microlocal calculus}\label{sec:microlocal-calculus}
A rigorous implementation of the strategy suggested in \S\ref{sec:weyls-law} might require working with positive-definite test functions, e.g., those obtained by convolving a function $f$ as above against its adjoint.  In \S\ref{sec:star-prod-extens}, we study in detail the asymptotics of the convolutions of two functions $f_1$ and $f_2$ as above in terms of the corresponding Fourier transforms $a_1$ and $a_2$ on $\mathfrak{g}^\wedge$ (i.e., both $a_1$ and $a_2$ are smooth bumps as in \eqref{eqn:a-smooth-bump-on} and Figure \ref{fig:tau-coordinates-intro-0}).  We provide a general calculus for doing so, valid for any connected reductive group $G$.  We hope this calculus will be more broadly useful.  It is rooted in a basic estimate to the effect that the convolution $f_1 \ast f_2$ corresponds to a certain star product $a_1 \star a_2$ admitting an asymptotic expansion
\begin{equation}\label{eqn:star-product-expansion-intro}
  a_1 \star a_2
  \sim
  a_1 a_2
  + a_1 \star^1 a_2 + a_1 \star^2 a_2 + \dotsb,
\end{equation}
with $\star^1$ a multiple of the Poisson bracket and $\star^j$ defined in general by a certain bidifferential operator.  The derivative bounds enjoyed by bumps as in \eqref{eqn:a-smooth-bump-on} turn out to force each successive term on the RHS of \eqref{eqn:star-product-expansion-intro} to be smaller than the previous terms, so that the expansion is not merely formal.  For example, using that the Poisson bracket is defined by vector fields tangent to the coadjoint orbits, we may check that $a_1 \star^1 a_2$ is of size $T^{-2 \eps}$.  Our results concerning this calculus may be understood as further steps in a direction suggested by Rieffel \cite{MR1064995}.

A first version of that calculus, inspired by the 
pseudodifferential calculus, was given in \cite{nelson-venkatesh-1}, but under hypotheses that would be prohibitively restrictive for our aims.  Indeed, that calculus applies to smooth bumps $a$ on regions roughly of the form $\left\{ \tau + \xi : |\xi| \ll T^{1/2+\eps } \right\}$, but not on somewhat rougher regions as in \eqref{eqn:a-smooth-bump-on}.  Such bumps are inadequate for the analysis of ``short'' families $\Pi$ required here.  For instance, in the ``dyadically spaced'' case, they would suffice for studying families defined by conditions like \eqref{eqn:small-family-usual}, but with $\O(1)$ replaced by $\O(T^{1/2})$.

We indicate why bumps as in \eqref{eqn:a-smooth-bump-on} represent a natural limit for our calculus:
\begin{enumerate}[(i)]
\item The coadjoint orbit $\mathcal{O}_\pi$ plays the role of the phase space for the representation $\pi$.  Balls of the form $\mathcal{O}_\pi \cap ( \tau + \O(T^{1/2}))$ have symplectic volume $\asymp 1$, corresponding to the Planck scale.  By the uncertainty principle, we cannot hope for a meaningful calculus involving smooth bumps on much smaller balls.  The resolution of our calculus is thus optimal (up to epsilons) in the $\xi '$-directions.
\item To understand the role of the $\xi ''$-directions, suppose for instance that $\pi$ belongs to the discrete series and $\tau$ is regular elliptic.  The representations $\pi '$ whose coadjoint orbits $\mathcal{O}_{\pi '}$ intersect the support of \eqref{eqn:a-smooth-bump-on} are then discrete series representations whose parameters differ from those of $\pi$ by $\O(T^\eps)$.  The number of such representations is $\O(T^\eps)$ (for a different value of $\eps$).  The indicated $\xi ''$-scale is thus the relevant one for projecting onto rather short families of representations.
\end{enumerate}

\subsection{Microlocalized vectors}\label{sec:micr-vect}
The function $a$ as in \eqref{eqn:a-smooth-bump-on}, being approximately the characteristic function of a set, satisfies $a^2 \approx a$.  From our calculus, it follows that the corresponding test function $f$ on $G$ satisfies $f \ast f \approx f$, and so the corresponding operator $\pi(f)$ is approximately idempotent.  The real-valuedness of $a$ forces $\pi(f)$ to be self-adjoint, while the noted symplectic measure calculation shows that $\trace(\pi(f)) \approx 1$.  Speaking informally, $\pi(f)$ is thus approximately a rank one orthogonal projection, and so corresponds approximately to some scaling class of unit vectors $v \in \pi$ for which $\pi(f) v \approx v$.

The relationship between $v$ and $\tau$ may be quantified more directly in terms of the action of Lie algebra elements $x \in \mathfrak{g}$ with $x = \O(T^{-1/2-\eps})$, under which $v$ enjoys the approximate equivariance property
\[
  \exp(x) v \approx e^{\langle x, \tau \rangle} v.
\]
This estimate persists for small $x$ lying within $\O(T^{-1/2-\eps})$ of the centralizer of $\tau$.  The test function $f$ should in turn be regarded as a proxy for the matrix coefficient $\langle g v, v \rangle$ (conjugated, normalized and truncated).

We say that such a vector $v$ is \emph{microlocalized} at the parameter $\tau \in \mathfrak{g}^\wedge$.  A key property of such vectors is that their matrix coefficients $\langle g v, v \rangle$, or equivalently, the corresponding test functions $f(g)$, concentrate near the centralizer of the parameter $\tau$.  Quantitatively, $f$ is essentially supported on the set \eqref{eqn:support-f-intro} consisting of small group elements that approximately centralize $\tau$.

One may understand our calculus and the associated notion of microlocalized vectors as giving an analytic archimedean analogue of the theory of types for $p$-adic groups.  The latter gives, among other things, many examples of open subgroups $J$ of $\GL_n(\mathbb{Z}_p)$ and one-dimensional representations $\chi$ of $J$ with the property that, up to twisting, there is at most one supercuspidal representation of $\GL_n(\mathbb{Q}_p)$ whose restriction to $J$ contains $\chi$; moreover, $\chi$ occurs in that restriction with multiplicity one.  In other words, the function $h$ on $\GL_n(\mathbb{Z}_p)$ given by $\chi^{-1}$ times the normalized characteristic function of $J$ has the property that, for each irreducible representation $\rho$ of $\GL_n(\mathbb{Q}_p)$, the operator $\rho(h)$ vanishes unless $\rho$ is a twist of the given supercuspidal, in which case $\rho(h)$ is a rank one projector with range spanned by some $(J,\chi)$-isotypic vector $u \in \rho$.  By comparison, we have noted in \S\ref{sec:microlocal-calculus} that our calculus produces functions $f$ supported near the identity element of the real group $G$ for which $\pi(f)$ is negligible unless $\pi$ belongs to $\O(T^\eps)$-many discrete series representations, in which case it is approximately a projector of rank $\O(T^{\eps})$.  Thus $f$ and $h$ are analogous.  The (informally and ambiguously defined) microlocalized vector $v \in \pi$ is analogous to the isotypic vector $u \in \rho$.

\subsection{Period formulas and
  matrix coefficient integral asymptotics}\label{sec:period-formulas-1}
Recall that our aim is to understand the moment \eqref{eqn:moment-of-branching-coefs} of branching coefficients $\mathcal{L}(\pi,\sigma)$ taken over $\pi$ in some family $\Pi$.  We have thus far discussed only the family, emphasizing the relevant local harmonic analysis on $G$.  We turn now to the branching coefficients.  They arise from a formula of the following shape: for automorphic forms $\varphi \in \pi$ and $\Psi \in \sigma$,
\begin{equation}\label{eqn:period-formula-intro}
  \left\lvert
    \int_{[H]}
    \varphi \bar{\Psi}
  \right\rvert^2
  \approx \mathcal{L}(\pi,\sigma)
  \mathcal{Q}(\varphi \otimes \Psi),
  \quad
  \mathcal{Q}(\varphi \otimes \Psi) :=
  \int_{h \in H}
  \langle h \varphi, \varphi  \rangle
  \langle \Psi, h \Psi  \rangle
  \, d h,
\end{equation}
where $\approx$ means ``up to unimportant inaccuracies'' (leading constants, local factors at ``auxiliary'' places, etc).  This formula amounts to the definition of $\mathcal{L}(\pi,\sigma)$ (see \S\ref{sec:branch-coeff}), and is all that enters into the proof of Theorem \ref{thm:main}; the fact that $\mathcal{L}(\pi,\sigma)$ is conjectured (and in many cases known) to be given by a special value of an $L$-function is relevant only for interpreting our results, as in Corollary \ref{cor:main}.

Bernstein--Reznikov \cite{MR2726097} initiated the systematic use of period formulas like \eqref{eqn:period-formula-intro} as a tool for estimating branching coefficients, emphasising spectral aspects.  Venkatesh \cite{venkatesh-2005} introduced related ideas in the level aspect, as well as the use of arithmetic amplification in the spirit of Duke--Friedlander--Iwaniec (see \cite{MR1923476} and references).  We mention also the influential works of Sarnak \cite{MR780071} and Iwaniec--Sarnak \cite{iwan-sar}.  A capstone application of this tool was the resolution of the subconvexity problem for $\GL_2$ \cite{michel-2009}.  The art in each such application lies in choosing vectors $\varphi$ and $\Psi$ for which one may prove that
\begin{itemize}
\item $\mathcal{Q}(\varphi \otimes \Psi)$ is not too small, and
\item $\int_{[H]} \varphi \bar{\Psi}$ is not too large.
\end{itemize}

The works mentioned above considered low rank examples.  The paper \cite{nelson-venkatesh-1} studied the higher rank setting and described
\begin{itemize}
\item an approach based on the orbit method for choosing vectors $\varphi$ to $\Psi$ to which the above strategy may be profitably applied (see \cite[\S1.10]{nelson-venkatesh-1}), and
\item a technique for asymptotically evaluating the local integrals $\mathcal{Q}(\varphi \otimes \Psi)$ (see \cite[\S1.9, \S19]{nelson-venkatesh-1}).
\end{itemize}
We briefly recall the main points.  One takes for $\varphi$ and $\Psi$ a pair of microlocalized unit vectors with corresponding parameters $\tau \in \mathcal{O}_\pi \subseteq \mathfrak{g}^\wedge$ and $\nu \in \mathcal{O}_\sigma \subseteq \mathfrak{h}^\wedge$, say.  If the restriction $\tau_H$ of $\tau$ to $\mathfrak{h}^\wedge$ is not close to $\nu$, then the approximate equivariance forces $\mathcal{Q}(\varphi \otimes \Psi)$ to be small.  The crucial case to understand is thus when $\nu = \tau_H$.  The ``no conductor dropping'' assumption now becomes relevant: it manifests at the level of the coadjoint orbits $\mathcal{O}_\pi \subseteq \mathfrak{g}^\wedge$ and $\mathcal{O}_\sigma \subseteq \mathfrak{h}^\wedge$ in that the ``relative coadjoint orbit''
\begin{equation}\label{eqn:O-pi-sigma-defn-intro}
  \mathcal{O}_{\pi,\sigma} := \left\{ \xi \in \mathcal{O}_{\pi} :
    \xi_H \in \mathcal{O}_\sigma \right\}
\end{equation}
is an $H$-torsor, i.e., transitive and free under the coadjoint action of $H$.  This says both that any two elements $\tau$ are equivalent under $H$ (so that little is lost in choosing a specific $\tau$) and that the centralizer in $H$ of $\tau$ is trivial.  The matrix coefficient concentration property noted in \S\ref{sec:micr-vect} then forces the integrand in the definition of $\mathcal{Q}(\varphi \otimes \Psi)$ to be concentrated roughly in the neighborhood $1 + \O(T^{-1/2})$ of the identity element of $H$, where it may be estimated using the Kirillov formula.  In particular, the integral has size $\approx (T^{-1/2})^{\dim(H)} = T^{-n^2/2}$.  We have oversimplified in this discussion; in practice, it becomes necessary to average over short (i.e., size $\O(T^\eps)$) famillies of microlocalized vectors.

In this paper, we choose vectors as in \cite{nelson-venkatesh-1} and appeal to the matrix coefficient integral asymptotics given in \cite[\S19]{nelson-venkatesh-1} as a black box (see \S\ref{sec:relat-char-asympt}).  The paper \cite{nelson-venkatesh-1} contains other ideas (e.g., related to ergodic theory) which do not enter here.

\subsection{Relative trace formula}\label{sec:relat-trace-form}
From the above discussion, we may construct
\begin{itemize}
\item a coadjoint element $\tau \in \mathfrak{g}^\wedge$ (the same choice of which works for all $\pi \in \Pi$, since the coadjoint orbits $\mathcal{O}_\pi$ are close to one another),
\item a unit vector $\Psi \in \sigma$, microlocalized at $\tau_H \in \mathfrak{h}^\wedge$, and
\item for each $\pi$ in our family $\Pi$, a unit vector $\varphi_\pi \in \pi$, microlocalized at $\tau$,
\end{itemize}
so that the integral representation
\begin{equation}\label{eqn:integral-rep-intro-branch-coef}
  \mathcal{L}(\pi,\sigma)
  \approx
  T^{n^2/2} \left\lvert \int_{[H]}
    \varphi_\pi \bar{\Psi} \right\rvert^2
\end{equation}
holds.  Since $\varphi_\pi$ is microlocalized at the parameter $\tau$, it satisfies $\pi(f) \varphi \approx \varphi$ for the test function $f$ described in \S\ref{sec:weyls-law}.

We have noted that the operator $\pi(f)$ is approximately a rank one idempotent, so that for an orthonormal basis $\mathcal{B}(\pi)$, we have $\sum_{\varphi \in \mathcal{B}(\pi)} \pi(f) \varphi(x) \overline{\varphi(y)} \approx \varphi_\pi(x) \overline{\varphi_\pi(y)}$.  The pretrace formula for $[G]$ applied to $f$ thus gives
\[
  \sum _{\pi \in \Pi } \varphi_\pi(x) \overline{\varphi_\pi(y)} \approx \sum _{\gamma \in \Gamma} f(x^{-1} \gamma y).
\]
Integrating both sides against $\Psi$, appealing to the integral representation \eqref{eqn:integral-rep-intro-branch-coef} and recalling the size \eqref{eqn:Pi-approx-cardinality} of $\Pi$, we obtain the basic formula
\begin{equation}\label{eqn:basic-pretrace-intro}
  |\Pi|^{-1}
  \sum _{\pi \in \Pi }
  \mathcal{L}(\pi,\sigma)
  \approx
  T^{-n/2}
  \int_{x,y \in [H]}
  \sum _{\gamma \in \Gamma }
  \bar{\Psi}(x)
  \Psi(y)
  f(x^{-1} \gamma y) \, d x \, d y.
\end{equation}
A comparison of \eqref{eqn:moment-of-branching-coefs} and \eqref{eqn:basic-pretrace-intro} now explains the sufficiency of \eqref{eqn:basic-pretrace-intro-0}.

We emphasize that in the body of the paper, we work most directly with the test function $f$ and the operators $\pi(f)$ rather than with the corresponding microlocalized vectors $\varphi_\pi$ (see for instance \S\ref{sec:constr-prop-test} and \S\ref{sec:analyt-test-vect}).  We hope that by having phrased this overview informally in terms of $\varphi_\pi$, it may usefully illustrate the rigorous arguments given in the body in terms of $f$.

\begin{remark}
  The present setup differs from that in \cite{nelson-venkatesh-1}, where an average of $\mathcal{L}(\pi,\sigma)$ for fixed $\pi$ over a large family of $\sigma$ is studied by spectrally decomposing the $L^2$-norm $\int_{[H]} |\varphi_\pi|^2$, averaging over a large family of microlocalized vectors $\varphi_\pi$, and appealing to Ratner theory.
\end{remark}

\begin{remark}\label{rmk:compare-iwaniec-sarnak}
  When $n=1$, the present setup is closely related to that of Iwaniec--Sarnak \cite{iwan-sar}.

  To explain, it is convenient to switch to the pair $(G,H) = (\U(1,1),\U(1))$, to which the above discussion applies without essential modification.  Modding out by the center of $G$ and ignoring issues of isogeny, we might as well pretend then that $(G,H) = (\SL_2(\mathbb{R}), \SO(2))$.

  Suppose that $\sigma$ is trivial, so that $\Psi$ is constant, and that $\pi$ belongs to the principal series, so that it contains an $H$-invariant unit vector $\varphi_{\pi,0}$.  Let us identify $G/H$ with the upper half-plane $\mathbb{H}$, so that $\varphi_{\pi,0}$ corresponds to some $L^2$-normalized Maass form on $\Gamma \backslash \mathbb{H}$.

  The integral $\int_{[H]} \varphi_\pi \bar{\Psi}$ is then a multiple of $\varphi_{\pi,0}(z)$, with $z \in \mathbb{H}$ the point stabilized by $H$.  For suitable $\pi_0$, our family $\Pi$ consists of $\pi$ with $r_{\pi} = T+ \O(1)$, with $1/4 + r_\pi^2$ the eigenvalue of $\varphi_{\pi,0}$.

  The LHS of \eqref{eqn:basic-pretrace-intro} is thus a multiple of $\sum_{r_\pi = T + \O(1)} |\varphi_{\pi,0}(z)|^2$.  These sums and their amplified variants were considered by Iwaniec--Sarnak \cite{iwan-sar} in their pioneering work on the sup norm problem.  Since $\Psi$ is constant, the RHS of \eqref{eqn:basic-pretrace-intro} is unchanged by replacing $f$ with its bi-$H$-invariant average, in which case it is essentially the same kernel function as considered by Iwaniec--Sarnak.

  Thus the general method described here specializes to that of Iwaniec--Sarnak (although, as noted in Remark \ref{rmk:unoptimized}, we have optimized our arguments to a much lesser extent).
\end{remark}

\begin{remark}
  It is instructive to consider also the closely related example $(G,H) = (\PGL_2(\mathbb{R}), \GL_1(\mathbb{R}))$, with the embedding $H \hookrightarrow G$ given by $y \mapsto \diag(y,1)$.  Suppose again that $\sigma$ is trivial.  The families $\Pi$ that we consider are then as in Remark \ref{rmk:compare-iwaniec-sarnak}.  The microlocalized vectors $\varphi_\pi$ are most conveniently described in the Kirillov model, where they are given by smooth bumps on the region $T + \O(T^{1/2})$ in $\mathbb{R}^\times$.  If we were in the (excluded, non-compact) case $(\Gamma, \Gamma_H) = (\PGL_2(\mathbb{Z}), \GL_1(\mathbb{Z}))$, then the sums $\sum_{\pi \in \Pi} \mathcal{L}(\pi,\sigma)$ that we consider would be the second moments of standard $L$-functions $\sum_{r_\pi = T + \O(1)} \left\lvert L(\pi, \tfrac{1}{2}) \right\rvert^2$ estimated first by Iwaniec \cite{Iwaniec1992}.  The method described here is similar to that of Iwaniec in that it also consists of bounding such amplified moments, but differs in that it directly averages and estimates a carefully chosen integral representation, while Iwaniec instead proceeds via the approximate functional equation, Kuznetsov formula and a couple applications of Poisson summation.
\end{remark}

\subsection{Local vs. convexity bounds}\label{sec:local-bound}
The method developed here may be understood as first giving, under our running ``no conductor dropping'' assumption, a ``local estimate'' for the branching coefficients $\mathcal{L}(\pi,\sigma)$ that recovers the convexity bound, and then improving that estimate to a subconvex bound via arithmetic amplification.  The first of these steps corresponds roughly to parts (i) and (ii) of Theorem \ref{thm:construct-test-function}, the second to part (iii) of the same theorem.  The branching coefficients $\mathcal{L}(\pi,\sigma)$ studied in this paper may be defined more generally than we have here (e.g., without assuming that our quotients are arithmetic or that our representations are Hecke eigenspaces), and our main local result (Theorem \ref{thm:construct-test-function}) gives the same local estimates in such generality.  Our local estimates are deduced from the integral representations involving microlocalized vectors given in \cite[\S19]{nelson-venkatesh-1}, and it is natural to ask whether there is a more direct approach.  It is also natural to ask whether the resulting estimates are sharp.  Here we address these two questions when the involved groups are compact.  We will see that many of our local estimates may be derived more directly, in sharp form, by considering multiplicities, much like in the first pages of \cite{sarnak-morawetz}.

Our observations apply in a more general setting, so let us temporarily set aside our running notation and assumptions and take now for $G$ any compact topological group and $H$ any closed subgroup, each equipped with the probability Haar measure.  Let $\pi$ and $\sigma$ be irreducible unitary representations of $G$ and $H$, respectively, satisfying the multiplicity one property
\begin{equation*}
  \dim \Hom_H(\pi,\sigma) = 1.
\end{equation*}
We may then identify $\sigma$ with a subrepresentation of $\pi$.  Let $\pr : \pi \rightarrow \sigma$ denote the orthogonal projection.

Relative to our prior discussion, we implicitly take $\Gamma$ and $\Gamma_H$ to be trivial subgroups in what follows.

Let $\iota_{\pi} : \pi \rightarrow L^2(G)$ and $\iota_{\sigma} : \sigma \rightarrow L^2(H)$ be equivariant linear isometries.  Inspired by \eqref{eqn:period-formula-intro}, we observe that there is a unique nonnegative real number $\mathcal{L}(\iota_{\pi}, \iota_{\sigma}) \geq 0$ such that for all $v \in \pi$ and $u \in \sigma$,
\begin{equation}\label{eqn:define-cal-L}
  \left\lvert
    \int _{h \in H}
    \iota_{\pi}(v)(h)
    \overline{\iota_{\sigma}(u)(h)} \, d h
  \right\rvert^2
  =
  \mathcal{L}(\iota_{\pi}, \iota_{\sigma})
  \int _{h \in H}
  \langle h v, v  \rangle
  \langle u,  h u  \rangle
  \, d h.
\end{equation}
Indeed, each side is a hermitian form on $\pi \otimes \sigma^\vee$ whose bilinearization is $H \times H$-invariant.  The multiplicity one assumption implies that the space of such forms is one-dimensional, so it suffices to show that the integral on the right hand side of \eqref{eqn:define-cal-L} is positive for some $u$ and $v$.  To see this, we appeal to the following identity, a consequence of the Schur orthogonality relations:
\begin{equation}\label{eqn:matrix-coeff-integral-vs-projections}
  \int _{h \in H} \langle h v, v \rangle \langle u, h u \rangle \, d h
  =
  \frac{1}{\dim \sigma }
  \left\lvert \left\langle \pr (v ), u  \right\rangle \right\rvert^2.
\end{equation}

\begin{proposition}[Local bound for $\mathcal{L}(\iota_{\pi}, \iota_{\sigma})$] \label{thm:local-bound-compact-case} We have $\mathcal{L}(\iota_{\pi}, \iota_{\sigma}) \leq \dim \pi$.  Equality holds for some $\iota_{\pi}$ and $\iota_{\sigma}$.
\end{proposition}
\begin{proof}
  Since $\pi$ and $\sigma$ are irreducible, the embeddings $\iota_{\pi}$ and $\iota_{\sigma}$ are given by taking inner products against some vectors $v_0 \in \pi$ and $u_0 \in \sigma$, i.e.,
  \begin{equation*}
    \iota_{\pi}(v)(g) = \langle g v, v_0 \rangle,
    \quad 
    \iota_{\sigma}(u)(h) = \langle h u, u_0 \rangle.
  \end{equation*}
  Our assumption that these embeddings are isometric translates, via the Schur orthogonality relations, to the normalizations
  \begin{equation}\label{eqn:v0-u0-norms}
    \|v_0\|^2 = \dim \pi,
    \quad
    \|u_0\|^2 = \dim \sigma.
  \end{equation}
  By evaluating the defining identity \eqref{eqn:define-cal-L} at $(v,u) = (v_0,u_0)$, we see that
  \begin{equation}\label{eqn:formula-for-cal-L}
    \mathcal{L}(\iota_{\pi}, \iota_{\sigma}) = \int _{h \in H} \langle h v_0, v_0 \rangle \langle u_0, h u_0 \rangle \, d h,
  \end{equation}
  provided the right hand side is nonzero.  Assuming for the moment that this is the case, we combine \eqref{eqn:formula-for-cal-L} with \eqref{eqn:matrix-coeff-integral-vs-projections} and clear denominators to obtain
  \begin{equation*}
    \mathcal{L} (\iota_{\pi}, \iota_{\sigma})
    =
    (\dim \pi )
    \frac{\left\lvert \left\langle \pr (v _0 ), u _0  \right\rangle \right\rvert^2}{\|v_0\| ^2 \|u _0 \| ^2 }.
  \end{equation*}
  Both sides of this identity vary continuously as the pair $(v_0,u_0)$ varies over the product of spheres defined by \eqref{eqn:v0-u0-norms}, so we may extend it continuously from the dense subset on which $\langle \pr (v_0), u_0 \rangle \neq 0$ to the full space of such pairs.  (We have used that in any sphere, the complement of a hyperplane is dense.)  Thus this identity holds in general, without our earlier assumption.  By Cauchy--Schwarz, its right hand side is at most $\dim \pi$, with equality attained precisely when $v_0$ lies in $\sigma \subseteq \pi$ and $u_0$ is a multiple of $v_0$.
\end{proof}

\begin{example}
  Suppose that $\sigma$ is the trivial representation $\mathbb{C}$.  Our assumptions then say that there is an essentially unique way to realize $\pi$ inside $L^2(H \backslash G)$, while our conclusion says that the largest $L^\infty$-norm of a unit vector in $\pi$ is $\sqrt{\dim \pi}$.  This case of our analysis is well-known in the sup norm literature (see \cite{sarnak-morawetz}).
\end{example}

\begin{remark}
  Since $\Hom_H(\pi,\sigma) \cong \Hom_H(\pi \otimes \sigma^\vee, \mathbb{C})$, our hypotheses apply also to the pair of representations $(\pi \otimes \sigma^\vee, \mathbb{C})$ of the product group $G \times H$ and the diagonal copy of $H$.  The local bound $\dim \pi$ for $(\pi, \sigma)$ is typically smaller than the local bound $\dim \pi \dim \sigma$ for $(\pi \otimes \sigma^\vee, \mathbb{C})$.  This is not too surprising: there are typically more ways to embed $\pi \otimes \sigma^\vee$ in $L^2(G \times H)$ than just the ``rank one'' products of embeddings of $\pi$ and $\sigma$.
\end{remark}

\begin{example}\label{examp:compact-unitary}
  Consider the ``GGP'' case $(G,H) = (\U(n+1), \U(n))$.  Let $\lambda$ (resp. $\mu$) denote the highest weight of $\pi$ (resp. $\sigma$).  Multiplicity one is equivalent to the interlacing condition
  \begin{equation*}
    \lambda_1 \geq \mu_1 \geq \lambda_2 \geq \mu_2 \geq \dotsb \geq \mu_n \geq \lambda_{n+1}.
  \end{equation*}
  The analytic conductor for $\mathcal{L}$ (i.e., for the $L$-function to which $\mathcal{L}$ corresponds in global arithmetic settings) is
  \begin{equation*}
    \mathcal{C} = \prod_{i=1}^{n+1} \prod_{j=1}^n (1 + |\lambda_i - \mu_j|)^2.
  \end{equation*}
  By the Weyl dimension formula, we have
  \begin{equation*}
    \dim \pi \asymp \prod_{i < j} (1 + |\lambda_i - \lambda_j|).
  \end{equation*}
  The interlacing condition implies that
  \begin{equation*}
    |\lambda_i - \mu_j| \leq
    \begin{cases}
      |\lambda_i - \lambda_{j+1}|   &  \text{ if } i \leq j, \\
      |\lambda_i - \lambda_j| & \text{ if } i > j.
    \end{cases}
  \end{equation*}
  It follows that
  \begin{equation*}
    \mathcal{C}^{1/4} \ll \dim \pi.
  \end{equation*}
  By Proposition \ref{thm:local-bound-compact-case}, we see that the convexity bound is always at least as strong as the local bound.

  In general, we have $\mathcal{C}^{1/4} \asymp \dim \pi$ if and only if
  \begin{equation}\label{eqn:no-cond-drop}
    1 + |\lambda_i - \mu_j| \asymp
    \begin{cases}
      1 + |\lambda_i - \lambda_{j+1}| & \text{ if } i \leq j, \\
      1 + |\lambda_i - \lambda_{j}| & \text{ if } i > j.
    \end{cases}
  \end{equation}
  Thus the local and convexity bounds are equivalent precisely when ``the conductor for $\pi \boxtimes \sigma^\vee$ does not drop'' in the precise sense afforded by \eqref{eqn:no-cond-drop}.

  Suppose now for simplicity that each $\lambda_i - \lambda_{i+1} \asymp T$.  Then $\mathcal{C}^{1/4} \asymp \dim \pi$ if and only if each $\lambda_i - \mu_j \asymp T$.  Some of the arguments of this paper (e.g., those leading to the first two parts of Theorem \ref{thm:construct-test-function}) may be understood as giving a different proof, using microlocalized vectors, of the bound $\mathcal{L}(\iota_{\pi}, \iota_{\sigma}) \ll T^{n(n+1)/2} \asymp \dim \pi$.  We see from Proposition \ref{thm:local-bound-compact-case} that this bound is sharp.

  In many ``conductor dropping'' cases, the local bound is worse than convexity.  For instance, if
  \begin{equation*}
    n = 2,
    \quad
    \lambda = (T, 0, -T), \quad
    \mu = (0,0),
  \end{equation*}
  then
  \begin{equation*}
    \dim \pi \asymp T^3,
    \quad
    \mathcal{C}^{1/4} \asymp T^2.
  \end{equation*}
\end{example}

\subsection{Reduction to volume bounds:
  remarks}\label{sec:volume-bounds-remarks}
We record here some miscellaneous remarks concerning the reduction to \eqref{eqn:off-diagonal-intro} and \eqref{eqn:volume-to-bound-intro-2} described in \S\ref{sec:proof-sketch}.

\begin{remark}
  The exploitation of invariance in the central direction is reminiscent of the conductor lowering trick applied by Munshi and others to the subconvexity problem on $\GL_3$ (see \cite{MR3369905,2018arXiv181000539M,2020arXiv201010153S,2019arXiv191209473L}).  It might be interesting to understand any relation more precisely.
\end{remark}

\begin{remark}
  \label{rmk:alternative-reduction-intro}
  There is another way to view the reduction to \eqref{eqn:volume-to-bound-intro} or \eqref{eqn:volume-to-bound-intro-2}.  Recall the integral representation \eqref{eqn:integral-rep-intro-branch-coef}.  In view of the $H_{\tau_H}$-equivariance of $\Psi$, it is natural to replace the microlocalized vector $\varphi_\pi$ with a modified vector $\varphi_{\pi,0}$, given by a weighted average of $\varphi_\pi$ under small elements of $H_{\tau_H}$, for which $\int _{[H]} \varphi_{\pi,0} \bar{\Psi} \approx \int _{[H]} \varphi_{\pi} \bar{\Psi}$.  We regard $\varphi_{\pi,0}$ as the more natural vector for the problem.  For example, in the $n=1$ case discussed in Remark \ref{rmk:compare-iwaniec-sarnak}, the modified vector $\varphi_{\pi,0}$ is a multiple of the normalized Maass form, which is clearly the natural vector for that case.  The microlocalized vectors $\varphi_{\pi}$ may be regarded as footholds into our understanding of the harmonic analytic difficulties and matrix coefficient integral asymptotics in higher rank.
  
  We apply the pretrace formula as before, but with $\varphi_\pi$ replaced by $\varphi_{\pi,0}$.  We reduce to estimating for $\gamma \in \Gamma - H Z$ the integrals
  \[
    \int_{x,y \in \mathcal{H}} \bar{\Psi}(x) \Psi(y) f_0(x^{-1} \gamma y) \, d x \, d y,
  \]
  where $f_0$ is given by a weighted average of $f$ on the left and right under small elements of $H_{\tau_H}$.  We may understand $f_0$ as a proxy for the matrix coefficient of $\varphi_{\pi,0}$.  Each factor in the integrand now transforms approximately under small elements of $H_{\tau_H}$ by a unitary character, so taking absolute values of the integrand leads to an integral that morally takes place on $(H/H_{\tau_H})^2$.

  Recall now from Perron--Frobenius theory that the spectral radius of a matrix with nonnegative entries is bounded by the largest row sum, and that for a symmetric matrix, the spectral radius bounds the operator norm.  If we formally apply this fact to the above integral with absolute value signs inserted (more precisely, a finite sum of such integrals having the required symmetry), then we are led to consider
  \[
    \max_{y \in \mathcal{H}} \int_{x \in \mathcal{H}} |f_0(x^{-1} \gamma y)| \, d x \lessapprox \max_{y \in \mathcal{H}} \int_{x \in \mathcal{H}} \int _{\substack{
        z \in H_{\tau_H} : \\
        z = \O(1) }} |f(x^{-1} \gamma y z)| \, d z \, d x.
  \]

  Nontrivial bounds for this last expression are essentially equivalent to \eqref{eqn:volume-to-bound-intro}, but this approach to the reduction may render it less mysterious: we are ``just'' estimating the matrix coefficients of the natural vector $\varphi_{\pi,0}$ and applying Perron--Frobenius.
\end{remark}

\subsection{Coadjoint reformulation: further discussion
}\label{sec:coadj-reform}
We explain here in more detail the content of the reformulation \eqref{eqn:coadjoint-reformulate-volume-intro} of the required volume bounds in terms of coadjoint orbits.

We may similarly reformulate \eqref{eqn:volume-to-bound-intro} in terms of the coadjoint action; it translates to an estimate like \eqref{eqn:coadjoint-reformulate-volume-intro}, but with $Z_H$ replaced by $H_{\tau_H}$.  We have noted already in \S\ref{sec:period-formulas-1} that the orbit $\Ad^*(H) \tau$ is the relative coadjoint orbit $\mathcal{O}_{\pi,\sigma}$ defined in \eqref{eqn:O-pi-sigma-defn-intro}.  It follows from the $H$-torsor property of $\mathcal{O}_{\pi,\sigma}$ that $\Ad^*(H_{\tau_H}) \tau$ is the fiber over $\tau_H$ in $\mathcal{O}_{\pi,\sigma}$:
\[
  \Ad^*(H_{\tau_H}) \tau = \mathcal{O}_{\pi,\sigma}(\tau_H) := \left\{ \xi \in \mathcal{O}_{\pi,\sigma} : \xi_H = \tau_H \right\}.
\]
The estimate \eqref{eqn:volume-to-bound-intro} thus boils down to exhibiting some approximate transversality between the varieties
\begin{equation}\label{eqn:transversality-orbits-intro}
  \Ad^*(\gamma) \mathcal{O}_{\pi,\sigma}(\tau_H)
  \quad
  \text{and}
  \quad
  \mathcal{O}_{\pi,\sigma}.
\end{equation}
These varieties have respective dimensions $n$ and $n^2$.  Both are contained in the $(n^2 + n)$-dimensional $G$-orbit $\mathcal{O}_\pi$.  It thus seems reasonable to expect that for ``generic'' $\gamma$ (in particular, $\gamma \notin H Z$), they should be literally transverse.  The point of \eqref{eqn:coadjoint-reformulate-volume-intro} is to exhibit an approximate form of such transversality using only the central direction in $H_{\tau_H}$, i.e., replacing the $n$-dimensional variety $\mathcal{O}_{\pi,\sigma}(\tau_H)$ with its one-dimensional subvariety $\Ad^*(Z_H) \tau$.

We can visualize \eqref{eqn:transversality-orbits-intro} directly only in the special case $n=1$.  That case is overly simplistic (e.g., because $Z_H = H_{\tau_H} = H$ and $\mathcal{O}_{\pi,\sigma}(\tau_H) = \mathcal{O}_{\pi,\sigma}$), but may nevertheless convey the flavor of the problem.  The visualization becomes slightly simpler to describe if we switch to the setting of compact unitary groups $(G,H) = (\U(2), \U(1))$.  Suppose then that $\pi$ is an irreducible representation of $\U(2)$ with trivial central character and highest weight $r \asymp T$ and that $\sigma$ is the trivial representation of $\U(1)$.  The varieties $\mathcal{O}_\pi$ and $\mathcal{O}_{\pi,\sigma}$ are then contained in the trace zero subspace $(\mathfrak{g}^\wedge )^0$ of $\mathfrak{g}^\wedge$, which we may identify with $\mathbb{R}^3$ in such a way that $\mathcal{O}_\pi$ is the Euclidean sphere
\[
  \mathcal{O}_{\pi} = \{(x,y,z) : x^2 + y^2 + z^2 = (r+1/2)^2 \}
\]
and so that the projection $(\mathfrak{g}^\wedge)^0 \rightarrow \mathfrak{h}^\wedge \cong \mathbb{R}$ is the vertical coordinate map $(x,y,z) \mapsto z$.  Then $\mathcal{O}_{\pi,\sigma} \subseteq \mathcal{O}_{\pi}$ is the equator
\[
  \mathcal{O}_{\pi,\sigma} = \{(x,y,0) : x^2 + y^2 = (r+1/2)^2\}.
\]
We may take for $\tau \in \mathcal{O}_{\pi,\sigma}$ any point along that equator.  The basic content of the required transversality is that the image of the equator under any nontrivial rotation of the sphere that fixes $\tau$ is transverse to the equator (see Figure \ref{fig:transversality}).

\begin{figure}
  \begin{tikzpicture} 
    \def\RadiusSphere{2} 
    \def\angEl{20} 
    \def\angAz{-20} 
    \filldraw[ball color=white] (0,0) circle (\RadiusSphere); \filldraw[fill=white] (0,0) circle (\RadiusSphere);

    \foreach \t in {0} { \DrawLatitudeCircle[\RadiusSphere]{\t} }

    \pgfmathsetmacro\H{\RadiusSphere*cos(\angEl)} 
    \coordinate (O) at (0,0);

    \def\angleLongitudeP{-110} 
    \def\angleLongitudeQ{-45} 
    \def\angleLatitudeQ{30} 
    \def\angleLongitudeA{-20} 

    \LongitudePlane[PLongitudePlane]{\angleLongitudeP}{\angAz} \LongitudePlane[QLongitudePlane]{\angleLongitudeQ}{\angAz} \LongitudePlane[ALongitudePlane]{\angleLongitudeA}{\angAz}

    \path[ALongitudePlane] (32.5:\RadiusSphere) coordinate (A'); \path[ALongitudePlane] (122.5:\RadiusSphere) coordinate (N'); \path[PLongitudePlane] (00:\RadiusSphere) coordinate (P);

 \begin{scope}[ x={(P)}, y={(A')}, z={(N')}]     
   \draw[blue] (-135:0.75) arc (-135:45:0.75) ; \draw[blue] (-135:0.75) arc (-135:-15:0.75) ; \coordinate (Q) at (-60:0.75);
 \end{scope}

 \NewLatitudePlane[equator]{\RadiusSphere}{\angEl}{00}; \path[equator] (-98:\RadiusSphere) coordinate (tau);

 \coordinate[mark coordinate] (N) at (tau); \draw[above left] node at (tau){$\tau$};

 \path[equator] (-180:\RadiusSphere) coordinate (OPS);

 \draw[left] node at (OPS){$\mathcal{O}_{\pi,\sigma}$};

 \NewLatitudePlane[slanted]{\RadiusSphere}{\angEl}{30}; \path[slanted] (0:\RadiusSphere*1.1) coordinate (OPSg); \draw[right, blue] node at (OPSg){$\Ad^*(\gamma) \mathcal{O}_{\pi,\sigma}(\tau_H) = \Ad^*(\gamma) \mathcal{O}_{\pi,\sigma}$};

 \draw[left] node at (0,0.8*\H){$\mathcal{O}_\pi$};

\end{tikzpicture}
\caption{ The required transversality for $(G,H) = (\U(2), \U(1))$.  }
\label{fig:transversality}
\end{figure}
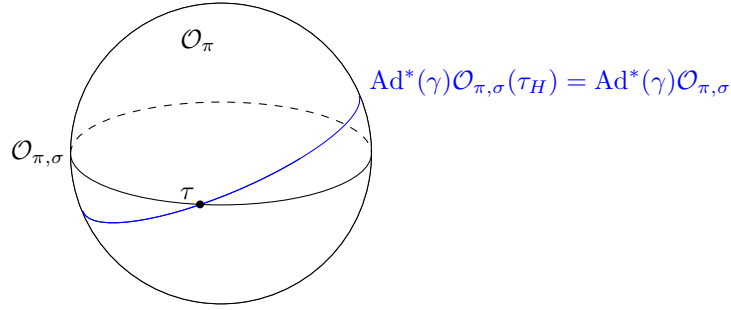

\begin{remark}
  
  There is another (closely related) way to arrive at the problem of controlling the transversality between the varieties \eqref{eqn:transversality-orbits-intro}.  For this remark, we assume familiarity with \cite[\S1]{nelson-venkatesh-1}.  The microlocal support of the vector $\varphi_{\pi,0}$ considered in Remark \ref{rmk:alternative-reduction-intro} is concentrated along the $H_{\tau_H}$-orbit $\mathcal{O}_{\pi,\sigma}(\tau_H)$.  The matrix coefficient of that vector at $\gamma$ should thus be related to the overlap between small neighborhoods of that orbit and its image under $\Ad^*(\gamma)$.  We noted in Remark \ref{rmk:alternative-reduction-intro} that the main point is to control the matrix coefficient of that vector along left $H$-cosets.  Since $\mathcal{O}_{\pi,\sigma}$ is an $H$-torsor, we arrive again at the problem of establishing some transversality between the varieties \eqref{eqn:transversality-orbits-intro}.
\end{remark}

\subsection{Endgame: discussion}
Recall from \S\ref{sec:proof-sketch} that -- having reduced the required moment estimate (\eqref{eqn:basic-pretrace-intro-0} or \eqref{eqn:moment-of-branching-coefs}) to volume bounds, and reduced the latter to a Lie algebra problem -- we must eventually confront the linear algebra problem addressed by Theorem \ref{thm:endgame}.  Here we outline the solution to that problem and record some miscellaneous remarks.

Theorem \ref{thm:endgame} is equivalent to the exclusion of the condition \eqref{eqn:intro-reduction-after-implicit-func-thm}.  Assuming for the sake of contradiction that that condition holds, we apply an infinitesimal form of the $H_{\tau_H}$-torsor property of $\mathcal{O}_{\pi,\sigma}(\tau_H)$ to deduce that
\begin{equation}\label{eqn:intro-after-applying-infinitesimal-torsor-property}
  \langle \mathfrak{h}_{\tau_H},   \ad^*(\ad(x) \mathfrak{z}_H)
  \tau \rangle = 0
  \text{ for some } x \in \mathfrak{g}_\tau - \mathfrak{z},
\end{equation}
where $\mathfrak{h}_{\tau_H} := \Lie(H_{\tau_H})$.  We may understand \eqref{eqn:intro-after-applying-infinitesimal-torsor-property} as an infinitesimal form of the failure of transversality in \eqref{eqn:transversality-orbits-intro}, specialized to central directions.  We interpret \eqref{eqn:intro-after-applying-infinitesimal-torsor-property} as the degeneracy of a certain bilinear form on $\mathfrak{h}_{\tau_H} \times \mathfrak{g}_{\tau}/\mathfrak{z}$.  We evaluate the determinant of that bilinear form with respect to a suitable basis.  That determinant turns out to be a nonzero multiple of the resultant of the characteristic polynomials of $\tau$ and $\tau_H$, hence vanishes precisely in the conductor dropping case.

The proof of Theorem \ref{thm:endgame} is then complete, modulo the verification of the determinantal identity.  There seem to be a few ways to do this (see Remark \ref{rmk:reduct-determ-ident-1}).  The given proof mimicks one of the standard ``divisibility and homogeneity'' calculations of the Vandermonde determinant, but with each step more complicated.  We appeal repeatedly to the results of \cite[\S14]{nelson-venkatesh-1}.  We eventually reduce to showing that there is \emph{some} $\tau$ for which the determinant in question does not vanish.  We conclude by exhibiting an explicit $\tau$ for which we may evaluate the determinant directly.  (We take for $\tau$ a regular nilpotent element for which $\tau_H$ is a regular semisimple element having the $n$th roots of unity as eigenvalues.)

\begin{remark}
  
  This last step is the only part of the proof of Theorem \ref{thm:main} specific to \emph{unitary} GGP pairs.  We leave to the reader the challenge of finding an analogous argument in the orthogonal case.  One should presumably focus there not on the (zero-dimensional) center $Z_H$, but instead on the unique one-parameter subgroup of the centralizer $H_{\tau_H}$ of the regular element $\tau_H$ for which some maximal isotropic subspace is an eigenspace.  The calculus arguments of \S\ref{sec:volume-estimates} apply, so the main challenge is to find a substitute for the linear algebra arguments of \S\ref{sec:some-invar-theory}.
\end{remark}

\begin{remark}
  We have noted already that the eigenvalue condition on $\tau$ in Theorem \ref{thm:endgame} comes from our ``no conductor dropping'' assumption.  Indeed, the eigenvalues of $\tau$ (resp. $\tau_H$) correspond to the parameters of $\pi$ (resp. $\sigma$) after rescaling by $T^{-1}$ and taking a limit.  It may be instructive to note moreover that any special case of Theorem \ref{thm:endgame} yields a corresponding special case of Theorem \ref{thm:main}.  For example, the case in which $\tau$ and $\tau_H$ are both regular semisimple corresponds to that in which the parameters of $\pi$ and $\sigma$ are ``dyadically spaced'' in the sense of \eqref{eqn:well-spaced-parameters}, so that the analogue of the ``avoidance of Weyl chamber walls'' condition \eqref{eqn:BlBu-Weyl-wall} holds.  On the other hand, the case in which $\tau$ and $\tau_H$ are both regular nilpotent (modulo the center) corresponds to the (twisted) $t$-aspect, as in Remark \ref{rmk:provable-examples}.
\end{remark}

\begin{remark}
  It would be interesting to understand whether ideas related to the case $n=2$ of Theorem \ref{thm:endgame} are implicit in existing proofs of spectral aspect subconvex bounds on $\GL_3$ \cite{MR4203038, MR4039487, 2020arXiv200607819K, 2020arXiv201010153S}.
\end{remark}

\section{General preliminaries}\label{sec:gener-prel-notat}

\subsection{Asymptotic notation}\label{sec:asymptotic-notation}

\subsubsection{Motivation}\label{sec:asymptotic-notation-motivation}
Analytic number theory papers often employ asymptotic notation in the following informal way:
\begin{itemize}
\item \emph{If $A \ll B$ and $B \ll C$, then $A \ll C$.}
\end{itemize}
Experienced readers know how to spell out such implications precisely:
\begin{itemize}
\item \emph{For each $c_1, c_2 \geq 0$ there exists $c_3 \geq 0$ so that for all $A,B,C$ with $|A| \leq c_1 |B|$ and $|B| \leq c_2 |C|$, we have $|A| \leq c_3 |C|$.}
\end{itemize}
This paper will involve the iterated application of many such implications.  To manage these coherently, it will be useful to take some care with foundational issues.

\subsubsection{Fixed elements}

We adopt E.\ Nelson's approach to nonstandard analysis \cite{MR469763}, but with the terminological substitution of ``fixed'' for ``standard.''  This choice of foundations should require little adjustment of habit for the reader familiar with analytic number theory papers.

Informally, a ``fixed'' quantity is an ``absolute constant,'' i.e., an element of some finite collection of objects in our mathematical universe, taken sufficiently large for the purposes of the paper.  Another informal perspective is that all quantities \emph{except} those labelled ``fixed'' are allowed by default to depend upon some implicit parameter tending off to $\infty$.  Formally, ``fixed'' is an undefined predicate equipped with some axioms for manipulating it, detailed in \cite{MR469763}, which yield a relatively consistent extension of ZFC.

This formulation of nonstandard analysis is well-suited to our aims.  It does not require constructing new objects, such as the hyperreals; one starts instead with the ``usual'' field of real numbers and introduces an undefined predicate ``fixed'' (or ``standard'') that applies to some of them.  We do not use any of the more sophisticated constructions of nonstandard analysis, but do require a coherent way to organize and iterate certain estimates.

\index{asymptotic notation!$\ll, \lll, \gg, \ggg, \O(\cdot), o(\cdot), \simeq$} The usual asymptotic notation and terminology is defined accordingly.  If $|A| \leq C |B|$ for some fixed $C \geq 0$, then we write $A \ll B$ or $B \gg A$ or $A = \O(B)$ (``$A$ is not much larger than $B$''); otherwise, we write $B \lll A$ or $A \ggg B$ or $B = o(A)$ (``$A$ is much larger than $B$'').  We write $A \asymp B$ as shorthand for $A \ll B \ll A$.  We often introduce absolute value signs, as in $|A| \asymp |B|$, for cosmetic reasons.  The notation $A \simeq B$ signifies that $A = B + o(1)$.  When an infinite exponent appears in an estimate, the meaning is that the indicated estimate holds with that exponent replaced by any fixed integer (e.g., $A \ll \h^\infty B$ means that $A \ll \h^N B$ for each fixed $N$).  Notation such as $A \ll_j B$ signifies that $A \ll B$ holds for all fixed $j$.  We write $\eps$ for a positive quantity, typically fixed and small enough.  ``Fix $n$'' carries the same meaning as ``Let $n$ be fixed.''

Experience suggests that the notation $A \lll B$ (or equivalently, $A = o(B)$) is more problematic than the more customary notation $\ll$ and $\gg$.  We emphasize that it carries the same content as ``not $A \gg B$.''  By introducing negations and taking contrapositives, any statement that involves $\lll$ or $\ggg$ may be translated to one that does not.  For example, the following statements are formally equivalent to each other (and true):
\begin{itemize}
\item \emph{If $A \lll B$ and $B \ll C$, then $A \lll C$.}
\item \emph{For each $c_2, c_3 \geq 0$ there exists $c_1 \geq 0$ so that for all $A,B,C$ with $|A| < c_1 |B|$ and $|B| \leq c_2 |C|$, we have $|A| < c_3 |C|$.}
\item \emph{If $A \gg C$ and $B \ll C$, then $A \gg B$.
  \item For each $c_2,c_3 \geq 0$ there exists $c_1 \geq 0$ so that for all $A,B,C$ with $|A| \geq c_3 |C|$ and $|B| \leq c_2 |C|$, we have $|A| \geq c_1 |B|$.}
\end{itemize}
We give a slightly more involved example of such translation in Remark \ref{rmk:translate-symbol-class-statement}.

\subsubsection{Some explication and translation}\label{sec:some-axioms}
While the notion ``fixed,'' and hence any asymptotic notation defined in terms of that notion, does not carry any intrinsic meaning, any statement that refers to that notion may be algorithmically translated to one that does not.  The algorithm is detailed in \cite[\S2]{MR469763}, but may be inferred from the informal equivalence between ``fixed'' and ``absolute constant.''  An example was given in \S\ref{sec:asymptotic-notation-motivation}.

We have attempted to write most of this paper like a typical paper in analytic number theory, working with asymptotic notation in the customary way and eschewing explicit mention of foundations.  In a couple places, we will be more pedantic, and it will be convenient then to refer back to the following consequences of the axioms of IST.
\begin{itemize}
\item \emph{Transfer}.  Let $A(x)$ be a statement in ZFC (i.e., a statement that does not contain ``fixed,'' even implicitly) that depends upon a quantity $x$ and possibly other fixed quantities (but no non-fixed quantities).  Assume that $A(x)$ holds for all fixed $x$.  Then $A(x)$ holds for all $x$.
\item \emph{Idealization}.  Let $A(x,y)$ be a statement in ZFC that depends upon quantities $x$ and $y$ (and possibly other quantities).  Assume that $A(x,y)$ is \emph{filtered} in $x$, i.e., $x$ ranges over some partially ordered set and if $x_1 \geq x_2$, then $A(x_1, y) \implies A(x_2, y)$.  Assume also that for each fixed $x$, there exists a $y$ so that $A(x,y)$ holds.  Then there exists $y$ so that $A(x,y)$ holds for all fixed $x$.
\end{itemize}
Each also has a dual form obtained by taking negations.

As illustration, one may combine transfer and idealization to verify the equivalence of the following pair of statements, for $A(x,y)$ filtered in $x$ satisfying the hypotheses of ``transfer'':
\begin{enumerate}[(i)]
\item For each $x$, there exists $y$ so that $A(x,y)$ holds.
\item There exists $y$ so that $A(x,y)$ holds for each fixed $x$.
\end{enumerate}
Many statements in this paper using asymptotic notation/terminology are of the form (ii) (or the corresponding dual form).

To give a further example, the axioms of IST imply that the following two assertions are equivalent for ZFC statements $A(x,y)$ and $B(x,z)$, with $A(x,y)$ filtered in $y$:
\begin{enumerate}[(i)]
\item For all $x$ for which $A(x,y)$ holds for some fixed $y$, we have that $B(x,z)$ holds for some fixed $z$.
\item For each $z$, there exists $y$ so that for all $x$ for which $A(x,y)$ holds, so does $B(x,z)$.
\end{enumerate}
For instance, the first italicized statement in \S\ref{sec:asymptotic-notation-motivation} is of the form (i), while the second is of the form (ii).  We refer again to \cite[\S2]{MR469763} for further details, examples and discussion concerning such translation.

\subsubsection{Classes}\label{sec:classes}
The discussion here becomes relevant only starting in \S\ref{sec:basic-symbol-classes}, so the reader may wish to refer back as needed.

``Illegal set formation'' is the forbidden practice of defining a set using ``fixed'' (or anything defined in terms of that notion, such as the above asymptotic notation).  For instance, there does not exist a subset $S \subseteq \mathbb{Z}$ consisting of precisely the integers $n$ with $n = \O(1)$; note that no axiom of ZFC allows one to construct such a set.  The customary way to work algebraically with predicates such as ``$n = \O(1)$'' (e.g., to assert that elements satisfying that predicate satisfy the ring axioms) is to introduce a dichotomy between ``internal'' and ``external'' sets, as in \cite[\S3]{MR469763}, \cite{MR723332} or \cite[\S3]{MR546178}.  For our limited purposes, a cheaper approach suffices.
\begin{definition}
  By a \emph{class}, we mean a pair $(X,P)$, where $X$ is a set and $P$ is a boolean formula that takes an element $x \in X$ as its argument and whose formulation may involve ``fixed.''
\end{definition}
We write ``$S$ is the class of all $x \in X$ satisfying $P(x)$'' or simply $S = \{x \in X : P(x)\}$ as longhand for $S = (X,P)$.  (In the language of the references noted above, classes define external subsets of internal sets.)

Classes play a purely syntactic role for us, giving a way to define symbolic expressions similar to ``$\O(B)$.''  For instance, given a scalar domain $X$ (e.g., $X = \mathbb{C}$) and a scalar $B \in X$, we may define the scalar class $\O(B)$ to be the pair $(X,P)$, where $P(A)$ is the predicate ``there is a fixed $C$ so that $|A| \leq C |B|$.''  This particular class will not play a role in this paper, but illustrates the basic idea.

We define only a handful of classes (Definitions \ref{defn:basic-symbol-class}, \ref{defn:new-symbol-class}, \ref{defn:oper-class} and mild derivatives thereof), but the notion provides a useful way to organize our estimates.

We may regard sets as classes (take for $P$ the predicate ``true'') and work with classes in much the same way we work with sets.  An element $x$ belongs to the class $(X,P)$ if $x \in X$ and if $P(x)$ is true.  One class $(X,P)$ is contained in another class $(Y,Q)$ if for every $x \in X$ satisfying $P(x)$, we have $x \in Y$ and $Q(x)$.  Two classes are equal if they contain the same elements, or equivalently, if each is contained in the other.  We may define intersections or unions of classes, class maps $f : (X,P) \rightarrow (Y,Q)$ between classes (i.e., subclasses of the product class satisfying the map axioms), and so on.

\subsection{Algebraic groups}\label{sec:algebraic-groups}
Let $F$ be a field of characteristic zero, and let $G$ be a linear algebraic group over $F$.  When $G$ is connected reductive, we often identify $G$ with its set of $F$-points: $G := G(F)$.  When $F$ is a global field, we write \index{adelic quotient $[G]$}
\begin{equation*} [G] := G(F) \backslash G(\mathbb{A})
\end{equation*}
for the corresponding adelic quotient, equip $G(\mathbb{A})$ with Tamagawa measure as in \cite[\S2.3.2]{2020arXiv200705601B}, and equip $[G]$ with the corresponding quotient measure.

\subsection{Unitary groups}\label{sec:unitary-groups}
Let $F$ be a field of characteristic zero.  Let $E/F$ be quadratic {\'e}tale $F$-algebra, thus either $E = F \times F$ or $E/F$ is a quadratic field extension.  Let $\iota$ denote the involution of $E$ fixing $F$.  Let $(V, \langle \rangle)$ be an $E$-vector space equipped with a nondegenerate $\iota$-hermitian form.  We write $\GL(V/E)$ for the group of $E$-linear automorphisms of $V$.  The subgroup
\[
  G := \Aut(V/E, \langle \rangle)
\]
consisting of automorphisms that preserve the given hermitian form is then an algebraic group.  It comes with a standard embedding $G \hookrightarrow\GL(V/E)$.  By a \emph{unitary group} over $F$, we mean a group arising in this way, together with its standard embedding.

In the split case $E = F \times F$, we may decompose $V = V^+ \oplus V^-$ as a sum of two $F$-vector spaces, with $E$ acting on the summands via the two projections to $F$.  We may identify $V^-$ with the dual of $V^+$ and
\[
  G \cong \GL(V^+)
\]
with the set of pairs $(g, {}^t g^{-1}) \in \GL(V^+) \times \GL(V^-)$ (see \cite[\S13.3]{nelson-venkatesh-1} for details).

\subsection{GGP pairs}
\label{sec:unitary-groups-ggp}
Retaining the setup of \S\ref{sec:unitary-groups}, let $e$ be an element of $V$ for which $E e$ is a free rank one $E$-module on which the form $\langle \rangle$ is nondegenerate.  Writing $V_H \subseteq V$ for the orthogonal complement of $E e$, we then have
\[
  V = V_H \oplus E e.
\]
The group
\[H := \Aut(V_H/E, \langle \rangle)\] embeds in $G$. 
By a \emph{unitary GGP pair} over $F$, we mean a pair $(G,H)$ arising in this way, together with its standard embeddings $H \hookrightarrow G \hookrightarrow \GL(V/E)$ and choice of $e \in E$.

We write $Z$ for the center of $G$ and
\begin{equation*}
  \bar{G} := G/Z
\end{equation*}
for the adjoint group.  We note that the intersection $Z \cap H$ is trivial, so we may regard $H$ also as a subgroup of $\bar{G}$.

In the split case $E = F \times F$, we may decompose $V = V^+ \oplus V^-$ and $V_H = V_H^+ \oplus V_H^-$ as in \S\ref{sec:unitary-groups}, and thereby identify
\[
  H \cong \GL(V_H^+) \hookrightarrow G \cong \GL(V^+).
\]

We may analogously define an \emph{orthogonal group} $G$ over $F$ and an \emph{orthogonal GGP pair} $(G,H)$ by taking $E := F$ in the above discussion (see \cite[\S13]{nelson-venkatesh-1} for details).  By a \emph{GGP pair}, we mean a unitary or orthogonal GGP pair.  (We focus in this paper on the unitary case, but many of our results apply just as well in the orthogonal case, so we state them in their natural generality.)

\subsection{Local distinction and matrix coefficient integrals}\label{sec:local-dist-matr}
Let $(G,H)$ be a GGP pair over a local field $F$.  Let $\pi$ and $\sigma$ be irreducible unitary representations of $G$ and $H$.  More precisely, we write $\pi$ and $\sigma$ for the spaces of smooth vectors in the given Hilbert spaces.

We say that the pair $(\pi,\sigma)$ is \emph{distinguished} if there is a nonzero $H$-invariant functional $\ell : \pi \rightarrow \sigma$.  It is known that the space of such functionals is then one-dimensional (see \cite{MR2874638}).

Suppose now that $\pi$ and $\sigma$ are tempered and that we are given a Haar measure $d h$ on $H$.  Then for $v \in \pi$ and $u \in \sigma$, the integral
\begin{equation*}
  Q(v \otimes u) :=
  \int _{h \in H}
  \langle h v, v  \rangle
  \langle u, h u  \rangle \, d h
\end{equation*}
converges absolutely and defines a \index{matrix coefficient integrals! quadratic, $\mathcal{Q}$} quadratic form $\mathcal{Q}$ on $\pi \otimes \sigma$ (see \cite[\S18]{nelson-venkatesh-1}, which refers to \cite[Prop 1.1]{MR2585578} and \cite[\S2]{MR3159075}).  It is expected that
\begin{enumerate}[(i)]
\item $\mathcal{Q}$ is not identically zero if and only if $(\pi,\sigma)$ is distinguished, and
\item $\mathcal{Q}$ is nonnegative-valued (rather than merely real-valued).
\end{enumerate}
Expectation (i) is known to hold in the unitary case \cite{2015arXiv150601452B}, so in that case, we may normalize the nonzero invariant functional $\ell : \pi \rightarrow \sigma$ for distinguished $(\pi,\sigma)$ by requiring that the identity
\begin{equation}\label{eqn:H-vs-ell-normalization}
  \mathcal{Q}(v \otimes u)
  = \pm |\langle \ell(v), u \rangle|^2
\end{equation}
hold for some sign $\pm 1$ depending only upon $(\pi,\sigma)$; in particular, $\mathcal{Q}$ is (positive or negative) definite.  Both expectations are known to hold in (at least) the following cases, in which we may provably normalize so that \eqref{eqn:H-vs-ell-normalization} holds with the expected sign $+1$:
\begin{itemize}
\item $F$ is non-archimedean, in which see \cite[Thm 5]{2015arXiv150601452B}, \cite[Prop 5.7]{Wald4}, \cite{SV}.
\item $F$ is archimedean and $G$ and $H$ are compact, in which the expectation follows readily from the Schur orthogonality relations for compact groups.
\item $F$ is archimedean and $\pi,\sigma$ arise by taking local components at the distinguished place $\mathfrak{q}$ of some constituent of the family $\mathcal{F}_T$ defined in \S\ref{sec:main-results}, with $T$ large enough.  In that case, the positivity follows from the relative character asymptotics established in \cite[\S19]{nelson-venkatesh-1} and recalled below in \S\ref{sec:relat-char-asympt}.
\end{itemize}

\subsection{Global notation and assumptions}\label{sec:assumpt-conc-s}
Let $F$ be a fixed number field.  We denote by $\mathbb{Z}_F$ its ring of integers and by $\mathbb{A}$ its adele ring.  For a place $\mathfrak{p}$ of $F$, we write $F_\mathfrak{p}$ for the completion.  If $\mathfrak{p}$ is finite, we write $\mathbb{Z}_\mathfrak{p} \subseteq F_\mathfrak{p}$ for the ring of integers and $q_{\mathfrak{p}}$ for the absolute norm.  For a finite set $S$ of places of $F$, we write $F_S := \prod_{\mathfrak{p} \in S} F_\mathfrak{p}$.  For an algebraic group $G$ over $F$, we sometimes abbreviate $G_\mathfrak{p} := G(F_\mathfrak{p})$ and $G_S = G(F_S)$.

Let $(G,H)$ be a fixed unitary GGP pair over $F$, with standard representation $G \hookrightarrow \GL(V/E)$.  We take for $S$ a fixed finite set of places of $F$ that is sufficiently large in the following senses.
\begin{itemize}
\item $S$ contains every archimedean place.
\item $G$ and $H$ admit smooth models over $\mathbb{Z}_F[1/S]$ and the inclusion $H \hookrightarrow G$ extends to a closed immersion.  We continue to denote simply by $G$ and $H$ these smooth models.  For each $\mathfrak{p} \notin S$, the groups $G(\mathbb{Z}_\mathfrak{p}) \leq G_\mathfrak{p}$ and $H(\mathbb{Z}_\mathfrak{p}) \leq H_\mathfrak{p}$ are then hyperspecial maximal compact subgroups.
\item The rings $\mathbb{Z}_F[1/S]$ and $\mathbb{Z}_E[1/S]$ (defined by localizing away from all finite primes in $S$) are principal ideal domains.
\item $G(F) G(F_S) \prod_{\mathfrak{p} \notin S} G(\mathbb{Z}_\mathfrak{p}) = G(\mathbb{A})$, and similarly for $H$.
\end{itemize}
The last two assumptions are likely technical, but convenient.

Under the above assumptions, we fix factorizations of the Tamagawa measures on $G(\mathbb{A})$ and $H(\mathbb{A})$ that assign volume one to $G(\mathbb{Z}_\mathfrak{p}) \subseteq G_\mathfrak{p}$ and $H(\mathbb{Z}_\mathfrak{p}) \subseteq H_\mathfrak{p}$ for all $\mathfrak{p} \notin S$.

\subsection{Branching coefficients}\label{sec:branch-coeff}
We retain the setting of \S\ref{sec:assumpt-conc-s}.  Recall the family $\mathcal{F}$ defined in \S\ref{sec:introduction}.  Let $(\pi,\sigma) \in \mathcal{F}$.  We equip $\pi$ (resp. $\sigma$) with the Petersson inner product defined by integrating over $[G/Z]$ (resp. $[H/Z_H]$).  As in \S\ref{sec:local-dist-matr}, we write $\pi$ and $\sigma$ for the spaces of smooth vectors in the given Hilbert spaces.

Writing $\pi_S \subseteq \pi$ and $\sigma_S \subseteq \sigma$ for the subspaces of vectors invariant by $\prod_{\mathfrak{p} \notin S} G(\mathbb{Z}_\mathfrak{p})$ and $\prod_{\mathfrak{p} \notin S} H(\mathbb{Z}_\mathfrak{p})$, respectively, we define quadratic forms $\mathcal{P}$ and $\mathcal{Q}$ on $\pi_S \otimes \sigma_S$ by the formulas
\begin{equation*}
  \mathcal{P}(v \otimes u)
  :=
  \left\lvert
    \int _{[H]}
    v \bar{u}
  \right\rvert^2
\end{equation*}
and, as in \S\ref{sec:local-dist-matr},
\begin{equation*}
  \mathcal{Q}(v \otimes u)
  :=
  \int _{h \in H(F_S)}
  \langle h v, v \rangle
  \langle u, h u \rangle \, d h.
\end{equation*}

As recalled in \S\ref{sec:local-dist-matr}, the local distinction of $(\pi,\sigma)$ implies that $\mathcal{Q}$ is not identically zero.  By multiplicity one, we may thus define a scalar \index{representations!branching coefficient $\mathcal{L}(\pi,\sigma)$} $\mathcal{L}(\pi,\sigma) \in \mathbb{R}_{\geq 0}$ by the relation
\begin{equation}\label{eqn:P-equals-L-H}
  \mathcal{P} = \mathcal{L}(\pi,\sigma)
  \cdot |\mathcal{Q}|.
\end{equation}
As discussed in \S\ref{sec:local-dist-matr}, it is expected that $\mathcal{Q}$ is positive-definite, so that the absolute value signs above are not necessary.  This expectation is likely provable in general, but not written down to the best of our knowledge.  It follows for $(\pi,\sigma) \in \mathcal{F}_T$ with $T$ large enough from the discussion of \S\ref{sec:local-dist-matr}, and we will apply \eqref{eqn:P-equals-L-H} only to such pairs $(\pi,\sigma)$.  For our proofs, we require only the obvious positivity of $\mathcal{P}$ (rather than that of $\mathcal{L}(\pi,\sigma)$, contrary to what the simplified presentation of \S\ref{sec:overview-proof} might suggest).

\subsection{Period formulas}\label{sec:period-formulas}
The conjectures of Ichino--Ikeda and N.\ Harris (see \cite{MR2585578, MR3159075} and \cite[\S25.5]{nelson-venkatesh-1}) assert that, with notation and assumptions as above, and with $L^{(S)}$ denoting the partial $L$-function obtained by omitting Euler factors in $S$,
\begin{equation*}
  \mathcal{L}(\pi,\sigma)
  =
  2^{- \beta}
  \frac{ L^{(S)}(\pi_E \otimes  \sigma_E^{\vee},1/2)}{L^{(S)}(\Ad, \pi \boxtimes {\sigma}^{\vee},1)}
  \Delta_G^{(S)}
\end{equation*}
for some specific $\beta = \O(1)$ and $\Delta_G^{(S)} \asymp 1$.  A proof is given in \cite[Thm 1.1.6.1]{2020arXiv200705601B} under the assumption that $\pi$ and $\sigma$ are everywhere tempered.  (In fact, an inspection of the proof suggests that it should suffice to assume that $\pi_\mathfrak{p}$ and $\sigma_\mathfrak{p}$ are tempered for all $\mathfrak{p} \in S$, with $S$ sufficiently large.)

\section{Statement of main local result}\label{sec:statement-main-local}
In \S\ref{sec:statement-main-local}, we formulate our main local result, Theorem \ref{thm:construct-test-function}.  In \S\ref{sec:reduction-proof}, we explain how Theorem \ref{thm:construct-test-function} implies our main global result, Theorem \ref{thm:main}.  The remainder of the paper will then be devoted to the proof of Theorem \ref{thm:construct-test-function}.
\subsection{Norms}
\subsubsection{}\label{sec:F-norms}
Let $F$ be a local field.  We denote by $|.|_F : F \rightarrow \mathbb{R}_{\geq 0}$ the modulus.  Let $V$ be a finite-dimensional vector space over $F$.  By an \emph{$F$-norm} on $V$, we mean a function $|.|_F : V \rightarrow \mathbb{R}_{\geq 0}$ given, for some $F$-basis $e_1,\dotsc,e_n$ of $V$, by the formula
\begin{equation*}
  |\sum c_j e_j|_F
  =
  \begin{cases}
    \max_j |c_j|_F & \text{ if $F$ is non-archimedean,}\\
    (\sum_j |c_j|^2_F)^{1/2} & \text{ if } F=\mathbb{R}, \\
    \sum_j |c_j|_F& \text{ if } F = \mathbb{C}.
  \end{cases}
\end{equation*}
In particular, $|.|_F$ is a continuous function such that
\begin{itemize}
\item $|v|_F = 0$ only if $v = 0$, and
\item $|t v|_F = |t|_F |v|_F$ for all $t \in F, v \in V$.
\end{itemize}
Any two $F$-norms are comparable in the sense that each is bounded from above and below by a positive real multiple of the other, so the precise choice is unimportant when $F$ and $V$ are fixed; in that case, we often write $|.|_F$ for some fixed but unspecified $F$-norm.

When $F = \mathbb{R}$ or $\mathbb{Q}_p$, we often write simply $|.|$ rather than $|.|_{\mathbb{R}}$.  For a general local field $F$ of characteristic zero, we write $|.| := |.|_F^{1/[F:F_0]}$, where $F_0$ is the subfield of $F$ isomorphic to $\mathbb{R}$ or $\mathbb{Q}_p$.

\subsubsection{}\label{sec:F-norms-global-to-local}
Let $V$ be a finite-dimensional vector space over a global field $F$.  By choosing a basis $e_1,\dotsc,e_n$ of $V$, we obtain from the above construction for each place $\mathfrak{p}$ of $F$ an $F_\mathfrak{p}$-norm $|.|_{F_\mathfrak{p}}$ on the completion $V_{\mathfrak{p}}$.

\subsubsection{}
If $E/F$ is an extension of local fields and $V$ is an $E$-vector space, then we may also regard $V$ as an $F$-vector space, hence we may speak of $E$-norms $|.|_E$ and $F$-norms $|.|_F$ on $V$.  Any such norms are related by the estimate (for fixed $E,F,V, |.|_F, |.|_E$)
\begin{equation}\label{eqn:relate-F-norms-inert}
  |v|_E \asymp |v|_F^{[E:F]}.
\end{equation}

\subsubsection{}
Suppose now that $E = F \times F$ is the split quadratic extension of a local field $F$ and that $V$ is an $E$-vector space.  We may then decompose
\[V = V^+ \oplus V^-,\] where $E$ acts on the summands via the two projections to $F$.  Writing $v = v^+ + v^-$ for the corresponding decomposition of a vector $v \in V$, we have (for fixed $F, V, |.|_F$)
\begin{equation}\label{eqn:F-norm-split-AM-GM}
  |v|_F \asymp |v^+|_F + |v^-|_F
  \gg \sqrt{|v^+|_F |v^-|_F}.
\end{equation}

\subsubsection{}
Let $G$ be a linear algebraic group over a local field $F$.  Suppose given some faithful embedding $G \hookrightarrow \SL_N(F)$.  By embedding $\SL_N(F)$ further in the space $M_N(F)$ of $N \times N$ matrices and choosing an $F$-norm $|.|_F$ on $M_N(F)$, we may define $|g - 1|_F$ for $g \in G$.

Suppose that $F$ and $G$ are fixed.  Let $|.|_F$ and $|.|_F'$ be any fixed pair of norms arising as above.  For $g \in G$, we then have $|g-1|_F \lll 1$ if and only if $|g-1|_F' \lll 1$, in which case $|g-1|_F \asymp |g-1|_F'$.  Thus $|g-1|_F$ gives a well-defined notion of the distance from small elements $g \in G$ to the identity element.

\subsection{Quantifying distance
  to the subgroup $H$}\label{sec:quant-dist-subgr}
Let $(G,H)$ be a fixed unitary GGP pair over a fixed local field $F$, with standard representation $G \hookrightarrow \GL(V/E)$ and accompanying decomposition
\begin{equation}\label{eqn:V-V-H-E-e-for-d-H}
  V = V_H \oplus E e.
\end{equation}
Recall that we denote by $Z$ the center of $G$ and by $\bar{G} = G/Z$ the adjoint group.  The purpose of this section is to record a quantification of the ``distance'' between an element $g \in \bar{G}$ and the subgroup $H \hookrightarrow \bar{G}$.  Since we will eventually need to perform explicit calculations involving this quantification, we will be explicit here.

The basic idea can be seen in the example $H = \GL_1(F) \hookrightarrow \bar{G} = \PGL_2(F)$.  An element of $\bar{G}$ is described by a scaling class of $2 \times 2$ matrices $\left(
  \begin{smallmatrix}
    a&b\\
    c&d
  \end{smallmatrix}
\right)$.  Such an element lies in $H$ precisely when the ratios $b/d$ and $c/d$ both vanish.  The sizes of those ratios thus quantify the distance from that element to $H$.

Turning to details, we may regard $V$ as a vector space over $E$ (which need not be a field) or over $F$.  We fix an $F$-norm $|.|_F : V \rightarrow \mathbb{R}_{\geq 0}$.  We then define $\tilde{d}_H : \PGL(V/E) \rightarrow [0,\infty]$ as follows.  Temporarily lifting $g$ to an element of $\GL(V/E)$, the decomposition of the vector $g e \in V$ with respect to \eqref{eqn:V-V-H-E-e-for-d-H} is given by
\[
  g e = \left( g e - \frac{\langle g e, e \rangle}{ \langle e, e \rangle} e \right) + \frac{\langle g e, e \rangle}{ \langle e, e \rangle} e.
\]
We define $\tilde{d}_H(g)$ to be the size of the vector obtained by dividing the $V_H$-component of $g e$ by the coefficient of $e/\langle e,e \rangle$ in the $E e$-component of $g e$:
\begin{equation*}
  \tilde{d}_H(g) :=
  \left|
    \frac{g e}{\langle g e, e \rangle}
    - \frac{e}{ \langle e, e \rangle}
  \right|_F,
\end{equation*}
with the convention that $\tilde{d}_H(g) := \infty$ if $\langle g e, e \rangle \notin E^\times$.  We then define \index{distance function $d_H$}
\[
  d_{H} : \PGL(V/E) \rightarrow [0,1],
\]
\begin{equation*}
  d_H(g) := \min(1,\tilde{d}_H(g) + \tilde{d}_H(g^{-1})).
\end{equation*}

We will be interested primarily in the restrictions of $\tilde{d}_H$ and $d_H$ to $\bar{G}$.  Since $H \hookrightarrow \bar{G}$ is the stabilizer of the line $E e$, it is clear that
\begin{equation*}
  H = \{g \in \bar{G} : d_H(g) = 0\},
\end{equation*}
so we may regard $d_H$ as quantifying how far $g \in \bar{G}$ is from $H$.

We may explicate these definitions in terms of matrix entries.  We choose an orthogonal $E$-basis $e_1,\dotsc,e_{n+1}$ of $V$ with $e_{n+1} = e$.  We accordingly identify $\GL(V/E) = \GL_{n+1}(E)$ and $\GL(V_H) = \GL_n(E)$, with the latter included in the former as the subgroup of upper-left block matrices.  For $g \in \PGL(V/E) = \PGL_{n+1}(E)$, we may write
\begin{equation}\label{eqn:g-a-b-c-d}
  g = \begin{pmatrix}
    a & b \\
    c & d
  \end{pmatrix}
\end{equation}
with $a,b,c,d$ matrices of respective dimensions $n \times n, n \times 1, 1 \times n, 1 \times 1$, well-defined up to simultaneous scaling.  We similarly write
\begin{equation}\label{eqn:g-inverse-a-b-c-d}
  g^{-1} = \begin{pmatrix}
    a' & b' \\
    c' & d'
  \end{pmatrix}.
\end{equation}
Then
\begin{equation*}
  \tilde{d}_H(g)
  \asymp
  |b/d|_F,
  \quad
  d_H(g)
  \asymp \min(1,|b/d|_F +  |c'/d'|_F),
\end{equation*}
where the implied constants quantify the ambiguity in defining $F$-norms.

To illustrate the definition of $d_H$ further, we explicate its restriction to $\bar{G}$ in the split case $E = F \times F$.  As explained in \S\ref{sec:unitary-groups-ggp}, we may identify
\[
  H = \GL_n(F) \hookrightarrow G = \GL_{n+1}(F) \hookrightarrow \GL(V/E) = \GL_{n+1}(F) \times \GL_{n+1}(F),
\]
with the final inclusion given by $g \mapsto (g, {}^t g^{-1})$.  For $g \in \bar{G} = \PGL_{n+1}(F)$, we have with the notation \eqref{eqn:g-a-b-c-d} and \eqref{eqn:g-inverse-a-b-c-d} that
\[
  \tilde{d}_H(g) \asymp |b/d|_F + |c'/d'|_F,
\]
\begin{equation*}
  d_H(g) \asymp
  \min(1,
  |b/d|_F
  +
  |b'/d'|_F
  + |c/d|_F
  + |c'/d'|_F).
\end{equation*}

The essential properties of $d_H$ are summarized in the following lemma.
\begin{lemma}\label{lem:d-H-essential-properties}
  Let $g \in \bar{G}$.
  \begin{enumerate}[(i)]
  \item We have
    \begin{equation}\label{eqn:d-H-Z-vs-dist-G-mod-Z}
      d_{H}(g) \ll |\Ad(g)-1|_F.
    \end{equation}  
  \item Let $u_1, u_2 \in \bar{G}$.  Suppose that $g$ belongs to a fixed compact set and that
    \[
      |\Ad(u_j)-1|_F \lll d_{H}(g) \quad (j=1,2).
    \]
    Then
    \begin{equation}\label{eqn:d-H-Z-compare-u-g-g}
      d_{H}(u_1 g u_2) \asymp d_{H}(g).
    \end{equation}
  \item Let $h_1, h_2 \in H$.  Suppose that $g,h_1,h_2$ each belong to some fixed compact set.  Then
    \begin{equation}\label{eqn:d-H-h1gh2}
      d_{H}(h_1 g h_2)
      \asymp d_{H}(g).
    \end{equation}
  \end{enumerate}
\end{lemma}
\begin{proof}
  We abbreviate $|.| := |.|_F$.  For $x \in \End(V/E)$, we set
  \[
    \nu(x) := \langle x e, e \rangle / \langle e, e \rangle
  \]
  and
  \[
    \mu(x) := x e - \nu(x) e
  \]
  so that $x e = \mu(x) +\nu (x) e$ with $\mu(x) \in V_H$ and $\nu (x) \in E$.  Then $\tilde{d}_H(g) \asymp |\mu(g)/\nu(g)|$, with the RHS well-defined.

  Suppose that $g \in G$ lies in a fixed compact set.  We may define $\tilde{d}_H(g)$ and $d_H(g)$ by composing with the quotient map $G \rightarrow \bar{G}$.  We claim that
  \begin{equation}\label{eqn:d-H-vs-mu}
    d_H(g) \asymp |\mu(g)|.
  \end{equation}
  To see this, note first that if $|\mu(g)| \asymp 1$, then $|\mu(g) / \nu(g)| \gg 1$, so $d_H(g) \asymp 1 \asymp |\mu(g)|$.  We may thus suppose that $|\mu(g)| \lll 1$.  Since $\mu(g)$ and $\nu(g)$ describe the rightmost column of $g \in \GL_{n+1}(E)$, we see then that $\nu(g)$ lies in a fixed compact subset of $E^\times$.  Since each entry of $g$ as well as $1/\det(g)$ is $\O(1)$, we deduce by Cramer's rule that $\nu(g^{-1})$ likewise lies in a fixed compact subset of $E^\times$ and that $|\mu(g^{-1})| \asymp |\mu(g)|$.  Thus $\tilde{d}_H(g) \asymp |\mu(g)|$ and $\tilde{d}_H(g^{-1}) \asymp |\mu(g^{-1})| \asymp |\mu(g)|$.  The estimate \eqref{eqn:d-H-vs-mu} follows.

  We now let $g \in \bar{G}$ and prove each assertion from the lemma in turn.
  \begin{enumerate}[(i)]
  \item Since $d_H(g) \leq 1$, we may assume that $|\Ad(g) - 1| \lll 1$, as the required estimate is otherwise trivial.  We may then lift $g$ to an element of $G$ with $|g-1| \asymp |\Ad(g)-1|$; in particular, $g \simeq 1$.  Writing $\mu(g) = (g-1) e - \nu(g-1) e$, we deduce that
    \[
      d_H(g) \asymp |\mu(g)| \ll |g-1| \ll |\Ad(g) - 1|,
    \]
    as required.
  \item Since $g$ lies in a fixed compact set, we may lift $g$ to an element of $G$ with the same property.  We may likewise lift $u_j$ to an element of $G$ with $|u_j - 1| \lll d_H(g) \asymp \mu(g)$ (by \eqref{eqn:d-H-vs-mu}).  Then $\mu(u_1 g u_2) = \mu(g) + o(\mu(g))$, hence $|\mu(u_1 g u_2)| \asymp |\mu(g)|$.  We conclude via \eqref{eqn:d-H-vs-mu}.

  \item Since $H$ stabilizes $e$ and $h_1$ lies in a fixed compact subset of $H$, we see that
    \[
      |\mu(h_1 g h_2)| = \left\lvert h_1 \left( \frac{g e}{\langle g e, e \rangle} - \frac{e}{\langle e, e \rangle} \right) \right\rvert \asymp |\mu(g)|.
    \]
    We conclude once again via \eqref{eqn:d-H-vs-mu}.
  \end{enumerate}

\end{proof}

\subsection{Construction
  of analytic test vectors: statements}\label{sec:constr-prop-test}
Fix a GGP pair $(G,H)$ over an archimedean local field $F$, with standard representation $G \hookrightarrow \GL(V/E)$.  Set $n + 1 := \dim_E(V)$, and write $\bar{G} = G/Z$ as usual.  We assume given fixed Haar measures on each group.  The following is the main local result of this paper.
\begin{theorem}\label{thm:construct-test-function}
  Let $T$ be a positive real with $T \ggg 1$.  Let $\pi$ and $\sigma$ be tempered irreducible unitary representations of $G$ and $H$, respectively.  Assume the following conditions (i.e., the local conditions imposed at the distinguished place ``$\mathfrak{q}$'' in the definition of ``$\mathcal{F}_T$'' from \S\ref{sec:introduction}):
  \begin{itemize}
  \item $(\pi,\sigma)$ is orbit-distinguished (\S\ref{sec:relat-coadj-orbits}),
  \item $T = \max(\{|\lambda_{\pi,i}|_F\} \cup \{|\lambda_{\sigma,j}|_F\})$ is the size of the largest archimedean Satake parameter of $\pi$ or $\sigma$ (\S\ref{sec:satake-param-arch}), and
  \item we have $|\lambda_{\pi,i} - \lambda_{\sigma,j}|_F \gg T$ for all $i$ and $j$, or equivalently, $C(\pi, \sigma) \asymp T^{2 n(n+1)}$.
  \end{itemize}
  
  Then for each fixed $\kappa > 0$, there is
  \begin{itemize}
  \item a test function $f \in C_c^\infty(G)$, supported in each fixed neighborhood of the identity in $G$, and
  \item a smooth unit vector $u \in \sigma$
  \end{itemize}
  with the following properties.

  Let $\omega_\pi : Z \rightarrow \U(1)$ denote the central character of $\pi$.  Define $f^\sharp \in C_c^\infty(\bar{G}, \omega_\pi^{-1})$ by the formula
  \begin{equation}\label{eqn:at-frakq-f-sharp-via-f0}
    f^\sharp(g) := \int _{z \in Z}
    \omega_\pi(z) (f \ast f^*)(z g) \, d z,
  \end{equation}
  where $f^*(g) := \overline{f(g^{-1})}$ and $f \ast f^*$ denotes the convolution product.
  \begin{enumerate}[(i)]
  \item \label{thm:f-item-1} With $\mathcal{Q}$ as in \S\ref{sec:local-dist-matr}, we have
    \begin{equation}\label{eqn:sum-v-H-pi-f-0-v-u}
      \sum _{v \in \mathcal{B}(\pi)}
      \mathcal{Q}(\pi(f) v \otimes u)
      \gg T^{-n^2/2-\kappa}.
    \end{equation}
    where $\mathcal{B}(\pi)$ denotes any orthonormal basis.
  \item \label{item:f-ell-one-norm-on-H} We have
    \[\int_{H} |f^\sharp| \ll T^{n/2+\kappa}.\]
  \item \label{thm:f-item-3} Let $\Omega$ be a fixed compact subset of $H$.  There is a fixed compact subset $\Omega '$ of $H$ so that the following assertions hold.  Let $\Psi_1, \Psi_2 : H \rightarrow \mathbb{C}$ be measurable functions satisfying $\Psi_j(g z) = \omega(z) \Psi_j(g)$ for some unitary character $\omega : Z_H \rightarrow \U(1)$ and all $g \in H, z \in Z_H$.  Let $\gamma \in \bar{G} - H$.  Then
    \begin{equation}\label{eq:int-_x-y}
      \int _{x, y \in \Omega } \left\lvert \overline{\Psi_1(x)} \Psi_2(y) f^\sharp(x^{-1} \gamma y) \right\rvert \, d x \, d y \ll \frac{T^{n/2 - 1/2 +\kappa}}{ d_H(\gamma)} \|\Psi_1\|_{L^2(\Omega')} \|\Psi_2\|_{L^2(\Omega')}.
    \end{equation}
  \end{enumerate}
\end{theorem}
\begin{proof}
  See \S\ref{sec:analyt-test-vect}.
\end{proof}
For details concerning the convergence of sums as in \eqref{eqn:sum-v-H-pi-f-0-v-u}, we refer to \cite[\S18]{nelson-venkatesh-1} and \S\ref{sec:relat-char-asympt}.

\begin{remark}\label{rmk:cancellation-in-integral}
  One could attempt a stronger estimate for the integral without absolute value signs.
\end{remark}

\section{Preliminaries for the proof of the reduction}\label{sec:prel-proof-reduct}
In this section we record some preliminaries for our deduction, given in \S\ref{sec:reduction-proof}, of Theorem \ref{thm:main} from Theorem \ref{thm:construct-test-function}.  Some of these results may be more broadly useful, but they will not be otherwise applied in this paper.

\subsection{Uniform distinction
  at auxiliary places}\label{sec:uniform-distinction}
Let $(G,H)$ be a GGP pair over a local field $F$.  We assume given a Haar measure $d h$ on $H$.  We define quadratic forms $\mathcal{Q}$ as in \S\ref{sec:local-dist-matr}.  All representations considered here are assumed smooth.  For a smooth unitary representation $V$ of some group, we write $\mathcal{B}(V)$ for an orthonormal basis consisting of vectors isotypic (i.e., generating a multiple of an irreducible representation) under some chosen maximal compact subgroup.

\begin{lemma}\label{lem:uniform-distinction-arch}
  Suppose that $F$ is archimedean and that $G$ and $H$ are compact.  For each irreducible unitary representation $\pi$ of $G$ and open neighborhood $\Omega$ of the identity element in $G$, there is a positive real $c$ and a test function $f \in C_c^\infty(\Omega)$ with the following property.  Let $\sigma$ be an irreducible unitary representation $H$ for which $(\pi,\sigma)$ is distinguished.  There is a unit vector $u \in \sigma$ with
  \begin{equation*}
    \sum _{v \in \mathcal{B}(\pi)}
    \mathcal{Q}(\pi(f) v, u)
    \geq c.
  \end{equation*}
\end{lemma}
\begin{proof}
  We take for $f$ a smooth bump with $\int_G f = 1$ that is supported close enough to the identity, where the meaning of ``close enough'' depends upon of $\pi$ and $\Omega$.  Since $G$ is compact, the representation $\pi$ is finite-dimensional.  The operator $\pi(f)$ thus approximates the identity operator on $\pi$, and so the above sum is at least
  \begin{equation}\label{eqn:sum-v-B-pi-H-v-u}
    \frac{1}{2} \sum_{v \in \mathcal{B}(\pi)}
    \mathcal{Q}(v,u).
  \end{equation}
  Let $c_\sigma$ denote the maximum of \eqref{eqn:sum-v-B-pi-H-v-u} over all unit vectors $u \in \sigma$.  By our distinction assumption, we have $c_\sigma > 0$.  Since $\pi|_H$ is finite-dimensional, there are only finitely many possibilities for the isomorphism class of $\sigma$.  We conclude taking for $c$ the minimum of $c_\sigma$ over $\sigma$.
\end{proof}

\begin{lemma}\label{lem:padic-uniform-distinction}
  Suppose that $F$ is non-archimedean.  For each pair of compact open subgroups $J_G \leq G$ and $J_H \leq H$, there is a positive real $c$ and a pair of arbitrarily small compact open subgroups $U_G \leq G$ and $U_H \leq H$ with the following property.  Let $\pi$ and $\sigma$ be tempered irreducible unitary representations of $G$ and $H$.  Assume that $(\pi,\sigma)$ is distinguished and that the invariant subspaces $\pi^{J_G}$ and $\sigma^{J_H}$ are nontrivial.  Then there are unit vectors $v \in \pi^{U_G}$ and $u \in \sigma^{U_H}$ so that $\mathcal{Q}(v,u) \geq c$.
\end{lemma}
\begin{proof}
  By uniform admissibility (see \cite[\S2.6.3]{michel-2009}), the dimensions of the spaces $\pi^{U_G}$ and $\sigma^{U_H}$ are bounded by a constant depending at most upon $U_G$ and $U_H$.  It is thus equivalent to produce $c > 0$ and $U_G, U_H$ so that
  \[\sum _{v \in \mathcal{B}(\pi^{U_G})}
    \sum _{u \in \mathcal{B}(\sigma^{U_H})} \mathcal{Q}(v,u) \geq c.\] To conclude, we apply the following result to the representation $\pi \otimes \sigma^\vee$ of $G \times H$.
\end{proof}
\begin{lemma}
  Suppose that $F$ is non-archimedean.  For each compact open subgroup $J$ of $G \times H$, there is a positive real $c$ and an arbitrarily small compact open subgroup $U$ of $G \times H$ so that for every tempered irreducible representation $\pi$ of $G \times H$ for which
  \begin{itemize}
  \item $\pi^J$ is nontrivial, and
  \item $\pi$ is distinguished, i.e., there is a nonzero $H$-invariant functional $\ell : \pi \rightarrow \mathbb{C}$,
  \end{itemize}
  we have
  \begin{equation}\label{eqn:sum-v-in-pi-U-int-h}
    \sum _{v \in \mathcal{B}(\pi^U)}
    \int _{h \in H}
    \langle h v, v \rangle \, d h \geq c.
  \end{equation}
\end{lemma}
\begin{proof}
  The proof is similar to that of \cite[\S24.5.1]{nelson-venkatesh-1}.
  
  Let $\pi$ satisfy the stated hypotheses.  It follows from the facts noted in \S\ref{sec:local-dist-matr} that the LHS of \eqref{eqn:sum-v-in-pi-U-int-h} converges absolutely, is nonnegative for all $U$, and is positive for some $U$.
  
  We recall (see \cite[\S23.4.3]{nelson-venkatesh-1}) that every tempered irreducible representation $\pi$ of $G \times H$ arises as a submodule of the normalized induction
  \begin{equation*}
    i_M \tau := \Ind_P^{G \times H} \tau
  \end{equation*}
  for some parabolic subgroup $P = M N$ of $G \times H$ and some irreducible representation $\tau$ of $M$ that is \emph{square-integrable}, i.e., has unitary central character and matrix coefficients square-integrable modulo the center.  The conjugacy class of the pair $(M,\tau)$ is moreover unique.  The unitary representation $i_M \tau$ decomposes as a finite direct sum of tempered irreducible representations of $G \times H$, one of which is $\pi$.  By strong multiplicity one (see \cite[\S24.1]{nelson-venkatesh-1} and references), $\pi$ is the only distinguished summand of $i_M \tau$.  It follows that
  \begin{equation*}
    \sum _{v \in \mathcal{B}(\pi^U)}
    \int _{h \in H}
    \langle h v, v \rangle \, d h
    = 
    \sum _{v \in \mathcal{B}((i_M \tau)^U)}
    \int _{h \in H}
    \langle h v, v \rangle \, d h.
  \end{equation*}

  Given a pair $(M,\tau)$ as above, we obtain another such pair by replacing $\tau$ with its twist $\tau_t$ by any unramified character $t$ of $M$.  As explained in \cite[\S23.8]{nelson-venkatesh-1}, we may find a finite list of pairs $(M,\tau)$ so that whenever $\pi^J \neq 0$, there is an element $(M,\tau)$ from this list so that $\pi \hookrightarrow i_M (\tau_t)$ for some $t$.  Since $\pi$ is distinguished, we know by the proof of \cite[\S24.1, Lemma 2]{nelson-venkatesh-1} that each of the representations $i_M(\tau_t)$ is distinguished.  We may thus find for each $t_0$ a compact open subgroup $U$ so that
  \begin{equation}\label{eqn:sum-v-int-h-t0}
    \sum _{v \in \mathcal{B}((i_M \tau_t)^{U})}
    \int _{h \in H}
    \langle h v, v \rangle \, d h > 0
  \end{equation}
  when $t = t_0$.  The LHS of \eqref{eqn:sum-v-int-h-t0} varies continuously in $t$, so the same choice of $U$ works uniformly in a small neighborhood of $t_0$.  Since the group of unramified characters of $M$ is compact, we may choose $U$ small enough that \eqref{eqn:sum-v-int-h-t0} holds for all $t$.  By continuity and compactness, the LHS of \eqref{eqn:sum-v-int-h-t0} is thus bounded from below by some $c > 0$.  This completes the proof.
\end{proof}

\subsection{Constructing an amplifier}\label{sec:constr-an-ampl}
An amplifier on general semisimple groups is constructed by Silberman--Venkatesh \cite[Lemma A.1]{SV-AQUE}.  The output of their construction is formulated qualitatively, so invoking it here as a black box would not lead to an effective numerical exponent $\delta$ as in Theorem \ref{thm:main}.  Blomer--Maga \cite[\S4]{MR3384442} give an explicit and quantitative amplifier in the setting of $\GL_n$.  We will recall their construction (in local language) and then carry out some estimates concerning the restriction to $\GL_n$ of the ``square'' of an amplifier on $\GL_{n+1}$.

Let $F$ be a non-archimedean local field with ring of integers $\mathfrak{o}$, maximal ideal $\mathfrak{p}$, uniformizer $\varpi \in \mathfrak{p}$, and $q := \# \mathfrak{o}/\mathfrak{p}$, so that $q$ is a power of some rational prime $p$.

\subsubsection{Preliminaries on the Hecke algebra}\label{sec:prel-hecke-algebra}
Let $n$ be a natural number.  Set $G := \GL_n(F)$, $K := \GL_n(\mathfrak{o})$.  We equip $G$ with the Haar measure assigning volume one to $K$.  The Hecke algebra $\mathcal{H}_G$ is the space of compactly-supported functions $t : K \backslash G / K \rightarrow \mathbb{C}$ equipped with the convolution product $\ast$.  A basis for $\mathcal{H}_G$ is indexed by $n$-tuples of integers $a = (a_1,\dotsc,a_n) \in \mathbb{Z}^n$ modulo permutations, with the corresponding basis element
\[
  T(a) \in \mathcal{H}_G
\]
given by the characteristic function of the double coset $K \diag(\varpi^{a_1}, \dotsc, \varpi^{a_n}) K$.  Rather than working modulo permutations, one can restrict to $a$ satisfying the dominance condition $a_1 \geq \dotsb \geq a_n$.  We introduce the abbreviation
\[T[j] := T((j,0,\dotsc,0))\] (here the implicit integer $n$ will always be clear from context).

For each Hecke algebra element $t \in \mathcal{H}_G$, we define the adjoint element $t^* \in \mathcal{H}_G$ by the formula $t^*(g) := \overline{t(g^{-1})}$.  Then
\begin{equation*}
  T(a_1,\dotsc,a_n)^*
  =
  T(-a_n,\dotsc,-a_1).
\end{equation*}

A \emph{Satake parameter} of $G$ is an element of $(\mathbb{C}^\times)^n / S(n)$, i.e., a multiset $\alpha = \{\alpha_1, \dotsc, \alpha_n\}$ of nonzero complex numbers.  We write $\pi_\alpha$ for the corresponding unramified induced representation of $G$, normalized so that $\pi_\alpha$ is tempered when each $|\alpha_j| = 1$.  The space $\pi_\alpha^K$ of $K$-fixed vectors is one-dimensional, so $\mathcal{H}_G$ acts on it via a character (i.e., algebra homomorphism)
\[
  \lambda_\alpha : \mathcal{H}_G \rightarrow \mathbb{C},
\]
i.e., for $v \in \pi_\alpha^K$ and $t \in \mathcal{H}_G$, we have $\pi_\alpha(t) v = \lambda_\alpha(t) v$.  The Satake isomorphism says that $\alpha \rightarrow \lambda_\alpha$ defines a bijection between the set of Satake parameters of $G$ and the set of characters of $\mathcal{H}_G$.  We note in passing that every irreducible representation of $G$ having a nontrivial $K$-fixed vector is isomorphic to some $\pi_\alpha$.

Let $\alpha$ be a Satake parameter.  For $\vartheta \geq 0$, we say that $\pi_\alpha$ is \emph{$\vartheta$-tempered} if
\begin{equation}\label{eqn:satake-parameter-unitary-central-character}
  |\alpha_1 \dotsb \alpha_n| = 1
\end{equation}
and
\begin{equation}\label{eqn:satake-parameter-bounds}
  q^{-\vartheta} \leq |\alpha_j| \leq q^{\vartheta}
  \text{ for } j = 1,\dotsc, n.
\end{equation}
The condition \eqref{eqn:satake-parameter-unitary-central-character} says that $\pi_\alpha$ has unitary central character.  We note that $\pi_\alpha$ is $0$-tempered if and only if $\pi_\alpha$ is tempered in the usual sense.

\subsubsection{Construction and properties}
Set
\[
  G := \GL_{n+1}(F), \quad H := \GL_n(F),\]
\[K := \GL_{n+1}(\mathfrak{o}), \quad K_H := \GL_n(\mathfrak{o}),\] with $H$ embedded in $G$ in the usual way as the upper-left block.  Thus, as discussed in \S\ref{sec:unitary-groups-ggp}, $(G,H)$ is a unitary GGP pair over $F$ attached to the split quadratic extension $E = F \times F$.  We equip $G$ and $H$ with measures as above, and define the corresponding Hecke algebras $\mathcal{H}_G$ and $\mathcal{H}_H$.  As usual, we write $Z$ for the center of $G$.  We equip $Z \cong F^\times$ with the Haar measure $d z$ assigning volume one to its maximal compact subgroup $\cong \mathfrak{o}^\times$.

Let $\omega$ be an unramified unitary character of $Z$.  We define a linear map
\begin{equation*}
  \res_{\omega} : \mathcal{H}_G \rightarrow \mathcal{H}_H
\end{equation*}
by the formula: for $g \in H$,
\begin{equation*}
  (\res_\omega t)(g)
  :=
  \int _{z \in Z}
  t(z g)
  \omega(z) \, d z.
\end{equation*}

\begin{lemma}
  Assume that $n$ is fixed.  There is a fixed $q_0$ so that for all $q \geq q_0$, the following assertions hold.
  
  The elements
  \begin{equation*}
    t_j := q^{-\frac{n j}{2}} T[j]
    \in \mathcal{H}_G
    \quad (j \geq 0)
  \end{equation*}
  have the following properties.
  \begin{enumerate}[(i)]
  \item For each Satake parameter $\alpha$ of $G$ for which $\pi_\alpha$ has unitary central character, we have
    \begin{equation}\label{eqn:blomer-maga}
      \max_{1 \leq j \leq n+1}
      |\lambda_\alpha(t_j)| \gg 1.
    \end{equation}
  \item We have
    \begin{equation}\label{eqn:blomer-maga-decomposition}
      t_j \ast t_j^*
      =
      \sum_{i=0}^j
      c_{i j}
      q^{-n i}
      T(i, 0, \dotsc, 0, -i)
    \end{equation}
    for some coefficients $c_{i j} \ll_{j} 1$.

  \item Fix $0 \leq \vartheta < 1/2$.  For each unramified unitary character $\omega$ of $Z$ and each Satake parameter $\beta$ of $H$ for which $\pi_\beta$ is $\vartheta$-tempered, we have
    \begin{equation}\label{eqn:estimate-res-h-j}
      \lambda_\beta(\res_\omega t_j)
      \ll_j
      q^{ - (\frac{1}{2} -\vartheta) j}
    \end{equation}
    and for $i \geq 0$
    \begin{equation}\label{eqn:estimate-res-h-j-ast-h-j}
      \lambda_\beta(\res_{\omega} q^{- n i }
      T(i,0,\dotsc,0,-i)) 
      \ll_i q^{ - (1- 2 \vartheta)) i} \leq 1.
    \end{equation}
    In particular,
    \begin{equation}\label{eqn:estimate-res-h-j-ast-h-j-2}
      \lambda_\beta(\res_\omega(t_j \ast t_j^*)) \ll_j 1.
    \end{equation}
  \end{enumerate}
\end{lemma}
\begin{proof}
  The estimate \eqref{eqn:blomer-maga} is essentially given by Blomer--Maga \cite[Cor 4.3]{MR3384442}.  We outline their proof, noting the minor differences required here.  Let $\Pi(n+1)$ denote the (finite) set consisting of all $a = (a_1,\dotsc,a_{n+1}) \in \mathbb{Z}_{\geq 0}^{n+1}$ with $a_1 \geq \dotsb \geq a_{n+1}$ and $\sum_k a_k = n+1$.  It is shown in \cite[Lem 4.2]{MR3384442} that for some $y_a \in \mathbb{Q}$ ($a \in \Pi(n+1)$) with $|y_a| \ll 1$, we have for large enough $q$ the identity in $\mathcal{H}_G$
  \begin{equation*}
    q^n
    \sum _{a \in \Pi(n+1)}
    y_a
    \prod_{j=1}^{n+1}
    T[a_j] =
    q^{\frac{(n+1)(n+2)}{2}} T(1,\dotsc,1).
  \end{equation*}
  Strictly speaking, the cited reference addresses the case $F = \mathbb{Q}_p$ in classical language, but the same argument applies in our setting.  Since $\pi_\alpha$ is assumed to have unitary central character, we have $|\lambda_\alpha(T(1,\dotsc,1))| = 1$.  The stronger condition $\lambda_\alpha(T(1,\dotsc,1)) = 1$ is assumed in \cite[\S4]{MR3384442}, but the weaker condition $|\lambda_\alpha(T(1,\dotsc,1))| = 1$ suffices to apply the proof of \cite[Lem 4.3]{MR3384442} to see that at least one of the quantities $|\lambda_\alpha(t_j)|$ is bounded from below by $(|\Pi(n+1)| \max_{a \in \Pi(n+1)} |y_a|)^{-1} \gg 1$, as required.

  The identity \eqref{eqn:blomer-maga-decomposition} is again due to Blomer--Maga \cite[Lemma 4.4]{MR3384442} (up to the cosmetic substitution $i \mapsto j-i$).

  The proofs of \eqref{eqn:estimate-res-h-j} and \eqref{eqn:estimate-res-h-j-ast-h-j} require a couple preliminaries to which we now turn.
  
  Let $a = (a_1,\dotsc,a_{n+1}) \in \mathbb{Z}^{n+1}$ and $d \in \mathbb{Z}$.  We observe that for $z = \varpi^{-d} \in Z \cong F^\times$, the double coset
  \[
    z K \diag(\varpi^{a_1},\dotsc,\varpi^{a_{n+1}}) K = K \diag(\varpi^{a_1-d},\dotsc,\varpi^{a_{n+1}-d}) K
  \]
  has nontrivial intersection with $H$ if and only if $d = a_j$ for some $j$, in which case that intersection is the double coset
  \[
    K_H \diag(\varpi^{a_1-d}, \dotsc, \varpi^{a_{j-1}-d}, \varpi^{a_{j+1}-d}, \dotsc, \varpi^{a_{n+1}-d}) K_H.
  \]
  We deduce the identity
  \begin{equation}\label{eqn:restricting-double-coset}
    \res_\omega T(a)
    =
    \sum _{j =1,\dotsc,n+1}^\sharp
    \omega(\varpi^{-a_j})
    T(a_1-a_j,\dotsc,a_{j-1}-a_j,a_{j+1}-a_j,\dotsc,a_{n+1}),
  \end{equation}
  where $\sharp$ signifies that the sum runs over indices $j$ such that the $a_j$ run over the distinct elements of the multiset $\{a_1,\dotsc,a_{n+1}\}$.

  We pause to record the estimate: for any Satake parameter $\beta$ of $H$ for which $\pi_\beta$ is $\vartheta$-tempered, any integers $b_1 \geq \dotsb \geq b_{n}$ and any fixed $\eps > 0$, we have
  \begin{equation}\label{eqn:lambda-alpha-T-a-estimate}
    \lambda_\beta(T(b))
    \ll
    e^{\eps \sum_{j=1}^{n} |b_j|}
    q^{-v(b) + \frac{n+1}{2} |b| }
    \max_{w \in W}
    |\beta_{w(1)}^{b_1}
    \dotsb \beta_{w(n)}^{b_n}|
  \end{equation}
  where
  \begin{equation*}
    |b| := \sum_{k=1}^n b_k, \quad
    v(b) := \sum _{k = 1}^{n} k b_k.  
  \end{equation*}
  This is likely well-known (in sharper forms), but since we could not quickly locate a suitable reference, we sketch a proof.  We use the formula (see \cite[(4.3)]{MR3384442} or \cite[(1.7)]{MR0282982})
  \begin{equation}\label{eqn:lambda-via-P}
    \lambda_\beta(T(b_1,\dotsc,b_n))
    =
    q^{-v(b) + \frac{n+1}{2} |b| }
    P_b(\beta),
  \end{equation}
  where
  \begin{equation}\label{eqn:P-alpha-formula}
    P_b(\beta)
    =
    c(b)
    \sum_{w \in S(n)}
    \beta_{w(1)}^{b_1}
    \dotsb 
    \beta_{w(n)}^{b_n}
    \prod_{i > j}
    \frac{
      1 - q^{-1} \beta_{w(j)} / \beta_{w(i)}
    }
    {
      1 - \beta_{w(j)}    /    \beta_{w(i)}
    }
  \end{equation}
  for some leading constant $c(b) \asymp 1$.  The $\vartheta$-temperedness assumption implies $\beta_j/\beta_i \ll q^{2 \vartheta} \ll q$, so the numerators of the fractions in \eqref{eqn:P-alpha-formula} are $\O(1)$.  The only issue in deducing \eqref{eqn:lambda-alpha-T-a-estimate} is that the denominator has singularities along the hyperplanes $\beta_i = \beta_j$ ($i \neq j$).  For $\delta > 0$, let $\mathcal{D}(\delta)$ denote the set of all Satake parameters $\beta$ that avoid these hyperplanes in the sense that $|1 - \beta_i/\beta_j| > \delta$.  For $\beta \in \mathcal{D}(\delta)$, the denominators in \eqref{eqn:P-alpha-formula} are bounded away from zero, so the required estimate \eqref{eqn:lambda-alpha-T-a-estimate} holds but with the implied constant depending upon $\delta$.  We deduce \eqref{eqn:lambda-alpha-T-a-estimate} in general using Cauchy's theorem, as follows.  For each $\beta$, define a function $f_\beta : \mathbb{C}^{n-1} \rightarrow \mathbb{C}$ by $f_\beta(s) := \lambda_{\beta(s)} (T(b))$, where $\beta(s) := \{\beta_1 e^{s_1}, \dotsc, \beta_n e^{s_n} \}$ and $s_n := - (s_1 + \dotsb + s_{n-1})$.  By a pigeonhole argument, we see that for each $r > 0$ there exists $r_0 > 0$ and $\delta > 0$ (independent of the local field $F$) so that for all $\beta$, there exists tuple of radii $\rho = (\rho_1,\dotsc,\rho_{n-1}) \in (r_0, r)^{n-1}$ for which $\beta(s)$ belongs to $\mathcal{D}(\delta)$ whenever $s$ lies in the polycircle $\mathcal{C}(\rho) := \{ s : |s_j| = \rho_j \}$.  By Cauchy's theorem, it follows that
  \begin{equation}\label{eqn:lambda-alpha-via-alpha-of-s}
    \lambda_{\beta}(T(b))
    =
    f(0)
    \ll (\rho_1 \dotsb \rho_{n-1})^{-1}
    \sup_{s \in \mathcal{C}(\rho)}
    | \lambda_{\beta(s)}(T(b)) |.
  \end{equation}
  We deduce \eqref{eqn:lambda-alpha-T-a-estimate} by fixing $r$ sufficiently small, applying the identity \eqref{eqn:lambda-via-P} to the RHS of the bound \eqref{eqn:lambda-alpha-via-alpha-of-s}, and observing that
  \begin{equation*}
    |\beta(s)_{w(1)}^{b_1}
    \dotsb \beta(s)_{w(n)}^{b_n}|
    \leq
    e^{r \sum_{j=1}^n |b_j|}
    |\beta_{w(1)}^{b_1}
    \dotsb \beta_{w(n)}^{b_n}|.
  \end{equation*}

  Below, we will frequently apply the identity
  \begin{equation*}
    \lambda_\beta(T(b_1 + d,\dotsc,b_n + d))
    =
    (\beta_1 \dotsb \beta_n)^d
    \lambda_\beta(T(b_1,\dotsc,b_n))
  \end{equation*}
  and the hypothesis \eqref{eqn:satake-parameter-unitary-central-character}.

  We now establish \eqref{eqn:estimate-res-h-j}.  We observe first using \eqref{eqn:restricting-double-coset} that
  \begin{equation}\label{eqn:res-T-j}
    \res_\omega T[j]
    =
    \omega(\varpi^{-j})
    T(0,-j,\dotsc,-j)
    +
    T(j,0,\dotsc,0).
  \end{equation}
  On the other hand, we verify readily using \eqref{eqn:lambda-alpha-T-a-estimate} that
  \begin{equation*}
    \lambda_\beta(T(j,0,\dotsc,0))
    \ll_j q^{\frac{nj}{2}- (\frac{1}{2} -\vartheta) j}
  \end{equation*}
  and
  \begin{equation*}
    \lambda_\beta(T(0,-j,\dotsc,-j))
    \asymp
    \lambda_\beta(T(j,0,\dotsc,0))
    \ll_j q^{\frac{n j}{2} - (\frac{1}{2} -\vartheta) j}.
  \end{equation*}
  This completes the proof of \eqref{eqn:estimate-res-h-j}.

  We now verify \eqref{eqn:estimate-res-h-j-ast-h-j}.  For $i=0$, we have the adequate estimate
  \[
    \lambda_\beta(\res_\omega(T(0,\dotsc,0))) = \lambda_\beta(T(0,\dotsc,0)) = 1 \ll 1.
  \]
  For $0 < i \leq j$, we have by \eqref{eqn:restricting-double-coset} that
  \begin{equation}\label{eqn:res-T-2-j-minus-i-etc}
    \begin{split}
      \res_\omega T(i, 0, \dotsc, 0, -i) &= \omega(\varpi^{-i}) T(-i,\dotsc,-i,-2 i)
      \\
      &+
      T(i,0,\dotsc,0,i) \\
      &+ \omega(\varpi^{i}) T(2i,i,\dotsc,i).
    \end{split}
  \end{equation}
  By \eqref{eqn:lambda-alpha-T-a-estimate}, we have
  \begin{equation*}
    \lambda_\beta(
    T(-i,\dotsc,-i,-2 i)
    )
    \asymp
    \lambda_\beta(T(0,\dotsc,0,-i))
    \ll_i
    q^{i (\frac{n}{2} - (\frac{1}{2} - \vartheta))},
  \end{equation*}
  \begin{equation*}
    \lambda_\beta(T(i,0,\dotsc,0,-i))
    \ll_i
    q^{i (n - (1 - 2 \vartheta))}
  \end{equation*}
  and
  \begin{equation*}
    \lambda_\beta(T(2i,
    i,\dotsc,i))
    \asymp
    \lambda_\beta(T(i,\dotsc,0,0))
    \ll_i
    q^{i (\frac{n}{2} - (\frac{1}{2} - \vartheta))}.
  \end{equation*}
  These identities and estimates combine to give \eqref{eqn:estimate-res-h-j-ast-h-j}.
\end{proof}

\subsection{Pretrace inequality}\label{sec:pretrace-inequality}
Let $(G,H)$ be a GGP pair over a number field $F$.  As usual, we write $Z$ for the center of $G$ and $\bar{G} = G/Z$ for the adjoint group.  Let $\pi$ be a cuspidal automorphic representation of $G(\mathbb{A})$ with unitary central character $\omega : Z(\mathbb{A}) \rightarrow \U(1)$.  We equip $\pi$ with the Petersson inner product defined by integrating over $[\bar{G}]$.  We write $\mathcal{B}(\pi)$ for an orthonormal basis consisting of vectors isotypic under some chosen maximal compact subgroup of $G(\mathbb{A})$.

For $f_1, f_2 \in C_c^\infty(G(\mathbb{A}))$, we write $f_1 \ast f_2$ for their convolution.  For $f \in C_c^\infty(G(\mathbb{A}))$, we set
\begin{equation*}
  f^*(g) := \overline{f(g^{-1})},
\end{equation*}
\begin{equation*}
  f^\omega (g) := \int _{z \in Z(\mathbb{A})} \omega(z) f(z g) \, d z.
\end{equation*}
\begin{lemma}\label{lem:pretrace-inequality}
  Let $f \in C_c^\infty(G(\mathbb{A}))$.  Let $\Psi : [H] \rightarrow \mathbb{C}$ be a measurable function of rapid decay.  Then
  \[
    \sum _{\varphi \in \mathcal{B}(\pi)} \left\lvert \int _{[H]} \pi(f) \varphi \cdot \bar{\Psi} \right\rvert^2 \leq \int _{x, y \in [H]} \bar{\Psi}(x) \Psi(y) \sum _{\gamma \in \bar{G}(F)} (f \ast f^*)^{\omega}(x^{-1} \gamma y) \, d x \, d y,
  \]
  where $\omega$ is the central character of $\pi$.  Each of the above iterated sums/integrals converges absolutely.  Moreover, the same estimate holds for any orthonormal basis $\mathcal{B}(\pi)$.
\end{lemma}
\begin{proof}
  Morally, we may conclude via the pretrace formula, which says that the RHS of the required inequality is the integral of the LHS over all $\pi$ occurring in the spectral decomposition of $L^2([G],\omega)$ -- including the continuous spectrum, as we have not assumed that $[G]$ is compact.  While such an argument could likely be implemented (following a preliminary discussion on spectral decomposition), it seems simpler to verify the required conclusion directly, as follows.  Using the rapid decay of $\Psi$ and the smoothness of $f$, we see that $\varphi \mapsto \int _{[H]} \pi(f) \varphi \cdot \bar{\Psi}$ defines a bounded linear functional on the Hilbert space $\pi$.  We may represent that functional by a unique vector $\varphi_0 \in \pi$.  By Parseval, we see that $\langle \varphi_0, \varphi_0 \rangle$ expands to the LHS of the required inequality, showing incidentally that the LHS is independent of the choice of $\mathcal{B}(\pi)$.  Our task is then to bound $\langle \varphi_0, \varphi_0 \rangle$ by the RHS.  By combining the invariance of $f$ under an open subgroup of the finite adelic points of $G$ with an application of the Dixmier--Malliavin lemma to the archimedean points of $G$, we see that $\varphi_0$ is smooth; since $\pi$ is cuspidal, it follows that $\varphi_0$ is of rapid decay (a feature which simplifies the details of the proof, but is not ultimately necessary).  By construction, $\langle \varphi_0, \varphi_0 \rangle = \int _{[H]} \pi(f) \varphi_0 \cdot \bar{\Psi}$.  By folding up this last integral, we obtain
  \[
    \langle \varphi_0, \varphi_0 \rangle = \int _{x \in [H]} \bar{\Psi}(x) \int _{y \in [\bar{G}]} \varphi_0(y) \sum _{\gamma \in \bar{G}(F)} f^{\omega}(x^{-1} \gamma y).
  \]
  Using the rapid decay of $\Psi$ and $\varphi_0$ and crudely bounding the sum over $\gamma$, we see that this iterated integral/sum converges absolutely.  We may thus rearrange it and apply Cauchy--Schwarz to obtain
  \[
    \langle \varphi_0, \varphi_0 \rangle \leq \int _{y \in [\bar{G}]} \left\lvert \int _{x \in [H]} \bar{\Psi}(x) \sum _{\gamma \in \bar{G}(F)} f^\omega(x^{-1} \gamma y) \, d x \right\rvert^2 \, d y.
  \]
  Opening the square, executing the $y$-integral and unfolding, we readily obtain the RHS of the required inequality.
\end{proof}

\section{Reduction of the proof of the
  main theorem}\label{sec:reduction-proof}
In this section, we prove Theorem \ref{thm:main}, assuming Theorem \ref{thm:construct-test-function}.

\subsection{Setting}
We recall the setting of \S\ref{sec:main-results}.  Let $F$ be a fixed number field.  Let $(G,H)$ be a fixed unitary GGP pair over $F$ with standard representation $G \hookrightarrow \GL(V/E)$.  We choose a distinguished archimedean place $\mathfrak{q}$ of $F$.  We recall from \S\ref{sec:main-results} the following further assumptions:
\begin{itemize}
\item $H$ is anisotropic, so that $[H]$ is compact.
\item $G(F_\mathfrak{p})$ and $H(F_\mathfrak{p})$ are compact for all archimedean places $\mathfrak{p}$ other than the distinguished archimedean place $\mathfrak{q}$.
\end{itemize}
We fix a finite set $S$ of places of $F$, taken large enough in the sense of \S\ref{sec:assumpt-conc-s}.

As usual, we write $Z$ for the center of $G$ and $\bar{G} = G/Z$.  We write $n+1$ for the $E$-dimension of the hermitian space $V$, so that $(G,H)$ is a form of $(\U_{n+1}, \U_n)$.  We denote by $\eps$ some fixed positive quantity, not necessarily the same in each occurrence.

Recall the families $\mathcal{F}_T$ defined in \S\ref{sec:main-results}.  Let $T$ be a large positive real, and let $(\pi,\sigma) \in \mathcal{F}_T$.  We assume that $\sigma$ is $\vartheta$-tempered for some fixed $\vartheta \in [0,1/2)$.  To prove Theorem \ref{thm:main}, we must show that
\begin{equation}\label{eqn:reduction-goal-bound-1}
  \mathcal{L}(\pi,\sigma)
  \ll
  T^{n(n+1)/2 - \delta}
\end{equation}
for suitable fixed $\delta > 0$.

\subsection{Reduction to bounds
  for global periods}
Recall from \S\ref{sec:branch-coeff} that $\sigma_S \subseteq \sigma$ denotes the unramified-outside-$S$ subspace.  We choose an isometric factorization $\sigma_S = \otimes_{\mathfrak{p} \in S} \sigma_\mathfrak{p}$.  Let $\Psi = \otimes_{\mathfrak{p} \in S} \Psi_\mathfrak{p} \in \sigma_S$ be a factorizable vector.  Let $f_0 = f_{0,S} \otimes f_{0}^S \in C_c^\infty(G(\mathbb{A}))$, with $f_{0,S} = \otimes_{\mathfrak{p} \in S} f_{0,\mathfrak{p}}$ and $f_0^S = \otimes_{\mathfrak{p} \notin S} 1_{G(\mathbb{Z}_\mathfrak{p})}$, be a factorizable test function, unramified outside $S$.  The period formula \eqref{eqn:P-equals-L-H} implies that
\begin{equation*}
  \sum _{\varphi \in \mathcal{B}(\pi)}
  \left\lvert \int _{[H]}
    \pi(f_0) \varphi \cdot \overline{\Psi}
  \right\rvert^2
  =
  \mathcal{L}(\pi,\sigma)
  \prod _{\mathfrak{p} \in S}
  I_\mathfrak{p},
\end{equation*}
where $\mathcal{B}(\cdot)$ is an orthonormal basis chosen as in \S\ref{sec:uniform-distinction} and
\begin{equation*}
  I_\mathfrak{p} :=
  \sum _{v \in \mathcal{B}(\pi_{0,\mathfrak{p}})}
  \mathcal{Q}_\mathfrak{p}(
  \pi(f_{0,\mathfrak{p}}) v \otimes
  \Psi_\mathfrak{p}),
\end{equation*}
with $\mathcal{Q}_\mathfrak{p}$ the quadratic form defined by integration over $H(F_\mathfrak{p})$ as in \S\ref{sec:local-dist-matr}.

For $\mathfrak{p} \in S - \{\mathfrak{q} \}$, we fix $f_{0,\mathfrak{p}} \in C_c^\infty(G(F_\mathfrak{p}))$ as follows.
\begin{itemize}
\item If $\mathfrak{p}$ is archimedean, we take for $f_{0,\mathfrak{p}}$ an ``approximate identity:'' we require that $\int _{G(F_\mathfrak{p})} f_{0,\mathfrak{p}} = 1$ and that $f_{0,\mathfrak{p}}$ be supported in a small enough neighborhood of the identity element of the compact Lie group $G(F_\mathfrak{p})$.
\item If $\mathfrak{p}$ is non-archimedean, we take for $f_{0,\mathfrak{p}}$ the normalized characteristic function of some small enough compact open subgroup $U_\mathfrak{p}$ of the $p$-adic group $G(F_\mathfrak{p})$.  The meaning of ``normalized'' is that $\int_{G(F_\mathfrak{p})} f_{0,\mathfrak{p}} = 1$.
\end{itemize}
Using Lemmas \ref{lem:uniform-distinction-arch} and \ref{lem:padic-uniform-distinction}, we may find for each $\mathfrak{p} \in S - \{\mathfrak{q} \}$ a unit vector $\Psi_\mathfrak{p} \in \sigma_\mathfrak{p}$ so that $I_\mathfrak{p} \gg 1$.

At the distinguished place $\mathfrak{q}$, we choose $f_{0,\mathfrak{q}}$ and $\Psi_\mathfrak{q} \in \sigma_\mathfrak{q}$ according to Theorem \ref{thm:construct-test-function}, applied to $(\pi_{0,\mathfrak{q}},\sigma_\mathfrak{q})$ and with the exponent $\kappa$ taken sufficiently small (in particular, $\kappa \leq \eps$).  We may and shall assume that $f_{0,\mathfrak{q}}$ is supported sufficiently close to the identity element.

With these choices, we obtain
\[
  \prod_{\mathfrak{p} \in S} I_\mathfrak{p} \gg T^{-n^2/2 - \eps}.
\]
Our task \eqref{eqn:reduction-goal-bound-1} thereby reduces to showing that
\begin{equation}\label{eqn:reduction-goal-bound-2}
  \sum _{\varphi \in \mathcal{B}(\pi)}
  \left\lvert \int _{[H]}
    \pi(f_0) \varphi \cdot \overline{\Psi}
  \right\rvert^2
  \ll
  T^{n/2 - \delta}.
\end{equation}

\subsection{Application of amplified pretrace inequality}\label{sec:ampl-pretr-form}

\subsubsection{Choice of primes at which to amplify}
Let $L$ be a positive parameter with $\log L \asymp \log T$.  We eventually take $L$ to be a \emph{small} positive power of $T$ (see \eqref{eqn:L-optimal-choice} below), but do not assume it is such yet.

Let $\mathbb{P}$ denote the set of finite primes $\mathfrak{p}$ of $F$ with the following properties:
\begin{itemize}
\item The absolute norm $q_\mathfrak{p}$ of $\mathfrak{p}$ lies in the interval $[L, 2 L]$.
\item $\mathfrak{p}$ splits the field extension $E/F$, so that $E_\mathfrak{p} \cong F_\mathfrak{p} \times F_\mathfrak{p}$.
\end{itemize}
By a weak form of the prime number theorem in arithmetic progressions for $F$, we have
\begin{equation}\label{eqn:weak-PNT}
  T^{-\eps} L
  \ll |\mathbb{P}| \ll
  L.
\end{equation}

\subsubsection{Some notation}
For $\mathfrak{p} \notin S$, we write $\mathcal{H}_{G(F_\mathfrak{p})}$ for the corresponding spherical Hecke algebra with respect to $G(\mathbb{Z}_\mathfrak{p})$, and $\mathcal{H}_{G}^S := \otimes_{\mathfrak{p} \notin S} ' \mathcal{H}_{G(F_\mathfrak{p})}$ for the restricted tensor product of these algebras.

For $\mathfrak{p} \in \mathbb{P}$, we have $G(F_\mathfrak{p}) \cong \GL_n(F_\mathfrak{p})$ (see \S\ref{sec:unitary-groups}).  We write $T_\mathfrak{p}(a)$ ($a \in \mathbb{Z}^{n+1}$) and $t_{\mathfrak{p},j}$ ($j \geq 0$) for the elements of $\mathcal{H}_{G(F_\mathfrak{p})}$ defined in \S\ref{sec:constr-an-ampl}.  By abuse of notation, we identify $t_{\mathfrak{p},j}$ with its image in $\mathcal{H}_G^S$.

We write $\lambda_\pi$ for the character of $\mathcal{H}_G^S$ attached to $\pi$.

\subsubsection{Amplification on the spectral side}
By \eqref{eqn:blomer-maga}, we have $\sum_{j=1}^{n+1} |\lambda_{\pi}(t_{\mathfrak{p},j})| \gg 1$ for each $\mathfrak{p} \in \mathbb{P}$.  By the pigeonhole principle, we may thus find $j \in \{1, \dotsc, n+1\}$ so that
\begin{equation*}
  \sum _{\mathfrak{p} \in \mathbb{P} }
  |\lambda_\pi(t_{\mathfrak{p},j})|
  \gg |\mathbb{P}|.
\end{equation*}
For $\mathfrak{p} \in \mathbb{P}$, we set $x_{\mathfrak{p}} := \overline{\sgn (\lambda_{\pi}(t_{\mathfrak{p},j}))}$, where $\sgn(0) := 0$ and $\sgn(z) :=z/|z|$ for $z \neq 0$.  In particular, $|x_{\mathfrak{p}}| \leq 1$.  For each $\varphi \in \pi$, we then have
\[
  \pi \left(f_{0,S} \otimes \sum_{\mathfrak{p} \in \mathbb{P}} x_\mathfrak{p} t_{\mathfrak{p},j} \right) \varphi = \left( \sum _{\mathfrak{p} \in \mathbb{P} } \left\lvert \lambda_{\pi}(t_{\mathfrak{p},j}) \right\rvert \right) \pi(f_0) \varphi.
\]
Setting
\[
  \mathcal{R}(j) := \sum _{\varphi \in \mathcal{B}(\pi) } \left\lvert \int _{[H]} \pi \left(f_{0,S} \otimes \sum_{\mathfrak{p} \in \mathbb{P}} x_\mathfrak{p} t_{\mathfrak{p},j} \right) \varphi \cdot \bar{\Psi} \right\rvert^2,
\]
the proof of \eqref{eqn:reduction-goal-bound-2} reduces to that of the estimate $\mathcal{R}(j) \ll T^{n/2-\delta} L^2$.

\subsubsection{Amplification on the geometric side}
Let $\omega$ denote the central character of $\pi$.  By Lemma \ref{lem:pretrace-inequality}, we have
\begin{equation*}
  \mathcal{R}(j) \leq \sum _{\mathfrak{p}_1, \mathfrak{p}_2 \in \mathbb{P} } x_{\mathfrak{p}_1} \bar{x}_{\mathfrak{p}_2} \mathcal{R}(\mathfrak{p}_1,\mathfrak{p}_2,j),
\end{equation*}
where
\[
  \mathcal{R}(\mathfrak{p}_1,\mathfrak{p}_2,j) := \int _{x, y \in [H]} \bar{\Psi}(x) \Psi(y) \sum _{\gamma \in \bar{G}(F)} f^{(\mathfrak{p}_1,\mathfrak{p}_2,j),\omega}(x^{-1} \gamma y)
\]
with $f^{(\mathfrak{p}_1,\mathfrak{p}_2,j)} := (f_{0,S} \otimes t_{\mathfrak{p}_1,j}) \ast (f_{0,S} \otimes t_{\mathfrak{p}_2,j})^*$.  In the special case $\mathfrak{p}_1 = \mathfrak{p}_2$, we recall from \eqref{eqn:blomer-maga-decomposition} that
\begin{equation*}
  t_{\mathfrak{p_1},j} \ast t_{\mathfrak{p}_1,j}^*
  =
  \sum_{i=0}^j
  c_{i j}
  q_{\mathfrak{p}_1}^{-n i}
  T_{\mathfrak{p}_1}(i, 0, \dotsc, 0, -i)
\end{equation*}
with $c_{i j} \ll 1$.  This yields a decomposition $\mathcal{R}(\mathfrak{p}_1,\mathfrak{p}_1,j) = \sum _{i=0}^j c_{i j} \mathcal{R}'(\mathfrak{p}_1,i)$ indexed by $0 \leq i \leq j \leq n+1$.  Since $|x_{\mathfrak{p}}| \leq 1$, our task reduces to verifying that
\begin{equation}\label{eqn:R-decomposed}
  \sum _{
    \mathfrak{p}_1 \neq \mathfrak{p}_2 \in \mathbb{P} 
  }
  \sum _{{\mathbf{j}}=1}^{n+1}
  |\mathcal{R}(\mathfrak{p}_1, \mathfrak{p}_2, {\mathbf{j}})|
  +
  \sum _{
    \mathfrak{p}_1 = \mathfrak{p}_2 \in \mathbb{P} 
  }
  \sum _{{\mathbf{j}}=0}^{n+1}
  |\mathcal{R}'(\mathfrak{p}_1, {\mathbf{j}})|
  \ll
  T^{n/2-\delta} L^2.
\end{equation}
(We have replaced $j$ with the bold symbol $\mathbf{j}$ since this variable will remain active for several pages.)

\subsubsection{Definitions of the test function $f$}\label{sec:defin-test-funct}
We henceforth focus on an individual term on the LHS of \eqref{eqn:R-decomposed}, corresponding to some specific values of ${\mathbf{j}},\mathfrak{p}_1,\mathfrak{p}_2$.  We emphasize that the test functions defined in what follows depend upon these values, and in particular, upon ${\mathbf{j}}$.  We define $f = \otimes f_\mathfrak{p} \in C_c^\infty(G(\mathbb{A}))$, as follows.
\begin{itemize}
\item For $\mathfrak{p} \in S$, we set $f_\mathfrak{p} := f_{0,\mathfrak{p}} \ast f_{0,\mathfrak{p}}^*$.
\item For $\mathfrak{p} \notin S \cup \{\mathfrak{p}_1, \mathfrak{p}_2\}$, we take $f_\mathfrak{p} := 1_{G(\mathbb{Z}_{\mathfrak{p}})}$.
\item For $\mathfrak{p}_1 \neq \mathfrak{p}_2$, we set
  \[
    f_{\mathfrak{p}_1} = t_{\mathfrak{p}_1,{\mathbf{j}}} = q_{\mathfrak{p}_1}^{- \frac{n {\mathbf{j}}}{2}} T_{\mathfrak{p}_1}({\mathbf{j}},0,\dotsc,0),
  \]
  \[
    f_{\mathfrak{p}_2} = t_{\mathfrak{p}_2,{\mathbf{j}}}^* = q_{\mathfrak{p}_2}^{- \frac{n {\mathbf{j}}}{2}} T_{\mathfrak{p}_2}(0,\dotsc, 0,-{\mathbf{j}}).
  \]
\item For $\mathfrak{p}_1 = \mathfrak{p}_2$, we set
  \[
    f_{\mathfrak{p}_1} = q_{\mathfrak{p}_1}^{-n {\mathbf{j}}} T_{\mathfrak{p}_1}({\mathbf{j}},0,\dotsc,0,-{\mathbf{j}}).
  \]
\end{itemize}
We set
\begin{equation*}
  f^\sharp := \otimes f_\mathfrak{p}^\sharp,
  \quad
  f_\mathfrak{p}^\sharp(g) := \int _{Z_\mathfrak{p} }
  \omega_\mathfrak{p}(z) f_\mathfrak{p}(z g ) \, d z,
\end{equation*}
so that $f^\sharp = f^{\omega}$ in the notation of \S\ref{sec:pretrace-inequality}.  Then for $\mathfrak{p}_1 \neq \mathfrak{p}_2$ (resp. $\mathfrak{p}_1 = \mathfrak{p}_2$), the quantity $\mathcal{R}(\mathfrak{p}_1,\mathfrak{p}_2,{\mathbf{j}})$ (resp. $\mathcal{R}'(\mathfrak{p}_1, {\mathbf{j}}$)) is given by
\begin{equation}\label{eqn:R-equals-M-plus-E}
  \mathcal{R}(f)
  :=
  \int _{x, y \in [H]}
  \overline{\Psi}(x) \Psi(y)
  \sum _{\gamma \in \bar{G}(F)}
  f^\sharp(x^{-1} \gamma y) \, d x \, d y
  = \mathcal{M} + \mathcal{E},
\end{equation}
where $\mathcal{M}$ denotes the contribution from $\gamma \in H(F) \hookrightarrow \bar{G}(F)$ and $\mathcal{E}$ denotes the remaining contribution.

We note for future reference that
\begin{equation}\label{eqn:f-sharp-L-infinity-outside-S}
  \prod_{\mathfrak{p} \notin S}
  \|f^\sharp_\mathfrak{p} \|_{\infty}
  \ll
  L^{-n {\mathbf{j}}},
\end{equation}
which follows from the definition of $f_\mathfrak{p}$ for $\mathfrak{p} \in \{\mathfrak{p}_1, \mathfrak{p}_2\}$.

\subsubsection{Plan}
We estimate $\mathcal{M}$ in \S\ref{sec:main-term-estimates} and $\mathcal{E}$ in \S\ref{sec:error-term-estimates}.  We then combine these estimates and sum them in \S\ref{sec:optimization} to deduce the required bound \eqref{eqn:R-decomposed}.

\subsection{Main term estimates}\label{sec:main-term-estimates}
Since the quotient $[H]$ is compact, the Petersson inner product on $\sigma$, defined in general by integration over $[H/Z_H]$, may be written
\begin{equation*}
  \langle \Psi_1, \Psi_2 \rangle = c \int_{[H]} \Psi_1 \bar{\Psi}_2
\end{equation*}
for some fixed $c > 0$.  To simplify notation, let us renormalize the inner product on $\sigma$ so that this last identity holds with $c = 1$; doing so has no effect on the estimates to be proved.  We may then unfold $\mathcal{M}$ to
\begin{align*}
  \mathcal{M} &=
                \int _{x \in [H]} \int_{y \in H(\mathbb{A})}
                \bar{\Psi}(x) \Psi(y)
                f^\sharp(x^{-1} y)
                \, d x \, d y
  \\
              &=
                \int _{g \in H(\mathbb{A})}
                f^\sharp(g) \langle g \Psi, \Psi  \rangle \, d g
  \\
              &= \prod_{\mathfrak{p} \in S \cup \{\mathfrak{p}_1, \mathfrak{p}_2\}}
                I_\mathfrak{p},
                \quad
                I_\mathfrak{p}  := 
                \int _{g \in H(F_\mathfrak{p})}
                f^\sharp_\mathfrak{p} (g) \langle g \Psi_\mathfrak{p} , \Psi_\mathfrak{p}   \rangle
                \, d g.
\end{align*}
We estimate each of these local integrals separately:
\begin{itemize}
\item Since $\Psi_\mathfrak{q}$ is a unit vector, we may bound $I_\mathfrak{q}$ by $\int_{H(F_\mathfrak{q})} |f^\sharp_{\mathfrak{q}}|$, which is $\ll T^{n/2+\eps}$ thanks to part \eqref{item:f-ell-one-norm-on-H} of Theorem \ref{thm:construct-test-function}.
\item For $\mathfrak{p} \in S - \{\mathfrak{q} \}$, we have $\int_{H(F_\mathfrak{p})} |f^\sharp_{\mathfrak{p}}| \ll 1$ and $\|\Psi_\mathfrak{p} \| = 1$, so $I_\mathfrak{p} \ll 1$.
\item In the case $\mathfrak{p}_1 \neq \mathfrak{p}_2$, we see from \eqref{eqn:estimate-res-h-j} that $I_{\mathfrak{p}_k} \ll q_{\mathfrak{p}_k}^{-(1/2-\vartheta) k} \ll L^{-(1/2-\vartheta)k}$ for $k=1,2$.
\item In the case $\mathfrak{p}_1 = \mathfrak{p}_2$, we see from \eqref{eqn:estimate-res-h-j-ast-h-j} that $I_{\mathfrak{p}_1} \ll q_{\mathfrak{p}_1}^{-(1-2 \vartheta) {\mathbf{j}}} \ll L^{-(1-2 \vartheta) {\mathbf{j}}}$.
\end{itemize}
Therefore in all cases,
\begin{equation*}
  \mathcal{M} \ll
  T^{n/2+\eps}
  L^{-(1-2 \vartheta) {\mathbf{j}}}.
\end{equation*}

\subsection{Error term estimates}\label{sec:error-term-estimates}
The estimation of $\mathcal{E}$ consists of a few steps.

\subsubsection{Preliminary reductions}\label{sec:error-prel-reduct}
We observe that the integrand in \eqref{eqn:R-equals-M-plus-E} is invariant in both variables under $\prod_{\mathfrak{p} \notin S} H(\mathbb{Z}_\mathfrak{p})$.  By our hypotheses concerning $S$ (see \S\ref{sec:assumpt-conc-s}) and strong approximation, we may thus replace the integrals over $[H]$ with integrals over the $S$-arithmetic quotient $\Gamma \backslash H_S$.  This quotient is compact thanks to our hypothesis that $H$ is anisotropic.  We fix a compact fundamental domain $\mathcal{D}$ for this quotient.  We fix a compact factorizable neighborhood $\Theta_S = \prod_{\mathfrak{p} \in S} \Theta_\mathfrak{p}$ of the identity element of $H_S$ with $\mathcal{D} \subseteq \Theta_S$.  Then by our normalization of measures,
\begin{equation*}
  \mathcal{E} \leq
  \int _{x, y \in \Theta_S}
  \left\lvert
    \Psi(x) \Psi(y)
    \sum _{\gamma \in \bar{G}(F)
      - H(F)}
    f^\sharp(x^{-1} \gamma y)
  \right\rvert \, d x \, d y.
\end{equation*}
Setting
\[\Sigma :=
  \{ \gamma \in \bar{G}(F) - H(F) : f^\sharp(x^{-1} \gamma y) \neq 0 \text{ for some } x,y \in \Theta_S \},
\]
we obtain
\begin{equation*}
  \mathcal{E} \leq
  |\Sigma|
  \max_{\gamma \in \Sigma}
  I(\gamma),
  \quad
  I(\gamma) :=
  \int _{x, y \in \Theta_S}
  \left\lvert
    \Psi(x) \Psi(y)
    f^\sharp(x^{-1} \gamma y)
  \right\rvert \, d x \, d y.
\end{equation*}
We estimate $|\Sigma|$ in \S\ref{sec:bounds-sigma} and $I(\gamma)$ in \S\ref{sec:bounds-igamma}.  The resulting bound for $\mathcal{E}$ is then recorded in \S\ref{sec:error-cal-E-summary}.

In what follows, we write
\[
  d_{H_\mathfrak{q}} : \bar{G}(F_\mathfrak{q}) \rightarrow [0,1]
\]
for the function attached in \S\ref{sec:quant-dist-subgr} to the GGP pair $(G(F_\mathfrak{q}), H(F_\mathfrak{q}))$.

\subsubsection{Bounds for $|\Sigma|$}\label{sec:bounds-sigma}

\begin{lemma}\label{lem:count-relevant-gamma}
  Recall the parameter $\mathfrak{j} \in \{0, \dotsc, n+1\}$ fixed in \S\ref{sec:defin-test-funct}.  We have
  \begin{equation}\label{eqn:Sigma-cardinality}
    |\Sigma|  \ll
    L^{2 (n+1)^2 {\mathbf{j}}}
  \end{equation}
  Moreover, for each $\gamma \in \Sigma$, we have
  \begin{equation}\label{eqn:d-H-gamma}
    d_{H_\mathfrak{q}}(\gamma) \gg
    L^{-{\mathbf{j}}}
  \end{equation}
\end{lemma}
\begin{remark}\label{rmk:opt-1}
  The estimate \eqref{eqn:Sigma-cardinality} is likely far from optimal (see the sentence following \eqref{eqn:cardinality-Sigma-D-n-plus-1} for details).  The factor of $2$ in the exponent could likely be removed with a bit more work, and the exponent could likely be reduced further for large $n$ by counting more carefully.
\end{remark}
\begin{remark}
  The content of Lemma \ref{lem:count-relevant-gamma} is best illustrated by passing to split unitary groups, i.e., general linear groups, where it amounts in the simplest case to (a weakened form of) the following easy exercise:

  \emph{For a natural number $\ell$, the number of matrices $\gamma \in \GL_{n+1}(\mathbb{Z})$ for which $\det(\gamma) = \ell^{n+1}$ and $|\ell^{-1} \gamma - 1| \leq \eps$ is $\O(\ell^{(n+1)^2})$.  Moreover, if such a matrix does not lie in $\GL_n(\mathbb{R})$ modulo the center, then for some $k \in \{1, \dotsc, n\}$, one of the entries $\gamma_{k,n+1}$ or $\gamma_{n+1,k}$ is a nonzero integer, hence $\gg 1$, while the entry $\gamma_{n+1,n+1}$ is $\asymp \ell$.}

  The proof of Lemma \ref{lem:count-relevant-gamma} is mildly more complicated due to the overhead coming from working $S$-arithmetically over a number field, but requires no significant new ideas.  The reader might thus wish to skip ahead to \S\ref{sec:bounds-igamma} on a first reading.
\end{remark}

The proof of Lemma \ref{lem:count-relevant-gamma} invokes the following preliminary result.  We write $\mathfrak{P}$ for a place of $E$ and, if $\mathfrak{P}$ is finite, $\mathcal{O}_\mathfrak{P} := \mathbb{Z}_{\mathfrak{P}}$ for the ring of integers of the completion $E_\mathfrak{P}$.
\begin{lemma}\label{lem:counting-PGL-n-plus-1-E}
  Fix a finite set $S_E$ of places of the number field $E$, containing every archimedean place.  For $\mathfrak{P} \in S_E$, suppose given a fixed compact neighborhood $\Omega_\mathfrak{P}$ of the identity element in $\PGL_{n+1}(E_\mathfrak{P})$.  For $\mathfrak{P} \notin S_E$, suppose given an element $D_\mathfrak{P} \geq 1$ of the value group of $E_\mathfrak{P}^\times$, with $D_\mathfrak{P} = 1$ for almost all $\mathfrak{P}$, and set
  \begin{equation*}
    \Omega_\mathfrak{P}
    := \text{ image in $\PGL_{n+1}(E_\mathfrak{P})$ of }
    \{g \in M_{n+1}(\mathcal{O}_\mathfrak{P}) : |\det
    g|_\mathfrak{P}   = 1 / D_\mathfrak{P} \},
  \end{equation*}
  so that $\Omega_\mathfrak{P} = \PGL_{n+1}(\mathcal{O}_\mathfrak{P})$ for almost all $\mathfrak{P}$.  Set
  \[
    \Omega := \prod_{\mathfrak{P}} \Omega_\mathfrak{P} \subseteq \PGL_{n+1}(\mathbb{A}_E), \quad D := \prod_{\mathfrak{P} \notin S_E} D_\mathfrak{P}.
  \]
  \begin{enumerate}[(i)]
  \item We have
    \begin{equation*}
      |\PGL_{n+1}(E) \cap \Omega| \ll D^{n+1}.
    \end{equation*}
  \item Suppose moreover that the following conditions holds.
    \begin{itemize}
    \item For each $\mathfrak{P} \in S_E$, the identity neighborhood $\Omega_\mathfrak{P}$ is of the form
      \begin{equation}\label{eqn:Omega-special-form}
        \Omega_{\mathfrak{P}}
        = 
        \Omega_{\mathfrak{P}}^1
        \Omega_{\mathfrak{P}}^0
        \Omega_{\mathfrak{P}}^1,
      \end{equation}
      where
      \begin{itemize}
      \item $\Omega_{\mathfrak{P}}^1$ is a fixed compact neighborhood of the identity element in $\GL_n(E_\mathfrak{P})$, included in $\PGL_{n+1}(E_{\mathfrak{P}})$ as the upper left $n \times n$ block, and
      \item $\Omega_{\mathfrak{P}}^0$ is a fixed sufficiently small compact neighborhood of the identity element in $\PGL_{n+1}(E_\mathfrak{P})$.
      \end{itemize}
    \item The set of places $S_E$ is large enough that $\mathcal{O}[1/S_E]$, the localization is a principal ideal domain.
    \end{itemize}
    Let $\gamma \in \PGL_{n+1}(E) \cap \Omega$.  Write
    \begin{equation}\label{eqn:gamma-a-b-c-d}
      \gamma = \begin{pmatrix}
        a & b \\
        c & d
      \end{pmatrix}
    \end{equation}
    with $a,b,c,d$ matrices of respective dimensions $n \times n, n \times 1, 1 \times n, 1 \times 1$, well-defined up to simultaneous scaling.  Then $d \neq 0$.  Moreover, for each subset $T$ of $S_E$, we have
    \begin{equation}\label{eqn:b-neq-0-impl}
      b \neq 0
      \implies
      \prod_{\mathfrak{P} \in T}
      |b/d|_{E_\mathfrak{P}} \gg D^{-\frac{1}{n+1}},
    \end{equation}
    \begin{equation}\label{eqn:c-neq-0-impl}
      c \neq 0
      \implies
      \prod_{\mathfrak{P} \in T}
      |c/d|_{E_\mathfrak{P}} \gg D^{-\frac{1}{n+1}}.
    \end{equation}
    where $|.|_{E_\mathfrak{P}}$ denotes any fixed $E_{\mathfrak{P}}$-norm.
  \end{enumerate}
\end{lemma}
\begin{remark}
  As we detail in the proof below, the estimate (i) follows by routine counting arguments.  The key feature in (ii) is that the stated estimates and the condition $d \neq 0$ hold if $\gamma$ is sufficiently close to the identity element of $\PGL_{n+1}(E_{\mathfrak{P}})$ and are unaffected by multiplying $\gamma$ on the left and right by elements of a fixed compact subset of $\GL_{n}(E_{\mathfrak{P}})$, exactly as in the split case of \eqref{eqn:d-H-h1gh2}.
\end{remark}
\begin{proof}
  We start with (i).  We fix a Haar measure $d z$ on $\mathbb{A}_E^\times \hookrightarrow \GL_{n+1}(\mathbb{A}_E)$ and a factorization $d z = \prod d z_\mathfrak{P}$ so that $d z_\mathfrak{P}$ assigns volume one to the unit group for all finite primes $\mathfrak{P}$.  For each $\mathfrak{P} \in S_E$, we may find a small neighborhood $\tilde{\Omega}_\mathfrak{P}$ of the identity element of $\GL_{n+1}(E_\mathfrak{P})$ so that for all $g \in \GL_{n+1}(E_\mathfrak{P})$ with image $[g] \in \PGL_{n+1}(E_\mathfrak{P})$, we have
  \begin{equation}\label{eqn:E-Omega-vs-tilde-Omega}
    1_{\Omega_\mathfrak{P} }([g])
    \ll
    \int _{z \in E_\mathfrak{P}^\times }
    1_{\tilde{\Omega}_\mathfrak{P} }(z g)
    \, d z.
  \end{equation}
  (Explicitly, take for $\tilde{\Omega}_{\mathfrak{P}}$ a large enough bounded subset of the preimage under $g \mapsto [g]$ of a bounded neighborhood of $\Omega_{\mathfrak{P}}$.)  For each $\mathfrak{P} \notin S_E$, we set
  \[
    \tilde{\Omega}_\mathfrak{P} := \{g \in M_{n+1}(\mathcal{O}_\mathfrak{P}) : |\det g|_\mathfrak{P} = 1/D_\mathfrak{P} \},
  \]
  so that \eqref{eqn:E-Omega-vs-tilde-Omega} holds.  Setting $\tilde{\Omega} := \prod \tilde{\Omega}_{\mathfrak{P}} \subseteq \GL_{n+1}(\mathbb{A}_E)$, we obtain
  \begin{align*}
    |\PGL_{n+1}(E) \cap \Omega|
    &=
      \sum _{\gamma \in \PGL_{n+1}(E)}
      1_{\Omega}(\gamma)
    \\
    &\ll
      \sum _{\gamma \in \PGL_{n+1}(E)}
      \int _{z \in \mathbb{A}_E^\times }
      1_{\tilde{\Omega}}(z \gamma) \, d z
    \\
    &=
      \int _{z \in \mathbb{A}_E^\times /E^\times }
      \sum _{\gamma \in \GL_{n+1}(E)}
      1_{\tilde{\Omega}}(z \gamma) \, d z.
  \end{align*}

  Suppose that $z,\gamma$ as above satisfy $z \gamma \in \tilde{\Omega}$.  Since the sets $\tilde{\Omega}_\mathfrak{P}$ ($\mathfrak{P} \in S_E$) are compact neighborhoods of the identity, we have for some fixed $B > 1$ that
  \[
    1/B < \prod_{\mathfrak{P} \in S_E} |\det (z_\mathfrak{P} \gamma)|_\mathfrak{P} < B.
  \]
  On the other hand, we have $\det(z) = z^{n+1}$, $\prod_{\mathfrak{P}} |\det (\gamma)|_\mathfrak{P} = 1$ (the product rule) and $\prod_{\mathfrak{P} \notin S_E} |\det (\gamma)|_\mathfrak{P} = D$.  It follows that the adelic absolute value $|z| = \prod_\mathfrak{P} |z_\mathfrak{P}|_\mathfrak{P}$ of $z$ satisfies
  \begin{equation}\label{eqn:z-half-two-n-D}
    1/B < |z|^{n+1} D < B,
  \end{equation}
  Since the idelic absolute value $|.| : \mathbb{A}_E^\times / E^\times \rightarrow \mathbb{R}^\times_+$ is measure-preserving up to a normalizing scalar, we have
  \begin{equation*}
    \vol \{z \in \mathbb{A}_E^\times / E^\times :    1/B < |z|^{n+1}  D < B \} \ll 1.
  \end{equation*}
  We thereby reduce to verifying that for each $z \in \mathbb{A}_E^\times$ satisfying \eqref{eqn:z-half-two-n-D}, we have
  \begin{equation}\label{eqn:count-tilde-Omega}
    |\GL_{n+1}(E) \cap z^{-1} \tilde{\Omega}| \ll D^{{n+1}}.
  \end{equation}

  Let $\gamma \in \GL_{n+1}(E) \cap z^{-1} \tilde{\Omega}$.  Then for all $i,j \in \{1,\dotsc,{n+1}\}$, the matrix entry $\gamma_{i j}$ satisfies $|\gamma_{i j}|_\mathfrak{P} \leq C_\mathfrak{P}$, where $C_\mathfrak{P} \ll |z|_\mathfrak{P}^{-1}$ for $\mathfrak{P} \in S_E$ and $C_\mathfrak{P} = |z|_\mathfrak{P}^{-1}$ for $\mathfrak{P} \notin S_E$.  By bounding separately the number of possibilities for each matrix entry of $\gamma$, it follows that
  \[|\GL_{n+1}(E) \cap z^{-1} \tilde{\Omega}| \ll N^{(n+1)^2}, \quad N := | \{x \in E : |x|_\mathfrak{P} \leq C_\mathfrak{P} \text{ for all } \mathfrak{P} \}.
  \]
  The required estimate \eqref{eqn:count-tilde-Omega} follows from the adelic Minkowski-type estimate
  \[N \ll 1 + \prod_{\mathfrak{P}} C_\mathfrak{P} \ll 1 + |z|^{-1} \ll D^{\frac{1}{n+1}}.\]

  We turn to (ii).  For notational simplicity, we abbreviate $|.|_{\mathfrak{P}} := |.|_{E_{\mathfrak{P}}}$.  By our hypothesis \eqref{eqn:Omega-special-form} concerning the $\Omega_\mathfrak{P}$ for $\mathfrak{P} \in S_E$, we have $d \neq 0$; moreover, $|b/d|_\mathfrak{P}, |c/d|_\mathfrak{P} \ll 1$.  By the product formula, the required implications \eqref{eqn:b-neq-0-impl} and \eqref{eqn:c-neq-0-impl} will follow if we can verify that
  \begin{equation}\label{eqn:prods-b-over-d-c-over-d-D-1-over-n-plus-1}
    \prod_{\mathfrak{P} \notin S_E}
    |b/d|_\mathfrak{P},
    \prod_{\mathfrak{P} \notin S_E}
    |c/d|_\mathfrak{P}
    \ll D^{\frac{1}{n+1}}.
  \end{equation}
  Here $b/d$ and $c/d$ are regarded as elements of the vector space $E^{n+1}$ or its dual, and we suppose that the various $E_{\mathfrak{P}}$-norms have been chosen as in \S\ref{sec:F-norms-global-to-local} using the standard basis.  We note for future reference that for $\mathfrak{P} \notin S_E$ and $v$ in $E^{n+1}$ or its dual, we have
  \begin{equation}\label{eqn:v-P-norm-integrality}
    \text{$|v|_{\mathfrak{P}} \leq 1$
      whenever each entry of $v$
      is $\mathfrak{P}$-integral.}
  \end{equation}
  
  To verify \eqref{eqn:prods-b-over-d-c-over-d-D-1-over-n-plus-1}, let us lift $\gamma$ to an element of $\GL_{n+1}(E)$.  By our (simplifying) assumption that $\mathcal{O}[1/S_E]$ is a principal ideal domain, we may choose the lift in such a way that $\gamma \in \tilde{\Omega}_\mathfrak{P}$ for all $\mathfrak{P} \notin S_E$.  By the proof of (i), we may find $z \in \mathbb{A}_E^\times$ with $|z| \asymp D^{-\frac{1}{n+1}}$ so that $z_\mathfrak{P} \gamma \in \tilde{\Omega}_\mathfrak{P}$ for all primes $\mathfrak{P}$.

  For $\mathfrak{P} \notin S_E$, both $\gamma$ and $z_\mathfrak{P} \gamma$ belong to $\tilde{\Omega}_\mathfrak{P}$; by comparing determinants, it follows that $|z_\mathfrak{P}|_\mathfrak{P} = 1$.  Thus $\prod_{\mathfrak{P} \in S_E} |z_\mathfrak{P} |_\mathfrak{P} = |z| \asymp D^{-\frac{1}{n+1}}$.
  
  For $\mathfrak{P} \in S_E$, we have $z_\mathfrak{P} \gamma \in \tilde{\Omega}_\mathfrak{P}$.  We may assume that the lift $\tilde{\Omega}_\mathfrak{P}$ is of the form $\Omega_{\mathfrak{P}}^1 \tilde{\Omega}_\mathfrak{P}^0 \Omega_{\mathfrak{P}}^1$ with $\tilde{\Omega}_{\mathfrak{P}}^0 \subseteq \GL_{n+1}(E_{\mathfrak{P}})$ sufficiently concentrated near the identity element.  It follows then that $|z_\mathfrak{P} d|_\mathfrak{P} \asymp 1$.
  
  Thus
  \[\prod_{\mathfrak{P} \notin S_E}
    |d|_\mathfrak{P} = \prod_{\mathfrak{P} \in S_E} |d|_\mathfrak{P}^{-1} \asymp D^{-\frac{1}{n+1}}.\] For $\mathfrak{P} \notin S_E$, we see from the condition $\gamma \in \tilde{\Omega}_\mathfrak{P} \subseteq M_{n+1}(\mathcal{O}_{\mathfrak{P}})$ and \eqref{eqn:v-P-norm-integrality} that $|b|_\mathfrak{P} \leq 1$.  Therefore $\prod_{\mathfrak{P} \notin S_E} |b/d|_\mathfrak{P} \ll D^{\frac{1}{n+1}}$, as required.  The same argument applies to $c/d$.
\end{proof}

\begin{proof}[Proof of Lemma \ref{lem:count-relevant-gamma}]
  We first prove \eqref{eqn:Sigma-cardinality}.  The aim is to apply Lemma \ref{lem:counting-PGL-n-plus-1-E}, taking for $S_E$ the set of places of $E$ lying above some place of $S$.  To that end, we will construct $\Omega_{\mathfrak{P}}$ so that
  \[
    \Sigma \subseteq \Omega := \prod_{\mathfrak{P}} \Omega_{\mathfrak{P}}.\]
  
  For each $\mathfrak{p} \in S$, we see from the compactness of the sets $\Theta_\mathfrak{p}$ and $\supp(f_\mathfrak{p}^{\sharp})$ that $\Sigma$ is contained in some fixed compact neighborhood of the identity element of $\bar{G}(F_\mathfrak{p})$.  Thus for the places $\mathfrak{P}$ of $E$ that lie over $S$, we may take for $\Omega_\mathfrak{P}$ a fixed compact neighborhood of the identity element.  Since $\Theta_\mathfrak{p}$ is contained in $H(F_\mathfrak{p})$ while $f_\mathfrak{p}^\sharp$ is supported near the identity element, we see moreover that $\Omega_{\mathfrak{P}}$ may be chosen to satisfy the hypothesis \eqref{eqn:Omega-special-form}.

  We recall from \S\ref{sec:ampl-pretr-form} that the primes $\mathfrak{p}_1, \mathfrak{p}_2 \in \mathbb{P}$ were chosen to split the extension $E/F$.  Using this feature of our choice, we could ``simplify'' some of the discussion below by noting that all inertial degrees must be $1$.  However, we prefer to carry out the analysis in full and ignore this feature of our choice.

  For places $\mathfrak{P}$ not lying over $S$, we define $\Omega_{\mathfrak{P}}$ as in Lemma \ref{lem:counting-PGL-n-plus-1-E} in terms of an element $D_{\mathfrak{P}}$ of the value group of $E_{\mathfrak{P}}^\times$, which we now describe case-by-case.
  \begin{itemize}
  \item If $\mathfrak{P}$ lies over some $\mathfrak{p} \notin S \cup \{\mathfrak{p}_1, \mathfrak{p}_2\}$, then we take $D_\mathfrak{P} := 1$.  We have
    \[\supp(f_{\mathfrak{p}}^\sharp)
      = \bar{G}(\mathbb{Z}_{\mathfrak{p}}) \subseteq \PGL_{n+1}(\mathcal{O}_{\mathfrak{P}}) = \Omega_{\mathfrak{P}}.
    \]
  \item If $\mathfrak{p}_1 \neq \mathfrak{p}_2$ and $\mathfrak{P}$ lies over $\mathfrak{p}_k$ for some $k \in \{1,2\}$, then we take
    \[D_{\mathfrak{P}} :=
      \begin{cases}
        q_{\mathfrak{p}_k}^{e {\mathbf{j}}}     &  \text{ if } k=1, \\
        q_{\mathfrak{p}_k}^{n e {\mathbf{j}}} & \text{ if } k=2,
      \end{cases}
    \]
    where $e \in \{1,2\}$ denotes the inertial degree of $\mathfrak{P}$ over $\mathfrak{p}_k$.  Since $\supp(f^\sharp_{\mathfrak{p}_1})$ is the image of
    \begin{align}
      \supp(f_{\mathfrak{p}_1})
      &= \supp(T_{\mathfrak{p}_1}({\mathbf{j}},0,\dotsc,0))
        \nonumber
      \\
      \label{constrain-support-f1}
      &\subseteq
        \{ g \in M_{n+1}(\mathbb{Z}_\mathfrak{p})
        :
        |\det g|_\mathfrak{p} = 1/ q_{\mathfrak{p}_k}^{\mathbf{j}}
        \}
    \end{align}
    and $\supp(f^\sharp_{\mathfrak{p}_2})$ is the image of
    \begin{align}
      \supp(f_{\mathfrak{p}_2}) 
      &= \supp(T_{\mathfrak{p}_2}({\mathbf{j}},\dotsc,{\mathbf{j}},0))
        \nonumber
      \\
      \label{constrain-support-f2}
      &\subseteq
        \{ g \in M_{n+1}(\mathbb{Z}_\mathfrak{p})
        :
        |\det g|_\mathfrak{p} = 1/ q_{\mathfrak{p}_k}^{n {\mathbf{j}}}
        \}
    \end{align}
    and $|\det g|_{\mathfrak{P}} = |\det g|_\mathfrak{p}^e$, we see that $\supp(f^\sharp_{\mathfrak{p}_k}) \subseteq \Omega_{\mathfrak{P}}$.
  \item If $\mathfrak{p}_1 = \mathfrak{p}_2$ and $\mathfrak{P}$ lies over $\mathfrak{p}_1$, then we take
    \[D_{\mathfrak{P}} := q_{\mathfrak{p}_1}^{e (n+1) {\mathbf{j}}}.\] Arguing as in the previous case and noting that
    \begin{align}
      \supp(f_{\mathfrak{p}_1})
      &= \supp(T_{\mathfrak{p}_1}({\mathbf{j}},0,\dotsc,-{\mathbf{j}})) \nonumber
      \\
      &= \supp(T_{\mathfrak{p}_1}(2 {\mathbf{j}},{\mathbf{j}},\dotsc,{\mathbf{j}},0)) \nonumber
      \\
      \label{constrain-support-f12}
      &\subseteq
        \{ g \in M_{n+1}(\mathbb{Z}_{\mathfrak{p}_1})
        :
        |\det g|_{\mathfrak{p}_1} = 1/ q_{\mathfrak{p}_1}^{(n+1) {\mathbf{j}}}
        \},
    \end{align}
    we see again that $\supp(f^\sharp_{\mathfrak{p}_k}) \subseteq \Omega_{\mathfrak{P}}$.
  \end{itemize}
  Let $D$ denote the product of $D_{\mathfrak{P}}$ taken over all such $\mathfrak{P}$.  Then
  \begin{equation*}
    D \asymp 
    L^{2 (n+1)  {\mathbf{j}}}.
  \end{equation*}
  For instance, suppose that $\mathfrak{p}_1 = \mathfrak{p}_2$ and that $\mathfrak{p}_1$ splits in $E$, say into two primes $\mathfrak{P}, \mathfrak{P} '$.  Then $D = D_{\mathfrak{P}} D_{\mathfrak{P}'}$.  Each of the factors $D_{\mathfrak{P}}, D_{\mathfrak{P} '}$ is $\asymp q_{\mathfrak{p}_1}^{(n+1) {\mathbf{j}}} \asymp L^{(n+1) {\mathbf{j}}}$, so their product satisfies the required estimate.  We argue similarly in the other cases.
  
  By Lemma \ref{lem:counting-PGL-n-plus-1-E}, we deduce that
  \begin{equation}\label{eqn:cardinality-Sigma-D-n-plus-1}
    |\Sigma| \leq
    |\PGL_{n+1}(E) \cap \Omega|
    \ll
    D^{n+1}.
  \end{equation}
  The proof of \eqref{eqn:Sigma-cardinality} is thus complete.

  To prove the required lower bound \eqref{eqn:d-H-gamma} for $d_{H_\mathfrak{q}}(\gamma)$, we write $\gamma$ as in \eqref{eqn:gamma-a-b-c-d} and see from the assumption $\gamma \notin H(F)$ that at least one of the conditions $b \neq 0$ or $c \neq 0$ holds.  We then apply \eqref{eqn:b-neq-0-impl} or \eqref{eqn:c-neq-0-impl}, taking for $T$ the set of primes $\mathfrak{Q}$ of $E$ lying over $\mathfrak{q}$.  Supposing for instance that $b \neq 0$, we obtain
  \begin{equation*}
    \prod_{\mathfrak{Q} | \mathfrak{q} }
    |b/d|_{E_\mathfrak{Q}}
    \gg
    D^{-\frac{1}{n+1}}.
  \end{equation*}
  We then apply \eqref{eqn:relate-F-norms-inert} (for $E_{\mathfrak{q}}$ a field) or \eqref{eqn:F-norm-split-AM-GM} (for $E_{\mathfrak{q}} \cong F_{\mathfrak{q}} \times F_{\mathfrak{q}}$) to deduce that
  \begin{equation*}
    |b/d|_{F_\mathfrak{q}} \gg
    D^{-\frac{1}{2 (n+1)}}.
  \end{equation*}
  Since $D \geq 1$, this yields the lower bound
  \begin{equation*}
    d_{H_\mathfrak{q}}(\gamma)
    \gg D^{-\frac{1}{2(n+1)}},
  \end{equation*}
  which translates to \eqref{eqn:d-H-gamma}.
\end{proof}

\begin{remark}
  The estimate \eqref{eqn:cardinality-Sigma-D-n-plus-1} could likely be improved by taking into account that $\Sigma \subseteq \bar{G}(F) \subsetneq \PGL_{n+1}(E)$ and that the containments \eqref{constrain-support-f1}, \eqref{constrain-support-f2} and \eqref{constrain-support-f12} are unlikely to be sharp when $n+1 \geq 3$ (compare with \cite[Lem 6.1]{MR3384442} and \cite{2014arXiv1405.6691B}).
\end{remark}

\subsubsection{Bounds for $I(\gamma)$}\label{sec:bounds-igamma}
\begin{lemma}\label{lem:bounds-igamma}
  Recall the parameter $\mathfrak{j} \in \{0, \dotsc, n+1\}$ fixed in \S\ref{sec:defin-test-funct}.  Let $\gamma \in \bar{G}(F) - H(F)$.  Then
  \[
    I(\gamma) \ll T^{n/2 - 1/2 + \eps} L^{-n {\mathbf{j}}} d_{H_\mathfrak{q}}(\gamma)^{-1/2}.
  \]
\end{lemma}
\begin{proof}
  Setting $S_0 := S - \{\mathfrak{q} \}$ and
  \[
    \Theta_{0} := \prod_{\mathfrak{p} \in {S_0}} \Theta_\mathfrak{p}, \quad f^\sharp_{0} := \otimes_{\mathfrak{p} \in {S_0}} f^\sharp_\mathfrak{p},
  \]
  we see by \eqref{eqn:f-sharp-L-infinity-outside-S} that
  \[
    I(\gamma) \ll L^{-n {\mathbf{j}}} \int _{x_{0}, y_{0} \in \Theta_{0}} \left\lvert f^\sharp_{0}(x_{0}^{-1} \gamma y_{0}) \right\rvert I(\gamma;x_{0}, y_{0}) \, d x_{0} \, d y_{0},
  \]
  where
  \[
    I(\gamma;x_{0}, y_{0}) := \int _{x_\mathfrak{q} , y_\mathfrak{q} \in \Theta_\mathfrak{q} } \left\lvert \Psi(x_{0} x_\mathfrak{q} ) \Psi(y_{0} y_\mathfrak{q} ) f_\mathfrak{q}^\sharp(x_\mathfrak{q} ^{-1} \gamma y_\mathfrak{q}) \right\rvert \, d x_\mathfrak{q} \, d y_\mathfrak{q}.
  \]
  Theorem \ref{thm:construct-test-function}, part \eqref{thm:f-item-3} implies that
  \[
    I(\gamma;x_{0}, y_{0}) \ll T^{n/2-1/2+\eps} d_{H_\mathfrak{q}}(\gamma)^{-1/2} \mathcal{N}(x_{0}) \mathcal{N}(y_{0}),
  \]
  where, with $\Theta_\mathfrak{q} '$ a fixed compact subset of $H(F_\mathfrak{q})$ large enough in terms of $\Theta_\mathfrak{q}$,
  \[
    \mathcal{N}(x_{0}) := \left( \int _{x_\mathfrak{q} \in \Theta_\mathfrak{q}' } |\Psi(x_{0} x_\mathfrak{q})|^2 \, d x_\mathfrak{q} \right)^{1/2}.
  \]
  Since $f_{0}^\sharp(x_{0}^{-1} \gamma y_{0}) \ll 1$, we obtain
  \[
    I(\gamma) \ll T^{n/2-1/2+\eps} L^{-n {\mathbf{j}}} d_{H_\mathfrak{q}}(\gamma)^{-1/2} \left( \int _{x_{0} \in \Theta_{0}} \mathcal{N}(x_{0}) \, d x_{0} \right)^2.
  \]
  By Cauchy--Schwarz and the estimate $\vol(\Theta_{0}) \ll 1$, we have
  \[
    \left( \int _{x_{0} \in \Theta_{0}} \mathcal{N}(x_{0}) \, d x_{0} \right)^2 \ll \int _{x_{0} \in \Theta_{0}} \mathcal{N}(x_{0})^2 \, d x_{0} = \|\Psi \|_{L^2(\Theta_S')}^2, \quad \Theta_S' := \Theta_{\mathfrak{q}}' \times \Theta_0.
  \]
  By covering $\Theta_S'$ by finitely many translates of a fundamental domain for $\Gamma_S \backslash H_S$, we see that $\|\Psi \|^2_{L^2(\Theta_S')} \ll \|\Psi \|^2 \ll 1$.  The required estimate follows.
\end{proof}

\subsubsection{Summary}\label{sec:error-cal-E-summary}
Combining Lemmas \ref{lem:count-relevant-gamma} and \ref{lem:bounds-igamma}, we obtain
\begin{align*}
  \mathcal{E}
  &\ll
    T^{n/2 - 1/2 + \eps} L^{-n {\mathbf{j}}}
    |\Sigma|
    \max_{\gamma \in \Sigma} d_{H_\mathfrak{q}}(\gamma)^{-1/2} \\
  &\ll
    T^{n/2 - 1/2 + \eps}
    L^{(2 (n+1)^2 - n + 1/2) {\mathbf{j}}}.
\end{align*}

\begin{remark}\label{rmk:opt-2}
  One could try to improve this estimate by taking into account that not all $\gamma \in \Sigma$ satisfy the worst case lower bound \eqref{eqn:d-H-gamma} for $d_H(\gamma)$.
\end{remark}

\subsection{Optimization}\label{sec:optimization}
The estimates obtained above for $\mathcal{M}$ and $\mathcal{E}$ combine to give
\begin{equation}\label{eqn:M-E-summary}
  \mathcal{M} + \mathcal{E} 
  \ll
  T^{n/2+\eps}
  \left(
    L^{-(1-2 \vartheta) {\mathbf{j}}}
    +
    T^{-1/2}
    L^{(2 (n+1)^2  -n + 1/2 ){\mathbf{j}}}
  \right).
\end{equation}
Recall our goal bound \eqref{eqn:R-decomposed}.  Recall from \eqref{eqn:weak-PNT} that $|\mathbb{P}|$ is approximately $L$.  By substituting \eqref{eqn:M-E-summary} into the decomposition \eqref{eqn:R-decomposed}, we deduce that
\begin{align*}
  \frac{\mathcal{R}}{L^2 T^{n/2+\eps}}
  &\ll
    \sum _{j=1}^{n+1}
    \left(
    L^{-(1-2 \vartheta) j}
    +
    T^{-1/2}
    L^{(2 (n+1)^2 - n+1/2 )j}
    \right)\\
  &\quad +
    L^{-1}
    \sum _{j=0}^{n+1}
    \left(
    L^{-(1-2 \vartheta) j}
    +
    T^{-1/2}
    L^{(2 (n+1)^2 - n+1/2 )j}
    \right)
  \\
  &\ll
    L^{-(1-2 \vartheta)}
    +
    T^{-1/2}
    L^{A},
\end{align*}
where
\[
  A := (2 (n+1)^2 - n + 1/2)(n+1).
\]
The optimal choice for $L$ is the solution to
\begin{equation*}
  L^{-(1-2 \vartheta)}
  =
  T^{-1/2}
  L^A,
\end{equation*}
i.e.,
\begin{equation}\label{eqn:L-optimal-choice}
  L
  =
  T ^{\frac{1}{2 (A + 1 - 2 \vartheta)} },
\end{equation}
which gives
\[
  \frac{\mathcal{R}}{L^2 T^{n/2+\eps}} \ll T ^{- \delta }, \quad \delta := \frac{1 - 2 \vartheta}{2 (A + 1 - 2 \vartheta)},
\]
as required.

This completes the proof that Theorem \ref{thm:construct-test-function} implies Theorem \ref{thm:main}.  The remainder of the paper is devoted to the proof of Theorem \ref{thm:construct-test-function}.

\section{Representation-theoretic
  preliminaries}\label{sec:repr-theor-prel}
The main purpose of \S\ref{sec:repr-theor-prel}--\S\ref{sec:relat-char-asympt} is to recall parts of the microlocal calculus developed in \cite[Parts 1,2,3]{nelson-venkatesh-1}, together with some crucial refinements (\S\ref{sec:refin-symb-class}, \S\ref{sec:proofs-star-product}).

\subsection{Lie groups}\label{sec:lie-groups}
Here we record some notation and preliminaries concerning Lie groups.  We are primarily interested in the case of real Lie groups, but discuss also the case of $p$-adic Lie groups.  The generality is motivated by the Lie-theoretic arguments of \S\ref{sec:bilinear-forms-estimates} and onwards; the results obtained by those arguments may be of independent interest, so we have formulated them over any local field of characteristic zero.

Let $F$ be a local field of characteristic zero.  In particular, $F$ is a non-discrete complete valued field in the sense of \cite[Part II, Ch. I]{MR2179691}.  We may thus speak of $n$-dimensional analytic manifolds $M$ over $F$, i.e., topological spaces equipped with a maximal atlas of charts $M \supseteq U \rightarrow F^n$ for which the transition maps are given locally by convergent power series (see \cite[Part II, Ch. III]{MR2179691} for details).  Many standard theorems from multivariable calculus hold in this setting, e.g., the implicit and inverse function theorems (see \S\ref{sec:inverse-function-theorem}).

We note that an analytic manifold over $F$ is naturally an analytic manifold over any subfield of $F$.  For this reason, we may often assume that $F$ is either $\mathbb{R}$ or $\mathbb{Q}_p$.

A Lie group $G$ over $F$ is a group equipped with the structure of an analytic manifold for which the multiplication and inversion maps are analytic.  We denote by $\mathfrak{g}$ the corresponding Lie algebra, by $\mathfrak{g}^*$ its linear dual, and by $\mathfrak{g}^\wedge$ its Pontryagin dual.  These are vector spaces over $F$ of the same dimension as $G$.  By choosing a nontrivial character $F \rightarrow \U(1)$, we could identify $\mathfrak{g}^\wedge$ with $\mathfrak{g}^*$, but will generally avoid doing so.  We denote typical elements of $\mathfrak{g}$ (resp. $\mathfrak{g}^\wedge$) by $x,y,z$ (resp. $\xi,\eta,\zeta,\tau$).

Suppose for the moment that $F$ is archimedean.  Without loss of generality, take $F = \mathbb{R}$.  We write
\begin{equation}\label{eq:ex-xi-in}
  e^{x \xi} \in \U(1)
\end{equation}
for the result of the canonical pairing between $x \in \mathfrak{g}$ and $\xi \in \mathfrak{g}^\wedge$.
We write $\langle x, \xi \rangle$ for the element of $i \mathbb{R}$ with the property that for all $t \in \mathbb{R}$, we have $\exp (t \langle x, \xi \rangle) = e ^{(t x ) \xi } = e ^{x (t \xi )}$.  We often abbreviate
\begin{equation}\label{eq:x-xi-:=}
  x \xi := \langle x, \xi \rangle.
\end{equation}
The notations \eqref{eq:ex-xi-in} and \eqref{eq:x-xi-:=} are then compatible: $e^{x \xi} = \exp(x \xi)$.  The pairing $\langle , \rangle$ identifies $\mathfrak{g}^\wedge$ with the imaginary dual $i \mathfrak{g}^*= \Hom(\mathfrak{g},i \mathbb{R})$.  We write $\mathfrak{g}_\mathbb{C} := \mathfrak{g} \otimes_{\mathbb{R}} \mathbb{C}$ for the complexification and $\mathfrak{g}_\mathbb{C}^*$ for its dual.  We regard $\mathfrak{g}^\wedge = i \mathfrak{g}^*$ as a real form of $\mathfrak{g}_{\mathbb{C}}^*$; it is the subspace fixed by the involution $\xi \mapsto - \overline{\xi}$.

Returning to the case of general $F$, the group $G$ acts on $\mathfrak{g}$ via the adjoint representation $x \mapsto \Ad(g) x$ and on $\mathfrak{g}^\wedge$ via the coadjoint representation $\xi \mapsto \Ad^*(g) \xi$.  We abbreviate these actions by $g \cdot x$ and $g \cdot \xi$ or simply $g \xi$.  We similarly denote the action of $G$ on functions on either space.  For instance, given $a : \mathfrak{g}^\wedge \rightarrow \mathbb{C}$, we set
\begin{equation*}
  (g \cdot a)(\xi) := a(\Ad^*(g)^{-1} \xi).
\end{equation*}
The Lie algebra $\mathfrak{g}$ acts on $\mathfrak{g}$ via the adjoint representation
\begin{equation*}
  \ad_x y = [x,y]
\end{equation*}
and on $\mathfrak{g}^\wedge$ via the coadjoint representation $\ad_x^* \xi$.  We often abbreviate the latter by $[x,\xi]$.

\subsection{Reductive groups}\label{sec:prelim-reductive-groups}
Suppose now that $G$ is a connected reductive group over a field $F$ of characteristic zero.  We may then speak of its Lie algebra $\mathfrak{g}$ and linear dual $\mathfrak{g}^*$.  When $F$ is a local field, we may regard $G$ also as a Lie group over $F$; in particular, we have defined the Pontryagin dual $\mathfrak{g}^\wedge$.

The vector spaces $\mathfrak{g}$, $\mathfrak{g}^*$ and (when defined) $\mathfrak{g}^\wedge$ are $G$-equivariantly isomorphic.  We denote by $[\mathfrak{g}], [\mathfrak{g}^*]$ and $[\mathfrak{g}^\wedge]$ the set of $F$-points of their geometric invariant quotients.\index{Lie algebra!GIT quotients $[\mathfrak{g}]$, $[\mathfrak{g}^*]$, $[\mathfrak{g}^\wedge]$} Thus, e.g., $[\mathfrak{g}^\wedge]$ is the set of $F$-algebra homomorphisms $R^G \rightarrow F$, where $R$ denotes the algebra of polynomial functions $\mathfrak{g}^\wedge \rightarrow F$.  These quotients are affine spaces over $F$ of dimension equal to the rank of $G$, and are (non-canonically) isomorphic to one another.

When $F$ is archimedean, we write $[\mathfrak{g}^*_{\mathbb{C}}]$ for the geometric invariant theory quotient of $\mathfrak{g}^*_{\mathbb{C}}$.  Then $[\mathfrak{g}^\wedge]$ is naturally a real form of $[\mathfrak{g}_\mathbb{C}^*]$, given as the fixed point set of the map $\lambda \mapsto - \bar{\lambda}$ on $[\mathfrak{g}_\mathbb{C}^*]$ descended from $\mathfrak{g}_\mathbb{C}^*$.  For details on these points, we refer to \cite[\S9.2]{nelson-venkatesh-1}.

We use a subscripted $\reg$, \index{Lie algebra!regular subset $\mathfrak{g}^\wedge_{\reg}$} as in $\mathfrak{g}_{\reg}$ or $\mathfrak{g}^\wedge_{\reg}$, to denote the subset of \emph{regular} elements, i.e., those whose centralizer in $\mathfrak{g}$ has minimal dimension.  We record in \S\ref{sec:interl-regul-elem} the essential properties of such elements.

For a vector space $V$ over $F$, \index{Lie algebra!$\dim, \dim_F, \rank, \rank_F$} we write $\dim_F(V)$ for its $F$-dimension and $\dim(V)$ for its $\mathbb{R}$-dimension, so that $\dim(V) = [F:\mathbb{R}] \dim_F(V)$.  This notation applies in particular to $V = \mathfrak{g}$.  We write $\rank_F(\mathfrak{g})$ (resp. $\rank(\mathfrak{g})$) for the $F$-dimension (resp. $\mathbb{R}$-dimension) of any Cartan subalgebra of $\mathfrak{g}$ (e.g., $\rank_F(\mathfrak{g}) = n$ if $G = \GL_n(F)$).

\subsection{Infinitesimal characters
  and Langlands parameters}\label{sec:infin-char-langl}
We retain the above setting, but suppose here that $F$ is archimedean.

We write \index{Lie algebra!enveloping algebra $\mathfrak{U}$, center $\mathfrak{Z}$} \index{representations!infinitesimal character $\lambda_\pi$} $\mathfrak{U}(\cdot)$ for the universal enveloping algebra and $\mathfrak{Z}(\cdot)$ for its center.  On any irreducible representation $\pi$ of $G$, the algebra $\mathfrak{Z}(\mathfrak{g}_\mathbb{C})$ acts via scalars.  The infinitesimal character $\lambda_\pi$ is the character of $\mathfrak{Z}(\mathfrak{g}_\mathbb{C})$ describing that action.  The Harish--Chandra isomorphism (see \cite[\S9.4]{nelson-venkatesh-1}) identifies $\mathfrak{Z}(\mathfrak{g}_\mathbb{C})$ with the ring of regular functions on $[\mathfrak{g}_\mathbb{C}^*]$.  In this way, we may regard $\lambda_\pi$ as an element of $[\mathfrak{g}_\mathbb{C}^*]$.  In the special case that $\pi$ is unitary, we have in fact $\lambda_\pi \in [\mathfrak{g}^\wedge]$ (see \cite[\S9.5]{nelson-venkatesh-1}).  There is a natural scaling action of $\mathbb{R}^\times$ (resp. $\mathbb{C}^\times$) on $[\mathfrak{g}^\wedge]$ (resp. $[\mathfrak{g}_\mathbb{C}^*]$), descended from the corresponding actions on $\mathfrak{g}^\wedge$ (resp. $\mathfrak{g}_\mathbb{C}^*$).

Let $G^\vee$ denote the complex dual group of $G$ -- regarding $G$ as an algebraic group over $F$ -- and $\mathfrak{g}^\vee$ its complex Lie algebra.  We may canonically identify
\begin{equation}\label{eqn:identification-inf-char-langlands}
  \mathfrak{g}^\vee \giit G^\vee
  \cong
  \mathfrak{t}^\vee / W
  \cong
  \begin{cases}
    \mathfrak{t}_\mathbb{C}^* / W \cong [\mathfrak{g}_\mathbb{C}^*]
    & \text{ if } F = \mathbb{R},  \\
    \mathfrak{t}^\wedge / W \cong [\mathfrak{g}^\wedge] & \text{ if } F = \mathbb{C},
  \end{cases}
\end{equation}
where $\mathfrak{t}_\mathbb{C} \leq \mathfrak{g}_\mathbb{C}$ (resp.  $\mathfrak{t} \leq \mathfrak{g}$) and $\mathfrak{t}^\vee \leq \mathfrak{g}^\vee$ are Cartan subalgebras and $W$ denotes the Weyl group.  The case $F = \mathbb{R}$ is discussed in \cite[\S15.1]{nelson-venkatesh-1}; for the case $F = \mathbb{C}$, we first identify $\mathfrak{t}^\vee$ canonically-modulo-$W$ with $\Hom(\mathfrak{t},\mathbb{C})$ and then identify the latter with $\mathfrak{t}^\wedge$ via the map $\mathbb{C} \rightarrow i \mathbb{R}$ sending a complex number to its imaginary component.

The identification \eqref{eqn:identification-inf-char-langlands} is compatible with the local Langlands correspondence in the following sense (see \cite[\S15.1]{nelson-venkatesh-1} and references).  Let $\pi$ be a tempered irreducible unitary representation of $G$.  It corresponds to a conjugacy class of representations $\phi_\pi : W_{F} \rightarrow {}^L G$ having bounded image.  Here $W_F$ denotes the Weil group of $F$ and ${}^L G = G^\vee \rtimes \Gal(\mathbb{C}/F)$ the Langlands dual group.  Using \eqref{eqn:identification-inf-char-langlands}, we may identify $\lambda_\pi$ with a semisimple $G^\vee$-conjugacy class in $\mathfrak{g}^\vee$.  The restriction of $\phi_\pi$ to the Weil group $W_\mathbb{C} =\mathbb{C}^\times \subseteq W_F$ of the complex numbers is then given for $t \in \mathbb{C}$ by the formula $\phi_\pi(\exp(t)) = \exp(2 i \Im(t \lambda_\pi))$, where $t \lambda_\pi = \Re(t \lambda_\pi) + i \Im(t \lambda_\pi)$.

\subsection{Eigenvalue multisets}\label{sec:eigenvalue-multisets}
Let $G$ be a unitary or orthogonal group over a field $F$ of characteristic zero.

For $x \in \mathfrak{g}$, we denote by $\ev(x)$ the ``eigenvalue multiset'' as defined in \cite[\S13.4.1]{nelson-venkatesh-1}.  Roughly speaking, it is the multiset of roots, taken in a given algebraic closure $\bar{F}$ of $F$, of the characteristic polynomial of the action of $x$ via the standard representation of $G$, except that in the odd orthogonal case, we subtract the ``obvious'' root zero with multiplicity one, noting that it always occurs.  This multiset depends only upon the image of $x$ in $[\mathfrak{g}]$, and so defines a function $\ev$ on $[\mathfrak{g}]$.

The spaces $\mathfrak{g}$ and $\mathfrak{g}^*$ are equivariantly isomorphic via the trace pairing with respect to the standard representation, while the spaces $\mathfrak{g}^*$ and $\mathfrak{g}^\wedge$ are equivariantly isomorphic via the choice of a nontrivial unitary character of $F$.  Having made such choices, we may extend the definition of $\ev$ to $[\mathfrak{g}]$ and to $[\mathfrak{g}^\wedge]$.

When $F$ is an archimedean local field, it will be convenient to work with a more canonical definition of $\ev(\lambda)$ for $\lambda \in [\mathfrak{g}^\wedge]$, given as follows.  The group $G^\vee$ comes equipped with a standard representation $G^\vee \hookrightarrow \GL_n(\mathbb{C})$.  In the unitary case, $n = \dim_E V$; in the orthogonal case, $n$ is even and $\dim_E V$ is either $n$ or $n+1$.  For $\lambda \in [\mathfrak{g}^\wedge]$, we \index{Lie algebra!eigenvalue multiset $\ev(\lambda)$} write
\begin{equation*}
  \ev(\lambda) = \{\lambda_1,\dotsc,\lambda_n\}
\end{equation*}
for the multiset of eigenvalues for the semisimple conjugacy class in $\mathfrak{g}^\vee \hookrightarrow \glLie_n(\mathbb{C})$ corresponding to $\lambda$ via \eqref{eqn:identification-inf-char-langlands}.


\subsection{Satake parameters
  for archimedean classical groups}\label{sec:satake-param-arch}
Let $G$ be a unitary or orthogonal group over an archimedean local field $F$.  We may associate to each tempered irreducible representation $\pi$ of $G$ a multiset
\begin{equation*}
\ev(\lambda_{\pi}) = \{\lambda_{\pi,1},\dotsc, \lambda_{\pi,n}\}
\end{equation*}
of complex numbers.  We refer to these as the \index{representations!archimedean Satake parameters $\lambda_{\pi,i}$} archimedean Satake parameters of $\pi$.

\subsection{Kirillov formula}\label{sec:kirillov-formula}
Let $\pi$ be an irreducible unitary representation of the connected reductive group $G$ over $\mathbb{R}$.  Set
\begin{equation*}
  2 d := \dim(\mathfrak{g}) - \rank(\mathfrak{g}).
\end{equation*}
We recall the statement of Rossmann's theorem \cite{MR508985, MR650379, MR587333} as summarized in \cite[\S6]{nelson-venkatesh-1}, which asserts the validity of the Kirillov formula \cite{MR1701415} for such $G$.
\begin{theorem}\label{thm:kirillov-formula}
  Assume that $\pi$ is tempered.  There \index{representations!coadjoint multiorbit $\mathcal{O}_\pi$} is a (unique) nonempty $G$-invariant subset $\mathcal{O}_\pi \subseteq \mathfrak{g}^\wedge_{\reg}$ with the following properties.
  \begin{enumerate}[(i)]
  \item The distributional character $\chi_\pi$ of $G$ is described for $x$ in a small enough neighborhood of the identity of $\mathfrak{g}$ by the formula \[\chi_\pi(\exp(x)) \sqrt{\jac(x)} = \int _{\xi \in \mathcal{O}_\pi } e^{x \xi},\] where $\jac(x)$ denotes the Jacobian of the exponential map as described in \cite[\S2.1]{nelson-venkatesh-1} and the integral is taken with respect to the symplectic measure on $\mathcal{O}_\pi$, normalized as in \cite[\S6.1]{nelson-venkatesh-1}.
  \item Let $\lambda_\pi \in [\mathfrak{g}^\wedge]$ denote the infinitesimal character of $\pi$, as defined in \S\ref{sec:infin-char-langl}.  Then $\mathcal{O}_\pi$ is contained in the preimage of $\lambda_\pi$.
  \end{enumerate}
\end{theorem}
We refer to $\mathcal{O}_\pi$ as the \emph{coadjoint multiorbit} attached to $\pi$.  The terminology reflects that in general, $\mathcal{O}_\pi$ is a finite union of coadjoint orbits.  We refer to \cite[\S6]{nelson-venkatesh-1} for further discussion of this point, as well as the precise definition and normalization of the symplectic measures.

\section{Basic operator assignment}
\label{sec:micr-analys-lie}
\label{sec:basic-oper-assignm}
We fix a Lie group $G$ over $\mathbb{R}$ and retain the notation of \S\ref{sec:repr-theor-prel}.

\subsection{Measures and Fourier transforms}
We fix Haar measures $d x$ on $\mathfrak{g}$ and $d \xi$ on $\mathfrak{g}^\wedge$ that are Fourier dual in the sense that for Schwartz functions $a$ on $\mathfrak{g}^\wedge$ and $\phi$ on $\mathfrak{g}$, the Fourier transforms given by \index{Lie algebra!Fourier transforms $a^\vee, \phi^\wedge$}
\begin{equation*}
  a^\vee(x) := \int _{\xi \in \mathfrak{g} ^\wedge }
  a(\xi) e^{- x \xi} \, d \xi,
  \quad
  \phi^\wedge(\xi)
  = \int _{x \in \mathfrak{g} }
  \phi(x)
  e^{x \xi} \, d x
\end{equation*}
define mutually inverse transforms.  We note that the Schwartz measure $a^\vee(x) \, d x$ on $\mathfrak{g}$ depends only upon $a$, not upon the choice of Haar measure.

\subsection{Nice cutoffs}\label{sec:nice-cutoffs}
We fix small enough symmetric neighborhoods $\mathcal{G}_1 \subset \mathcal{G}_0$ of the origin in $\mathfrak{g}$, with $\mathcal{G}_1$ taken small enough in terms of $\mathcal{G}_0$.  We assume in particular that
\begin{itemize}
\item the exponential map defines an analytic diffeomorphism between $\mathcal{G}_1$ and its image, and
\item $\exp(x) \exp(y) \in \exp(\mathcal{G}_1)$ whenever $x,y \in \mathcal{G}_0$.
\end{itemize}

By a \emph{nice cutoff}, \index{Lie algebra!nice cutoff $\chi$} we mean a smooth function $\chi : \mathfrak{g} \rightarrow [0,1]$ that is supported in $\mathcal{G}_0$, satisfies $\chi(x) = \chi(-x)$, and is identically $1$ in some neighborhood of the origin.

\subsection{The wavelength parameter $\h$}\label{sec:wavel-param-h}
\index{wavelength parameter $\h$} Many of the definitions below depend upon the choice of a positive ``wavelength parameter'' $\h \in (0,1]$.  One can take, e.g., $\h = 1$, but we will be interested primarily in the case that $\h \lll 1$.  Since we will work primarily with one value of $\h$ at a time, we will often suppress the dependence of our definitions upon $\h$.

\subsection{From
  Schwartz functions on $\mathfrak{g}^\wedge$
  to distributions on $G$}
Let $a$ be a Schwartz function on $\mathfrak{g}^\wedge$ and let $\chi$ be a nice cutoff.  We define the rescaling \index{symbols!rescaling $a_{\h}(\xi) := a(\h \xi)$} $a_{\h}$ of $a$ by
\begin{equation*}
  a_{\h}(\xi) := a(\h \xi).
\end{equation*}
We obtain the rescaled inverse Fourier transform $a_{\h}^\vee$, given explicitly by
\begin{equation*}
  a_{\h}^\vee(x) = \h^{-\dim(\mathfrak{g})} a^\vee(x/\h).
\end{equation*}
The cutoff $\chi a_{\h}^\vee$ of the latter is then a smooth compactly-supported function on $\mathfrak{g}$.  We denote by \index{operators!$\widetilde{\Opp}$}
\[
  \widetilde{\Opp}_{\h}(a:\chi) := \exp_* (\chi a_{\h}^\vee \, d x)
\]
the pushforward under the exponential map of the smooth compactly-supported distribution $\chi(x) a_{\h}^\vee(x) \, d x$ on $\mathfrak{g}$.  Thus $\widetilde{\Opp}_{\h}(a:\chi)$ is a smooth distribution on $G$, supported in a small neighborhood of the identity element and given there by a smooth multiple of any (left or right) Haar measure $d g$ on $G$.  Explicitly,
\begin{equation*}
  \int _{g \in G}
  \widetilde{\Opp}_{\h}(a:\chi)(g)  \phi(g) 
  =
  \int _{x \in \mathfrak{g}}
  \chi(x) a_{\h}^\vee(x)
  \phi(\exp(x))
  \, d x
  \quad 
  \text{ for } \phi : G \rightarrow \mathbb{C}.
\end{equation*}
When $G$ is unimodular, we identify $\widetilde{\Opp}_{\h}(a,\chi)$ with an element of $C_c^\infty(G)$ by dividing by some fixed (left and right) Haar measure $d g$.

\subsection{Definition of operator assignment}
Given a representation $\pi$ of $G$ (on, e.g., a locally convex space), we define an operator $\Opp_{\h}(a:\pi, \chi) \in \End(\pi)$ \index{operators!$\Opp$} by integrating: for $v \in \pi$,
\begin{align*}
  \Opp_{\h}(a:\pi,\chi)
  v
  &= \int _{g \in G}
    \widetilde{\Opp}_{\h}(a:\chi)(g) \pi(g) v \\
  &=
    \int _{x \in \mathfrak{g}}
    \chi(x) a_{\h}^\vee(x) \pi(\exp(x))  v
    \, d x.
\end{align*}

\subsection{Adjoints}\label{sec:adjoints}
The significance of the evenness condition $\chi(x) = \chi(-x)$ imposed in \S\ref{sec:nice-cutoffs} is that if $\pi$ is a unitary representation, then the adjoint of such an operator is given by
\begin{equation}\label{eqn:adjoint-opp-a}
  \Opp_{\h}(a:\pi,\chi)^*
  = 
  \Opp_{\h}(\bar{a} : \pi , \chi).
\end{equation}

\subsection{Star products and composition}\label{sec:star-product:-basic}
We recall the star product considered in \cite{nelson-venkatesh-1} (and earlier, in \cite{MR1064995}).  Let $\phi_1, \phi_2 \in C_c^\infty(\mathfrak{g})$ be supported close enough to the origin.  The distribution $\phi_j(x) \, d x$ pushes forward under the exponential map to a smooth distribution $f_j$ on $G$ supported near the identity.  The convolution $f_1 \ast f_2$ is another smooth distribution on $G$ supported near the identity.  We define $\phi_1 \star \phi_2 \in C_c^\infty(\mathfrak{g})$ to be the function supported near the origin for which $(\phi_1 \star \phi_2)(x) \, d x$ pushes forward to $f_1 \ast f_2$.  Thus for any $\Psi \in C_c^\infty(G)$, we have
\[
  \int_{x_1, x_2 \in \mathfrak{g}} \phi_1(x_1) \phi_2(x_2) \Psi(\exp(x_1) \exp(x_2)) \, d x_1 \, d x_2 = \int_{x \in \mathfrak{g}} (\phi_1 \star \phi_2)(x) \Psi(\exp(x)) \, d x.
\]

Given a nice cutoff $\chi$ and Schwartz functions $a,b$ on $\mathfrak{g}^\wedge$, we define the star product $a \star b$ to be the Schwartz function on $\mathfrak{g}^\wedge$ given by the formula \index{symbols!star product $\star, \star_{\h}$}
\begin{equation}\label{eqn:unscaled-star-product}
  a \star b := (\chi a^\vee \star \chi b ^\vee)^\wedge.
\end{equation}
The definition depends implicitly upon $\chi$.  The rescaled star product $a \star_{\h} b$ is characterized by
\begin{equation}\label{eqn:rescaled-star-product}
  (a \star_{\h} b)_{\h} = a_{\h} \star b_{\h}.
\end{equation}

For small elements $x,y \in \mathfrak{g}$, we write
\begin{equation}\label{eq:x-ast-y-:=-logexpx-expy-}
  x \ast y := \log(\exp(x) \exp(y))
\end{equation}
for the small element of $\mathfrak{g}$ for which $\exp(x \ast y) = \exp(x) \exp(y)$.  If $\chi$ and $\chi '$ are nice cutoffs with the property that $\chi '( x \ast y) = 1$ whenever $\chi(x) \chi(y) \neq 0$, then
\begin{equation*}
  \widetilde{\Opp}_{\h}(a:\chi)
  \ast
  \widetilde{\Opp}_{\h}(b:\chi)
  =
  \widetilde{\Opp}_{\h}(a \star_{\h} b:\chi'),
\end{equation*}
hence for a representation $\pi$ of $G$,
\begin{equation}\label{eq:composition-for-basic-operator-assignment}
  \Opp_{\h}(a:\pi,\chi)
  {\Opp}_{\h}(b:\pi,\chi)
  =
  {\Opp}_{\h}(a \star_{\h} b:\pi,\chi').
\end{equation}

\subsection{Abbreviations}
We abbreviate $\Opp_{\h}(a:\pi,\chi)$ by $\Opp_{\h}(a:\pi)$ or simply $\Opp_{\h}(a)$ when $\chi$ and $\pi$ are clear by context, and similarly for $\widetilde{\Opp}_{\h}$.  We note that the precise choice of $\chi$ is often unimportant (see for instance \cite[\S5.4]{nelson-venkatesh-1}).  We abbreviate further by simply \index{operators!$\Opp$}
\[
  \Opp(a)
\]
when the wavelength parameter $\h$ is understood.\footnote{This convention differs from that in \cite{nelson-venkatesh-1}, where $\Opp$ refers to the ``unscaled'' operator assignment, i.e., the case $\h = 1$.  We refer here to the latter more verbosely as $\Opp_{1}$.}

\section{Symbols
  and star product asymptotics}\label{sec:symbols-star-product}
We retain the setting and notation of \S\ref{sec:basic-oper-assignm}.

\subsection{Spaces of symbols}\label{sec:spaces-symbols}
We fix a basis of $\mathfrak{g}^\wedge$.  With respect to this choice, we define for each multi-index
\[
  \alpha = (\alpha_1,\dotsc,\alpha_{\dim(\mathfrak{g})}) \in \mathbb{Z}_{\geq0}^{\dim(\mathfrak{g})}\] a differential operator $\partial^\alpha$ on $C^\infty(\mathfrak{g}^\wedge)$ (see \cite[\S4.1]{nelson-venkatesh-1} for details).

For each integer $m$ and multi-index $\alpha$, we define a seminorm $\nu_{m,\alpha}$ on $C^\infty(\mathfrak{g}^\wedge)$ (valued in the extended nonnegative reals) by the formula
\[
  \nu_{m,\alpha}(a) := \sup_{\xi \in \mathfrak{g}^\wedge} \langle \xi \rangle^{|\alpha| - m} \left\lvert \partial^\alpha a(\xi) \right\rvert,
\]
where
\begin{equation*}
  |\alpha| := \sum \alpha_j
\end{equation*}
denotes the order of $\alpha$ and
\begin{equation*}
  \langle \xi \rangle := (1 + |\xi|^2)^{1/2}.
\end{equation*}
We set\footnote{The space $\underline{S}^m$ was denoted in \cite[\S4.3]{nelson-venkatesh-1} by $S^m$.  We do not adopt the latter notation here.}  \index{symbols!underlying space $\underline{S}^m$}
\[
  \underline{S}^m := \{ a \in C^\infty(\mathfrak{g}^\wedge): \nu_{m,\alpha}(a) < \infty \text{ for all } \alpha \}.
\]
We extend the definition of $\underline{S}^m$ to $m = \infty$ (resp. $m=-\infty$) by taking the union (resp.  intersection) over all integers $m$.  We write more verbosely $\underline{S}^m(\mathfrak{g}^\wedge)$ when we wish to indicate which group is being considered.  For $m < \infty$, the space $\underline{S}^m$ is a Frechet space equipped with a distinguished family of seminorms, while $\underline{S}^\infty$ is an inductive limit of such spaces.  We note that $\underline{S}^{-\infty}$ is the Schwartz space on $\mathfrak{g}^\wedge$.

\subsection{Star product:
  formal expansion and
extension to symbols}\label{sec:star-prod-extens}
Recall from \S\ref{sec:star-product:-basic} that the star product $\star$ depends upon the choice of a nice cutoff $\chi$.  For small $x,y \in \mathfrak{g}$, write
\begin{equation*}
  \{x,y\} := x \ast y - x - y,
\end{equation*}
with $x \ast y = \log(\exp(x) \exp(y))$ as in \S\ref{sec:star-product:-basic}.  For the moment, let $a$ and $b$ be Schwartz functions on $\mathfrak{g}^\wedge$.  The rescaled star product admits the integral representation (see \cite[(4.2)]{nelson-venkatesh-1})
\begin{equation}\label{eqn:rescaled-star-integral-rep}
  a \star_{\h} b(\zeta)
  = \int_{x,y \in \mathfrak{g}}
  a_{\h}^\vee(x)
  b_{\h}^\vee(y)
  e^{\langle x, \zeta/h  \rangle}
  e^{\langle y, \zeta/\h  \rangle}
  e^{\langle \{x,y\}, \zeta/\h \rangle}
  \chi(x) \chi(y) \, d x \, d y.
\end{equation}

The analytic function $e^{\{ x , y \} \zeta} = e^{\langle \{ x , y \} ,\zeta \rangle}$ admits the Taylor expansion
\begin{equation}\label{eqn:taylor-expand-coarse}
  e^{\{x,y\} \zeta}
  = \sum_{\alpha,\beta,\gamma} c_{\alpha \beta \gamma}
  x^\alpha y^\beta \zeta^\gamma,
\end{equation}
where the multi-indices $\alpha,\beta,\gamma$ range over $\mathbb{Z}_{\geq 0}^{\dim(\mathfrak{g})}$ and we write
\[
  x^\alpha = \prod x_j^{\alpha_j}, \quad y^\beta = \prod y_j^{\beta_j} \text{ and } \zeta^\gamma = \prod \zeta_j^{\gamma_j}
\]
for the corresponding monomials, defined using coordinates $\mathfrak{g} \cong \mathbb{R}^{\dim(\mathfrak{g})}$ and $\mathfrak{g}^\wedge \cong i \mathbb{R}^{\dim(\mathfrak{g})}$ with respect to which the natural pairing $\mathfrak{g} \otimes \mathfrak{g}^\wedge \rightarrow i \mathbb{R}$ is given by $x \xi = \sum_{j=1}^{\dim(\mathfrak{g})} x_j \xi_j$.

Using the estimate
\begin{equation*}
  \{x,y\} \ll |x| \cdot |y|
  \quad (\text{for } x, y \lll 1),
\end{equation*}
we may see as in \cite[\S4.2]{nelson-venkatesh-1} that the coefficient $c_{\alpha \beta \gamma}$ in the Taylor expansion \eqref{eqn:taylor-expand-coarse} of $e^{\{x,y\} \zeta}$ vanishes unless
\begin{equation}\label{eqn:gamma-less-alpha-beta}
  |\gamma| \leq \min(|\alpha|, |\beta|).
\end{equation}
Grouping the terms in \eqref{eqn:taylor-expand-coarse} by their homogeneity degree $j = |\alpha| + |\beta| - |\gamma|$ (which is nonnegative in view of \eqref{eqn:gamma-less-alpha-beta}) and applying Fourier inversion to \eqref{eqn:rescaled-star-integral-rep} suggests the formal asymptotic expansion
\begin{equation}\label{eqn:star-prod-formal-suggestion}
  a \star_{\h} b
  \sim
  \sum _{j \geq 0} \h^j a \star^j b,
\end{equation}
where $\star^j$ denotes the finite bidifferential operator defined by
\begin{equation*}
  a \star^j b(\zeta) = \sum _{|\alpha| + |\beta| - |\gamma| = j}
  c_{\alpha \beta \gamma}
  \zeta^\gamma \partial^\alpha a(\zeta) \partial^\beta b(\zeta).
\end{equation*}

Various rigorous forms of \eqref{eqn:star-prod-formal-suggestion} were verified in \cite{nelson-venkatesh-1}.  For instance:
\begin{theorem}\label{thm:basic-star-product-extension-properties}
  The star product $\star$ and its rescaled variant $\star_{\h}$, defined initially for Schwartz functions, extend uniquely to a compatible family of continuous bilinear maps
  \begin{equation}\label{eqn:basic-star-h-underlined-thm-stmt}
    \star_{\h}
    : \underline{S}^{m_1} \times
    \underline{S}^{m_2} \rightarrow \underline{S}^{m_1 + m_2}
  \end{equation}
  with the convention $\infty + (-\infty) := -\infty$.  For each $(a,b) \in \underline{S}^{m_1} \times \underline{S}^{m_2}$ and $J \in \mathbb{Z}_{\geq 0}$, we have
  \[
    a \star_{\h} b \equiv \sum _{0 \leq j < J} \h^j a \star^j b \mod{\underline{S}^{m_1 + m_2 - J}},
  \]
  where the remainder term defines a continuous map $\underline{S}^{m_1} \times \underline{S}^{m_2} \rightarrow \underline{S}^{m_1 + m_2 - J}$.
\end{theorem}
\begin{proof}
  See \cite[Thm 1]{nelson-venkatesh-1}.
\end{proof}

\subsection{Basic symbol classes}\label{sec:basic-symbol-classes}
As motivation, let us explicate the continuity of the map \eqref{eqn:basic-star-h-underlined-thm-stmt}, restricting for simplicity to the case that $m_1$ and $m_2$ are finite.
\begin{itemize}
\item \emph{Let $G$ be a Lie group.  Let $m_1, m_2 \in \mathbb{Z}$.  Let $\h \in (0,1]$.  Let $\chi$ be a nice cutoff, so that $\star$ and $\star_{\h}$ are defined.  For each multi-index $\gamma$, there exists $C \geq 0$ and a finite family $\mathcal{N}$ of pairs of multi-indices $(\alpha,\beta)$ so that for all $(a,b) \in \underline{S}^{m_1} \times \underline{S}^{m_2}$, we have}
  \[\nu_{m_1+m_2,\gamma}(a \star_{\h} b)
    \leq C \max_{(\alpha,\beta) \in \mathcal{N}} \nu_{m_1,\alpha}(a) \nu_{m_2,\beta}(b).\]
\end{itemize}
In this statement, $\mathcal{N}$ and $C$ may depend freely upon $\h$.  The purpose of this section is to quantify that dependence under additional hypotheses on $a$ and $b$.

\index{symbols!basic class $S_\delta^m$}
\begin{definition}\label{defn:basic-symbol-class}
  For fixed $0 \leq \delta < 1$ and $m \in \mathbb{Z}$, we define the following class (\S\ref{sec:classes}) of symbols:
  \[S^m_\delta := \left\{ a \in \underline{S}^ \infty : \nu_{m,\alpha}(a) \ll \h^{- \delta |\alpha|} \text{ for each fixed } \alpha \right\}.
  \]
  Thus $S^m_\delta$ consists of those $a \in \underline{S}^\infty $ satisfying
  \begin{equation}\label{eqn:S-m-delta-defn}
    \partial^\alpha a(\xi) \ll \h^{-\delta |\alpha|}
    \langle \xi  \rangle^{m-|\alpha|}
  \end{equation}
  for each fixed multi-index $\alpha$ and all $\xi \in \mathfrak{g}^\wedge$.  (More pedantically, $S^m_\delta$ is the pair $(\underline{S}^\infty,P)$, where $P(a)$ is the predicate ``for each fixed $\alpha$ there is a fixed $C$ so that $|\partial^\alpha a(\xi)| \leq C \h^{-\delta |\alpha|} \langle \xi \rangle^{m-|\alpha|}$ for all $\xi$.'')

  We extend the definition to $m = \pm \infty$ by taking for $S^\infty_\delta$ the union of $S^m_\delta$ over all fixed $m$ and for $S^{-\infty}_\delta$ the intersection of $S^m_\delta$ over all fixed $m$.  Thus in all cases, $S^m_\delta \subseteq \underline{S}^\infty$.  We write more verbosely $S^{m}_\delta(\mathfrak{g}^\wedge)$ or $S^{m}_\delta[\h]$ or $S^{m}_\delta(\mathfrak{g}^\wedge)[\h]$ when we wish to indicate which group and/or wavelength parameter we are considering.
\end{definition}

\begin{remark}
  This definition is a mild reformulation of that given in \cite[\S4.4]{nelson-venkatesh-1}.  There, we worked instead with functions $a(\xi,\h)$ of two variables $(\xi,\h) \in \mathfrak{g}^\wedge \times (0,1]$, smooth in the first variable, satisfying $|\partial^\alpha a(\xi,\h)| \leq C_\alpha \h^{-\delta |\alpha|} \langle \xi \rangle^{m-|\alpha|}$.  The two definitions have similar content, and the reader will lose little by working with whichever feels more comfortable.  The practical disadvantage of the definition of \cite[\S4.4]{nelson-venkatesh-1} is that its user must require many quantities to be ``$\h$-dependent.''  This disadvantage would be more severe for the applications pursued in this paper.  We discuss how to translate between the two formulations in the proof of Theorem \ref{thm:nv-star-prod-asymp-h-dependent}.
\end{remark}

\begin{remark}\label{rmk:translate-symbol-class-statement}
  We emphasize, following \S\ref{sec:asymptotic-notation}, that ``let $a \in S_\delta^{m}$'' itself carries no semantic content, but any complete statement involving $S_\delta^m$ does.  For example, statement (i) below, formulated in terms of symbol classes, is formally equivalent to statement (ii), formulated explicitly.  Both statements are moreover equivalent to (iii), an apparent strengthening of (ii) that follows by linearity.  (Each statement also happens to be true and readily provable using integration by parts.)
  \begin{enumerate}[(i)]
  \item \emph{Fix a Lie group $G$.  Let $\h \in (0,1]$.  Fix $\delta \in [0,1)$.  Let $a \in S_\delta^{-\infty}$.  Then the integral defining $a^\vee$ converges absolutely, and we have $a^\vee(x) \ll \lvert \h^{\delta} x \rvert^{-N}$ for all fixed $N \geq 0$ and all $x$.}
  \item \emph{Let $G$ be a Lie group. Let $\delta \in [0,1)$.  Let $(C_{m,\alpha})$ be a collection of nonnegative reals indexed by $m \in \mathbb{Z}$ and $\alpha \in \mathbb{Z}_{\geq 0}^{\dim(\mathfrak{g})}$.  For each $N \geq 0$, there is a finite subset $\mathcal{N}$ of $\mathbb{Z} \times \mathbb{Z}_{\geq 0}^{\dim(\mathfrak{g})}$ and $C \geq 0$ with the following property.  Let $\h \in (0,1]$.  Let $a \in \underline{S}^{\infty}$ with $|\partial^{\alpha} a(\xi)| \leq C_{m,\alpha} \h^{- \delta |\alpha|} \langle \xi \rangle^{m-|\alpha|}$ for $(m,\alpha) \in \mathcal{N}$ and all $\xi$.  Then the integral defining $a^\vee$ converges absolutely, and we have $|a^\vee(x)| \leq C \lvert \h^{\delta} x \rvert^{-N}$ for all $x$.}
  \item \emph{Let $G$ be a Lie group. Let $\delta \in [0,1)$.  For each $N \geq 0$, there is a finite subset $\mathcal{N}$ of $\mathbb{Z} \times \mathbb{Z}_{\geq 0}^{\dim(\mathfrak{g})}$ and $C \geq 0$ with the following property.  Let $\h \in (0,1]$.  Let $a \in \underline{S}^{\infty}$ with
      \[
        \nu(a) := \max _{(m,\alpha) \in \mathcal{N} } \sup_{\xi} \h^{\delta |\alpha|} \langle \xi \rangle^{|\alpha|-m} |\partial^{\alpha} a(\xi)| < \infty.
      \]
      Then the integral defining $a^\vee$ converges absolutely, and we have $|a^\vee(x)| \leq C \nu(a) \lvert \h^{\delta} x \rvert^{-N}$ for all $x$.}
  \end{enumerate}
  Every theorem that we state involving such classes admits a similar translation.  We hope the passage from the first statement to the second is intuitively clear thanks to the informal association between ``fixed'' and ``absolute constant,'' and refer again to \cite[\S2]{MR469763} for details on how to carry out such translations algorithmically.
\end{remark}

For a positive real $c$, we write $c S_\delta^m$ for the class of symbols of the form $c v$, with $v \in S_\delta^m$.  We write $\h^\infty S_\delta^m$ for the intersection over all fixed $N$ of the classes $\h^N S_\delta^m$.  The class $\h^\infty S^{-\infty}_{\delta}$ is independent of $\delta$; it consists of $a \in \underline{S}^\infty$ satisfying $\partial^\alpha a(\xi) \ll \h^N \langle \xi \rangle^{-N}$ for all fixed $\alpha$ and $N$.  This class, which should be regarded as consisting of ``negligible'' symbols, will be denoted simply by \index{symbols!negligible class $\h^\infty S^{-\infty}$}
\[
  \h^\infty S^{-\infty}.
\]

\index{symbols!star product $\star, \star_{\h}$}
\begin{theorem}\label{thm:nv-star-prod-asymp-h-dependent}
  Fix a nice cutoff $\chi$.  Fix $\delta \in [0,1/2)$ and $m_1, m_2 \in \mathbb{Z} \cup \{\pm \infty \}$.  The star product, extended as in Theorem \ref{thm:basic-star-product-extension-properties}, induces a class map
  \begin{equation}\label{eqn:star-h-symbol-class-mapping-old}
    \star_{\h} : S^{m_1}_{\delta} \times S^{m_2}_{\delta}
    \rightarrow S^{m_1+m_2}_{\delta}.
  \end{equation}
  For $(a,b) \in S^{m_1}_{\delta} \times S^{m_2}_{\delta}$ and any fixed $j \in \mathbb{Z}_{\geq 0}$, we have
  \begin{equation}\label{eqn:star-j-old-mapping-property}
    a \star^j b \in \h^{-2 \delta j} S^{m_1 + m_2 - j}_{\delta}.
  \end{equation}
  Moreover, for any fixed $J \in \mathbb{Z}_{\geq 0}$,
  \begin{equation}\label{eqn:star-old-asymp-expn}
    a \star_{\h} b
    \equiv \sum _{0 \leq j < J}
    \h^j a \star^j b \mod{ \h^{(1 - 2 \delta) J}
      S^{m_1 + m_2 - J}_{\delta}}.
  \end{equation}
\end{theorem}
\begin{proof}
  As we will explain, this follows formally from \cite[Thm 1]{nelson-venkatesh-1}.
  
  Recall from \S\ref{sec:spaces-symbols} the spaces $\underline{S}^m$, defined in \cite[\S4.3]{nelson-venkatesh-1}, consisting of smooth functions $a : \mathfrak{g}^\wedge \rightarrow \mathbb{C}$ such that for each $\alpha$, there exists $C_\alpha(a) \geq 0$ so that $|\partial^\alpha a(\xi)| \leq C_\alpha(a) \langle \xi \rangle^{m-|\alpha|}$ for all $\xi$.  We now record another definition, this time from \cite[\S4.4]{nelson-venkatesh-1}, which we do not otherwise use in this paper.  Given a subset $\mathcal{H}$ of $(0,1]$ having $0$ as an accumulation point, let $\underline{S}^m_\delta$ denote the space of functions $a : \mathcal{H} \times \mathfrak{g}^\wedge \rightarrow \mathbb{C}$, smooth in the second variable, so that for each $\alpha$, there exists $C_\alpha(a) \geq 0$ so that for each $\h \in \mathcal{H}$, the specialization $a[\h] : \mathfrak{g}^\wedge \rightarrow \mathbb{C}, \xi \mapsto a(\h,\xi)$ satisfies
  \begin{equation*}
    |\partial^\alpha a[\h](\xi)| \leq C_\alpha(a) \h^{-\delta |\alpha|} \langle \xi \rangle^{m-|\alpha|}
  \end{equation*}
  for all $\xi$.  For $(a,b) \in \underline{S}_\delta^{m_1} \times \underline{S}_{\delta}^{m_2}$, we define $a \star_{\h} b : \mathcal{H} \times \mathfrak{g}^\wedge \rightarrow \mathbb{C}$ by requiring that $(a \star_{\h} b)[\h] = a[\h] \star_{\h} b[\h]$ for each $\h \in \mathcal{H}$.  We extend these definitions to infinite exponents $m = \pm \infty$ by taking unions or intersections.  Each of these spaces comes equipped with an ``evident topology'' defined by a family of seminorms (detailed in \emph{loc.\ cit.}).\footnote{ The space $\underline{S}^m_\delta$ was denoted ``$S^m_{\delta}$'' in \cite[\S4.4]{nelson-venkatesh-1}.  The latter notation is used differently here.  }

  We verify \eqref{eqn:star-h-symbol-class-mapping-old} in detail; we may similarly deduce \eqref{eqn:star-j-old-mapping-property} and \eqref{eqn:star-old-asymp-expn} from the corresponding assertions in \cite[Thm 1]{nelson-venkatesh-1}.  We consider the case $m_1, m_2 \in \mathbb{Z}$; the cases in which some $m_j = \pm \infty$ follow similarly.  As shown in \cite[Thm 1]{nelson-venkatesh-1} (and noted partially in Theorem \ref{thm:basic-star-product-extension-properties}), the star product and its rescaling define compatible families of continuous bilinear maps
  \begin{equation}\label{eqn:old-ast-S-m1-m2-sum}
    \star : \underline{S}^{m_1} \times \underline{S}^{m_2} \rightarrow \underline{S}^{m_1 + m_2}.
  \end{equation}
  \begin{equation}\label{eqn:old-ast-S-m1-m2-sum-h}
    \star_{\h} : \underline{S}^{m_1}_{\delta} \times \underline{S}^{m_2}_{\delta}
    \rightarrow \underline{S}^{m_1 + m_2}_{\delta}.
  \end{equation}
  We will see that \eqref{eqn:star-h-symbol-class-mapping-old} is a formal consequence of these mapping properties.

  For clarity, we employ here the more verbose notation $S^m_\delta = S^m_\delta[\h]$ for our symbol classes.  Let $\h \in (0,1]$ and $(a,b) \in S^{m_1}_{\delta}[\h] \times S^{m_2}_\delta[\h]$.  By Theorem \ref{thm:basic-star-product-extension-properties}, the rescaled star product $a \star_{\h} b$ lies in the space $\underline{S}^{m_1 + m_2}$.  We must check that it in fact lies in the class $S_{\delta}^{m_1 + m_2}[\h]$.  In other words, we must verify the following:
  \begin{itemize}
  \item \emph{Fix a Lie group $G$ and a nice cutoff $\chi$.  Let $\h \in (0,1]$.  Fix $\delta \in [0,1/2)$.  Let $a \in S_\delta^{m_1}[\h]$ and $b \in S_\delta^{m_2}[\h]$.  Then $a \star_{\h} b \in S_\delta^{m_1 + m_2}[\h]$.}
  \end{itemize}
  As in Remark \ref{rmk:translate-symbol-class-statement}, this is formally equivalent to:
  \begin{itemize}
  \item \emph{Let $G$ be a Lie group.  Let $\chi$ be a nice cutoff.  Let $\delta \in [0,1/2)$.  For each $m \in \mathbb{Z}$ and $\alpha \in \mathbb{Z}_{\geq 0}^{\dim(\mathfrak{g})}$ and $\h \in (0,1]$, we define a seminorm $\mu_{ m, \alpha}[\h]$ on $C^\infty(\mathfrak{g}^\wedge)$ by the formula}
    \[
      \mu_{m,\alpha}[\h](a) := \sup_{\xi} \h^{\delta |\alpha|} \langle \xi \rangle^{|\alpha| - m} |\partial^{\alpha} a(\xi)|.
    \]
    \emph{ Write $\mu_{m,\alpha}$ for the map $\h \mapsto \mu_{m,\alpha}[\h]$.  Set
      \begin{equation*}
        \mathcal{N}(m) := \{ \mu_{m,\alpha} : \alpha \in \mathbb{Z}_{\geq 0}^{\dim(\mathfrak{g})} \}.
      \end{equation*}
    }

    \emph{Let $m_1, m_2 \in \mathbb{Z}$.  For each $\mu \in \mathcal{N}({m_1 + m_2})$, there exists $C \geq 0$ and finite subsets $\mathcal{N}_1 \subseteq \mathcal{N}(m_1)$ and $\mathcal{N}_2 \subseteq \mathcal{N}(m_2)$ with the following property.  For each $(a,b) \in \underline{S}^{\infty} \times \underline{S}^{\infty}$ and $\h \in (0,1]$ for which the quantities}
    \[
      \nu_{\mathcal{N}_1}[\h](a) := \max_{\nu \in \mathcal{N}_1} \nu[\h](a) \quad \text{ and } \quad \nu_{\mathcal{N}_2}[\h](b) := \max_{\nu \in \mathcal{N}_2} \nu[\h](b)
    \]
    \emph{are finite, the rescaled star product $a \star_{\h} b$ satisfies}
    \[
      \mu[\h](a \star_{\h} b) \leq C \nu_{\mathcal{N}_1}[\h](a) \nu_{\mathcal{N}_2}[\h](b).
    \]
  \end{itemize}
  Suppose this fails.  We may then find a Lie group $G$, a nice cutoff $\chi$, and elements $\delta \in [0,1/2)$, $m_1, m_2 \in \mathbb{Z}$, $\mu \in \mathcal{N}(m_1 + m_2)$ with the following property: for each $C \geq 0$ and all finite subsets $\mathcal{N}_1 \subseteq \mathcal{N}(m_1)$ and $\mathcal{N}_2 \subseteq \mathcal{N}(m_2)$, there exists $(a,b,\h) \in \underline{S}^{\infty} \times \underline{S}^{\infty} \times (0,1]$ so that $\mu(a \star_{\h} b) > C \nu_{\mathcal{N}_1}[\h](a) \nu_{\mathcal{N}_2}[\h](a)$.  (We incorporate into this inequality the assertion that each factor on the RHS is finite.)  By a diagonalization argument and linearity, we may find a sequence of elements $(a_j, b_j, \h_j) \in \underline{S}^{\infty} \times \underline{S}^{\infty} \times (0,1]$ so that for all finite subsets $\mathcal{N}_1 \subseteq \mathcal{N}(m_1)$ and $\mathcal{N}_2 \subseteq \mathcal{N}(m_2)$, we have
  \begin{equation}\label{eqn:supposition-a-j-b-j-h-j}
    \sup_j
    \nu_{\mathcal{N}_1}[\h_j](a_j) < \infty,
    \quad
    \sup_j
    \nu_{\mathcal{N}_2}[\h_j](b_j) < \infty,
    \quad 
    \lim_{j \rightarrow \infty}
    \mu[\h_j](a_j \star_{\h_j} b_j)
    = \infty.
  \end{equation}
  By passing to a subsequence, we may reduce to the following cases.
  \begin{enumerate}[(i)]
  \item $\inf_j \h_j > 0$.  In this case, $\star_{\h}$ and $\star$ differ mildly, and our supposition \eqref{eqn:supposition-a-j-b-j-h-j} contradicts the continuity of \eqref{eqn:old-ast-S-m1-m2-sum}.
  \item $\h_j \rightarrow 0$.  By passing to a further subsequence, we may assume that the sequence $\h_j$ is decreasing.  Set $\mathcal{H} := \{h_j\}$, and define $a,b : \mathcal{H} \times \mathfrak{g}^\wedge \rightarrow \mathbb{C}$ by requiring that $a[\h_j] = a_j$ and $b[\h_j] = b_j$.  By \eqref{eqn:supposition-a-j-b-j-h-j}, we then have $a \in \underline{S}^{m_1}_{\delta}$ and $b \in \underline{S}^{m_2}_{\delta}$, but $a \star_{\h} b \notin \underline{S}^{m_1 + m_2}_{\delta}$, contrary to \eqref{eqn:old-ast-S-m1-m2-sum-h}.\qedhere
  \end{enumerate}

\end{proof}

\subsection{Refined symbol classes}\label{sec:refin-symb-class}
Supposing now that the Lie group $G$ is a connected reductive group over $\mathbb{R}$ (so that centralizers of regular elements behave as nicely as possible), we introduce some refinements of the basic symbol classes described above.  For motivation concerning the introduction of such refinements, we refer to \S\ref{sec:microlocal-calculus}.

\subsubsection{Coordinates
  tailored to regular coadjoint
  elements}\label{sec:coord-tail-regul}
Recall from \S\ref{sec:prelim-reductive-groups} that $\mathfrak{g}^\wedge_{\reg}$ denotes the set of regular elements in $\mathfrak{g}^\wedge$.  Let $\tau \in \mathfrak{g}^\wedge_{\reg}$.  \index{Lie algebra!$\mathfrak{g}_\tau, \mathfrak{g}_\tau^\perp, \mathfrak{g}_\tau^{\flat}, \mathfrak{g}_\tau^{\perp \flat}$} The centralizer $\mathfrak{g}_\tau$ in $\mathfrak{g}$ is then an abelian subalgebra of dimension equal to the rank of $\mathfrak{g}$ (see \cite{Kurtzke}).  As $\tau$ varies in $\mathfrak{g}^\wedge_{\reg}$, the centralizer $\mathfrak{g}_{\tau}$ varies smoothly in the Grassmannian of $\rank(\mathfrak{g})$-dimensional subspaces of $\mathfrak{g}$.

We may form the complement $\mathfrak{g}_\tau^\perp \subseteq \mathfrak{g}^\wedge$ of $\mathfrak{g}_\tau$ with respect to the canonical duality $\mathfrak{g} \otimes \mathfrak{g}^\wedge \rightarrow i \mathbb{R}$.  It has codimension equal to the rank of $\mathfrak{g}$.

The tangent space in $\mathfrak{g}^\wedge$ to the coadjoint orbit $G \cdot \tau$ at $\tau$ is the coset $\tau + \mathfrak{g}_\tau^\perp$.  This fact follows from the identities
\begin{equation}\label{eqn:adjoint-map-adjointness-zeta-centralizer}
  \langle [x,y], \tau \rangle = \langle x, [y,\tau] \rangle =
  - \langle y, [x,\tau] \rangle
  \quad \text{ for } x,y \in \mathfrak{g}
\end{equation}
and a comparison of dimensions.

We fix an inner product on $\mathfrak{g}^\wedge$.  We use this inner product to define an orthogonal complement $\mathfrak{g}_{\tau}^{\perp \flat} \subseteq \mathfrak{g}^\wedge$ to $\mathfrak{g}_\tau^\perp$.  We have the direct sum decomposition
\begin{equation}\label{eqn:decompose-via-zeta}
  \mathfrak{g}^\wedge = \mathfrak{g}_{\tau}^\perp \oplus \mathfrak{g}_{\tau}^{\perp \flat}.
\end{equation}

\begin{remark}
  It is convenient, but not essential, to take for $\mathfrak{g}_{\tau}^{\perp \flat}$ the orthogonal complement of $\mathfrak{g}_{\tau}^{\perp}$ with respect to an inner product; we could take instead any linear complement arising from a smooth assignment $\tau \mapsto \mathfrak{g}_{\tau}^{\perp \flat}$.
\end{remark}

Let $\mathfrak{g}_\tau^{\flat} \subseteq \mathfrak{g}$ denote the complement of $\mathfrak{g}_\tau^{\perp \flat}$ with respect to the canonical duality between $\mathfrak{g}$ and $\mathfrak{g}^\wedge$.  We obtain the following decomposition, dual to \eqref{eqn:decompose-via-zeta}:
\begin{equation}\label{eqn:decompose-via-zeta-2}
  \mathfrak{g} = \mathfrak{g}_\tau^{\flat}
  \oplus \mathfrak{g}_\tau.
\end{equation}

We equip $\mathfrak{g}$ with the inner product dual to that on $\mathfrak{g}^\wedge$.  Then both decompositions \eqref{eqn:decompose-via-zeta} and \eqref{eqn:decompose-via-zeta-2} are orthogonal.  In particular, $\mathfrak{g}_\tau^{\flat}$ is the orthogonal complement of $\mathfrak{g}_\tau$.

For $x \in \mathfrak{g}$ and $\xi ' \in \mathfrak{g}^\wedge$, we write $x = (x',x'')$ and $\xi = (\xi ', \xi '')$ for the coordinates induced by the above decompositions, so that
\begin{align*}
  &\xi ' \in
    \mathfrak{g}_\tau^\perp,
  &&\xi ''
     \in \mathfrak{g}_{\tau} ^{\perp \flat}, \\
  &x' \in \mathfrak{g}_\tau^{\flat},
    \quad
  &&x'' \in \mathfrak{g}_\tau.
\end{align*}
\index{Lie algebra!$\tau$-coordinates} We refer to these as \emph{$\tau$-coordinates}.  We note that $\langle x, \xi \rangle = \langle x ', \xi ' \rangle + \langle x '', \xi '' \rangle$.

\setlength{\unitlength}{1.5cm}
\begin{figure}
  \begin{picture}(4,3)

    \put(-1,0){\vector(1,0){6}}
    \put(-1,0){\vector(0,1){1.5}}

    {%
      \thicklines
      \color{black}%
      \multiput(2.1,0.2)(0,0.1){10}{\line(0,1){0.05}}
      \put(2.1,1.3){$\tau + \mathfrak{g}_\tau^{\perp \flat}$}
    }

    {%
      \thicklines
      \color{black}%
      \multiput(0,0.5)(0.1,0){40}{\line(1,0){0.05}}
      \put(4.2,0.4){$\tau + \mathfrak{g}_\tau^\perp = \tau + [\mathfrak{g},\tau]$
      }

      {%
        \thicklines
        \color{black}%
        \qbezier(0,1)(2,0)(4,1)
        \put(-0.6,1.1){$G \cdot \tau$}
      }

      \color{black}
      \put(-1.4,1.4){$\xi''$}
      \put(4.9,-0.25){$\xi'$}
      \put(2,0.3){$\tau$}
      \put(1.95,0.5){\circle*{0.1}}
    }

  \end{picture}
  \caption{ The coadjoint orbit $G \cdot \tau$ near $\tau$ in $\tau$-coordinates.  }
  \label{fig:tau-coordinates}
\end{figure}
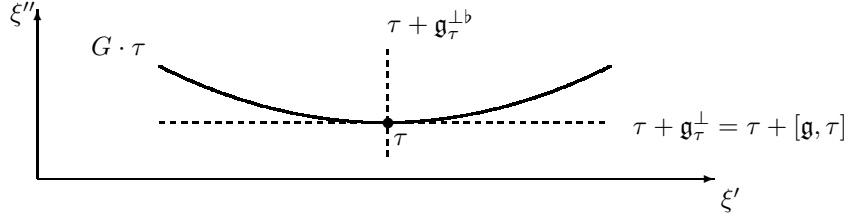

\begin{example}
  Suppose that $G = \GL_3(\mathbb{R})$.  Let us identify $\mathfrak{g}^\wedge$ with $\mathfrak{g}$ via the trace pairing divided by the imaginary unit $i$, and equip $\mathfrak{g}^\wedge$ with the inner product defined by taking the standard Euclidean inner product of the matrix entries.  We consider two representative examples:
  \begin{enumerate}[(i)]
  \item Suppose first that $\tau$ is regular semisimple, say a diagonal matrix with distinct entries.  Then
    \[
      \mathfrak{g}_\tau^\perp = \begin{pmatrix}
        & \ast  & \ast \\
        \ast &  & \ast  \\
        \ast & \ast &
      \end{pmatrix}, \quad \mathfrak{g}_\tau^{\perp \flat} = \begin{pmatrix}
        \ast &  &  \\
        & \ast &  \\
        & & \ast
      \end{pmatrix},
    \]
    \[
      \mathfrak{g}_\tau^{\flat} = \begin{pmatrix}
        & \ast & \ast \\
        \ast &  & \ast \\
        \ast & \ast &
      \end{pmatrix}, \quad \mathfrak{g}_\tau =
      \begin{pmatrix}
        \ast &  &  \\
        & \ast &  \\
        & & \ast
      \end{pmatrix}.
    \]
  \item Suppose next that $\tau$ is regular nilpotent, say
    \[
      \tau = \begin{pmatrix}
        0 & 0 & 0 \\
        1 & 0 & 0 \\
        0 & 1 & 0
      \end{pmatrix}.
    \]
    Then
    \[
      \mathfrak{g}_\tau^\perp = \left\{
        \begin{pmatrix}
          a & b & 0 \\
          c & d & -b \\
          e & f & -a-d
        \end{pmatrix}
      \right\}, \quad \mathfrak{g}_{\tau}^{\perp \flat} = \left\{ \begin{pmatrix}
          a & b & c \\
          & a & b \\
          & & a
        \end{pmatrix} \right\},
    \]
    \[
      \mathfrak{g}_\tau^{\flat} = \left\{
        \begin{pmatrix}
          a & b & c \\
          d & e & f \\
          0 & -d & -a-e
        \end{pmatrix}
      \right\}, \quad \mathfrak{g}_\tau = \left\{
        \begin{pmatrix}
          a &  &  \\
          b & a &  \\
          c & b & a
        \end{pmatrix}
      \right\}.
    \]
  \end{enumerate}
  We note that $\tau$ does not belong to a consistent summand of \eqref{eqn:decompose-via-zeta}: we have $\tau \in \mathfrak{g}_\tau^{\perp \flat}$ in case (i), $\tau \in \mathfrak{g}_\tau^{\perp}$ in case (ii).
\end{example}

\subsubsection{Euclidean coordinates}\label{sec:eucl-coord}
By choosing orthonormal bases for each of the spaces occurring in the decompositions \eqref{eqn:decompose-via-zeta} and \eqref{eqn:decompose-via-zeta-2}, we equip those spaces with Euclidean coordinates.  We assume our bases chosen so that the coordinates on the mutually dual spaces in the indicated decompositions are compatible in the sense that
\begin{equation}\label{eqn:coordinates-x-xi-primes-dual}
  \langle x', \xi '  \rangle
  = \sum_j x_j' \xi_j',\quad 
  \langle x'', \xi ''  \rangle
  = \sum_j x_j'' \xi_j''.
\end{equation}

\subsubsection{Definition and informal discussion of refined symbol classes}
\begin{definition}\label{defn:new-symbol-class}
  Assume that the wavelength parameter \index{symbols!refined class $S^\tau_{\delta ', \delta ''}$} satisfies $\h \lll 1$.  Let $\delta', \delta ''$ be fixed quantities satisfying
  \begin{equation}\label{eqn:delta-delta-prime-hypotheses}
    0 < \delta' < \delta '' < 2 \delta ' < 1.
  \end{equation}
  Let $\tau$ be an element of $\mathfrak{g}^\wedge$ that belongs to some fixed compact subset of $\mathfrak{g}^\wedge_{\reg}$.  We then define the symbol class (\S\ref{sec:classes})
  \[
    S_{\delta ', \delta ''}^\tau\] to consist of all $a \in \underline{S}^{-\infty}$ (i.e., Schwartz functions) with the following properties.
  \begin{enumerate}[(i)]
  \item For some fixed $R > 0$, we have $\supp(a) \subseteq \{\xi \in \mathfrak{g}^\wedge: |\xi - \tau| \leq R \h^{\delta'' - \delta' } \}$.
  \item For all fixed multi-indices $\alpha' \in \mathbb{Z}_{\geq 0}^{\dim(\mathfrak{g})-\rank(\mathfrak{g})}$ and $\alpha'' \in \mathbb{Z}_{\geq 0}^{\rank(\mathfrak{g})}$ and all $\xi \in \mathfrak{g}^\wedge$, we have
    \begin{equation}\label{eqn:symbol-class-defn-estimates}
      \partial^{\alpha'}_{\xi '}
      \partial^{\alpha''}_{\xi ''}
      a(\xi)
      \ll
      \h^{-\delta' |\alpha'|-\delta'' |\alpha''|},
    \end{equation}
    where $\xi = \xi ' + \xi ''$ are the $\tau$-coordinates and the differential operators $\partial^{\alpha'}_{\xi '}, \partial^{\alpha''}_{\xi ''}$ are defined using Euclidean coordinates as in \S\ref{sec:eucl-coord}.  (The definition does not depend upon the choice of orthonormal bases used to define those coordinates: any two choices are related via an orthogonal change of variables, the partial derivatives of which are $\O(1)$.)
  \end{enumerate}
  We write more verbosely $S_{\delta ', \delta ''}^{\tau}(\mathfrak{g}^\wedge)$ or $S_{\delta ', \delta ''}^{\tau}[\h]$ or $S_{\delta ', \delta ''}^{\tau}(\mathfrak{g}^\wedge)[\h]$ when we wish to indicate which group and/or wavelength parameter we are considering.
\end{definition}

Notation as at the end of \S\ref{sec:basic-symbol-classes} may be applied to the refined classes.  For instance, we may define $c S^\tau_{\delta ', \delta ''}$ for a positive real $c$.

Informally, $S_{\delta ', \delta ''}^\tau$ consists of symbols that are supported quite close to $\tau$ and that oscillate with wavelength
\begin{itemize}
\item at least $\h^{\delta'}$ in the $\xi '$-directions tangent to the coadjoint orbit $G \cdot \tau$ and
\item at least $\h^{\delta''}$ in the $\xi''$-directions transverse to that coadjoint orbit.
\end{itemize}
Since $\delta ''$ is larger than $\delta'$, less regularity is required in the transverse directions than in the tangent directions.  In applications, we will take $\delta'$ close to $1/2$ and $\delta ''$ close to $1$.

For the sake of visualization, a good example to keep in mind is when $G = \SO(3)$, so that $\mathfrak{g}^\wedge$ identifies with $\mathbb{R}^3$ in such a way that the coadjoint orbits are the Euclidean spheres, with $G$ acting via rotation.  If $\tau = (0,0,1)$ is the ``north pole'' of the unit sphere and $\langle , \rangle$ is the standard Euclidean inner product, then the $\tau$-coordinates of $\xi = (\xi_1,\xi_2,\xi_3)$ are $\xi ' = (\xi_1, \xi_2)$ and $\xi '' = x_3$.  For $\delta ' \approx 1/2$ and $\delta '' \approx 1$, a typical element of $S^{\tau}_{ \delta ', \delta ''}$ is roughly a smooth bump on the coin-shaped domain
\[
  \{\xi \in \mathbb{R}^3 : 1 - \h < |\xi| < 1 + \h, |\xi - \tau | < \h^{1/2}\}.
\]

We note that $S^\tau_{\delta ', \delta ''} \subseteq S^{-\infty}_{\delta ''} \cap \underline{S}^{-\infty}$.

\subsubsection{Insensitivity to basepoint}
The significance of the exponent $\delta '' - \delta '$ in the support condition (i) is to ensure that $S_{\delta ', \delta ''}^\tau$ is insensitive to the precise choice of ``basepoint'' $\tau$, in the following sense:
\begin{lemma}\label{lem:symbol-class-change-basepoint}
  Let $\h, \delta', \delta ''$ be as in Definition \ref{defn:new-symbol-class}.  Let $\tau_1, \tau_2$ be elements of $\mathfrak{g}^\wedge$ that belong to some fixed compact subset of $\mathfrak{g}^\wedge_{\reg}$.  If
  \begin{equation}\label{eqn:tau1-tau2-close}
    \tau_1 = \tau_2 + \O(\h^{\delta'' - \delta '}),
  \end{equation}
  then the symbol classes coincide:
  \[S^{\tau_1}_{\delta', \delta''} = S^{\tau_2}_{\delta', \delta ''};\] otherwise, their intersection is trivial:
  \[\cap_{j=1,2} S^{\tau_j}_{ \delta', \delta ''} = \{0\}.\]

  More precisely, if \eqref{eqn:tau1-tau2-close} holds, then for each smooth function $a : \mathfrak{g}^\wedge \rightarrow \mathbb{C}$ and each $\xi = \tau_j + \O(\h^{\delta '' - \delta'})$, the estimate \eqref{eqn:symbol-class-defn-estimates} holds with respect to $\tau_1$-coordinates if and only if it holds with respect to $\tau_2$-coordinates.
\end{lemma}
A useful consequence of this result is that to check whether a smooth function $a$ supported on $\tau_1 + \O(\h^{\delta '' - \delta '})$ belongs to the class $S^{\tau_1}_{\delta', \delta ''}$, it suffices to verify for each $\tau_2 = \tau_1 + \O(\h^{ \delta '' - \delta '})$ that the derivative bounds \eqref{eqn:symbol-class-defn-estimates} hold at $\xi = \tau_2$ with respect to $\tau_2$-coordinates.

\begin{proof}
  The support condition (i) and our assumption $\h \lll 1$ imply that the intersection is trivial if \eqref{eqn:tau1-tau2-close} is not satisfied, so suppose \eqref{eqn:tau1-tau2-close} holds.  The support condition (i) is then the same for both symbol classes, so it is enough to verify that if a smooth function $a : \mathfrak{g}^\wedge \rightarrow \mathbb{C}$ satisfies \eqref{eqn:symbol-class-defn-estimates} with respect to the coordinates defined by $\tau_1$, then it satisfies the same with respect to the coordinates defined by $\tau_2$.  Set
  \begin{equation*}
    m := \dim(\mathfrak{g}) - \rank(\mathfrak{g}),
    \quad 
    n := \rank(\mathfrak{g}).
  \end{equation*}
  Let
  \begin{equation}\label{eqn:euclidean-isomorphism-insens}
    \mathfrak{g}_{\tau_j}^{\perp} \cong \mathbb{R}^m,
    \quad 
    \mathfrak{g}_{\tau_j}^{\perp \flat} \cong \mathbb{R}^n
  \end{equation}
  denote the isomorphisms arising from the Euclidean coordinates defined in \S\ref{sec:eucl-coord}.  Write $\xi = \xi_1' + \xi_1'' = \xi_2' + \xi_2''$ for the respective $\tau_1$- and $\tau_2$-coordinates.  With respect to \eqref{eqn:euclidean-isomorphism-insens}, these coordinates are related via an invertible linear change of variables
  \begin{equation*}
    \begin{pmatrix}
      \xi_1'  \\
      \xi_1''
    \end{pmatrix}
    =
    \begin{pmatrix}
      A & B \\
      C & D
    \end{pmatrix}
    \begin{pmatrix}
      \xi_2'  \\
      \xi_2''
    \end{pmatrix}
  \end{equation*}
  with $A,B,C,D$ matrices of respective dimensions $m \times m, m \times n, n \times m, n \times n$.  Since the assignment $\tau \mapsto \mathfrak{g}_{\tau}^{\perp}$ is smooth and in particular Lipschitz, we deduce from \eqref{eqn:tau1-tau2-close} the operator norm estimates
  \begin{equation}\label{eqn:A-B-C-D-bounds}
    A,D =
    1 + \O(\h^{\delta '' - \delta' })
    =
    \O(1),
    \quad B,C = \O(\h^{\delta '' - \delta' }).
  \end{equation}
  Let $f_1 : \mathbb{R}^m \oplus \mathbb{R}^n \rightarrow \mathbb{C}$ be the function describing $a$ with respect to the coordinates $\xi \mapsto (\xi_1 ', \xi_1 '')$.  Then $f_2(x,y) := f_1(A x + B y, C x + D y)$ describes $a$ with respect to the coordinates $\xi \mapsto (\xi_2 ', \xi_2 '')$, and our task is to show that if the estimates
  \begin{equation*}
    \partial_x^{\alpha} \partial_y^ \beta f_j(x,y) \ll
    \h^{-\delta' |\alpha| - \delta '' |\beta|}
    \quad
    \text{($\alpha \in \mathbb{Z}_{\geq 0}^m,
      \beta \in \mathbb{Z}_{\geq 0}^n$ fixed)}
  \end{equation*}
  hold for $j=1$, then they hold also for $j=2$.

  We explain the required implication in the notationally simplest case $m=n=1$, in which $A,B,C,D$ are scalars and $\alpha, \beta$ are integers.  (This case does not actually arise for any $\mathfrak{g}$, but the calculus exercise we are solving makes sense for general $m$ and $n$.)  In that case, $\partial_x^\alpha \partial_y^\beta f_2(x,y)$ evaluates via the chain rule to
  \begin{equation*}
    \sum _{k=0}^{\alpha}
    \sum _{l=0}^{\beta}
    \binom{\alpha}{k}
    \binom{\beta}{l}
    A^{\alpha-k}
    C^{k}
    B^l
    D^{\beta - l}
    \partial_1^{\alpha-k+ l}
    \partial_2^{\beta-l+k}
    f_1(A x + B y, C x + D y),
  \end{equation*}
  where for clarity we write $\partial_1, \partial_2$ for derivatives in the first and second variables.  Invoking our hypotheses concerning $f_1$, we reduce the required bound for $f_2$ to the estimate
  \begin{equation*}
    A^{\alpha-k}
    C^{k}
    B^l
    D^{\beta - l}
    \h^{- \delta' (-k+l) - \delta '' (-l+k)} \ll 1.
  \end{equation*}
  To verify this, 
  we invoke \eqref{eqn:A-B-C-D-bounds} in the weaker form
  \begin{equation*}
    A,B,D = \O(1),
    \quad C = \O(\h^{\delta '' - \delta'}).
  \end{equation*}
  and recall that $\delta ''\geq \delta'$.  The case of general $(m,n)$ is treated similarly by inserting indices throughout the above argument.
\end{proof}

\subsubsection{Basic properties}\label{sec:symbol-classes-basic-properties}
We record some basic mapping properties concerning the symbol classes just introduced, to be applied repeatedly in what follows.  Each property follows readily from the definition.
\begin{lemma}
  Let $\h, \delta ', \delta '', \tau$ be as in Definition \ref{defn:new-symbol-class}.  Abbreviate $S := S^{\tau}_{ \delta ', \delta ''}$.
  \begin{enumerate}[(i)]
  \item Let $P : \mathfrak{g}^\wedge \rightarrow \mathbb{C}$ be a fixed polynomial.  Then multiplication by $P$ defines a class map
    \begin{equation*}
      P : S \rightarrow S.
    \end{equation*}
  \item Set $m := \dim(\mathfrak{g}) - \rank(\mathfrak{g})$, $n := \rank(\mathfrak{g})$.  Fix $\alpha' \in \mathbb{Z}_{\geq 0}^{m}$, $\alpha'' \in \mathbb{Z}_{\geq 0}^{n}$.  Define the partial differential operators $\partial_{\xi '}^{\alpha'}, \partial_{\xi ''}^{\alpha''}$ as in \eqref{eqn:symbol-class-defn-estimates}, using the coordinates defined by $\tau$.  Then
    \begin{equation*}
      \partial_{\xi '}^{\alpha'} : S \rightarrow \h^{- \delta ' |\alpha'|} S,
    \end{equation*}
    \begin{equation*}
      \partial_{\xi '}^{\alpha''} : S \rightarrow \h^{- \delta '' |\alpha''|} S.
    \end{equation*}
  \end{enumerate}
\end{lemma}

\subsubsection{``Completing'' the refined symbol classes}

By
\begin{equation*}
  S^{\tau}_{\delta', \delta''}
  + \h^\infty S^{-\infty}
\end{equation*}
we mean the class of symbols of the form $a + b$ with $a \in S^{\tau}_{ \delta ', \delta ''}$ and $b \in \h^\infty S^{-\infty}$.  It may be regarded informally as a ``completion'' of $S^{\tau}_{ \delta ', \delta ''}$ obtained by adjoining the ``negligible'' symbols.  This ``completed'' class admits a more direct characterization:
\begin{lemma}\label{lem:characterize-completed-symbols}
  Let $a \in \underline{S}^{\infty}$.  Then $a$ belongs to $S^{\tau}_{ \delta ', \delta ''} + \h^\infty S^{-\infty}$ if and only if there is a fixed $R > 0$ so that the following conditions hold.
  \begin{itemize}
  \item For all fixed $\alpha, N$ and all $\xi$ with $|\xi - \tau | > R \h^{\delta '' - \delta '}$, we have
    \[\partial^\alpha a(\xi) \ll \h^N \langle \xi
      \rangle^{-N},\] where as before $\langle \xi \rangle = ( 1 + |\xi|^2)^{1/2}$.
  \item For all fixed $\alpha, \beta$ and all $\xi$ with $|\xi - \tau | < 2 R \h^{\delta '' - \delta '}$, we have in $\tau$-coordinates
    \[\partial_{\xi'}^{\alpha} \partial_{\xi ''}^{\beta}
      a(\xi) \ll \h^{-\delta ' |\alpha| - \delta '' |\beta|}.\]
  \end{itemize}
\end{lemma}
\begin{proof}
  The necessity of these conditions is clear from the definitions.  Conversely, suppose these conditions hold.  We may construct an element $p$ of the basic symbol class $S^{-\infty}_{\delta '' - \delta '}$ satisfying
  \begin{itemize}
  \item $p(\xi) = 1$ for $|\xi - \tau| < 3 R \h^{\delta '' - \delta '}$,
  \item $p(\xi) = 0$ for $|\xi - \tau| > 4 R \h^{\delta '' - \delta '}$.
  \end{itemize}
  We then verify readily that $p a$ belongs to $S^{\tau}_{ \delta ', \delta ''}$ and $(1-p) a$ belongs to $\h^\infty S^{-\infty}$.  We conclude by noting that $a = p a + (1 - p) a$.
\end{proof}

\subsubsection{Star product asymptotics}
We will establish analogues of \eqref{eqn:star-j-old-mapping-property} and \eqref{eqn:star-old-asymp-expn} for the refined symbol classes.  The possibility of doing so boils down to the fact that the BCHD formula (Theorem \ref{thm:BCHD}, below) consists of Lie polynomials; by contrast, the proof of Theorem \ref{thm:nv-star-prod-asymp-h-dependent} made direct use only of the analyticity of that formula and the fact that $x \ast y = x + y + \O(|x| \cdot |y|)$.

\begin{theorem}\label{thm:refined-star-prod}
  Fix a nice cutoff $\chi$.  Assume that $\h \lll 1$.  Fix $\delta ', \delta ''$ satisfying \eqref{eqn:delta-delta-prime-hypotheses}.  Write $j, J$ for fixed elements of $\mathbb{Z}_{\geq 0}$ and $\tau, \tau_1, \tau_2$ for elements of some fixed compact subset of $\mathfrak{g}^\wedge_{\reg}$.
  \begin{enumerate}[(i)]
  \item \label{item:star-prod-1} $\star_{\h}$ enjoys the mapping property
    \begin{equation}\label{eqn:star-h-new-mapping-property-super-localized}
      \star_{\h} :
      S^{\tau_1}_{\delta', \delta''} 
      \times 
      S^{\tau_2}_{\delta', \delta''} 
      \rightarrow
      S^{\tau_1}_{\delta', \delta ''}
      \cap S^{\tau_2}_{\delta', \delta ''}
      + \h^{\infty} S^{-\infty}.
    \end{equation}
  \item \label{item:star-prod-2} If $\tau_1 - \tau_2 \ggg \h^{\delta '' - \delta '}$, then $a \star^j b = 0$ for all $a \in S^{\tau_1}_{\delta',\delta''}$, $b \in S^{\tau_2}_{\delta ', \delta ''}$.
  \item \label{item:star-prod-3} $\star^j$ enjoys the mapping property
    \begin{equation}\label{eqn:star-j-new-mapping-property}
      \star^j :
      S_{\delta ', \delta ''}^{\tau }
      \times 
      S_{\delta ', \delta ''}^{\tau }
      \rightarrow \h^{- 2 \delta' j} 
      S^{\tau}_{\delta', \delta ''}.
    \end{equation}
    For any $a,b \in S_{\delta ', \delta ''}^\tau$, we have the asymptotic expansion
    \begin{equation}\label{eqn:star-j-asymp-expn-new}
      a \star_{\h} b
      \equiv  \sum _{0 \leq j < J}
      \h^j a \star^j b
      \mod{
        \h^{(1- 2 \delta') J} S^{\tau}_{\delta', \delta''}
        +
        \h^{\infty}
        S^{-\infty}}.
    \end{equation}

  \item \label{item:star-prod-4} Fix $\delta \in [0,\delta']$.  We have the mapping properties
    \begin{equation*}
      \star_{\h} : S^{\infty}_{\delta} \times S^{\tau}_{ \delta
        ', \delta ''}
      \rightarrow S^{\tau}_{ \delta ', \delta ''}
      + \h^\infty S^{-\infty},
    \end{equation*}
    \begin{equation*}
      \star^j : S^{\infty}_{\delta} \times S^{\tau}_{ \delta
        ', \delta ''}      \rightarrow \h^{-(\delta +  \delta ')
        j}
      S^{\tau}_{ \delta ', \delta ''}
    \end{equation*}
    and, for $(a,b) \in S^{\infty}_{\delta} \times S^{\tau}_{ \delta ', \delta ''}$, the asymptotic expansion
    \begin{equation*}
      a \star_{\h} b
      \equiv  \sum _{0 \leq j < J}
      \h^j a \star^j b
      \mod{
        \h^{(1- \delta - \delta') J}  S^{\tau}_{ \delta ', \delta ''} 
        +
        \h^{\infty} S^{-\infty}}.
    \end{equation*}
    Analogous results hold with $S^{\infty}_{\delta} \times S^{\tau}_{ \delta ', \delta ''}$ replaced by $S^{\tau}_{ \delta ', \delta ''} \times S^{\infty}_{\delta}$.
  \end{enumerate}
\end{theorem}
\begin{proof}
  We indicate here the overall structure of the proof, the details of which are given below.

  Assertion \eqref{item:star-prod-2} follows immediately from the support condition \eqref{item:star-prod-1} in Definition \ref{defn:new-symbol-class}.

  As for \eqref{item:star-prod-3}, we establish in \S\ref{sec:mapp-prop-homog} the mapping property \eqref{eqn:star-j-new-mapping-property} for $\star^j$, then in \S\ref{sec:estimates-remainder} the asymptotic expansion \eqref{eqn:star-j-asymp-expn-new}.  The mapping property \eqref{eqn:star-h-new-mapping-property-super-localized} for $\star_{\h}$ follows in the special case $\tau_1 = \tau_2$ by specializing the asymptotic expansion \eqref{eqn:star-j-asymp-expn-new} to $J=0$.

  By the first part of Lemma \ref{lem:symbol-class-change-basepoint}, the mapping property \eqref{eqn:star-h-new-mapping-property-super-localized} in the case that \eqref{eqn:tau1-tau2-close} holds reduces to the case $\tau_1 = \tau_2$ already considered.  To complete the proof of the general case of \eqref{eqn:star-h-new-mapping-property-super-localized}, it remains to show that for
  \begin{equation*}
    \tau_1 - \tau_2 \ggg \h^{\delta ' - \delta ''},
  \end{equation*}
  one has $a \star_{\h} b \in \h^\infty S^{-\infty}$ for all $(a,b) \in S^{\tau_1}_{\delta ', \delta ''} \times S^{\tau_2}_{\delta ', \delta ''}$.  The proof in that case is implicit in the proof of \eqref{eqn:star-j-asymp-expn-new}, but for expository purposes we have treated it separately in \S\ref{sec:case-disj-supp}.
  
  We discuss part \eqref{item:star-prod-4} in \S\ref{sec:mixed-case}.  The proof amounts to an ``interpolation'' between the proofs given here and the proof of Theorem \ref{thm:nv-star-prod-asymp-h-dependent} given in \cite[\S7.3]{nelson-venkatesh-1}.
\end{proof}

\subsection{Proofs of refined star product asymptotics}\label{sec:proofs-star-product}
Here we complete the proof of Theorem \ref{thm:refined-star-prod} following the outline indicated above.  On a first reading, we recommend skipping this section and proceeding to \S\ref{sec:operator-classes}.

\subsubsection{Taylor coefficient support conditions}
We retain the notation and setting of \S\ref{sec:star-prod-extens}.  The support condition \eqref{eqn:gamma-less-alpha-beta} was sufficient for the purposes of \cite{nelson-venkatesh-1}.  We require here a stronger support condition.  Let $\tau \in \mathfrak{g}^\wedge_{\reg}$.  Recall from \S\ref{sec:star-prod-extens} the definition of $\{x,y\}$.  Let
\begin{equation}\label{eqn:taylor-expand-finer}
  e^{\{x,y\} \tau}
  = \sum_{\alpha', \alpha '', \beta', \beta '',\gamma} c_{\alpha' \alpha '' \beta' \beta '' \gamma}
  (x')^{\alpha '}
  (x'')^{\alpha ''}
  (y')^{\beta'}
  (y'')^{\beta''}
  \tau^\gamma
\end{equation}
denote the Taylor expansion in $\tau$-coordinates, where monomials such as $(x')^{\alpha '} = \prod _j (x'_j)^{\alpha _j'}$ are defined with respect to the Euclidean coordinates of \S\ref{sec:eucl-coord}.
\begin{lemma}\label{lem:constraint-coefficients-via-BCDH}
  Suppose that $c_{\alpha ' \alpha '' \beta ' \beta '' \gamma } \neq 0$.  Then
  \[
    |\alpha '| + |\beta '| \geq 2 |\gamma|,
  \]
  hence, with
  \begin{equation}\label{eq:j-:=-alpha}
    j := |\alpha '| + |\beta '| + |\alpha ''| + |\beta '' | - |\gamma|,
  \end{equation}
  we have
  \begin{equation}\label{eqn:alpha-beta-prime-2-gamma}
    |\alpha '| + |\beta '|
    +
    2 |\alpha ''| + 2 |\beta ''|
    \leq 2 j
  \end{equation}
\end{lemma}
\begin{proof}
  The two required inequalities are restatements of one another, so we focus on the first.  We use the fact (Theorem \ref{thm:BCHD}, below) that each homogeneous component of the Taylor expansion of $\{x,y\}$ is a finite iterated Lie polynomial in $x$ and $y$, i.e., a linear combination of expressions $L(x,y) = \ad_{z_1} \dotsb \ad_{z_{n-1}} z_n$ with each $z_j$ equal to either $x$ or to $y$.  By expanding the exponential series, we reduce to verifying for each such expression that if we write
  \[
    L(x,y) \tau = \sum_{\alpha', \alpha '', \beta', \beta '',\gamma} C_{\alpha' \alpha '' \beta' \beta '' \gamma} (x')^{\alpha '} (x'')^{\alpha ''} (y')^{\beta'} (y'')^{\beta''} \tau^\gamma,
  \]
  then
  \[
    C_{\alpha' \alpha '' \beta' \beta '' \gamma} \neq 0 \implies |\alpha '| + |\beta '| \geq 2 |\gamma|.
  \]
  Suppose $C_{\alpha' \alpha '' \beta' \beta '' \gamma} \neq 0$.  Set $n := |\alpha '| + |\alpha ''| + |\beta '| + |\beta ''|$.  Then $n$ is the degree of the Lie monomial $L$ (e.g., $n=3$ if $L(x,y) = [x,[x,y]]$), and we may find nonzero elements $x_1,\dotsc,x_n$ of $\mathfrak{g}_\tau^\flat$ or of $\mathfrak{g}_\tau$ so that
  \begin{itemize}
  \item the number of $x_j$ lying in $\mathfrak{g}_\tau^{\flat}$ is $|\alpha '| + |\beta '|$,
  \item the number of $x_j$ lying in $\mathfrak{g}_\tau$ is $|\alpha ''| + |\beta ''|$, and
  \item $\langle [x_1,[x_2,\dotsc,[x_{n-1},x_n] ] ], \tau \rangle \neq 0$.
  \end{itemize}
  We have $|\gamma| = 1$, so our task is to verify that $|\alpha '| + |\beta '| \geq 2$.  Suppose otherwise that $|\alpha'| + |\beta '| \leq 1$.  Then $|\alpha '' | + | \beta ''| \geq n-1$, i.e., at least $n-1$ of the $x_j$ centralize $\tau$.  Since $\mathfrak{g}_\tau$ is a subalgebra, it follows that at least one of the elements $x := x_1$ or $y := [x_2,\dotsc,[x_{n-1},x_n]]$ centralizes $\tau$, and yet $\langle [x,y], \tau \rangle = 0$.  The required contradiction then follows from \eqref{eqn:adjoint-map-adjointness-zeta-centralizer}.
\end{proof}

We recall a form of the BCHD formula, used in the above proof (see, e.g., \cite[\S I.IV.7 and II.V.4]{MR2179691}):
\begin{theorem}\label{thm:BCHD}
  There is a neighborhood $U$ of the origin in $\mathfrak{g}$ so that for all $x,y \in U$, the group law $x \ast y$ (see \eqref{eq:x-ast-y-:=-logexpx-expy-}) is defined and admits the absolutely convergent expansion
  \begin{equation}\label{eq:x-ast-y-=-sum_n-geq-1-z_n-}
    x \ast y = \sum_{n \geq 1} z_n,
  \end{equation}
  where $z_n$ is a linear combination of $n$-fold iterated commutators of $x$ and $y$.  For example,
  \begin{equation*}
    z_1 = x + y,
    \quad
    z_2 = \frac{[x,y]}{2},
    \quad
    z_3 =
    \frac{[x, [x,y]] + [y, [y x]]}{12},
    \quad
    z_4 =
    \frac{[y, [x, [y,x]]]}{24}.
  \end{equation*}
\end{theorem}

\subsubsection{ Mapping properties of homogeneous components }\label{sec:mapp-prop-homog}
\begin{proof}[Proof of \eqref{eqn:star-j-new-mapping-property}.]
  Let $a,b \in S^{\tau}_{\delta ', \delta ''}$.  We must verify that $a \star^j b$ belongs to the symbol class $\h^{-2 \delta'j } S^{\tau}_{ \delta ', \delta ''}$.  Since $\star^j$ is a finite order differential operator, the support condition in the definition of that symbol class clearly holds.  It is also clear that $a \star^j b \in \underline{S}^{-\infty}$.  Setting $m := \dim(\mathfrak{g}) - \rank(\mathfrak{g})$ and $n := \rank(\mathfrak{g})$, we must thus check that for all fixed multi-indices $\eta' \in \mathbb{Z}_{\geq 0}^{m}, \eta '' \in \mathbb{Z}_{\geq 0}^{n}$ and all $\xi = \tau + \O(\h^{\delta'' - \delta '})$, we have
  \begin{equation}\label{eqn:symbol-mapping-star-j-goal}
    \partial_{\xi '}^{\eta'} \partial_{\xi ''}^{\eta ''} (a \star^j b)(\xi)
    \ll
    \h^{
      - \delta ' |\eta '|
      - \delta '' | \eta ''|
      -2 \delta' j
    }.
  \end{equation}
  By Lemma \ref{lem:symbol-class-change-basepoint}, we may and shall assume that $\xi = \tau$.

  We begin by verifying \eqref{eqn:symbol-mapping-star-j-goal} in the special case $\eta ' = 0$, $\eta '' = 0$.  By expanding the definition of $\star^j$, we reduce to checking that for all fixed $\alpha ', \alpha '', \beta ', \beta '', \gamma$ satisfying $c_{\alpha ' \alpha '' \beta ' \beta '' \gamma} \neq 0$ and \eqref{eq:j-:=-alpha}, we have the following estimate, stated in $\tau$-coordinates:
  \begin{equation}\label{eqn:key-upper-bound-for-star-j-mapping}
    \left(
      \partial_{\xi '}^{\alpha '}
      \partial_{\xi ''}^{\alpha ''}
      a(\tau)
    \right)
    \left(
      \partial_{\xi '}^{\beta '}
      \partial_{\xi ''}^{\beta ''}
      b(\tau)
    \right)
    \tau^{\gamma}
    \ll
    \h^{- 2 \delta' j}.
  \end{equation}
  By Lemma \ref{lem:constraint-coefficients-via-BCDH}, the inequality \eqref{eqn:alpha-beta-prime-2-gamma} holds; this is all that we will use concerning $\alpha ', \alpha '', \beta ', \beta '' ,\gamma $.  By our hypothesis that $\tau = \O(1)$, we have in particular $\tau^\gamma \ll 1$.  By the definition of the symbol class $S^{\tau}_{ \delta ', \delta ''}$, we see that
  \begin{equation*}
    \partial_{\xi '}^{\alpha '}
    \partial_{\xi ''}^{\alpha ''}
    a(\tau)
    \ll \h^{- \delta ' |\alpha '| - \delta '' | \alpha ''|},
    \quad 
    \partial_{\xi '}^{\beta '}
    \partial_{\xi ''}^{\beta ''}
    b(\tau)
    \ll \h^{- \delta ' |\beta '| - \delta '' | \beta ''|}.
  \end{equation*}
  It will thus suffice to verify under the stated conditions that
  \begin{equation}\label{eqn:tons-of-primes-alpha-beta}
    \delta' |\alpha '| +
    \delta '' |\alpha ''|
    +
    \delta' |\beta '| +
    \delta '' |\beta ''|
    \leq 2 \delta' j.
  \end{equation}
  To see this, we use that $\delta '' \leq 2 \delta '$ and apply \eqref{eqn:alpha-beta-prime-2-gamma}.

  The proof is complete in the special case $\eta ' = 0, \eta '' = 0$.  The general case is deduced similarly, using the product rule to evaluate the application of the differential operator $\partial_{\xi '}^{\eta'} \partial_{\xi ''}^{\eta ''}$ to the LHS of \eqref{eqn:key-upper-bound-for-star-j-mapping} and recalling that $\tau = \O(1)$.
\end{proof}

\subsubsection{Estimates for the
  remainder}\label{sec:estimates-remainder}
We now prove \eqref{eqn:star-j-asymp-expn-new}.  The proof is similar in overall structure to \cite[Thm 6]{nelson-venkatesh-1}.  The present proof is simpler because our symbols have uniformly bounded support, but more complicated because of the more complicated nature of the symbol class.

\emph{Reduction to tail estimates.}  Set
\begin{equation*}
  r := a \star_{\h} b - \sum_{0 \leq j < J} \h^j a \star^j b.
\end{equation*}
By Theorem \ref{thm:basic-star-product-extension-properties}, we have $r \in \underline{S}^{-\infty}$.  By the criterion of Lemma \ref{lem:characterize-completed-symbols}, we must verify that for some fixed (large enough) $R > 0$,
\begin{itemize}
\item for all $\zeta \in \mathfrak{g}^\wedge$ with $|\zeta - \tau| > R \h^{\delta '' - \delta '}$, all fixed multi-indices $\gamma$ and each fixed $N$, we have (with $\langle \zeta \rangle = (1 = |\zeta|^2)^{1/2}$, as before)
  \begin{equation}\label{eqn:partial-zeta-gamma-r-zeta-trivz}
    \partial_\zeta^\gamma r(\zeta) \ll \h^N \langle \zeta  \rangle^{-N},
  \end{equation}
  and
\item for all $\zeta \in \mathfrak{g}^\wedge$ with $|\zeta - \tau| < 2 R \h^{\delta '' - \delta '}$ and all fixed multi-indices $\gamma', \gamma ''$, we have
  \begin{equation}\label{eqn:required-estimate-for-r}
    \partial_{\xi '}^{\gamma'}
    \partial_{\xi ''}^{\gamma ''} r(\zeta)
    \ll
    \h^{(1 - 2 \delta') J - \delta ' |\gamma '|
      - \delta '' | \gamma '' | }.
  \end{equation}
\end{itemize}
In fact, since $1 - 2 \delta' > 0$, the factors $\h^{(1- 2 \delta') J}$ decrease as $J$ increases.  Thus, by the mapping properties \eqref{eqn:star-j-new-mapping-property} established above, we see that the terms $\h^j a \star^j b$ ($j \leq J$) satisfy the analogue of the estimate \eqref{eqn:required-estimate-for-r} required by $r$.  These terms are supported on the intersection of the supports of $a$ and $b$.  Choosing $R$ large enough that $\{\xi : |\xi - \tau | < R \}$ contains those supports, we reduce to showing that for each fixed $\gamma$ and $N$, the estimate \eqref{eqn:partial-zeta-gamma-r-zeta-trivz} holds provided that $J$ is large enough but fixed.  We record the details below in the notationally simplest case $\gamma = 0$.  The general case may be verified by differentiating each step of the proof, noting that our inputs from \cite{nelson-venkatesh-1} apply to derivatives (see \cite[\S7.7, Proof of Prop]{nelson-venkatesh-1} for further details).

\emph{Reduction to the case of localized symbols.}  We may assume -- after smoothly decomposing $a$ and $b$ into $\h^{-\O(1)}$ pieces -- that for some $\omega_1, \omega_2 \in \mathfrak{g}^\wedge$ satisfying the estimates
\begin{equation*}
  \omega_1, \omega_2 =
  \tau + \O(\h^{\delta '' - \delta '}),
\end{equation*}
the symbols $a$ and $b$ satisfy the support conditions
\begin{equation}\label{eqn:support-a-vs-omega-1}
  \begin{split}
    \supp(a) &\subseteq \{ \xi \in \mathfrak{g}^\wedge : |\xi' - \omega_1'| \leq \h^{\delta '}, |\xi'' - \omega_1''| \leq \h^{\delta ''} \},
    \\
    \supp(b) &\subseteq \{ \xi \in \mathfrak{g}^\wedge : |\xi' - \omega_2'| \leq \h^{\delta '}, |\xi'' - \omega_2''| \leq \h^{\delta ''} \},
  \end{split}
\end{equation}
with $\tau$-coordinates $\xi = (\xi ', \xi '')$ as usual.  In view of our assumptions that $\h \lll 1$ and $\delta ', \delta '' > 0$, these support conditions constrain the symbols $a$ and $b$ to small neighborhoods of their respective ``basepoints'' $\omega_1$ and $\omega_2$.  It may help to recall here that
\begin{equation*}
  \delta '' - \delta ' < \delta ' < \delta ''.
\end{equation*}

We pause to describe the Fourier transforms $a_{\h}^\vee, b_{\h}^\vee$.  Set
\begin{equation}\label{eqn:notation-for-describing-fourier-transforms-of-superlocalized-symbols}
  A' := \h^{-1 + \delta '},
  \quad 
  A'' := \h^{-1 + \delta ''},
  \quad
  m := \dim(\mathfrak{g}) - \rank(\mathfrak{g}),
  \quad 
  n := \rank(\mathfrak{g}).
\end{equation}
It follows then readily from the definition of $S^{\tau}_{ \delta ', \delta ''}$, the support conditions \eqref{eqn:support-a-vs-omega-1} of $a$ and $b$ and elementary Fourier analysis that with respect to $\tau$-coordinates $x = (x',x'')$ on $\mathfrak{g}$, we may write
\begin{equation}\label{eqn:a-h-vee-fourier-transform}
  a_{\h}^\vee(x)
  = e^{- \langle x, \omega_1/\h \rangle}
  (A')^m
  (A'')^n
  \phi_1(A' x', A'' x''),
\end{equation}
\begin{equation*}
  b_{\h}^\vee(x)
  = e^{- \langle x, \omega_2/\h \rangle}
  (A')^m
  (A'')^n
  \phi_2(A' x', A'' x''),
\end{equation*}
where $\phi_1, \phi_2$ are Schwartz functions on $\mathfrak{g}$ with each fixed Schwartz seminorm of size $\O(1)$.  From this description we see that for any fixed nonnegative integers $k', k'', l', l''$,
\begin{equation}\label{eqn:moment-estimates-lotsa-primes}
  \int_{x,y \in \mathfrak{g}}
  \left\lvert
    a_{\h}^\vee(x)
    b_{\h}^\vee(y)
  \right\rvert
  |x'|^{k'}
  |x''|^{k''}
  |y'|^{l'}
  |y''|^{l''}
  \, d x \, d y
  \ll
  (A')^{-k'-l'}
  (A'')^{-k''-l''}.
\end{equation}

Set \[Q := \h^{-1} \langle \zeta \rangle \in \mathbb{R}_{\geq 1},\] so that our goal bound (specialized as indicated to $\gamma ' =0 , \gamma'' = 0$) may be written
\begin{equation}\label{eqn:required-estimate-for-r-2}
  r(\zeta)
  \ll
  Q^{-N}.
\end{equation}
Since $\h \lll 1$, we have $Q \ggg 1$.

\emph{Reduction to the case of controlled arguments.}  We fix $\eps > 0$ sufficiently small in terms of $\delta '$ and $\delta ''$.  We consider first the case that $|\zeta| \geq Q^{\eps}$.  In that case -- in view of the estimate $Q \ggg 1$ and the assumed support conditions \eqref{eqn:support-a-vs-omega-1} on $a$ and $b$ -- we have $a \star ^j b (\zeta) = 0$ for all $j$ (in particular, for all $j \leq J$).  Indeed, $a \star^j b$ is a linear combination of products of derivatives of $a$ and $b$; those derivatives have disjoint support, hence their products vanish.

We thereby reduce to showing that $a \star_{\h} b(\zeta) \ll Q^{-N}$.  This estimate may be verified by a crude application of partial integration to the integral representation \eqref{eqn:rescaled-star-integral-rep}, repeating the argument around \cite[(7.15), (7.16)]{nelson-venkatesh-1} verbatim.  We thereby reduce to the case $|\zeta| < Q^{\eps}$.

\emph{Reduction to the case of nearby arguments.}  We next to reduce to the critical case
\begin{equation}\label{eqn:critical-case-zeta-near-tau}
  \zeta = \tau + O(\h^{\delta '' - \delta'}).
\end{equation}
Suppose otherwise that $\zeta - \tau \ggg \h^{\delta '' - \delta'}$.  Since $a$ and $b$ are both supported on $\tau + \O(\h^{\delta '' - \delta '})$, we then have $a \star^j b (\zeta) = 0$ for all $j$, so our task is to check that
\[a \star_{\h} b(\zeta) \ll \h^N.\] We start with the integral representation \eqref{eqn:rescaled-star-integral-rep}.  The estimate \eqref{eqn:moment-estimates-lotsa-primes} says informally that $a_{\h}^\vee(x) b_{\h}^\vee(y)$ is concentrated on $x', y' \ll \h^{1 - \delta '}$ and $x'', y'' \ll \h^{1 - \delta ''}$.  Quantitatively, let us define smooth cutoffs $\chi '$ (resp. $\chi ''$) on $\mathfrak{g}_\tau^{\flat}$ (resp. $\mathfrak{g}_\tau$) by using the coordinates defined in \S\ref{sec:eucl-coord} to transport some fixed smooth cutoffs on the corresponding Euclidean spaces.  Setting \[B' := \h^{-\eps} (A')^{-1}, \quad B'' := \h^{-\eps} (A'')^{-1}.
\]
we deduce from \eqref{eqn:moment-estimates-lotsa-primes} that the error incurred in the integral representation \eqref{eqn:rescaled-star-integral-rep} by replacing our original cutoff $\chi(x) \chi(y)$ with the reduced cutoff
\begin{equation*}
  f(x,y) := \chi'(x'/B') \chi''(x''/B'') \chi'(y'/B')
  \chi''(y''/B'')
\end{equation*}
is $\O(\h^\infty)$.  Having shrunk the cutoff in this way, we open the Fourier integrals defining $a_{\h}^\vee(x)$ and $b_{\h}^\vee(y)$ to express $a \star_{\h} b(\zeta)$ up to negligible error as
\begin{equation*}
  \h^{-2 \dim(\mathfrak{g})}
  \int_{\xi,\eta}
  a(\xi) b(\eta)
  \left(
    \int_{x,y}
    f(x,y)
    e^{(x (\zeta - \xi) + y (\zeta - \eta) + \{ x, y \} \zeta )/\h}
    \, d x \, d y
  \right)
  \, d \xi \, d \eta
\end{equation*}
Since $a(\xi) b(\eta)$ vanishes unless
\begin{equation}\label{eqn:xi-eta-size-before-ibp}
  \xi, \eta = \tau + \O(\h^{\delta '' - \delta '})
\end{equation}
in which case $a(\xi) b(\eta) \ll 1$, it is enough to show that if \eqref{eqn:xi-eta-size-before-ibp} holds and $|\zeta - \tau| \ggg \h^{\delta '' - \delta'}$, then
\begin{equation}\label{eq:int_x-y-fx-y-ex-zeta-xi-+-y-zeta-eta-+--x-y}
  \int_{x,y}
  f(x,y)
  e^{(x (\zeta - \xi) + y (\zeta - \eta) + \{ x, y \} \zeta )/\h}
  \, d x \, d y
  \ll \h^N.
\end{equation}
To see this, we integrate by parts, but in a more refined way than in \cite[\S7]{nelson-venkatesh-1}.  We record the details below in Lemma \ref{lem:outsourced-IBP}.

\emph{Taylor expansion.}  Having reduced to the critical case \eqref{eqn:critical-case-zeta-near-tau}, we introduce the notation
\[
  \Omega(x,y,\zeta) := e ^{\{x, y \} \zeta } = \sum _{ \substack{
      \alpha, \beta, \gamma : \\
      |\gamma| \leq \min(|\alpha|,|\beta|) } } c_{\alpha \beta \gamma} x^{\alpha} y^{\beta} \zeta^{\gamma} = \sum _{j \geq 0} \Omega_j(x,y,\zeta),
\]
where $\Omega_j$ denotes the homogeneous component obtained by restricting the summation to $|\alpha| + |\beta| - |\gamma| = j$, so that
\[
  a \star^j b(\zeta) = \int _{x , y} a^\vee(x) b^\vee(y) e^{ x \zeta} e^{y \zeta} \Omega_j(x,y,\zeta) \, d x \, d y.
\]
We write $\Omega^{(J)} := \Omega - \sum_{0 \leq j < J} \Omega_j$ for the remainder obtained by subtracting from $\Omega$ its first $J$ homogeneous components.  Using the integral representation \eqref{eqn:rescaled-star-integral-rep}, we may then split $r(\zeta) = r'(\zeta) + r''(\zeta)$, where
\begin{equation*}
  r'(\zeta) :=
  \int_{x,y}
  a_{\h}^\vee(x) b_{\h}^\vee(y)
  \Omega^{(J)}(x,y,\zeta/\h)
  \, d x \, d y,
\end{equation*}
\begin{equation*}
  r''(\zeta) :=
  \int_{x,y}
  a_{\h}^\vee(x) b_{\h}^\vee(y)
  (\chi(x) \chi(y) - 1)
  \Omega(x,y,\zeta/\h)
  \, d x \, d y.
\end{equation*}
Since $|\Omega(x,y,\zeta)| = 1$ and $\chi$ is identically $1$ near the origin, we obtain from \eqref{eqn:moment-estimates-lotsa-primes} the estimate $r''(\zeta) \ll (A' A'')^{-M}$ for any fixed $M$.  Since $\max(\delta ' , \delta '') < 1$, it follows that $r''(\zeta) \ll \h^M$ for any fixed $M$.  Our restriction $|\zeta| \leq Q^{\eps}$ implies that, say, $Q \ll \h^{-2}$, so this last estimate for $r''(\zeta)$ suffices.

It remains to estimate $r'(\zeta)$.  For this we Taylor expand $\Omega^{(J)}$ as in \cite[\S7.4]{nelson-venkatesh-1}, but incorporating the refined Taylor coefficient support condition afforded by Lemma \ref{lem:constraint-coefficients-via-BCDH}.  Thanks to the condition \eqref{eqn:critical-case-zeta-near-tau} and Lemma \ref{lem:symbol-class-change-basepoint}, we may and shall assume that $\zeta = \tau$.  We apply Lemma \ref{lem:constraint-coefficients-via-BCDH} to write $\Omega_j(x,y,\tau/\h)$ as
\[
  \sum_{ \substack{ \alpha', \alpha '', \beta', \beta '',\gamma :
      \\
      |\alpha '| + |\alpha ''| + |\beta '| + |\beta ''|
      - |\gamma| = j,  \\
      |\gamma| \leq \min(|\alpha '| + |\alpha ''|, |\beta '| + |\beta ''|),
      \\
      |\alpha ' | + |\beta' | \geq 2 |\gamma| } } c_{\alpha' \alpha '' \beta' \beta '' \gamma} (x')^{\alpha '} (x'')^{\alpha ''} (y')^{\beta'} (y'')^{\beta''} (\tau/\h)^\gamma,
\]
with coordinates as in the formulation of that Lemma.  Using the analyticity of $\{x,y\}$ and the estimate $\{x,y\} \ll |x| \cdot |y|$, we see that $c_{\alpha' \alpha '' \beta' \beta '' \gamma} \ll R^{j}$ for some fixed $R > 0$ (with the implied constant independent of $j$).  We claim that for each term as in the above sum and all $x,y = \O(1)$, we have
\begin{equation}\label{eqn:epic-bound-by-rho}
  (x')^{\alpha '}
  (x'')^{\alpha ''}
  (y')^{\beta'}
  (y'')^{\beta''}
  (\tau/\h)^\gamma
  \ll
  \rho^j
\end{equation}
where we abbreviate
\begin{equation*}
  \rho :=
  \max(|x'|,|y'|,|x''|,|y''|,|x'|^2/\h,|y'|^2/\h).
\end{equation*}
Indeed, since $|\tau| \asymp 1$, the LHS of \eqref{eqn:epic-bound-by-rho} is majorized by
\begin{equation*}
  \max(|x'|, |y'|)^{|\alpha '| + |\beta '| - 2 |\gamma|}
  \max(|x''|,|y''|)^{|\alpha ''| + |\beta ''|}
  (\max(|x'|,|y'|)^2 /\h)^{|\gamma|}
\end{equation*}
which is bounded in turn, thanks to the inequality $|\alpha '| + |\beta '| - 2 |\gamma| \geq 0$, by
\begin{equation*}
  \rho ^{|\alpha '| + |\beta '| - 2 |\gamma|}
  \rho ^{|\alpha ''| + |\beta ''|}
  \rho^{|\gamma|}
  = \rho^j,
\end{equation*}
as required.  Bounding the number of relevant multi-indices crudely by $(1 + j)^{\O(1)}$, we deduce that
\begin{equation*}
  \Omega_j(x,y,\tau/\h) \ll (1 + j)^{\O(1)} (R \rho)^j.
\end{equation*}
We now argue as in \cite[\S7.4]{nelson-venkatesh-1} -- using the trivial bounds $\Omega(x,y,\tau/\h) \ll 1$, $\Omega_j(x,y,\tau/\h) \ll_j 1$ when $\rho \geq 1/2 R$ and summing over $j \geq J$ when $\rho \leq 1/2 R$ -- to see that
\begin{equation*}
  \Omega^{(J)}(x,y,\tau/\h)
  \ll \rho^J.
\end{equation*}
Our task thereby reduces to verifying that if $J$ is fixed but large enough in terms of $N$, then
\begin{equation}\label{eqn:required-moment-bound-with-rho}
  \int_{x,y} |a_{\h}^\vee(x)
  b_{\h}^\vee(y)| \rho^J \, d x \, d y
  \ll \h^N.
\end{equation}
The moment bound \eqref{eqn:moment-estimates-lotsa-primes} gives
\begin{align*}
  \int_{x,y} |a_{\h}^\vee(x)
  b_{\h}^\vee(y)| \rho^J \, d x \, d y
  &\ll
    \max
    \left(
    (A')^{-1},
    (A'')^{-1},
    (A')^{-2}/\h
    \right)^J
  \\
  &=
    \max
    \left(
    \h^{1 - \delta '},
    \h^{1 - \delta ''},
    \h^{1 - 2 \delta '}
    \right)^J.
\end{align*}
Since $\delta ' < \delta '' < 2 \delta ' < 1$, this last expression is indeed $\leq \h^N$ for large enough $J$.  The proof of \eqref{eqn:star-j-asymp-expn-new} is thus complete.  \qed

\subsubsection{Integration by parts}
We postponed in the above proof the following application of integration by parts.
\begin{lemma}\label{lem:outsourced-IBP}
  Let $\h \lll 1$, let $\delta', \delta ''$ be fixed quantities satisfying \eqref{eqn:delta-delta-prime-hypotheses}, let $\eps > 0$ be fixed but small enough in terms of $\delta ', \delta ''$, let $\tau$ belong to a fixed compact subset of $\mathfrak{g}^\wedge_{\reg}$, let $\xi, \eta \in \mathfrak{g}^\wedge$ with $\xi = \tau + \O(\h^{\delta '' - \delta '})$, and let $\zeta \in \mathfrak{g}^\wedge$ be of the form $\O(\h^{-\eps})$.  Suppose moreover that
  \[\zeta - \tau \ggg \h^{\delta '' - \delta '}.\]
  Set
  \[B' := \h^{1 - \delta ' - \eps } \leq B'' := \h^{1 - \delta '' - \eps } \lll 1.
  \]
  Let $f \in C_c^\infty(\mathfrak{g} \times \mathfrak{g})$ satisfy the support condition (in $\tau$-coordinates)
  \begin{equation*}
    f(x,y) \neq 0
    \quad \implies \quad  x', y' \ll B', \quad 
    x'', y'' \ll B''
  \end{equation*}
  and the derivative bounds
  \begin{equation*}
    \partial_{x'}^{\alpha'}
    \partial_{x''}^{\alpha''}
    \partial_{y'}^{\beta'}
    \partial_{y''}^{\beta''}
    f(x,y) \ll
    (B')^{-|\alpha'| - |\beta'|}
    (B'')^{-|\alpha''| - |\beta''|}
  \end{equation*}
  for all fixed multi-indices $\alpha ', \alpha '', \beta ', \beta ''$.  Define the phase function
  \[
    \phi(x,y,\zeta) := x (\zeta -\xi) + y (\zeta -\eta) + \{x, y\} \zeta \in i \mathbb{R}.
  \]
  Then for each fixed $\gamma \in \mathbb{Z}_{\geq 0}^{\dim(\mathfrak{g})}$ and $N \geq 0$, we have
  \begin{equation}\label{eqn:f-x-y-e-phi-blah-ibp-triv}
    \partial_\zeta^{\gamma}
    \int_{x,y} f(x,y)
    e^{\phi(x,y,\zeta)/\h} \, d x \, d y
    \ll \h^N
  \end{equation}
\end{lemma}
Lemma \ref{lem:outsourced-IBP} implies the estimate \eqref{eq:int_x-y-fx-y-ex-zeta-xi-+-y-zeta-eta-+--x-y} required earlier.  The proof of Lemma \ref{lem:outsourced-IBP} invokes the following useful lemma from \cite[\S8]{MR3127809}.
\begin{lemma}\label{lem:IBP}
  Let $Y \geq 1$, $X, Q, U, R > 0$, 
  and suppose that $w$ 
  is a smooth function with support on a compact interval $E \subseteq R$ such that for all $j \in \mathbb{Z}_{\geq 0}$,
  \begin{equation*}
    w^{(j)}(t) \ll_j X U^{-j}.
  \end{equation*}
  Suppose $\Phi$ is a smooth real-valued function on $E$ such that
  \begin{equation*}
    |\Phi'(t)| \geq R
  \end{equation*}
  for some $R > 0$, and
  \begin{equation*}
    \Phi^{(j)}(t) \ll_j Y Q^{-j}, \qquad \text{for } j=2, 3, \dots.
  \end{equation*}
  Then for all $A \geq 0$,
  \begin{equation*}
    \int_{t \in \mathbb{R} }
    w(t) e^{i \Phi(t)} dt \ll_A \vol(E) X [(QR/\sqrt{Y})^{-A} + (RU)^{-A}].
  \end{equation*}
\end{lemma}

\begin{proof}[Proof of Lemma \ref{lem:outsourced-IBP}]
  We may and shall assume that $\gamma = 0$.  Indeed, we may reduce to this case by absorbing factors such as $\{x,y\}^{\gamma}$ into the weight function $f$.
  
  At least one of the following possibilities occurs:
  \begin{enumerate}[(i)]
  \item $|\zeta ' - \tau '| \gg |\zeta - \tau|$.
  \item $|\zeta '' - \tau ''| \gg |\zeta - \tau|$.
  \end{enumerate}

  Consider first case (i).  We may then choose a unit speed one-parameter subgroup $t \mapsto z_t'$ of $\mathfrak{g}_\tau^{\flat}$ so that $\partial_{t} z_t' (\zeta - \tau) \gg |\zeta - \tau|$.  Set $e_1 := \partial_t z_t'$; it is the unit vector for which $z_t ' = t e_1$.  We extend $e_1$ to an orthonormal basis $e_1,\dotsc,e_{\dim(\mathfrak{g})}$ of $\mathfrak{g}$, and set $W := \oplus_{j \geq 2} \mathbb{R} e_j$.  The LHS of \eqref{eqn:f-x-y-e-phi-blah-ibp-triv} may then be written as the iterated integral
  \begin{equation*}
    \int_{ x \in W, y \in \mathfrak{g} }
    \int_{t \in \mathbb{R}}
    f(x + z_t', y)
    e^{\phi(x + z_t',y,\zeta)/\h}
    \, d t
    \, d x \, d y.
  \end{equation*}
  It suffices to show that the inner integral over $t$ is $\ll \h^N$ for all $x,y$ in the support of the outer integral.  We accordingly focus on individual $x,y \in \mathfrak{g}$ satisfying
  \begin{equation}\label{eqn:x-y-g-hypotheses-ibp}
    \text{$x',y' \ll B'$ and $x'',y'' \ll B''$}
  \end{equation}
  and set
  \begin{equation*}
    w(t) := f(x + z_t', y),
    \quad
    \Phi(t) := \phi(x + z_t', y,\zeta) / i \h.
  \end{equation*}
  We aim to show that
  \begin{equation}\label{eqn:ibp-case-y-prime}
    \int_{t \in \mathbb{R} }
    w(t)
    e^{i \Phi(t)} \, d t
    \ll \h^N.
  \end{equation}
  We will do so via Lemma \ref{lem:IBP}.  To apply that lemma, we must estimate the derivatives of $\Phi$.  To that end, we claim first that for $t \ll B'$ (equivalently, $z_t' \ll B'$), we have
  \begin{equation}\label{eqn:partial-t-brackets-zeta-B-primes}
    \partial_t \{ x + z_t', y \} \zeta \ll
    B' |\zeta| + B'' |\zeta - \tau|.
  \end{equation}
  To verify \eqref{eqn:partial-t-brackets-zeta-B-primes}, we expand the analytic function $\{,\}$ as a Taylor series.  From our assumed upper bounds on $x,y,\zeta$, we see that the error incurred in $\partial_t \{x + z _t ', y \} \zeta$ by replacing $\{, \}$ with the first $J$ homogeneous components of its Taylor series is $\O(\h^N)$, for any fixed $N$, provided that $J$ is fixed but large enough in terms of $N$.  Since our hypotheses imply that the RHS of \eqref{eqn:partial-t-brackets-zeta-B-primes} is bounded from below by some fixed power of $\h$, it is thus enough to verify for each fixed iterated Lie monomial $L$ of degree at least two that
  \begin{equation}\label{eqn:partial-t-L-x-t-prime-etc}
    \partial_t L( z_t', y ) \zeta \ll
    B' |\zeta| + B'' |\zeta - \tau|.
  \end{equation}
  We verify this first for $L$ the commutator bracket $[,]$.  By linearity, we may consider separately the cases $y =y '$ and $y = y''$.  In the first case, we have $\partial_t [z_t', y'] \ll B'$, which gives the adequate estimate $\partial_t [z_t', y] \zeta \ll B' |\zeta|$.  In the second case, the element $y''$ centralizes $\tau$, hence
  \begin{align*}
    \langle [z_t', y''], \zeta  \rangle
    &=
      \langle z_t', [y'', \zeta]  \rangle
    \\
    &=
      \langle z_t', [y'', \zeta - \tau ]  \rangle
    \\
    &=
      \langle [z_t',y''], \zeta - \tau  \rangle,
  \end{align*}
  giving again the adequate estimate $\partial_t [z_t', y''] \zeta \ll B'' |\zeta - \tau|$.  This completes the proof of \eqref{eqn:partial-t-L-x-t-prime-etc} for $L$ of degree two.  We argue similarly for $L$ of higher degree.  For instance, in the case of degree three, we obtain the adequate estimates
  \begin{itemize}
  \item $\partial_t [z_t', [z_t', y']] \zeta \ll (B')^2 |\zeta|$,
  \item $\partial_t [y_1', [z_t', y_2']] \zeta \ll (B')^2 |\zeta|$,
  \item $\partial_t [z_t', [z_t', y'']] \zeta \ll B' B'' |\zeta|$, and
  \item $\partial_t [y_1'', [z_t', y_2'']] \zeta \ll (B'')^2 |\zeta - \tau|$
  \item $\partial_t [y'', [z_t', y']] \zeta \ll B' B'' |\zeta|$.
  \end{itemize}
  for $y', y_1', y_2' \ll B'$ in $\mathfrak{g}_\tau^{\flat}$ and $y'', y_1'', y_2'' \ll B''$ in $\mathfrak{g}_\tau$.  In general, if our monomial features at least one $y'_j$ or at least two $z_t'$, then we obtain an adequate estimate by taking into account the sizes of $z_t'$, $y'_j$ and $|\zeta|$.  Otherwise, our monomial features one $z_t'$, many $y_j''$ and no $y_j'$.  We then use that each $y_j''$ centralizes $\tau$ to replace $\zeta$ by $\zeta - \tau$ and argue as in the case $L(x,y) = [x,y]$.  This completes the proof of \eqref{eqn:partial-t-L-x-t-prime-etc}, hence that of \eqref{eqn:partial-t-brackets-zeta-B-primes}.

  We observe next that
  \begin{equation}\label{eqn:h-prime-at-least-blah}
    \Phi'(t) \gg
    |\zeta - \tau |/\h.
  \end{equation}
  To see this, we first rewrite the definition in the form
  \begin{equation}\label{eqn:i-h-h-prime-t}
    i \h \Phi'(t) =
    \partial_t z_t' (\zeta - \tau)
    +
    \partial_t z_t' (\tau - \xi)
    +
    \partial_t \{x + z_t', y\} \zeta.
  \end{equation}
  By the construction of $z_t'$, the first term on the RHS of \eqref{eqn:i-h-h-prime-t} is $\gg |\zeta - \tau|$.  The second term is
  \[\ll \h^{\delta '' - \delta '}
    \lll |\zeta - \tau|.\] The third term is
  \[
    \ll B' |\zeta| + B'' |\zeta - \tau| \lll |\zeta - \tau|
  \]
  by \eqref{eqn:partial-t-brackets-zeta-B-primes} and our hypothesis $|\zeta| \ll \h^{-\eps}$.  This completes the verification of \eqref{eqn:h-prime-at-least-blah}.

  We now apply Lemma \ref{lem:IBP} with
  \begin{equation}\label{eqn:XURQ-choices}
    X = 1,
    U = B',
    \quad 
    R \asymp |\zeta - \tau|/\h,
    \quad Y =
    \frac{B'' |\zeta|}{\h}
    \ll  B'' \h^{-1+\eps},
    Q = 1.
  \end{equation}
  By \eqref{eqn:h-prime-at-least-blah}, we may find such an $R$ with $|\Phi'(t)| \geq R$.  We note that for fixed $j \geq 2$, we have $\Phi^{(j)}(t) \ll Y$, or equivalently, $\partial_t^j \{x + z_t', y\} \zeta \ll B'' |\zeta|$.  To see this, we expand $\{,\}$ as above in a Taylor series consisting of iterated Lie monomials and use that $y \ll B' + B'' \asymp B'' \leq 1$.  We note also that $w$ is a smooth function on $\mathbb{R}$ that is supported on an interval $E$ with $\vol(E) \ll B'$ and satisfies the derivative bounds $w^{(j)}(t) \ll_j (B')^{-j}$.  The hypotheses of Lemma \ref{lem:IBP} are thus satisfied.  We observe that
  \[
    R U \gg \h^{\delta '' - 2 \delta ' - \eps } \geq \h^{-\eps}
  \]
  and
  \[
    \frac{(Q R)^2}{Y} \gg \frac{(\h^{\delta '' - \delta ' }/\h)^2}{B'' \h^{-1-\eps}} = \h^{2 \eps - 2 ( 1 - \delta ') - ( 2 \delta ' - \delta '')} \geq \h^{-\eps}.
  \]
  Thus the required estimate \eqref{eqn:ibp-case-y-prime} follows from the conclusion of Lemma \ref{lem:IBP} upon taking $A$ sufficiently large in terms of $\eps$ and $N$.  This completes the proof of Lemma \ref{lem:outsourced-IBP} in case (i).

  The proof in case (ii) is similar but slightly simpler.  We choose a unit speed one-parameter subgroup $t \mapsto z_t''$ of $\mathfrak{g}_\tau$ so that $\partial_t z_t''(\zeta - \tau) \gg |\zeta - \tau|$.  The estimate \eqref{eqn:partial-t-brackets-zeta-B-primes} remains valid for $t \ll B''$, with essentially the same proof: we reduce to estimates for monomials as before and argue separately according as some $y_j'$ appears (in which case we get the bound $\ll B' |\zeta|$) or none appears (in which case we obtain $\ll B'' |\zeta - \tau|$).  (The second case could even be eliminated altogether by exploiting the fact that $\mathfrak{g}_\tau$ is abelian in our setup.)  We make the same choices \eqref{eqn:XURQ-choices} as before (improving $U$ to $B''$ if desired) and conclude once again via Lemma \ref{lem:IBP}.
\end{proof}

\subsubsection{The case of disjoint supports}
\label{sec:case-disj-supp}
We have noted following the statement of Theorem \ref{thm:refined-star-prod} that in order to complete the proof of \eqref{eqn:star-h-new-mapping-property-super-localized}, it suffices to show for
\begin{equation}\label{eqn:tau1-tau2-far-apart}
  \tau_1 - \tau_2 \ggg \h^{\delta '' - \delta '}
\end{equation}
that $a \star_{\h} b \in \h^\infty S^{-\infty}$ for $(a,b) \in S^{\tau_1}_{\delta ', \delta ''} \times S^{\tau_2}_{\delta ', \delta ''}$.  We must check that $\partial_{\zeta}^\gamma (a \star_{\h} b)(\zeta) \ll \h^N \langle \zeta \rangle^{-N}$ for fixed $\gamma,N$ and all $\zeta \in \mathfrak{g}^\wedge$.  To see this, we reduce as in \S\ref{sec:estimates-remainder} to the case that $a,b$ satisfy the support conditions \eqref{eqn:support-a-vs-omega-1} for some $\omega_j = \tau_j + \O(\h^{\delta '' - \delta '})$ and that $|\zeta| < Q^{\eps}$.  We see then by \eqref{eqn:tau1-tau2-far-apart} that $|\zeta - \tau_j| \ggg \h^{\delta '' - \delta '}$ for some $j=1,2$.  If $j=1$, then we conclude by applying the argument following \eqref{eqn:critical-case-zeta-near-tau} verbatim.  If $j=2$, then we apply the same argument but with the roles of $a$ and $b$ reversed.

\subsubsection{The mixed case}\label{sec:mixed-case}

Here we discuss the proof of part \eqref{item:star-prod-4} of Theorem \ref{thm:refined-star-prod}.  We have noted already that the proof amounts to an interpolation between the proofs of part \eqref{item:star-prod-3} and of \cite[\S7.3]{nelson-venkatesh-1}, so we will be brief.

As in the proof of part \eqref{item:star-prod-3}, it is enough to verify
\begin{itemize}
\item the claimed mapping properties for $\star^j$, and
\item the claimed asymptotic expansion for $a \star_{\h} b$, where (say) $a \in S_{\delta ', \delta ''}^{\tau}$ and $b \in S_\delta^m$ (the same argument applies with the roles of $a$ and $b$ reversed).
\end{itemize}
The former may be verified exactly as in \S\ref{sec:mapp-prop-homog}.  For the latter, we reduce as in \S\ref{sec:estimates-remainder} to verifying that the remainder $r := a \star_{\h} b - \sum _{0 \leq j < J} \h^j a \star^j b$ enjoys the estimate $\partial_\zeta^\gamma r(\zeta) \ll \h^N \langle \zeta \rangle^{-N}$, as in \eqref{eqn:partial-zeta-gamma-r-zeta-trivz}, provided that $J$ is fixed large enough in terms of $N$.  By decomposing our symbols into localized pieces, we may reduce as in \cite[\S7.7]{nelson-venkatesh-1} and \S\ref{sec:estimates-remainder} to the case that
\begin{itemize}
\item $a$ satisfies the support condition \eqref{eqn:support-a-vs-omega-1} for some $\omega_1 = \tau + \O(\h^{\delta '' - \delta '})$, while
\item $b$ is supported on $\omega_2 + \O(\h^\delta \langle \omega_2 \rangle)$ for some $\omega_2 \in \mathfrak{g}^\wedge$.
\end{itemize}
We then have with $(A', A'', m, n)$ as in \eqref{eqn:notation-for-describing-fourier-transforms-of-superlocalized-symbols} and $A := \h^{-1 + \delta} \langle \omega_2 \rangle$ the moment bound
\begin{equation}\label{eqn:moment-estimates-lotsa-primes-mixed}
  \int_{x,y \in \mathfrak{g}}
  \left\lvert
    a_{\h}^\vee(x)
    b_{\h}^\vee(y)
  \right\rvert
  |x'|^{k'}
  |x''|^{k''}
  |y'|^{l'}
  |y''|^{l''}
  \, d x \, d y
  \ll
  (A')^{-k'}
  (A'')^{-k''}
  A^{-l' - l''},
\end{equation}
with notation as in \eqref{eqn:moment-estimates-lotsa-primes}.  We set $Q := \h^{-1} \langle \zeta \rangle \langle \omega_2 \rangle$.  As in \cite[\S7.7]{nelson-venkatesh-1}, the required estimate follows by trivially estimating the integral representation \eqref{eqn:rescaled-star-integral-rep} unless $|\omega_2| \leq Q^{\eps}$.  We reduce further to the case $|\zeta| \geq Q^{2 \eps}$ by integrating by parts (crudely) as in \cite[\S7.7]{nelson-venkatesh-1}.  We reduce further to the case $\zeta = \tau + \O(\h^{\delta '' - \delta '})$ by the same argument as in the reduction to \eqref{eqn:critical-case-zeta-near-tau}, noting that Lemma \ref{lem:outsourced-IBP} imposes no constraints on the variable $\eta$.  Arguing exactly as in \S\ref{sec:estimates-remainder}, we reduce finally to verifying the estimate \eqref{eqn:required-moment-bound-with-rho}.  For this we appeal to the moment bound \eqref{eqn:moment-estimates-lotsa-primes-mixed} and the fact that $\max(2 \delta, 2 \delta ', \delta '') < 1$.

\section{Operators}\label{sec:operator-classes}
Let $\pi$ be a unitary representation of the fixed Lie group $G$ over $\mathbb{R}$.

\subsection{Spaces of operators}\label{sec:spaces-operators}
\index{Lie algebra!Laplacian $\Delta, \Delta_G$} We set
\begin{equation*}
  \Delta := 1 - \sum _{x \in \mathcal{B}(\mathfrak{g}) } x^2 \in \mathfrak{U}(\mathfrak{g})
\end{equation*}
for some fixed basis $\mathcal{B}(\mathfrak{g})$ of $\mathfrak{g}$.  We write more verbosely $\Delta_G$ when we wish to indicate which group is being considered.  As discussed in \cite[\S3.1]{nelson-venkatesh-1}, $\pi(\Delta)$ is a densely-defined self-adjoint positive operator, with bounded inverse having operator norm $\leq 1$.

We write $\pi^{\infty}$ for the space of smooth vectors.  For $s \in \mathbb{Z}$, we denote by $\pi^s$ the Hilbert space completion of $\pi^\infty$ with respect to the inner product
\begin{equation*}
  \langle v_1, v_2 \rangle_{\pi^s} := \langle \pi(\Delta)^s v_1, v_2 \rangle.
\end{equation*}
\index{representations!Sobolev space $\pi^s$} Up to natural identifications,
\[
  \pi^\infty = \cap \pi^s \leq \dotsb \leq \pi^{s+1} \leq \pi^s \leq \pi^{s-1} \leq \dotsb.
\]
The space $\pi^{-\infty}$ of distributional vectors is defined to be the union of the spaces $\pi^s$.  For future reference, we record \cite[(3.1)]{nelson-venkatesh-1}: for fixed $s \in \mathbb{Z}_{\geq 0}$,
\begin{equation}\label{eqn:Delta-s-positive-easy}
  \|v\|_{\pi^s}^2
  \asymp
  \sum _{r = 0}^s
  \sum _{x_1,\dotsc,x_r \in \mathcal{B}(\mathfrak{g})}
  \|\pi(x_1 \dotsb x_r) v\|^2.
\end{equation}

By an \emph{operator} on $\pi$, we mean a linear map $T : \pi^{\infty} \rightarrow \pi^{-\infty}$.  We denote by $\underline{\Psi}^m$ the space of operators $T$ such that for each $s \in \mathbb{Z}$, $n \in \mathbb{Z}_{\geq 0}$ and $x_1,\dotsc,x_n \in \mathfrak{g}$, the following iterated commutator defines a bounded map between the indicated Hilbert spaces:
\begin{equation}\label{eqn:defn-Psi-m}
  [\pi(x_1),[\pi(x_2),[\dotsc,[\pi(x_n),T] \dotsb]]]
  :
  \pi^s \rightarrow \pi^{s-m}.
\end{equation}
\index{operators!underlying space $\underline{\Psi}^m$} We extend the definition to $m = \pm \infty$ by taking the union or intersection over all integers $m$.  (For instance, elements of $\underline{\Psi}^{-\infty}$ may be understood as ``smoothing operators'': they map any space $\pi^s$ into $\pi^{-\infty}$.)  We write more verbosely $\underline{\Psi}^m(\pi)$ when we wish to indicate the representation under consideration.  As in \S\ref{sec:spaces-symbols}, for $m < \infty$, the space $\underline{\Psi}^m$ is a Frechet space equipped with a distinguished family of seminorms, while $\underline{\Psi}^\infty$ is an inductive limit of such spaces.  We refer to \cite[\S3]{nelson-venkatesh-1} for general discussion of the spaces $\underline{\Psi}^m$, parts of which will be recalled below as needed.

\subsection{Operator classes}\label{sec:operator-classes-basic-properties}

\begin{definition}\label{defn:oper-class}
  For fixed $m \in \mathbb{Z}$ and $\delta \in [0,1)$, we write\footnote{ The notation here differs from that in \cite[\S3]{nelson-venkatesh-1}.  There, we used ``$\Psi^m$'' in place of $\underline{\Psi }^m$ and ``$\Psi^m_\delta$'' for the space of $\h$-dependent operators whose norms are bounded in the indicated manner.  The class $\Psi^m_\delta$ defined here plays a similar role.  The relationship between these definitions is as described in the proof of Theorem \ref{thm:nv-star-prod-asymp-h-dependent}.  } \index{operators!operator class $\Psi^m_\delta$}
  \[
    \Psi^m_\delta
  \]
  for the class (\S\ref{sec:classes}) of all $T \in \underline{\Psi }^{\infty}$ such that for each fixed $s,n$ and $x_1,\dotsc,x_n$, the map \eqref{eqn:defn-Psi-m} has operator norm $\O(\h^{-\delta n})$.  We extend the definition to $m = \pm \infty$ as in Definition \ref{defn:basic-symbol-class}, thus $\Psi_\delta^{\infty}$ (resp. $\Psi_\delta^{-\infty}$) is the union (resp. intersection) of $\Psi_\delta^m$ over fixed integers $m$.  We write more verbosely $\Psi_\delta^m(\pi)$ when we wish to indicate the representation under consideration.
\end{definition}

As in \S\ref{sec:basic-symbol-classes}, we may define for a positive real $c$ the class $c \Psi_\delta^m$, and we denote by $\h^\infty \Psi_\delta^m$ the intersection of $\h^N \Psi_\delta^m$ over all fixed $N$.  The classes $\h^\infty \Psi_\delta^{m}$ are independent of $\delta$ and will be denoted simply by $\h^\infty \Psi^m$.

\begin{lemma}\label{lem:composition-operator-classes}
  Composition of operators induces, for each fixed $m_1, m_2$,
  \begin{equation*}
    \Psi_\delta^{m_1} \times \Psi_\delta^{m_2}
    \rightarrow \Psi_\delta^{m_1 + m_2},
  \end{equation*}
  with the convention $\infty + (-\infty) := -\infty$.
\end{lemma}
\begin{proof}
  This follows from the proof of \cite[\S3.4]{nelson-venkatesh-1}.  Indeed, for $(T_1,T_2) \in \Psi_{\delta}^{m_1} \times \Psi_{\delta}^{m_2}$ and fixed $x_1,\dotsc,x_n \in \mathfrak{g}$, we may write $[\pi(x_1),\dotsc,[\pi(x_n),T_1 T_2]]$ as a linear combination of compositions $T_1' T_2'$, where $T_1' = [\pi(y_1),\dotsc,[\pi(y_{n_1}),T_1]]$ and $T_2' = [\pi(z_1),\dotsc,[\pi(z_{n_1}),T_2]]$ with $n_1 + n_2 = n$ and $y_i, z_j \in \mathfrak{g}$ fixed.  We must check that the composition $T_1' T_2' : \pi^s \rightarrow \pi^{s-m}$ has operator norm $\O(\h^{-\delta n})$.  To do so, we factor that composition as $\pi^s \xrightarrow{T_2'} \pi^{s- m_2} \xrightarrow{T_1'} \pi^{s - m}$ and apply our hypotheses to each factor.
\end{proof}

\begin{lemma}\label{lem:Delta-xj-vs-Psi}
  For fixed $m \in \mathbb{Z}$,
  \begin{equation}
    \pi(\Delta)^m \in \Psi_0^{2 m}.
  \end{equation}
  For fixed $x_1,\dotsc, x_m \in \mathfrak{g}$,
  \begin{equation}\label{eqn:lie-alg-ops-op-class}
    \pi(x_1 \dotsb x_m) \in \Psi_0^m.
  \end{equation}
\end{lemma}
\begin{proof}
  This follows from \cite[\S3.5]{nelson-venkatesh-1}.
\end{proof}
Using the self-adjointness of $\pi(\Delta)$, it is straightforward to check that $\Psi^m$ is preserved under taking adjoints.

\subsection{Operators attached to symbols}\label{sec:oper-attach-symb}
Recall that $\underline{S}^{-\infty}$ is the Schwartz space on $\mathfrak{g}^\wedge$.  As explained in \cite[\S5.1]{nelson-venkatesh-1}, the assignment \index{operators!$\Opp$}
\begin{equation}\label{eqn:basic-Opp-assignment}
  \begin{split}
    \Opp : \underline{S}^{-\infty} \rightarrow \{\text{operators on } \pi \}
    \\
    a \mapsto \Opp(a) := \Opp_{\h}(a:\pi,\chi),
  \end{split}
\end{equation}
defined in \S\ref{sec:basic-oper-assignm} with respect to some nice cutoff $\chi$, extends naturally to the space of symbols $\underline{S}^\infty$.  In particular, it extends to every symbol class defined in \S\ref{sec:symbols-star-product}.  The extension may be characterized in terms of matrix coefficients: for smooth vectors $u,v \in \pi$,
\begin{equation*}
  \langle \Opp(a) u, v \rangle
  =
  \int_{\xi \in \mathfrak{g}^\wedge} a(\h \xi)
  \left(\int_{x \in \mathfrak{g}}
    e^{-x \xi}
    \chi(x) \langle \pi(\exp(x)) u, v \rangle \, d x \right) \, d \xi.
\end{equation*}
Note that the parenthetical integral over $x$ defines a Schwartz function of $\xi$, so the remaining integral over $\xi$ converges absolutely.

Each polynomial function $p : \mathfrak{g}^\wedge \rightarrow \mathbb{C}$ (equivalently, element $p$ of the symmetric algebra $\Sym(\mathfrak{g}_\mathbb{C})$) defines an element of $\underline{S}^{\infty}$.  The image of such an element under $\Opp$ is described as follows.
\begin{lemma}\label{lem:polyn-symb}
  For each $p \in \Sym(\mathfrak{g}_\mathbb{C}) \subseteq \underline{S}^{\infty}$, we have
  \[
    \Opp(p) = \pi(\sym(p_{\h})),
  \]
  where $p_{\h}$ denotes the rescaling $p_{\h}(\xi) := p(\h \xi)$ and $\sym : \Sym(\mathfrak{g}_\mathbb{C}) \rightarrow \mathfrak{U}(\mathfrak{g}_\mathbb{C})$ denotes the symmetrization map, i.e., the linear isomorphism sending each monomial to the average of its permutations.
\end{lemma}
\begin{proof}
  See \cite[\S5.2, \S8.1]{nelson-venkatesh-1}.
\end{proof}
For applications of Lemma \ref{lem:polyn-symb}, it is useful to note that $\sym(p^k) = \sym(p)^k$ for $k \in \mathbb{Z}_{\geq 0}$.  Moreover, for $x \in \mathfrak{g}$ -- regarded as a linear function $x : \mathfrak{g}^\wedge \rightarrow i \mathbb{R}$, or equivalently, as a degree one element of $\Sym(\mathfrak{g}_\mathbb{C})$ -- we have $\sym(x_{\h}) = \h x$ (regarded as an element of $\mathfrak{U}(\mathfrak{g})$), hence $\Opp(x) = \h \pi(x)$.

Throughout this section, we retain the abbreviation \eqref{eqn:basic-Opp-assignment}.  Our aim is to recall from \cite{nelson-venkatesh-1} the effect of $\Opp$ on the basic symbol classes $\underline{S}^m$ and $S^m_\delta$ and to describe its effect on the refined classes $S^\tau_{\delta ', \delta ''}$.

\begin{lemma}
  \label{lem:chi-doesnt-matter-underline}
  For any nice cutoffs $\chi_1, \chi_2$ and any $a \in \underline{S}^\infty$, we have
  \[\Opp_{\h}(a;\pi,\chi_1)- \Opp_{\h}(a;\pi,\chi_2) \in
    \underline{\Psi}^{-\infty}.\] This difference defines a continuous map $\underline{S}^\infty \rightarrow \underline{\Psi}^{-\infty}$.
\end{lemma}
\begin{theorem}\label{thm:opp-S-m-underlined}
  We have
  \[
    \Opp(\underline{S}^m) \subseteq \underline{\Psi }^m,
  \]
  and the induced map is continuous.  In particular, elements of $\Opp(\underline{S}^\infty)$ act on $\pi^\infty$, and so may be composed.  For $(a,b) \in \underline{S}^{m_1} \times \underline{S}^{m_2}$, the composition formula \eqref{eq:composition-for-basic-operator-assignment} remains valid, and we have
  \[
    \Opp(a) \Opp(b) \equiv \Opp(a \star_{\h} b) \mod{\underline{\Psi }^{-\infty}}.
  \]
\end{theorem}
\begin{proof}[Proof of Lemma \ref{lem:chi-doesnt-matter-underline} and Theorem \ref{thm:opp-S-m-underlined}]
  The rescaling map $a \mapsto a_{\h}$ defines a topological automorphism of $\underline{S}^m$, so it suffices to consider the case $\h = 1$.  The required conclusions are given then by \cite[\S5.4]{nelson-venkatesh-1} and \cite[Thm 2]{nelson-venkatesh-1}.
\end{proof}

\begin{lemma}
  \label{lem:chi-doesnt-matter}
  For fixed nice cutoffs $\chi_1$ and $\chi_2$, fixed $\delta \in [0,1)$ and any $a \in S^\infty_{\delta}$, we have
  \begin{equation*}
    \Opp_{\h}(a;\pi,\chi_1)
    -
    \Opp_{\h}(a;\pi,\chi_2)
    \in \h^\infty \Psi^{-\infty} \cap \underline{\Psi }^{-\infty}.
  \end{equation*}
\end{lemma}
\begin{theorem}\label{thm:oper-assignm-comp-basic}
  Fix a nice cutoff $\chi$.  Fix $\delta \in [0,1/2)$.  For $m \in \mathbb{Z}$, we have
  \begin{equation}\label{eqn:opp-mapping-S-m-delta}
    \Opp(S^m_\delta)
    \subseteq
    \h^{\min(m,0)}
    \Psi_\delta^m.
  \end{equation}
  For $a, b \in S^\infty_\delta$, we have
  \begin{equation}\label{eqn:composition-formula-1}
    \Opp(a)
    \Opp(b)
    \equiv \Opp(a \star_{\h} b)
    \mod{ \h^\infty \Psi^{-\infty} \cap \underline{\Psi } ^{- \infty }}
  \end{equation}
  For each fixed $m_1, m_2 < \infty$ and $N \in \mathbb{Z}_{\geq 0}$ there is a fixed $J \in \mathbb{Z}_{\geq 0}$ so that for all $(a,b) \in S^{m_1}_\delta \times S^{m_2}_\delta$,
  \begin{equation}\label{eqn:composition-formula-2}
    \Opp(a)
    \Opp(b)
    \equiv \sum _{0 \leq j < J}
    \h^j \Opp(a \star^j b)
    \mod{\h^N \Psi_\delta^{-N}}.
  \end{equation}
\end{theorem}
\begin{proof}[Proof of Lemma \ref{lem:chi-doesnt-matter} and Theorem \ref{thm:oper-assignm-comp-basic}]
  We appeal to \cite[\S5.4]{nelson-venkatesh-1} and \cite[\S5.6]{nelson-venkatesh-1}, translating to the present formulation as in the proof of Theorem \ref{thm:nv-star-prod-asymp-h-dependent}.
\end{proof}

We turn to the main new result of this section.
\begin{theorem}\label{thm:oper-assignm-comp-refined}
  Fix a connected real reductive group $G$ and a nice cutoff $\chi$.  Let $\h, \delta ', \delta '', \tau$ be as in Definition \ref{defn:new-symbol-class}.  Then for each fixed $m \in \mathbb{Z}$,
  \begin{equation}\label{eq:opp-of-refined-symb-into-Psi}
    \Opp(S^\tau _{\delta', \delta ''})
    \subseteq
    \h^{m} \Psi_{\delta '}^{m} \cap \underline{\Psi}^{-\infty}.
  \end{equation}
  In particular, $\Opp(S^\tau _{\delta', \delta ''}) \subseteq \Psi_{\delta '}^0$.  Moreover, the composition formulas \eqref{eqn:composition-formula-1} and \eqref{eqn:composition-formula-2} remain valid for $a,b \in S^\infty_{\delta '} \cup S^\tau_{\delta ', \delta ''}$.
\end{theorem}
\begin{proof}
  Using the refined star product asymptotics afforded by Theorem \ref{thm:refined-star-prod}, we may complete the proof of Theorem \ref{thm:oper-assignm-comp-refined} exactly as in \cite[\S8.6]{nelson-venkatesh-1}.  For the sake of completeness, we record the details here.  The proof occupies the remainder of \S\ref{sec:oper-attach-symb}.

  We verify first that the composition formula \eqref{eqn:composition-formula-1} remains valid for $a,b \in S_{\delta '}^{\infty} \cup S_{\delta ', \delta ''}^\tau$.  Since in particular $a,b \in \underline{S}^{\infty}$, we have by Theorem \ref{thm:opp-S-m-underlined} that the composition formula \eqref{eq:composition-for-basic-operator-assignment} remains valid.  On the other hand, we see from the star product asymptotics given in Theorem \ref{thm:refined-star-prod} and the crude inclusion $S^\tau_{ \delta ', \delta ''} \subseteq S^{-\infty}_{\delta ''}$ that $a \star_{\h} b \in S^{\infty}_{\delta ''}$.  By Lemma \ref{lem:chi-doesnt-matter}, it follows that
  \[
    \Opp_{\h}(a \star_{\h} b: \pi, \chi ') \equiv \Opp(a \star_{\h} b) \mod{\h^\infty \Psi^{-\infty}}.
  \]
  The required formula \eqref{eqn:composition-formula-1} then follows from \eqref{eq:composition-for-basic-operator-assignment}.

  For the proof of the remaining assertions, a key step is to verify the following consequence of \eqref{eq:opp-of-refined-symb-into-Psi}.

\begin{lemma}\label{lem:refined-symb-basic-opp-bound}
  Let $\|.\|$ denote the operator norm on the space of linear maps $\pi \rightarrow \pi$.  For $a \in S^\tau_{\delta ', \delta''}$, we have
  \[
    \|\Opp(a)\| \ll 1.
  \]
\end{lemma}
\begin{proof}[Proof of Lemma \ref{lem:refined-symb-basic-opp-bound}]
  Let $\mathcal{N}$ denote the norm on the Schwartz space $\mathcal{S}(\mathfrak{g}^\wedge)$ given by the $L^1$-norm of the Fourier transform:
  \begin{equation*}
    \mathcal{N}(a) := \|a^\vee \|_{L^1(\mathfrak{g})}.
  \end{equation*}
  This norm is dilation-invariant: $\mathcal{N}(a) = \mathcal{N}(a_{\h})$.  It follows readily from the definition of $\Opp$ that $\|\Opp(a)\| \leq \mathcal{N}(a)$.  We have the crude bound
  \begin{equation}\label{eqn:crude-bound}
    \mathcal{N}(a) \ll \h^{-\O(1)}
  \end{equation}
  where the implied constant depends at most upon $\dim(\mathfrak{g})$ and $(\delta ', \delta '')$.  While $\mathcal{N}(a)$ can be quite large (e.g., a fixed positive power of $\h^{-1}$) for elements $a$ of $S^\tau_{\delta ', \delta ''}$, we will see below that any such element may be decomposed into ``almost-orthogonal'' pieces on which $\mathcal{N}$ is $\O(1)$.  We will then conclude by the Cotlar--Stein lemma \cite[Lem 18.6.5]{MR2304165}, which we recall here:
  \begin{itemize}
  \item \emph{Let $V_1, V_2$ be Hilbert spaces.  Let $T_j : V_1 \rightarrow V_2$ be a sequence of bounded linear operators.  Assume that}
    \begin{equation}\label{eq:cotlar-stein-hypothesis}
      \sup_j \sum_k
      \|T_j^* T_k\|^{1/2}
      \leq C,
      \quad 
      \sup_j \sum_k
      \|T_j T_k^*\|^{1/2} \leq C,
    \end{equation}
    \emph{Then the series $T := \sum T_j$ converges in the Banach space of bounded linear operators from $V_1$ to $V_2$, and has operator norm $\|T\| \leq C$.}
  \end{itemize}

  Let $a \in S^\tau_{\delta ', \delta ''}$.  Recall that $a$ is supported on elements of the form $\tau + \O(\h^{ \delta '' - \delta '})$.  For each element $\omega$ of that form, denote by $\mathcal{D}(\omega)$ the rectangular domain
  \begin{equation*}
    \mathcal{D}(\omega)
    :=
    \left\{ \xi \in \mathfrak{g}^\wedge : |\xi' - \omega'| \leq  \h^{\delta '}, |\xi ''- \omega ''| \leq \h^{\delta ''} \right\}.
  \end{equation*}
  Note that, by the definition of $S^\tau_{\delta ', \delta ''}$ (Definition \ref{defn:new-symbol-class}), the variation of $a$ on each such domain is mild.  We say that $a$ is \emph{localized at $\omega$} if it is supported on $\mathcal{D}(\omega)$.  This notion appeared implicitly in \S\ref{sec:estimates-remainder}, before the estimate \eqref{sec:estimates-remainder}.  By the one-dimensional analogue of that estimate, we see that if $a$ is localized at $\omega$, then $\mathcal{N}(a) \ll 1$.
  
  By taking a suitable smooth partition of unity, we may decompose any element $a \in S^\tau_{\delta ', \delta ''}$ as a finite sum $a = \sum _{\omega \in \Omega } a_\omega$, where
  \begin{itemize}
  \item $\Omega$ is a finite subset of $\mathfrak{g}^\wedge$, consisting of elements $\omega$ of the form $\tau + \O(\h^{\delta '' - \delta '})$, of cardinality $\# \Omega \ll \h^{-\O(1)}$,
  \item the domains $\mathcal{D}(\omega)$ ($\omega \in \Omega$) have uniformly bounded overlaps in the sense that
    \begin{equation}\label{eqn:Omega-well-spaced}
      \sup_{\omega_1 \in \Omega}
      \# \left\{ \omega_2 \in \Omega :
        \mathcal{D}(\omega_1) \cap
        \mathcal{D}(\omega_2) \neq \emptyset \right\} \ll
      1,
    \end{equation}
    and
  \item each summand $a_\omega$ lies in $S^\tau_{\delta ', \delta ''}$ and is localized at the corresponding element $\omega$.
  \end{itemize}
  
  Turning to the proof of the lemma, let $a \in S^\tau _{\delta ', \delta ''}$.  We decompose $a = \sum a_\omega$ as above.  Then $\Opp(a) = \sum \Opp(a_\omega)$ and $\Opp(a_\omega)^* = \Opp(\bar{a}_\omega)$ (cf.\ \S\ref{sec:adjoints}).  By the composition formula \eqref{eq:composition-for-basic-operator-assignment} for Schwartz functions, we see that
  \[
    \|\Opp(a_{\omega_1})^* \Opp(a_{\omega_2})\| \leq \mathcal{N}(\bar{a}_{\omega_1} \star_{\h} a_{\omega_2}).
  \]
  By the Cotlar--Stein lemma, we reduce to verifying that
  \begin{equation*}
    \sup_{\omega_1 \in \Omega}
    \sum _{\omega_2 \in \Omega }
    \mathcal{N} (\bar{a}_{\omega_1} \star_{\h}
    a_{\omega_2})^{1/2}
    \ll 1
  \end{equation*}
  (together with the analogous estimate in which complex conjugation is applied instead to $a_{\omega_2}$).  To see this, we fix $N \in \mathbb{Z}_{\geq 0}$ sufficiently large and evaluate $\bar{a}_{\omega_1} \star_{\h} a_{\omega_2}$ using the star product asymptotics \eqref{eqn:star-j-asymp-expn-new} as the sum $b_{\omega_1, \omega_2} + r_{\omega_1,\omega_2}$, where
  \[
    b_{\omega_1, \omega_2} := \sum _{0 \leq j < J} \h^j \bar{a}_{\omega_j} \star^j a_{\omega_2}
  \]
  denotes the result of truncating the asymptotic expansion up to some fixed index $J$, taken large enough in terms of $N$, and
  \[
    r_{\omega_1,\omega_2} \in \h^{(1- 2 \delta') J} S^{\tau}_{\delta', \delta''} + \h^{\infty} S^{-\infty}
  \]
  denotes the remainder.  By the triangle inequality for $\mathcal{N}$ and the bound $\sqrt{x + y} \ll \sqrt{x} + \sqrt{y}$ for $x,y \geq 0$, we reduce to verifying that each of the quantities
  \[
    \sup_{\omega_1 \in \Omega} \sum _{\omega_2 \in \Omega } \mathcal{N} ( b_{\omega_1,\omega_2})^{1/2}, \quad \sup_{\omega_1 \in \Omega} \sum _{\omega_2 \in \Omega } \mathcal{N} ( r_{\omega_1,\omega_2})^{1/2}
  \]
  is $\O(1)$.  Since $1 - 2 \delta' > 0$, the crude bound \eqref{eqn:crude-bound} and our choice of $J$ yield the strong estimate $\mathcal{N}(r_{\omega_1,\omega_2}) \ll \h^N$, which gives an acceptable contribution after summing over $\omega_2$.  On the other hand, we have $b \in S^\tau_{\delta ', \delta ''}$ and $\supp(b_{\omega_1,\omega_2}) \subseteq \mathcal{D}(\omega_1) \cap \mathcal{D}(\omega_2)$.  In particular, $b_{\omega_1,\omega_2}$ is localized at both $\omega_1$ and $\omega_2$.  We noted above that these conditions imply that $\mathcal{N}(b_{\omega_1,\omega_2}) \ll 1$ and that for each $\omega_1$, we have $b_{\omega_1,\omega_2} = 0$ for all $\omega_2$ outside some set of cardinality $\O(1)$.  The required bound follows.  The proof of Lemma \ref{lem:refined-symb-basic-opp-bound} is thus complete.
\end{proof}

We now resume the proof of Theorem \ref{thm:oper-assignm-comp-refined}, turning our attention to the estimate \eqref{eq:opp-of-refined-symb-into-Psi}.  Let $a \in S^\tau _{\delta ', \delta ''}$.  By Theorem \ref{thm:opp-S-m-underlined}, we have $\Opp(a) \in \underline{\Psi }^{-\infty}$.  We must verify that for all fixed $m,s \in \mathbb{Z}$ and $x_1,\dotsc,x_n \in \mathfrak{g}$, the operator norm of
\[
  [\pi(x_1),\dotsc,[\pi(x_n),\Opp(a)]] : \pi^s \rightarrow \pi^{s-m}
\]
is $\O(\h^{m - n \delta'})$.  We first reduce to the case $n=0$.  For symbols $x,y \in \underline{S}^\infty$, define the normalized star product commutator
\begin{equation*}
  \mathcal{C}(x,y) := x \star_{\h} y - y \star_{\h} x.
\end{equation*}
The star product asymptotics imply that for $x \in \mathfrak{g}$ and $y \in S^\tau_{\delta ', \delta ''}$, we have $\mathcal{C}(x,y) \in \h^{1 - \delta '} S^\tau_{\delta ', \delta ''} + \h^\infty S^{-\infty}$.  It follows that the iterated star product commutator
\begin{equation*}
  b := \h^{ -n + n \delta '} \mathcal{C}(x_1, \dotsc, \mathcal{C}(x_n, a))
\end{equation*}
lies in $S^\tau_{\delta ', \delta ''} + \h^\infty S^{-\infty}$.  Each fixed element $x \in \mathfrak{g}$ defines a linear function $\mathfrak{g}^\wedge \rightarrow i \mathbb{R}$, hence a symbol $x \in S^1_0$.  Recall from Lemma \ref{lem:polyn-symb} that $\pi(x) = \h^{-1} \Opp(x)$.  It follows then by iterated application of the composition formula \eqref{eqn:composition-formula-1} that $[\pi(x_1),\dotsc,[\pi(x_n),\Opp(a)]] \equiv \Opp(b) \mod{\h^\infty \Psi^{-\infty}}$.  Theorem \ref{thm:oper-assignm-comp-basic} gives $\Opp(\h^\infty S^{-\infty}) \subseteq \h^\infty \Psi^{-\infty}$.  It follows that the $n=0$ case of the required operator norm bound implies the general case.

It remains to show that $\Opp(a) : \pi^s \rightarrow \pi^{s-m}$ has operator norm $\ll \h^m$, i.e., that for all smooth vectors $v \in \pi$, we have
\begin{equation}\label{eqn:opp-a-pi-s-pi-s-minus-m}
  \h^{-m} \|\Opp(a) v\|_{\pi^{s-m}} \ll
  \|v\|_{\pi^s}.
\end{equation}

We consider first the special case $s=m \geq 0$, in which our task is to show that
\begin{equation}\label{eqn:special-case-s-equals-m-geq-0}
  \h^{-m} \|\Opp(a) v\| \ll \|v\|_{\pi^m}.
\end{equation}
Let us fix an element $z \in \mathfrak{g}$ with $z(\tau) \asymp 1$.  Then $z(\xi) \asymp 1$ for all $\xi$ in the support of $a$.  We may construct a symbol $q \in S^\tau_{\delta ', \delta ''}$ that is an approximate quotient for $a$ with respect to $z^{m}$ in the sense that
\[
  q \star_{\h} z^{m} \equiv a \mod{\h^N S^\tau_{\delta ', \delta ''} + \h^\infty S^{-\infty}},
\]
where $N$ is fixed large enough in terms of $m$.  To do so, we take for $q$ the series $\sum _{0 \leq j < J} \h^j q_j$, where the terms $q_j$ are the components of the formal solution:
\[
  q_0 := \frac{a}{z^{m}}, \quad q_1 := \frac{- a \star^1 q_0}{z^{m}}, \quad q_2 := \frac{- a \star^2 q_0 + a \star^1 q_1}{z^{m}}.
\]
and so on.  We verify readily -- by induction on $j$, using the quotient rule for derivatives and the fact that $z(\xi) \asymp 1$ for $\xi \in \supp(a)$ -- that $q_j \in \h^{- \delta ' j} S^\tau_{\delta ', \delta ''}$ and that $q$ has the claimed properties.  The operator norm for $\pi \rightarrow \pi$, hence also for $\pi^m \rightarrow \pi^0$, of any element of $\h^N S^\tau_{\delta ', \delta ''} + \h^\infty S^{-\infty}$ is $\O(\h^N)$, as follows from Theorem \ref{thm:oper-assignm-comp-basic} and (a weak form of) Lemma \ref{lem:refined-symb-basic-opp-bound}.  It will thus suffice to verify the following modification of \eqref{eqn:special-case-s-equals-m-geq-0}, obtained by replacing $\Opp(a)$ by its approximation $\Opp(q) \Opp(z^m) = \h^{m} \Opp(q) \pi(z)^m$:
\[
  \|\Opp(a) z^m v\| \ll \|v\|_{\pi^m}.
\]
By Lemma \ref{lem:refined-symb-basic-opp-bound}, we have $\|\Opp(a) z^m v\| \ll \|z^m v\|$.  By \eqref{eqn:Delta-s-positive-easy}, we have $\|z^m v\| \ll \|v\|_{\pi^m}$.  The proof of \eqref{eqn:special-case-s-equals-m-geq-0} is thus complete.

Turning to the general case of \eqref{eqn:opp-a-pi-s-pi-s-minus-m}, let us fix $k \in \mathbb{Z}_{\geq 0}$ large enough in terms of $s$ and $m$.  It will be enough then to verify the modification of \eqref{eqn:opp-a-pi-s-pi-s-minus-m} obtained by replacing $v$ with its image under the invertible operator $\Delta^k$.  The required estimate expands out to
\begin{equation*}
  \h^{2 m}
  \langle \Delta^{s-m} \Opp(a) \Delta^k v,
  \Opp(a) \Delta^k v \rangle
  \ll
  \|v\|_{\pi^{s + 2k}}^2.
\end{equation*}

We consider first the case $s-m \geq 0$.  We expand $\Delta$ as a sum of monomials, appeal to the skew-adjointness of the action of $\mathfrak{g}$ on $\pi$, and apply Cauchy--Schwarz.  We reduce in this way to verifying that for all fixed $x_1,\dotsc,x_{s-m} \in \{1 \} \cup \mathfrak{g}$,
\begin{equation*}
  \h^{-m}
  \|x_1 \dotsb x_{s-m} \Opp(a) \Delta^k v\|
  \ll
  \|v\|_{\pi^{s + 2k}}.
\end{equation*}
To that end, we assume our coordinates and basis $\mathcal{B}(\mathfrak{g})$ as in \S\ref{sec:spaces-operators} chosen compatibly, so that by Lemma \ref{lem:polyn-symb}, we have $\Delta = \Opp(p)$ with $p(\xi) = \h^2 + |\xi|^2$.  We then apply the composition formula to write
\begin{equation}\label{eqn:x1-xsmm-oppa-Delta-k}
  x_1 \dotsb x_{s-m} \Opp(a) \Delta^k
  \equiv \h^{-(s-m + 2k)} \Opp(x_1 \star_{\h} \dotsb \star_{\h}
  x_{s-m} \star_{\h} a \star_{\h} p^k),
\end{equation}
where $\equiv$ denotes congruence modulo $\h^\infty \Psi^{-\infty}$.  (Strictly speaking, the star product $\star_{\h}$ is not quite associative due to the cutoffs, so we should specify that this last expression is evaluated from left to right, say.  The formal expansion is associative, so the asymptotic expansion is independent of the order of evaluation.)  The remainder term in this congruence contributes acceptably, since the $\pi^{s+k} \rightarrow \pi^0$ operator norm of any element of $\h^\infty \Psi^{-\infty}$ is $\O(\h^\infty)$.  We likewise incur an acceptable error in replacing the iterated star product $x_1 \star_{\h} \dotsb \star_{\h} p^k$ with its approximation $b \in S_{\delta ', \delta ''}^{\tau}$ obtained by approximating each star product with its asymptotic expansion taken to order $J$, provided that $J$ is fixed sufficiently large in terms of $m, s, \delta '$.  Our task thereby reduces to verifying for all $b \in S^\tau_{\delta ', \delta ''}$ that $\h^{-(s+2k)} \|\Opp(b) v\| \ll \|v\|_{\pi^{s+2k}}$, which follows from the special case \eqref{eqn:special-case-s-equals-m-geq-0} treated above.

We turn to the case $s-m \leq 0$.  The argument is quite similar, so we will be brief.  Using the composition formula, we may replace $\Opp(a) \Delta^k$ with $\h^{-2 k} \Opp(b)$, where $b \in S^\tau_{\delta ', \delta ''}$ is a suitable truncation of the asymptotic expansion of $a \star_{\h} p^k$ with $p(\xi) = \h^2 + |\xi|^2$ as before.  We fix $z \in \mathfrak{g}$ as above.  We approximately divide $b$ on the left by $z^{m-s}$ and on the right by $z^{s+2k}$, yielding $q \in S^\tau_{\delta ', \delta ''}$ for which $\Opp(b)$ is approximated by $\Opp(z^{m-s}) \Opp(q) \Opp(z^{s+2k}) = \h^{m + 2 k} \pi(z^{m-s}) \Opp(q) \pi(z^{s+2k})$.  We thereby reduce to verifying that
\[
  \langle z^{m-s} \Delta^{s-m} z^{m-s} \Opp(q) z^{s + 2 k} v, \Opp(q) z^{s + 2 k} v \rangle \ll \|v\|_{\pi^{s+2k}}^2.
\]
By Lemma \ref{lem:Delta-xj-vs-Psi}, the operator $\pi(z^{m-s} \Delta^{s-m} z^{m-s})$ lies in $\Psi^0_0$, and in particular has $\pi \rightarrow \pi$ operator norm $\O(1)$.  By Cauchy--Schwarz, we reduce to verifying that $\|\Opp(q) z^{s + 2 k} v \| \ll \|v\|_{\pi^{s+2k}}$.  This last estimate follows (again) from Lemma \ref{lem:refined-symb-basic-opp-bound} and \eqref{eqn:Delta-s-positive-easy}.

The proof of \eqref{eq:opp-of-refined-symb-into-Psi} is now complete.  We may deduce the remaining assertion \eqref{eqn:composition-formula-2} from \eqref{eqn:composition-formula-1}, \eqref{eqn:star-j-asymp-expn-new} and \eqref{eq:opp-of-refined-symb-into-Psi} by choosing $J$ large enough that $(1 - 2 \delta') J \geq N$.
\end{proof}

\subsection{Trace estimates}
With notation and assumptions as in \S\ref{sec:kirillov-formula}, we recall the rescaled asymptotic form of the Kirillov formula given in \cite[(12.4)]{nelson-venkatesh-1}.  Throughout this section, the nice cutoff $\chi$ is fixed.
\begin{theorem}\label{thm:asymptotic-kirillov}
  Assume that $\pi$ is tempered.  Fix $\delta \in [0,1)$ and $J, N \in \mathbb{Z}_{\geq 0}$.  For all $a \in S^{-\infty}_\delta$, we have
  \[
    \h^d \trace(\Opp(a)) = \sum _{0 \leq j < J} \h^j \int _{\h \mathcal{O}_\pi } \mathcal{D}_j a + \O (\h^{(1-\delta) J} \langle \h \lambda_\pi \rangle^{-N} ),
  \]
  where
  \begin{itemize}
  \item $d$ is as in \S\ref{sec:kirillov-formula},
  \item $\mathcal{D}_j$ is a fixed constant coefficient differential operator on $\mathfrak{g}^\wedge$ of pure degree $j$,
  \item $\int_{\h \mathcal{O}_\pi}$ denotes the integral over the rescaled coadjoint multiorbit $\h \mathcal{O}_\pi$ with respect to its symplectic measure, normalized as in \cite[\S6.1]{nelson-venkatesh-1}, and
  \item $\langle \h \lambda_\pi \rangle = (1 + |\h \lambda_\pi|^2)^{1/2} \geq 1$ denotes the norm of the rescaled infinitesimal character, as in \cite[\S9.8]{nelson-venkatesh-1}.
  \end{itemize}
\end{theorem}
As a consequence, we obtain strong trace estimates when the support of the symbol is disjoint from the rescaled coadjoint multiorbit.
\begin{corollary}\label{cor:trace-disjoint-supports}
  Assume that $\pi$ is tempered.  Let $a \in S^{-\infty}_{\delta}$, with $0 \leq \delta < 1$ fixed.  If $\supp(a) \cap \h \mathcal{O}_\pi = \emptyset$, then
  \[\trace(\Opp(a)) \ll \h^\infty \langle \h \lambda_\pi
    \rangle^{-\infty} \ll \h^\infty.\]
\end{corollary}
Since the refined symbol classes $S^\tau_{\delta', \delta''}$ are contained in $S^{-\infty}_{\delta ''}$, Theorem \ref{thm:asymptotic-kirillov} and Corollary \ref{cor:trace-disjoint-supports} apply to elements of those classes.  It is crucial here that these results apply to $\delta < 1$ (rather than merely to $\delta < 1/2$, as in Theorem \ref{thm:oper-assignm-comp-basic}).

The following result, given in \cite[Thm 9 (iii)]{nelson-venkatesh-1}, is particularly useful for ``clean-up.''  It does not require any temperedness assumption on $\pi$ (although we will only apply it here in tempered cases).
\begin{theorem}\label{thm:trace-class-Psi-neg-N}
  Fix $N \in \mathbb{Z}_{\geq 0}$ large enough in terms of $\dim(\mathfrak{g})$.  Then each operator $T$ on $\pi$ with $\|T\|_{\pi^{0} \rightarrow \pi^N} \ll 1$ is trace class, with trace norm \[\|T\|_1 \ll \langle \lambda_{\pi} \rangle^{d-N} \leq 1.\]
\end{theorem}

\section{Stability}\label{sec:stability}
Let $(G,H)$ be a GGP pair over a field $F$ of characteristic zero.

We retain the notation of \S\ref{sec:repr-theor-prel}.  We denote the restriction map of linear duals $\mathfrak{g}^* \rightarrow \mathfrak{h}^*$ by $\xi \mapsto \xi_H$.  When $F$ is a local field, we denote in the same way the restriction map of Pontryagin duals $\mathfrak{g}^\wedge \rightarrow \mathfrak{h}^\wedge$.  We work primarily with Pontryagin duals over a local field $F$.  It will occasionally be convenient to pass to an algebraic closure of $F$, which need not be a local field; in that setting, it becomes more natural to work with linear duals.

\subsection{Definition and characterization}\label{sec:stability-defn}

\begin{definition}\label{defn:stable-locus-inf-chars}
  Following \cite[\S14]{nelson-venkatesh-1}, we say that $(\lambda,\mu) \in [\mathfrak{g}^*] \times [\mathfrak{h}^*]$ is \emph{stable} if
  \[\ev(\lambda) \cap \ev(\mu) = \emptyset,\]
  where $\ev(\cdot)$ denotes the eigenvalue multiset defined in \S\ref{sec:eigenvalue-multisets}.
\end{definition}
When $F$ is local, the Pontryagin duals $\mathfrak{g}^\wedge$ and $\mathfrak{h}^\wedge$ may be compatibly identified with $\mathfrak{g}^*$ and $\mathfrak{h}^*$, respectively, after choosing a nontrivial character $F \rightarrow \U(1)$.  ``Compatibly'' means that the identifications are equivariant and define a commuting square
\begin{equation*}
  \begin{CD}         
    \mathfrak{g}^\wedge  @>\cong>> \mathfrak{g}^*\\
    @VVV  @VVV \\
    \mathfrak{h}^\wedge @>>\cong> \mathfrak{h}^*.
  \end{CD}
\end{equation*}
In this way, Definition \ref{defn:stable-locus-inf-chars} applies also to $(\lambda,\mu) \in [\mathfrak{g}^\wedge] \times [\mathfrak{h}^\wedge]$.

We may reformulate the local assumptions at $\mathfrak{q}$ in the definition of $\mathcal{F}_T$ (\S\ref{sec:main-results}), or equivalently, those concerning the archimedean Satake parameters of $\pi$ and $\sigma$ in Theorem \ref{thm:construct-test-function}, using the following equivalence, implicit in \cite[\S15]{nelson-venkatesh-1}.
\begin{lemma}\label{lem:stability-inf-chars-satake-params}
  Assume that $F$ is an archimedean local field.  Let $T$ and $\h$ be positive reals with $T \ggg 1$ and $\h \asymp T^{-[F:\mathbb{R}]}$.  Let $(\pi,\sigma)$ be tempered irreducible unitary representations of a fixed GGP pair $(G,H)$ over $F$.  Let $\lambda_\pi \in [\mathfrak{g}^\wedge]$ and $\lambda_\sigma \in [\mathfrak{h}^\wedge]$ denote the infinitesimal characters (\S\ref{sec:infin-char-langl}).  Let $\lambda_{\pi,i}$ and $\lambda_{\sigma,j}$ denote the archimedean Satake parameters (\S\ref{sec:satake-param-arch}).  The following conditions are equivalent:
  \begin{enumerate}[(i)]
  \item $\max ( \{|\lambda_{\pi,i}|_{F}\} \cup \{|\lambda_{\sigma,j}|_{F}\}) \asymp T$ and $| \lambda_{\pi,i} - \lambda_{\sigma,j}|_F \gg T$ for all $i,j$.
  \item $(\h \lambda_\pi, \h \lambda_\sigma)$ lies in some fixed compact subset of $\{\text{stable } (\lambda,\mu) \in [\mathfrak{g}^\wedge] \times [\mathfrak{h}^\wedge] \}$.
  \end{enumerate}
\end{lemma}

We next recall a standard definition from geometric invariant theory.
\begin{definition}\label{defn:stable-elements}
  We write $\mathfrak{g}_{\stab}^*$ for \index{Lie algebra!stable subset $\mathfrak{g}^*_{\stab}$} the set of all $\xi \in \mathfrak{g}^*$ that are \emph{stable} with respect to $H$, i.e., have finite $H$-stabilizer and (Zariski) closed $H$-orbit.
\end{definition}
We make the analogous definition of $\mathfrak{g}^\wedge_{\stab} \subseteq \mathfrak{g}^\wedge$ in the case of a local field.

We summarize some results from \cite[\S14, \S17.3]{nelson-venkatesh-1} relating the above definitions.
\begin{lemma}\label{lem:stability-summary}
  The set $\mathfrak{g}_{\stab}^*$ consists of all $\xi \in \mathfrak{g}^*$ with $\ev(\xi) \cap \ev(\xi_H) = \emptyset$, hence coincides with the preimage of $\{\text{stable } (\lambda,\mu) \in [\mathfrak{g}^*] \times [\mathfrak{h}^*] \}$ under the map $\xi \mapsto ([\xi],[\xi_H])$.  In particular, it is a dense open subset of $\mathfrak{g}^*$.
\end{lemma}

\subsection{Basic consequences}\label{sec:cons-stab}
We recall from \cite[\S14, \S17.3]{nelson-venkatesh-1} that the map
\begin{equation*}
\mathfrak{g}^*_{\stab} \rightarrow \{\text{stable } (\lambda,\mu)\}
\end{equation*}
is a principal $H$-bundle over its image.  For general $F$, this bundle is locally trivial in the {\'e}tale topology.  When $F$ is a local field, this bundle is thus locally trivial with respect to the usual metric topology.  In particular, if $(\lambda,\mu)$ is stable, then the fiber
\begin{equation}\label{eqn:O-lambda-mu}
  \mathcal{O}^{\lambda,\mu} :=
  \{\xi \in \mathfrak{g}^* :
  [\xi] = \lambda, [\xi_H] = \mu 
  \}
\end{equation}
is either empty, or is an \emph{$H$-torsor}, i.e., a closed $H$-invariant subset of $\mathfrak{g}^*$ on which $H$ acts simply-transitively.\footnote{ Some readers have told us that they found the ``argument'' of \cite[\S17.3]{nelson-venkatesh-1} insufficiently detailed, so we elaborate upon the key step here.  Let $F$ be a field of characteristic zero with algebraic closure $\bar{F}$.  Let $G$ be an algebraic group that acts on a variety $X$, with $G$, $X$ and the action all defined over $F$.  Suppose that the action of $G(\bar{F})$ on $X(\bar{F})$ is simply-transitive and that $X(F)$ is nonempty.  Then the action of $G(F)$ on $X(F)$ is likewise simply-transitive.  In verifying this, the main step is to check that if $g \in G(\bar{F})$ and $x_0 \in X(F)$ satisfy $g x_0 \in X(F)$, then in fact $g \in G(F)$.  It is enough to check that $g$ is fixed by every element of $\Gal(\bar{F}/F)$.  Since the action of $G(\bar{F})$ on $X(\bar{F})$ is free and defined over $F$, it is enough to check the same for $g x_0$, which follows from our hypothesis that $g x_0 \in X(F)$.  }

Suppose now that $F$ is local.  We analogously define $\mathcal{O}^{\lambda,\mu} \subseteq \mathfrak{g}^\wedge$ for $(\lambda,\mu) \in [\mathfrak{g}^\wedge] \times [\mathfrak{h}^\wedge]$.  Let us suppose that $H$ has been equipped with a Haar measure.  We then equip $\mathcal{O}^{\lambda,\mu}$ with the measure $d \Haar_H$ given by pushforward of that Haar measure via the orbit map, thus
\[
  \int _{\mathcal{O}^{\lambda,\mu}} a := \int _{\xi \in \mathcal{O}^{\lambda,\mu}} a(\xi) \, d \Haar_{H}(\xi) := \int _{s \in H} a(s \cdot \tau) \, d s.
\]
We recall from \cite[\S17.3.2]{nelson-venkatesh-1} (written in the archimedean case, but the argument is general) that integration defines a continuous map
\begin{equation*}
  \{\text{stable } (\lambda,\mu) \in [\mathfrak{g}^\wedge] \times [\mathfrak{h}^\wedge] \} \times \mathcal{S}(\mathfrak{g}^\wedge) \rightarrow \mathbb{C},
\end{equation*}
\begin{equation*}
  (\lambda,\mu,a) \mapsto \int_{\mathcal{O}^{\lambda,\mu}} a,
\end{equation*}
where $\mathcal{S}(\mathfrak{g}^\wedge)$ denotes the Schwartz space.

\subsection{Relative coadjoint orbits}\label{sec:relat-coadj-orbits}
Suppose now that $F$ is an archimedean local field.  Let $\pi$ and $\sigma$ be tempered irreducible unitary representations of $G$ and $H$, respectively.  Let $\mathcal{O}_\pi \subseteq \mathfrak{g}^\wedge_{\reg}$ and $\mathcal{O}_\sigma \subseteq \mathfrak{h}^\wedge_{\reg}$ denote their coadjoint multiorbits, as described in Theorem \ref{thm:kirillov-formula}.  We set \index{representations!relative coadjoint orbit $\mathcal{O}_{\pi,\sigma}$}
\begin{equation*}
  \mathcal{O}_{\pi,\sigma}
  :=
  \{\xi \in \mathcal{O}_\pi : \xi_H \in \mathcal{O}_\sigma \}.
\end{equation*}
We say that $(\pi,\sigma)$ is \emph{orbit-distinguished} if $\mathcal{O}_{\pi,\sigma}$ is nonempty.
\begin{remark}
  As remarked in \cite[\S22.2]{nelson-venkatesh-1}, it is likely provable that distinction in the traditional sense is asymptotically equivalent to orbit-distinction, e.g., in the setting of Theorem \ref{thm:construct-test-function} with $T$ large.  (That orbit-distinction implies distinction follows from Theorem \ref{thm:rel-char-asymptotics}; the converse should then follow by combining Theorem \ref{thm:rel-char-asymptotics} with strong multiplicity one for archimedean $L$-packets \cite{2015arXiv150601452B, 2020arXiv200913947L}.)  Since it is straightforward to check orbit-distinction in examples, we are content to leave it in the hypotheses of our main results.  We note that for complex groups, orbit-distinction always holds, while for compact unitary groups, it is an asymptotic form of the interlacing condition describing the classical branching laws.
\end{remark}

If $(\lambda_\pi, \lambda_\sigma)$ is stable and $\mathcal{O}_{\pi,\sigma}$ is nonempty, then it follows from the discussion of \S\ref{sec:cons-stab} that $\mathcal{O}_{\pi,\sigma} = \mathcal{O}^{\lambda_\pi, \lambda_\sigma}$ is an $H$-torsor, and so we may integrate over it as above; otherwise, we define any integral over $\mathcal{O}_{\pi,\sigma}$ to be zero.  We define in the same way integrals over the rescalings $\h \mathcal{O}_{\pi,\sigma} = \mathcal{O}^{\h \lambda_{\pi}, \h \lambda_{\sigma}}$ of such sets.

\section{Relative character
  asymptotics}\label{sec:relat-char-asympt}
Here we recall the main results of \cite[\S19]{nelson-venkatesh-1}.

\subsection{Setup}
Let $(G,H)$ be a fixed GGP pair over an archimedean local field $F$.  We assume given a fixed Haar measure on $H$.  Let $\pi$ and $\sigma$ be tempered irreducible unitary representations of $G$ and $H$, respectively.  Let $\sigma^\vee$ denote the contragredient.  We obtain the (tempered irreducible unitary) representation $\pi \otimes \sigma^\vee$ of
\[
  M := G \times H.
\]

\subsection{Formal definitions
  and identities}\label{sec:relat-trac-form}
We write $\mathcal{B}(\pi \otimes \sigma^\vee), \mathcal{B}(\pi), \mathcal{B}(\sigma)$ for orthonormal bases consisting of eigenvectors for $\Delta_M, \Delta_G, \Delta_H$, with $\Delta$ as in \S\ref{sec:spaces-operators}.  Such eigenvectors are in particular smooth.

For an operator $T$ on $\pi \otimes \sigma^\vee$, we define \index{matrix coefficient integrals!hermitian, $\mathcal{H}$}
\[
  \mathcal{H}(T) := \sum _{v \in \mathcal{B}(\pi \otimes \sigma^\vee)} \int _{h \in H} \langle h T v, v \rangle \, d h
\]
provided that the sum converges absolutely.

Working formally for the moment, if $T_1$ (resp. $T_2$) is an operator on $\pi$ (resp. $\sigma$), then we may naturally define operators $T_2^\vee$ on $\sigma^\vee$ and $T_1 \otimes T_2^\vee$ on $\pi \otimes \sigma^\vee$, and we have
\[
  \mathcal{H}(T_1 \otimes T_2^\vee) = \sum _{v \in \mathcal{B}(\pi)} \sum _{u \in \mathcal{B}(\sigma)} \int _{h \in H} \langle h T_1 v, v \rangle \langle T_2 u, h u \rangle \, d h.
\]
Moreover, writing $T_j^*$ for the adjoint operator,
\begin{align}\label{eq:H-T1T1s-T2T2s}
  \mathcal{H}(T_1 T_1^* \otimes (T_2 T_2^*)^\vee )
  &=
    \sum _{v \in \mathcal{B}(\pi)}
    \sum _{u \in \mathcal{B}(\sigma)}
    \int _{h \in H}
    \langle h T_1 v,  T_1 v \rangle
    \langle T_2 u, h T_2 u \rangle
    \, d h
  \\
  \nonumber
  &=
    \sum _{v \in \mathcal{B}(\pi)}
    \sum _{u \in \mathcal{B}(\sigma)}
    \mathcal{Q} (T_1 v \otimes T_2 u),
\end{align}
where the quadratic form $\mathcal{Q}$ is defined as in \S\ref{sec:local-dist-matr}.  In particular, writing $1$ for the identity operator,
\begin{equation}\label{eq:H-T1-T1star-otimes-1}
  \mathcal{H}(T_1 T_1^* \otimes 1)
  =
  \sum _{v \in \mathcal{B}(\pi)}
  \sum _{u \in \mathcal{B}(\sigma)}
  \mathcal{Q} (T_1 v \otimes u).
\end{equation}

\subsection{Convergence criteria and \emph{a priori}
  bounds}\label{sec:conv-crit-empha}
In this section, for an operator $T$ on one of the representations $\pi \otimes \sigma^\vee, \pi$ or $\sigma$ and a nonnegative integer $N$, we write $\nu_N(T) \in [0,\infty]$ for the operator norm of $\Delta^N T \Delta^N$, where $\Delta$ is respectively $\Delta_M, \Delta_G$ or $\Delta_H$.

Suppose for instance that $T$ is an operator on $\pi$.  Then $\nu_N(T)$ is the operator norm of the map $\pi^{-N} \rightarrow \pi^{N}$; in particular, $\nu_N(T)$ is finite and $\O(1)$ whenever $T$ belongs to the operator class $\Psi_\delta^{-2 N}$ (for any fixed $\delta$).
\begin{lemma}\label{lem:convergence-relative-character}
  Fix a natural number $N$ large enough in terms of $(G,H)$.
  \begin{enumerate}[(i)]
  \item Let $T$ be an operator on $\pi \otimes \sigma^\vee$ for which $\nu_N(T) = \O(1)$.  Then $\mathcal{H}(T)$ is defined and $\O(1)$.
  \item Let $T$ be an operator on $\pi$ for which $\nu_N(T) = \O(1)$.  Then $\mathcal{H}(T \otimes 1)$ is defined and $\O(1)$.
  \item
    \label{item:convergence-relative-character-3}
    Let $T_1$ and $T_2$ be operators on $\pi$ and $\sigma$, respectively, for which $\nu_N(T_1), \nu_N(T_2) = \O(1)$.  Then $\mathcal{H}(T_1 \otimes T_2^\vee)$ is defined and $\O(1)$; moreover, the formal identities stated in \S\ref{sec:relat-trac-form} are valid.
  \item The quadratic form $\mathcal{Q}$ extends continuously to the tensor product of Sobolev spaces $\pi^N \otimes \sigma^N$.
  \item
    \label{item:convergence-relative-character-4}
    The identities \eqref{eq:H-T1T1s-T2T2s} and \eqref{eq:H-T1-T1star-otimes-1} hold more generally taking for $\mathcal{B}(\pi)$ and $\mathcal{B}(\sigma)$ any orthonormal bases, provided that for \eqref{eq:H-T1-T1star-otimes-1}, the basis $\mathcal{B}(\sigma)$ is contained in $\sigma^N$.
  \end{enumerate}
\end{lemma}
\begin{proof}
  For (i), see just after \cite[(19.14)]{nelson-venkatesh-1}.  For (ii), see just before \cite[(18.5)]{nelson-venkatesh-1}.  The translation to the present formulation is as in the proof of Theorem \ref{thm:nv-star-prod-asymp-h-dependent}.

  For (iii), we observe first, by writing $\Delta_M$ in terms of $\Delta_G$ and $\Delta_H$ for suitably chosen bases, that $\nu_N(T_1 \otimes T_2^\vee) = \O(1)$; we then apply (i) to obtain the first assertion.  To complete the proof of (iii), it remains to verify the formal identities.  We start with \eqref{eq:H-T1-T1star-otimes-1}.  The RHS of \eqref{eq:H-T1-T1star-otimes-1} may be written
  \begin{equation}\label{eq:sum-_v_1-in}
    \sum _{v_1 \in \mathcal{B}(\pi)} \sum _{u \in \mathcal{B}(\sigma)} \int _{h \in H} \langle h T_1 v_1, T_1 v_1 \rangle \langle u, h u \rangle \, d h.
  \end{equation}
  We expand the second copy of $T_1 v_1$ as $\overline{T_1 v_1} = \sum _{v_2 \in \mathcal{B}(\pi)} \langle v_2, T_1 v_1 \rangle \overline{v _2}$ and rewrite the inner product as $\langle T_1^* v_2, v_1\rangle$.  By arguments as in \cite[\S18 and App. A]{nelson-venkatesh-1} (which boil down to the uniform trace class property of $\Delta^{-N}$ given, e.g., in Theorem \ref{thm:trace-class-Psi-neg-N}), we have
  \begin{equation}\label{eq:sum-_v_2-in}
    \sum _{v_2 \in \mathcal{B}(\pi)} \int_{h \in H} |\langle T_1^* v_2, v_1 \rangle \langle h T_1 v_1, v_2 \rangle \langle u, h u \rangle| \, d h < \infty.
  \end{equation}
  For the convenience of the reader, we outline the argument.  We observe first from \cite[Lemma A.2]{nelson-venkatesh-1} that for some $s \in \mathbb{Z}_{\geq 0}$ depending only upon $(G,H)$, we have
  \begin{equation*}
    \left\langle h T_1 v_1, v_2 \right\rangle
    \ll
    \Xi_G(h) \lVert T v_1 \rVert _{\pi ^s }
    \lVert v _2  \rVert _{\pi ^s },
    \quad
    \left\langle u, h u \right\rangle
    \ll
    \Xi_H(h) \lVert u \rVert _{\sigma ^s}.
  \end{equation*}
  Since $\int_H \Xi_G|_{H} \cdot \Xi_H < \infty$ (see \cite[(18.1)]{nelson-venkatesh-1}), our task reduces to checking that
  \begin{equation*}
    \sum _{v _2 \in \mathcal{B}(\pi)}
    \left\lvert
      \left\langle T _1 ^\ast v _2, v _1  \right\rangle
    \right\rvert
    \lVert v _2  \rVert _{\pi _s} < \infty.
  \end{equation*}
  By Cauchy--Schwarz, we have $\lvert \left\lvert T _1 ^\ast v _2, v _1 \right\rvert \rvert \leq \nu_N(T_1) \lVert \Delta^{-N} v_2 \rVert \lVert \Delta^{-N} v_1 \rVert$, so our task reduces to checking that the sum $\sum _{v \in \mathcal{B}(\pi)} \lVert \Delta^{-N} v_2 \rVert \, \lVert v_2 \rVert_{\pi^s}$ converges.  Writing $\lambda$ for the $\Delta$-eigenvalue of $v_2$, the summand is $\lambda^{s/2 - N}$, the sum of which converges for $N$ large enough (see \cite[Lemma A.3]{nelson-venkatesh-1}).

  The absolute convergence \eqref{eq:sum-_v_2-in} permits us to swap the sum over $v_2$ with the integral over $h$, giving the following rearrangement of \eqref{eq:sum-_v_1-in}:
  \[
    \sum _{v_1 \in \mathcal{B}(\pi)} \sum _{u \in \mathcal{B}(\sigma)} \left( \sum _{v_2 \in \mathcal{B}(\pi)} \langle T_1^* v_2, v_1 \rangle \int _{h \in H} \langle h T_1 v_1, v_2\rangle \langle u, h u \rangle \, d h \right).
  \]
  Again applying arguments as in \cite[\S18 and App. A]{nelson-venkatesh-1}, we see that
  \[
    \sum _{v_1 \in \mathcal{B}(\pi)} \sum _{u \in \mathcal{B}(\sigma)} \sum _{v_2 \in \mathcal{B}(\pi)} \left\lvert \langle T_1^* v_2, v_1 \rangle \int _{h \in H} \langle h T_1 v_1, v_2\rangle \langle u, h u \rangle \, d h \right\rvert < \infty.
  \]
  We may thus rearrange the outer summations, giving
  \[
    \sum _{v_2 \in \mathcal{B}(\pi)} \sum _{u \in \mathcal{B}(\sigma)} \left( \sum _{v_1 \in \mathcal{B}(\pi)} \langle T_1^* v_2, v_1 \rangle \int _{h \in H} \langle h T_1 v_1, v_2\rangle \langle u, h u \rangle \, d h \right).
  \]
  By reversing the reasoning applied in the previous step, we see that
  \[
    \sum _{v_1 \in \mathcal{B}(\pi)} \langle T_1^* v_2, v_1 \rangle \int _{h \in H} \langle h T_1 v_1, v_2\rangle \langle u, h u \rangle \, d h = \int _{h \in H} \langle h T_1 T_1^* v_2, v_2\rangle \langle u, h u \rangle \, d h,
  \]
  which completes the proof.  The proof of \eqref{eq:H-T1T1s-T2T2s} is similar, using, e.g., that
  \[
    \sum _{v_1, v_2 \in \mathcal{B}(\pi)} \sum _{u_1, u_2 \in \mathcal{B}(\sigma)} \left\lvert \int _{h \in H} \langle h T_1 v_1, T_1 v_2 \rangle \langle T_2 u_2, h T_2 u_1 \rangle \right\rvert < \infty.
  \]

  For (iv), combine \cite[A.2.2]{nelson-venkatesh-1} and \cite[(18.1)]{nelson-venkatesh-1}.  The proof of (v) is similar to that of (iii), cf.\  \cite[Lem 10.1]{MR855239} or \cite[Appendix 4]{MR2154720}.
\end{proof}

\subsection{Stability in the product setting}
We write $\mathfrak{m}$ for the Lie algebra of $M$.  We say that an element of $\mathfrak{m}^\wedge$ is \emph{stable} if under the diagonal action of $H$, it has finite stabilizer and closed orbit.  We note that $\xi \in \mathfrak{g}^\wedge$ is stable in the sense of \S\ref{sec:stability-defn} if and only if the element $(\xi,-\xi_H) \in \mathfrak{m}^\wedge$ is stable in the present sense (see \cite[\S19.4]{nelson-venkatesh-1} for further discussion).

\subsection{Main estimates}\label{sec:main-estimates}
\begin{theorem}\label{thm:rel-char-asymptotics}
  Let $0 \leq \delta < 1/2$ be fixed.  Let $\pi$ and $\sigma$ be tempered irreducible unitary representations of $G$ and $H$, as above.
  \begin{enumerate}[(i)]
  \item Let $a \in S^{-\infty}_\delta(\mathfrak{m}^\wedge)$ be supported in some fixed compact collection of stable elements of $\mathfrak{m}^\wedge$.  Set $A := \Opp_{\h}(a:\pi \otimes \sigma^\vee )$.  Then for each fixed $J \geq 0$, we have
    \begin{equation}\label{eqn:relative-character-asymptotics}
      \mathcal{H}(A)
      =
      \sum _{0 \leq j < J}
      \h^j
      \int _{\xi \in \h \mathcal{O}_{\pi,\sigma}}
      \mathcal{D}_j a(\xi, -\xi_H)
      \,
      d \Haar_{H}(\xi)
      + \O(\h^{(1-2 \delta) J}),
    \end{equation}
    where the $\mathcal{D}_j$ are fixed differential operators on $\mathfrak{m}^\wedge$ with the following properties:
    \begin{itemize}
    \item $\mathcal{D}_0 a = a$ for all smooth $a : \mathfrak{m}^\wedge \rightarrow \mathbb{C}$.
    \item $\mathcal{D}_j$ has order $\leq 2 j$ and homogeneous degree $j$: $\mathcal{D}_j (a_{\h}) = \h^j (\mathcal{D}_j a)_{\h}$.
    \end{itemize}
  \item Let $a \in S^{-\infty}_\delta(\mathfrak{g}^\wedge)$ be supported in some fixed compact collection of stable elements of $\mathfrak{g}^\wedge$.  Set $A := \Opp_{\h}(a:\pi )$, so that $A \otimes 1$ is an operator on $\pi \otimes \sigma^\vee$.  Then
    \begin{equation}\label{eqn:relative-character-asymptotics-2}
      \mathcal{H}(A \otimes 1)
      =
      \sum _{0 \leq j < J}
      \h^j
      \int _{\xi \in \h \mathcal{O}_{\pi,\sigma}}
      \mathcal{D}_j a(\xi)
      \,
      d \Haar_{H}(\xi)
      + \O(\h^{(1-2 \delta) J}),
    \end{equation}
    where the $\mathcal{D}_j$ are fixed differential operators on $\mathfrak{g}^\wedge$ with properties analogous to those stated in (i).
  \end{enumerate}
\end{theorem}
\begin{proof}
  These are \cite[Theorems 19.4 and 19.3]{nelson-venkatesh-1}, respectively.
\end{proof}

As explained in \cite[\S19.3]{nelson-venkatesh-1}, such estimates apply readily to compositions: one just substitutes the composition formula \eqref{eqn:composition-formula-2} into \eqref{eqn:relative-character-asymptotics} or \eqref{eqn:relative-character-asymptotics-2}, using Lemma \ref{lem:convergence-relative-character} as an \emph{a priori} bound to clean up the remainder terms.

\section{Lie-theoretic interludes}
Let $F$ be a local field of characteristic zero.  In what follows, ``manifold'' means ``analytic manifold over $F$'' in the sense recalled in \S\ref{sec:lie-groups}.  Here we record some miscellaneous results that will be invoked throughout the remainder of the paper.

\subsection{Inverse function theorem}\label{sec:inverse-function-theorem}
Let $f : X \rightarrow X'$ be an analytic map of manifolds.  The inverse function theorem \cite[II.III.10]{MR2179691} implies that if the derivative $(d f)_x$ of $f$ at a point $x \in X$ is injective (resp. surjective), then there are local coordinates near $x$ with respect to which $f$ is given by an injective (resp. surjective) linear map between vector spaces.  In particular:
\begin{itemize}
\item If $(d f)_x$ is injective, then $f$ is bi-Lipschitz in some neighborhood of $x$.  To make ``bi-Lipschitz'' precise, we might suppose that $X$ (resp. $X'$) is realized as a submanifold of some normed vector space $V$ (resp. $V'$).
\item If $(d f)_x$ is surjective, then $f$ is open in some neighborhood of $x$.
\end{itemize}

These observations hold uniformly for certain families of such maps.  Let $S$ be a manifold, let $V$ and $V'$ be normed vector spaces, and let
\begin{equation*}
  \rho : V \times S \rightarrow S,
  \quad 
  \rho' : V' \times S \rightarrow S
\end{equation*}
denote the projection maps.  Let $X$ (resp. $X'$) be a submanifold of $V \times S$ (resp. $V' \times S$) such that the restricted projection $\rho|_X$ (resp. $\rho '|_{X'}$) is submersive, i.e., has everywhere surjective derivative.  For each $s \in S$, the fiber $\rho^{-1}(s)$ (resp. $(\rho')^{-1}(s)$) is then a submanifold which may be written $X_s \times \{s\}$ (resp. $X_s' \times \{s\}$) for some submanifold $X_s \subseteq V$ (resp. $X_s' \subseteq V'$).  Let $f : X \rightarrow X'$ be an analytic map such that $\rho'(f(x)) = \rho(x)$ for all $x \in X$.  Then the maps $f_s : X_s \rightarrow X_s'$ characterized by the identities $f(x,s) = (f_s(x),s)$ are analytic.  We may understand $(f_s)_{s \in S}$ informally as an analytic family of analytic maps between submanifolds of normed vector spaces.  We then have the following uniform form of the inverse function theorem.
\begin{lemma}\label{lem:sub-gln-2:injective-surjective-neighborhood-criteria}
  Let $(x_0,s_0) \in X$.
  \begin{enumerate}[(i)]
  \item If $(d f_{s_0})_{x_0}$ is injective, then there is a neighborhood $(x_0,s_0) \in U \subseteq X$ and $C > c > 0$ so that for all $(x_1,s), (x_2,s) \in U$, we have
    \begin{equation*}
      c |x_1 - x_2| \leq |f_s(x_1) - f_s(x_2)| \leq C |x_1 - x_2|.
    \end{equation*}
  \item If $(d f_{s_0})_{x_0}$ is surjective, then for each neighborhood $(x_0,s_0) \in U \subseteq X$, there is a neighborhood $f(x_0,s_0) \in U' \subseteq X'$ so that $f(U) \supseteq U'$.
  \end{enumerate}
\end{lemma}
\begin{proof}
  The conclusion is local on $S$, so after possibly shrinking $S$, we may assume that it is realized as a submanifold of a normed vector space.  Using the submersivity of $\rho|_{X}$ and $\rho '|_{X'}$, we verify readily that if the fiberwise derivative $(d f_{s_0})_{x_0}$ is injective (resp. surjective), then so is the derivative $(d f)_{(x_0,s_0)}$ taken on the ``entire'' space.  In either case, the inverse function theorem applies to $f$, giving that it is bi-Lipschitz (resp. open) in a neighborhood of $(x_0,s_0)$.  Assertion (ii) is then immediate, while assertion (i) follows upon specializing the bi-Lipschitz property to pairs of elements of $X$ having the same projection to $S$.
\end{proof}

These lemmas have analogues obtained by replacing the basepoints by compact sets.  For example, given a compact subset $E$ of $X$ such that $(d f_{s_0})_{x_0}$ is injective (resp. surjective) for all $(x_0,s_0) \in E$, we can find a neighborhood $U$ of $E$ so that the respective conclusions hold.  We can also streamline the statements by employing asymptotic notation and terminology:
\begin{lemma}\label{lem:FTOC-families-nonstandard}
  Let $(f_s)_{s \in S}$ be as above, with $(S,V,V',X,X',f)$ fixed.  Let $E$ be a fixed compact subset of $X$.
  \begin{enumerate}[(i)]
  \item \label{itm:sub-gln-2:ftoc-inj} Assume that for each $(x,s) \in E$, the map $(d f_s)_x$ is injective.  Then there is a fixed neighborhood $E \subseteq U \subseteq X$ so that for all $(x_1,s), (x_2,s) \in U$, we have
    \begin{equation}\label{eq:f_sx_1-f_sx_2-asymp-x_1-x_2.--7}
      |f_s(x_1) - f_s(x_2)| \asymp |x_1 - x_2|.
    \end{equation}
    In particular, the above estimate holds for all $(x_1,s), (x_2,s) \in X$ for which there exists $(x,s) \in E$ with $x_1, x_2 \simeq x$.
  \item \label{itm:sub-gln-2:ftoc-surj} Assume that for each all $(x,s) \in E$, the map $(d f_s)_x$ is surjective.  Then for each fixed neighborhood $E \subseteq U \subseteq X$, there is a fixed neighborhood $f(E) \subseteq U' \subseteq X'$ so that $f(U) \supseteq U'$.  In particular, for each $(x',s) \in X'$ for which there is some $(x,s) \in E$ with $(x',s) \simeq f(x,s)$ (quantified via first components), there exists $(x, s) \in X$ with $f(x,s) = (x',s)$ that is contained in each fixed neighborhood of $E$.
  \end{enumerate}
\end{lemma}
\begin{proof}
  The initial statements (up to ``in particular'') follow immediately from the transfer axiom (\S\ref{sec:some-axioms}) applied to Lemma \ref{lem:sub-gln-2:injective-surjective-neighborhood-criteria}.  We verify the subsequent statements:
  \begin{enumerate}[(i)]
  \item Suppose $(x_1,s), (x_2,s) \in X$ satisfy $x_1,x_2 \simeq x$ for some $(x, s) \in E$.  Let $U$ be a fixed neighborhood of $E$ on which the estimate \eqref{eq:f_sx_1-f_sx_2-asymp-x_1-x_2.--7} holds. Since $E$ is compact, we have $(x,s) \simeq (x_0,s_0)$ for some fixed $(x_0,s_0) \in E$.  Then $U$ is a neighborhood of $(x_0,s_0)$ and $(x_1,s), (x_2,s) \simeq (x_0,s_0)$, hence $(x_1,s), (x_2,s) \in U$, and so \eqref{eq:f_sx_1-f_sx_2-asymp-x_1-x_2.--7} holds.
  \item Let $(x', s) \in X'$ with $(x',s) \simeq f(z)$ for some $z \in E$.

    Let $U$ be any fixed neighborhood of $E$.  Using the compactness of $E$, there is again some fixed $(x_0,s_0) \in E$ with $(x',s) \simeq f(x_0,s_0)$.  We know that there is a fixed neighborhood $U'$ of $f(E)$ contained in $f(U)$.  Since $f(x_0,s_0)$ is fixed and $(x',s) \simeq f(x_0,s_0) \in U'$, we have $(x',s) \in U'$, hence $(x',s) \in f(U)$.  We may thus find some $(x,s) \in U$ with $(x',s) = f(x,s)$.

    In summary, we have shown that for each fixed neighborhood $U$ of $E$, there is some $(x,s) \in U$ so that $f(x,s) = (x',s)$.  By idealization (\S\ref{sec:some-axioms}), there is thus some $(x,s) \in X$ with $f(x,s) = (x',s)$ that is contained in each fixed neighborhood of $E$.
  \end{enumerate}
\end{proof}

\subsection{Orthogonal complements}\label{sec:orth-compl-tau-coord}
The purpose of this section is to explain how the definition of ``$\tau$-coordinates'' (\S\ref{sec:coord-tail-regul}) may be extended to the non-archimedean case.  The only archimedean-specific input to that definition was the construction of orthogonal complements $\mathfrak{g}_\tau^{\flat}$ and $\mathfrak{g}_\tau^{\perp \flat}$ to the subspaces $\mathfrak{g}_\tau \leq \mathfrak{g}$ and $\mathfrak{g}_\tau^{\perp} \leq \mathfrak{g}^\wedge$, respectively.  We define below a notion of ``orthogonal complement'' that serves as a suitable substitute in the non-archimedean case.

Recall that $F$ is a local field.  Let $V$ be a finite-dimensional vector space over $F$ equipped with an $F$-norm $|.|_F$ (\S\ref{sec:F-norms}).  Let $V_1$ be a subspace of $V$, and let $V_2$ be a complementary subspace, so that $V = V_1 \oplus V_2$.  We say that $V_2$ is an \emph{orthogonal complement} to $V_1$ if for all $(v_1,v_2) \in V_1 \times V_2$, we have
\begin{equation*}
  \lVert v_1 + v_2 \rVert_F =
  \begin{cases}
    (\lVert v_1 \rVert_{\mathbb{R}}^2 + \lVert v_2 \rVert_{\mathbb{R}}^2)^{1/2} & F = \mathbb{R} , \\
    \lVert v_1 \rVert_{\mathbb{C}} + \lVert v_2 \rVert_{\mathbb{C}} & F = \mathbb{C} , \\
    \max (\lVert v_1 \rVert_F, \lVert v_2 \rVert_F) & F \text{ nonarchimedean}.
  \end{cases}
\end{equation*}
Equivalently, $|.|_F$ can be defined as in \S\ref{sec:F-norms} using a basis of $V$ that is a union of bases of $V_1$ and $V_2$.  In particular, for all $v_1 \in V_1$, we have
\begin{equation}\label{eq:minimum-norm-for-orth-compl}
  \lVert v_1 \rVert = \min _{v_2 \in V_2} \lVert v_1 + v_2 \rVert.
\end{equation}

Write $n := \dim(V)$.  For $m \in \{0, \dotsc, n\}$, let $\Gr_m(V)$ denote the Grassmannian manifold consisting of $m$-dimensional subspaces.
\begin{lemma}\label{lem:orth-compl-general}
  There is an analytic map
  \begin{equation*}
    \flat : \Gr_m(V) \rightarrow \Gr_{n-m}(V)
  \end{equation*}
  such that for each $W \in \Gr_m(V)$, its image $W^{\flat} \in \Gr_{n-m}(V)$ under $\flat$ is an orthogonal complement to $W$.
\end{lemma}
\begin{example}\label{example:orth-compl-padic}
  Suppose $F = \mathbb{Q}_p$, $V = F^2$ equipped with the standard norm $\lvert (x,y) \rvert_F = \max \{\lvert x \rvert_F, \lvert y \rvert_F\}$, and $m=1$.  Every line $W$ in $V$ is generated either by $(x,1)$ for some $x \in \mathbb{Z}_p$ or by $(1,y)$ for some $y \in p \mathbb{Z}_p$.  Let us define $W^{\flat}$ to be the line generated in the first case by $(1,0)$ and in the second case by $(0,1)$.  Then the assignment $W \mapsto W^{\flat}$ satisfies the conclusions of Lemma \ref{lem:orth-compl-general}.
\end{example}
\begin{proof}[Proof of Lemma \ref{lem:orth-compl-general}]
  When $F$ is archimedean, the $F$-norm on $V$ is induced from a unique inner product on $V$ (in the complex case, the square of such an inner product), and we can take for $W^{\flat}$ the orthogonal complement of $W$, in the usual sense, with respect to that inner product.  Consider now the non-archimedean case.  The existence of an orthogonal complement $W^{\flat}$ to a given subspace $W$ is given by \cite[II.1, Cor to Prop 3]{MR0427267}.  We sketch here a construction, along the lines of Example \ref{example:orth-compl-padic}, which is locally constant, hence analytic.  Let $\mathfrak{o}$ denote the ring of integers of $F$, $\mathfrak{p}$ its maximal ideal, and $k := \mathfrak{o} / \mathfrak{p}$ the residue field. Set $L := \{v \in V : \lvert v \rvert_F \leq 1\}$; it is an $\mathfrak{o}$-submodule of rank $n$.  Set $\bar{L} := L / \mathfrak{p} L$; it is a $k$-vector space of dimension $n$.  Choose, for each $m$-dimensional subspace $M$ of $L$, a linear complement $M'$ together with an $\mathfrak{o}$-submodule $\tilde{M}' \leq L$ with image $M'$.  Given $W \in \Gr_m(V)$, let $M \leq \bar{L}$ denote the image of $W \cap L$, and define $W^{\flat}$ to be the $F$-vector space generated by $\tilde{M}'$.  This assignment defines a locally constant map $\flat$ with the desired properties.
\end{proof}

\subsection{Regular elements}\label{sec:interl-regul-elem}
We record some technical estimates to be applied below.  Let $F$ be a fixed local field of characteristic zero, and let $G$ be a fixed connected reductive group over $F$.  Recall from \S\ref{sec:prelim-reductive-groups} that $\mathfrak{g}^\wedge_{\reg} \subseteq \mathfrak{g}^\wedge$ denotes the subset of regular elements, i.e., those whose centralizer has minimal dimension (equal to $\rank(\mathfrak{g})$, the dimension of any Cartan subalgebra).  It is an open subset.  For our purposes, the most important property of this subset is the following result of Kostant (see \cite[Thm 0.1]{MR0158024}).
\begin{theorem}\label{thm:kostant-regular-elements}
  For $\tau \in \mathfrak{g}^\wedge$, the derivative of the map $\mathfrak{g}^\wedge \rightarrow [\mathfrak{g}^\wedge]$ is surjective at $\tau$ if and only if $\tau \in \mathfrak{g}^\wedge_{\reg}$.
\end{theorem}

We will frequently apply the facts summarized in the following lemma, each assertion of which follows readily from Theorem \ref{thm:kostant-regular-elements}.
\begin{lemma}\label{lem:easy-properties-regular-elements}
  Let $\tau \in \mathfrak{g}^\wedge_{\reg}$.  Let $U$ be a small enough open neighborhood of $\tau$.  Let $\gamma$ denote the restriction to $U$ of the map $\mathfrak{g} \rightarrow [\mathfrak{g}^\wedge]$.  Then $\gamma : U \rightarrow \gamma(U)$ is a fibered manifold, i.e., a surjective map whose derivative is surjective at each point.  The fibers of $\gamma$ are precisely the intersections with $U$ of the coadjoint orbits.
\end{lemma}

We next record a straightforward estimate for the symplectic measure of a small ball in a coadjoint orbit based at a regular element.  We state this result only in the archimedean case because that case is the only one in which we have set up the relevant preliminaries, and also the only one required by our present applications.
\begin{lemma}\label{lem:bounds-symplectic-measure-ball-regular}
  Assume that $F$ is archimedean.  Let $\tau$ belong to a fixed compact subset of $\mathfrak{g}^\wedge_{\reg}$.  Let $G \cdot \tau$ denote the coadjoint orbit containing $\tau$ and $\omega_{G \cdot \tau}$ the corresponding symplectic measure, normalized as in \cite[\S6.1]{nelson-venkatesh-1}.  For $r > 0$, set $B_\tau(r) := \{\xi \in \mathfrak{g}^\wedge : |\tau - \xi| < r\}$.  Then for $r \lll 1$, we have
  \begin{equation*}
    \omega_{G \cdot \tau}(B_\tau(r))
    \asymp r^{\dim(\mathfrak{g}) - \rank(\mathfrak{g})}.
  \end{equation*}
\end{lemma}
\begin{proof}
  By restriction of scalars, we may and shall assume that $F = \mathbb{R}$.  Set $n := \rank(\mathfrak{g})$ and $m := \dim(\mathfrak{g}) - \rank(\mathfrak{g})$.  Let $\tau \in \mathfrak{g}^\wedge_{\reg}$.  We choose a small enough neighborhood $U$ of $\tau$.  We have noted (lemma \ref{lem:easy-properties-regular-elements}) that the fibers of $U \rightarrow [\mathfrak{g}^\wedge]$ are then the intersections with $U$ of coadjoint orbits.  By shrinking $U$ if necessary, we may suppose given a trivialization $\kappa : U \hookrightarrow \mathbb{R}^m \times \mathbb{R}^n$ of the fibered manifold $U \rightarrow \image(U) \subseteq [\mathfrak{g}^\wedge]$, i.e., a coordinate chart under which the map $\mathfrak{g} \rightarrow [\mathfrak{g}^\wedge]$ corresponds to the projection onto $\mathbb{R}^n$.  We may find a smooth function $\mu : \kappa(U) \rightarrow \mathbb{R}_{>0}$ so that the symplectic measures on those fibers are given by $\mu$ times the Lebesgue measure on $\mathbb{R}^m$.  (We use here that if $2 d = \dim(\mathfrak{g}) - \rank(\mathfrak{g})$, then the symplectic measure on a regular coadjoint orbit is induced by a constant multiple of the $d$th power of the canonical symplectic form, which is nonvanishing on $\mathfrak{g}^\wedge_{\reg}$.)  By shrinking $U$ if necessary, we may assume that $\mu$ is bounded from above and below by positive scalars.  The required estimates follow in the special case that $\Omega$ is a small enough neighborhood of $\tau$, then in general by compactness.
\end{proof}

We record another simple consequence of Lemma \ref{lem:easy-properties-regular-elements}:
\begin{lemma}\label{lem:submersivity-orbit-map-regular-elt}
  Let $\tau$ belong to a fixed compact subset of $\mathfrak{g}^\wedge_{\reg}$.  Write $\mathcal{O}_\tau := G \cdot \tau$ for its coadjoint orbit.  Let $\xi \in \mathfrak{g}^\wedge$ with $\xi \lll 1$ and $\tau + \xi \in \mathcal{O}_\tau$.  Then there exists $x \in \mathfrak{g}_\tau^{\flat}$ with
  \begin{equation*}
    \exp(x) \tau = \tau + \xi, \quad x \asymp \xi.
  \end{equation*}
\end{lemma}
\begin{proof}
  The derivative at the identity of the orbit map $G \rightarrow \mathcal{O}_\tau$, $g \mapsto g \tau$ is the map
  \begin{equation}\label{eq:mathfr-right-t_ta-=-mathfr-tau-quad-x-mapsto-x-tau}
    \mathfrak{g} \rightarrow T_\tau(\mathcal{O}_\tau) = [\mathfrak{g},\tau], \quad x \mapsto [x,\tau].
  \end{equation}
  That derivative has kernel $\mathfrak{g}_\tau$, so its restriction to the complementary subspace $\mathfrak{g}_\tau^{\flat}$ is an isomorphism.  We now appeal to Lemma \ref{lem:FTOC-families-nonstandard} with $(S,V,V') = (\mathfrak{g}^\wedge_{\reg}, \mathfrak{g}, \mathfrak{g}^\wedge)$, with $X$ and $X'$ the analytic bundles over $S$ whose fibers above $s \in S$ are $\mathfrak{g}_s^{\flat}$ and $G \cdot s$, respectively, with $f(x,s) = \exp(x) s$, and with $E = \left\{ (0,s) : s \in \mathfrak{S} \right\}$ for some fixed compact $\mathfrak{S} \subseteq \mathfrak{g}^\wedge_{\reg}$ that contains $\tau$.  The derivatives at the origin of the maps $f_s$ are the maps \eqref{eq:mathfr-right-t_ta-=-mathfr-tau-quad-x-mapsto-x-tau}, which we have noted to be bijective.  The element $\tau + \xi \in X'_{\tau}$ satisfies $(\tau + \xi,\tau) \simeq (\tau,\tau) = f(0,\tau) \in f(E)$, so by part \eqref{itm:sub-gln-2:ftoc-surj} of Lemma \ref{lem:FTOC-families-nonstandard}, there exists $(x,\tau) \in X$ with $f(x,\tau) = (\tau + \xi, \tau)$, i.e., $\exp(x) \tau = \tau + \xi$, and with $(x,\tau)$ contained in each fixed neighborhood of $E$.  In particular, $x \simeq 0$.  By part \eqref{itm:sub-gln-2:ftoc-inj} of Lemma \ref{lem:FTOC-families-nonstandard}, we then have
  \begin{equation}\label{eq:f_taux-f_tau0-asymp}
    |\xi| = |\exp(x) \tau - \tau| = |f_\tau(x) - f_\tau(0)| \asymp |x - 0| = |x|,
  \end{equation}
  as required.
\end{proof}

We record a convenient relationship between $\tau$-coordinates (\S\ref{sec:coord-tail-regul}, \S\ref{sec:orth-compl-tau-coord}) and the extent to which a Lie group element centralizes $\tau$:
\begin{lemma}\label{lem:approximate-x-prime-various}
  Let $\tau$ belong to a fixed compact subset of $\mathfrak{g}^\wedge_{\reg}$.  Let $x = (x',x'')$ denote $\tau$-coordinates.  Let $x \in \mathfrak{g}$.  Then
  \begin{equation*}
    x' \asymp [x,\tau].
  \end{equation*}
  If moreover $x \lll 1$, then
  \begin{equation*}
    x' \asymp [x,\tau] \asymp \exp(x) \tau - \tau.
  \end{equation*}
\end{lemma}
\begin{proof}
  We show first that
  \begin{equation*}
    x ' \asymp [x,\tau].
  \end{equation*}
  We have $[x,\tau] = [x',\tau]$, so we may assume that $x = x' \in \mathfrak{g}_\tau^{\flat}$.  Since both sides are linear in $x$, there is no loss of generality in assuming that $x \lll 1$.  We then argue as in the proof of Lemma \ref{lem:submersivity-orbit-map-regular-elt}, using that the map $\mathfrak{g}_\tau^{\flat} \rightarrow [\mathfrak{g},\tau]$, $x \mapsto [x,\tau]$ is injective.

  To complete the proof, we open the exponential series, giving
  \begin{equation*}
    \exp(x) \tau - \tau = [x,\tau] + \sum _{j \geq 2} \frac{1}{j!} (\ad_x^*)^{j-1} [x,\tau].
  \end{equation*}
  We use now that $x \lll 1$ and $\tau \ll 1$ to see that the sum over $j$ is much smaller than $[x,\tau]$.  It follows that the RHS is $\asymp [x,\tau]$, as required.
\end{proof}

We compare the coadjoint orbit $G \cdot \tau$ containing a regular element $\tau$ to its tangent plane $\tau + \mathfrak{g}_\tau^\perp$ at that element:
\begin{lemma}\label{lem:parabola-y-x-squared}
  Let $\tau$ belong to a fixed compact subset of $\mathfrak{g}^\wedge_{\reg}$.  Let $\xi, \eta \in \mathfrak{g}^\wedge$.  Suppose $\xi \lll 1$ and $\tau + \xi \in G \cdot \tau$.  Then, with $\tau$-coordinates $\xi = (\xi ', \xi '')$, we have $|\xi ''| \ll |\xi '|^2$.
\end{lemma}
\begin{proof}
  The reader is encouraged to review Figure \ref{fig:tau-coordinates} (\S\ref{sec:refin-symb-class}), which suggests a convincing ``proof by picture'' of the required estimate.  (Conversely, the present lemma may be understood as justifying that picture, hence the accuracy of the terminology ``coin-shaped'' that was used starting in \S\ref{sec:weyls-law}.)  We record a written verification for completeness.  By Lemma \ref{lem:submersivity-orbit-map-regular-elt}, we may write $\tau + \xi = \exp(x ) \tau$ with $x \in \mathfrak{g}_\tau^{\flat}$ and $x \asymp \xi \lll 1$.  Expanding the exponential series, we obtain
  \begin{equation*}
    \xi = [\tau,x] + \sum _{j \geq 2} \frac{1}{j!}  (\ad_x^*)^{j-1} [ \tau,x].
  \end{equation*}
  By Lemma \ref{lem:approximate-x-prime-various}, we have $[\tau,x] \asymp x \asymp \xi$.  Since $x \lll 1$, the sum over $j \geq 2$ is thus of size $\O(|\xi|^2)$.  Since $[\tau,x] \in [\mathfrak{g},\tau] = \mathfrak{g}_\tau^\perp$, it follows from the definition of $\tau$-coordinates that $\xi '' = \O(|\xi|^2)$.  We conclude via the relation $|\xi| \asymp |\xi ' | + |\xi ''|$ and the hypothesis $\xi \lll 1$.
\end{proof}

\section{Construction of analytic test
  vectors}\label{sec:analyt-test-vect}
The purpose of this section is to prove our main local result, Theorem \ref{thm:construct-test-function}, modulo a technical estimate which we postpone.

\subsection{Setup}\label{sec:construction-setup}
We adopt the setting of Theorem \ref{thm:construct-test-function}: $(G,H)$ is a fixed GGP pair over an archimedean local field $F$, $T \ggg 1$ is a positive real, and $(\pi,\sigma)$ is a pair of tempered irreducible unitary representations satisfying certain assumptions.  In view of Lemma \ref{lem:stability-inf-chars-satake-params}, those assumptions imply the following:
\begin{itemize}
\item Setting\index{wavelength parameter $\h$}
  \begin{equation}\label{eq:h-:=-t}
    \h := T^{-1/[F:\mathbb{R}]} \lll 1, \quad \text{ so that } T = \h^{-[F:\mathbb{R}]},
  \end{equation}
  the pair $(\h \lambda_\pi, \h \lambda_\sigma)$ of rescaled infinitesimal characters lies in a fixed compact subset of $\{\text{stable } (\lambda,\mu) \in [\mathfrak{g}^\wedge] \times [\mathfrak{h}^\wedge] \}$.
\item $\mathcal{O}_{\pi,\sigma}$ is nonempty.
\end{itemize}
We have noted in \S\ref{sec:relat-coadj-orbits} that $\mathcal{O}_{\pi,\sigma}$ is then an $H$-torsor.

We must show that for each fixed $\kappa > 0$, there is a test function $f \in C_c^\infty(G)$ and a smooth unit vector $u \in \sigma$, with $f$ supported in each fixed neighborhood of the identity element, so that the three properties enunciated in Theorem \ref{thm:construct-test-function} are satisfied.  We construct $f$ in \S\ref{sec:construction-f} and $u$ in \S\ref{sec:pass-an-indiv}.  We verify the first two properties in \S\ref{sec:proof-part-eqref} and \S\ref{sec:proof-part-eqref-1}.  The proof of the third property, which is more involved, is given in \S\ref{sec:proof-theor-refthm:c}.

We fix $0 < \delta = \delta ' < 1/2$ and $\delta ' < \delta '' < 2 \delta '$ (so that the condition \eqref{eqn:delta-delta-prime-hypotheses} is satisfied), with both $1/2 - \delta'$ and $1 - \delta ''$ taken small enough in terms of $\kappa$.

Throughout this section, $N$ is a fixed sufficiently large positive integer.

\subsection{Choice of $\tau$}

Let $\tau \in \mathfrak{g}^\wedge$ be an element of the rescaled orbit $\h \mathcal{O}_{\pi,\sigma}$ of (say) minimal Euclidean norm.  By the condition on $(\h \lambda_\pi, \h \lambda_\sigma)$ stated above and the ``principal bundle'' consequence of stability, $\tau$ belongs to a fixed compact subset of $\mathfrak{g}_{\stab}^\wedge$.  In view of the following lemma, $\tau$ (resp.  its restriction $\tau_H$) belongs to a fixed compact subset of $\mathfrak{g}_{\reg}^\wedge$ (resp. $\mathfrak{h}_{\reg}^\wedge$).
\begin{lemma}\label{lem:stab-impl-reg}
  Let $\tau \in \mathfrak{g}^\wedge_{\stab}$.  Then $\tau \in \mathfrak{g}^\wedge_{\reg}$ and $\tau_H \in \mathfrak{h}^\wedge_{\reg}$.
\end{lemma}
\begin{proof}
  This is \cite[\S14.3, Lem 2]{nelson-venkatesh-1}.
\end{proof}

\subsection{Construction of basic
  symbols}\label{sec:constr-basic-symb}
We may find symbols $a \in S^{-\infty}_\delta (\mathfrak{g}^\wedge)$ and $b \in S^{-\infty}_\delta (\mathfrak{h}^\wedge)$, valued in the unit interval $[0,1]$, that are smooth bumps on $\tau + \O(\h^\delta)$ and $\tau_H + \O(\h^\delta)$, respectively, in the sense that
\[
  a(\xi) \neq 1 \implies |\xi - \tau| \gg \h^\delta, \quad a(\xi) \neq 0 \implies |\xi - \tau| \ll \h^\delta,
\]
\[
  b(\eta) \neq 1 \implies |\eta - \tau_H| \gg \h^\delta, \quad b(\eta) \neq 0 \implies |\eta - \tau_H| \ll \h^\delta.
\]

\subsection{Application of relative character estimates}
Set
\begin{equation}\label{eq:A-B-defns-rel-char}
  \begin{split}
    A &:= \Opp_{\h}(a:\pi) \in \Psi_\delta^0(\pi) \cap \h^{-N}\Psi_\delta^{-N}(\pi),
    \\
    B &:= \Opp_{\h}(b:\sigma) \in \Psi_\delta^0(\sigma) \cap \h^{-N}\Psi_\delta^{-N}(\sigma),
  \end{split}
\end{equation}
where the memberships follow from \eqref{eqn:opp-mapping-S-m-delta}.  Recall from \S\ref{sec:local-dist-matr} the quadratic form $\mathcal{Q}$ on the smooth subspace of $\pi \otimes \sigma$ given by
\[
  \mathcal{Q}(v \otimes u) = \int_{s \in H} \langle s v, v \rangle \langle u, s u \rangle \, d s.
\]
As noted in Lemma \ref{lem:convergence-relative-character}, $\mathcal{Q}$ extends continuously to the tensor product of Sobolev spaces $\pi^N \otimes \sigma^N$.  We write $\mathcal{B}(\pi)$ and $\mathcal{B}(\sigma)$ for any orthonormal bases.
\begin{lemma}
  We have
  \begin{equation}\label{eqn:sum-v-u-Opp-a-Opp-b}
    \sum _{v \in \mathcal{B}(\pi)}
    \sum _{u \in \mathcal{B}(\sigma)}
    \mathcal{Q}(A v \otimes B u)
    \asymp
    \h^{\delta \dim(H)},
  \end{equation}
  with the LHS absolutely convergent.
\end{lemma}
\begin{proof}
  Setting $M := G \times H$, we apply the results of \S\ref{sec:relat-char-asympt} to the symbol $c \in S^{-\infty}_{\delta}(\mathfrak{m}^\wedge )$ given by $c(\xi,\eta) := a(\xi) b(-\eta)$.  We choose the nice cutoff on $\mathfrak{h}$ implicit in the definition of $\Opp$ to be the product of those on $\mathfrak{g}$ and $\mathfrak{h}$.  We see then that
  \[
    C := \Opp_{\h}(c : \pi \otimes \sigma^\vee) = A \otimes B^\vee
  \]
  and that the LHS of \eqref{eqn:sum-v-u-Opp-a-Opp-b} is $\mathcal{H}(C C^*)$.  The required absolute convergence follows from Lemma \ref{lem:convergence-relative-character}.  We note that $c$ is supported in a small neighborhood of $(\tau,-\tau_H)$, hence is supported in some fixed compact collection of stable elements.  The results of \S\ref{sec:main-estimates} give (for fixed $N$ and $J$, with $J$ large enough in terms of $N$)
  \begin{equation*}
    \mathcal{H}(C C^*)
    =
    \sum _{0 \leq j_1, j_2 < J}
    \h^{j_1 + j_2}
    \int _{\xi \in \h \mathcal{O}_{\pi,\sigma}}
    \mathcal{D}_{j_1} (c \star_{j_2} \bar{c})(\xi, -\xi_H)
    \,
    d \Haar_{H}(\xi)
    + \O(\h^{N}).
  \end{equation*}
  The $(j_1,j_2) = (0,0)$ term is
  \[
    \int _{\xi \in \h \mathcal{O}_{\pi,\sigma}} |a(\xi)|^2 |b(\xi_H)|^2 \, d \Haar_{H}(\xi) = \int _{s \in H} |a(s \cdot \tau )|^2 |b(s \cdot \tau_H)|^2 \, d s
  \]
  By the construction of $a$ and $b$ and the ``principal bundle'' consequence of stability, this term is $\asymp \h^{\delta \dim(H)}$.  The general $(j_1,j_2)$ term is an integral of a function of size $\h^{(1 - 2 \delta) (j_1+j_2)}$ over a domain of volume $\ll \h^{\delta \dim(H)}$, and is thus
  \[
    \ll \h^{(1 - 2 \delta) (j_1+j_2) + \delta \dim(H)}.
  \]
  Therefore the $(0,0)$ term dominates and the required estimate \eqref{eqn:sum-v-u-Opp-a-Opp-b} follows.
\end{proof}

\subsection{Construction of refined symbols}
Recall from the end of \S\ref{sec:construction-setup} that we have chosen some parameters $\delta ', \delta ''$.  Using that $\delta ' > \delta '' - \delta '$, we see that the symbol $a$ belongs to $S_{\delta ', \delta ''}^{\tau}$.  We may smoothly decompose $a = a' + a''$, where $a', a'' \in S_{\delta ', \delta ''}^{\tau }$ satisfy the following support conditions given in terms of $\tau$-coordinates (\S\ref{sec:coord-tail-regul}):
\begin{align}
  a ' (\tau + \xi ) \neq 0
  \quad   \implies
  \quad 
  &\xi ' \ll \h^{\delta '},
    \quad
    \xi '' \ll \h^{\delta ''},
    \nonumber
  \\
  \label{eqn:support-of-a-double-prime}
  a''(\tau + \xi) \neq 0
  \quad   \implies
  \quad 
  &\xi ' \ll \h^{\delta '},
    \quad
    \h^{\delta ''} \ll \xi '' \ll \h^{\delta '}.
\end{align}
Then
\[
  A = A' + A'', \quad A' := \Opp_{\h}(a':\pi), \quad A'' := \Opp_{\h}(a'':\pi).
\]

\setlength{\unitlength}{1.5cm}
\begin{figure}
  \begin{picture}(4,3)

    \put(-1,0){\vector(1,0){6}}
    \put(-1,0){\vector(0,1){1.5}}

    {%
      \thicklines
      \color{black}%
      \multiput(1.1,0.3)(0,0.1){12}{\line(0,1){0.05}}
      \multiput(3.0,0.3)(0,0.1){12}{\line(0,1){0.05}}
    }

    {%
      \thicklines
      \color{black}%
      \multiput(1,0.3)(0.1,0){20}{\line(1,0){0.05}}
      \multiput(1,0.7)(0.1,0){20}{\line(1,0){0.05}}
      \multiput(1,1.5)(0.1,0){20}{\line(1,0){0.05}}
      \put(3.2, 1.1){$\supp(a'')$}
      \put(3.2, 0.5){$\supp(a')$}

      {%
        \thicklines
        \color{black}%
        \qbezier(0,1)(2,0)(4,1)
        \put(-0.6,1.1){$G \cdot \tau$}
      }

      \color{black}
      \put(-1.4,1.4){$\xi''$}
      \put(4.9,-0.25){$\xi'$}
      \put(2,0.3){$\tau$}
      \put(1.95,0.5){\circle*{0.1}}
    }

  \end{picture}
  \caption{ A schematic of the decomposition $a = a' + a''$.  }
  \label{fig:a-decompose}
\end{figure}
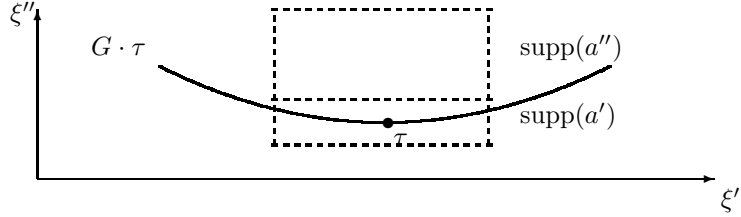

\begin{lemma}\label{lem:A-double-prime-negligible}
  $A'' \in \h^\infty \Psi^{-\infty}(\pi)$.
\end{lemma}
\begin{proof}
  The key observation is that, as suggested by Figure \ref{fig:a-decompose}, the support of $a''$ is disjoint from the rescaled coadjoint orbit $\h \mathcal{O}_\pi$.  To see this, suppose $\tau + \xi$ lies in the support of $a''$ and also in $\h \mathcal{O}_\pi$.  We see then from \eqref{eqn:support-of-a-double-prime} that $\xi \lll 1$.  By Lemma \ref{lem:easy-properties-regular-elements} and the fact that $\mathcal{O}_\pi$ is contained in the fiber of $\lambda_\pi$ (Theorem \ref{thm:kirillov-formula}), it follows that $\tau + \xi \in G \cdot \tau$.  By Lemma \ref{lem:parabola-y-x-squared}, we obtain $|\xi''| \ll |\xi '|^2$.  From \eqref{eqn:support-of-a-double-prime}, we deduce that
  \[
    \h^{\delta ''} \ll |\xi ''| \ll |\xi|^2 \ll \h^{2 \delta '}.
  \]
  We obtain a contradiction in view of the inequality $2 \delta ' > \delta ''$ and our assumption $\h \lll 1$.

  Recall from \S\ref{sec:spaces-operators} the notation $\Delta = 1 - \sum _{x \in \mathcal{B}(\mathfrak{g})} x^2 \in \mathfrak{U}(\mathfrak{g})$.  By Theorem \ref{thm:opp-S-m-underlined}, we have $A'' \in \underline{\Psi }^{-\infty}(\pi)$, so to deduce the required conclusion that $A''$ belongs to $\h^\infty \Psi^{-\infty}(\pi)$, it suffices to show that for each fixed $m \in \mathbb{Z}_{\geq 0}$, the bounded operator $\pi \rightarrow \pi$ defined by
  \begin{equation*}
    S := \pi(\Delta)^m A'' \pi(\Delta)^m
  \end{equation*}
  has operator norm of size $\O(\h^\infty)$.  Since the operator norm is bounded by the Hilbert--Schmidt norm, it is enough to check that
  \[
    \trace(S S^*) \ll \h^\infty.
  \]
  To that end, let $p \in S^2_0(\mathfrak{g}^\wedge)$ denote the symbol given by the polynomial function
  \[p(\xi) = 1 - \sum _{x \in \mathcal{B}(\mathfrak{g})} \langle x, \xi \rangle^2 \geq 0.\] Since $\Delta$ is the symmetrization of $p$, we see that $\Delta^m$ is likewise the symmetrization of $p^m \in S^{2 m}_0(\mathfrak{g}^\wedge)$, hence by Lemma \ref{lem:polyn-symb} that
  \[\pi(\Delta)^m = \pi(\Delta^m) =
    \Opp_{1}(p^m) = \h^{-2 m} \Opp_{\h}(p^m).
  \]
  Fix $N \in \mathbb{Z}_{\geq 0}$ large enough in terms of $m$.  By the composition formula \eqref{eqn:composition-formula-2}, we may write
  \[
    S S^* = \h^{-8 m} \Opp_{\h}(q) + \mathcal{E}
  \]
  with (recall that $\delta ' = \delta $)
  \[\mathcal{E} \in \h^N \Psi^{-N}_\delta(\pi)\]
  and with $q \in S_{\delta ', \delta ''}^\tau \subseteq S_{\delta ''}^{-\infty}$ obtained by truncating the asymptotic expansion of \[p^m \star_{h} a'' \star_{\h} p^m \star_{\h} p^m \star_{\h} \overline{a''} \star_{\h} p^m
  \]
  (with the same conventions concerning order of evaluation as in \eqref{eqn:x1-xsmm-oppa-Delta-k}).  Then $\supp(q) \subseteq \supp(a'')$, hence
  \[
    \supp(q) \cap \h \mathcal{O}_\pi = \emptyset.
  \]
  By Corollary \ref{cor:trace-disjoint-supports}, we have $\trace(\Opp_{\h}(q)) \ll \h^\infty$.  By Theorem \ref{thm:trace-class-Psi-neg-N}, we have $\trace(\mathcal{E}) \ll \h^N$.  Since $N$ was arbitrary, the proof is now complete.
\end{proof}

\begin{lemma}\label{lem:discard-A-double-prime}
  We have
  \begin{equation}\label{eqn:Q-A-prime-B}
    \sum _{v \in \mathcal{B}(\pi)}
    \sum _{u \in \mathcal{B}(\sigma)}
    \mathcal{Q}(A ' v \otimes B u)
    \asymp
    \h^{\delta \dim(H)}.
  \end{equation}
\end{lemma}
\begin{proof}
  We will show that the left hand sides of \eqref{eqn:sum-v-u-Opp-a-Opp-b} and \eqref{eqn:Q-A-prime-B} differ by $\O(\h^\infty)$.  The difference in question is
  \[
    \sum _{u \in \mathcal{B}(\sigma)} \sum _{v \in \mathcal{B}(\pi)} \left( \mathcal{Q}(A v \otimes B u) - \mathcal{Q}(A ' v \otimes B u) \right) = \mathcal{H}(D \otimes E^\vee),
  \]
  where
  \[
    D := A' (A'')^* + A'' (A')^* + A'' (A'')^*, \quad E := B B^*.
  \]
  By Theorem \ref{thm:oper-assignm-comp-refined}, we have $A' \in \Psi^0_\delta(\pi)$ (recall that $\delta = \delta '$).  By Lemma \ref{lem:A-double-prime-negligible}, we have $A'' \in \h^\infty \Psi^{-\infty}(\pi)$.  As noted in \eqref{eq:A-B-defns-rel-char}, we have $B \in \h^{-N} \Psi_\delta^{-N}(\sigma)$.  By Lemma \ref{lem:composition-operator-classes}, it follows that $D \in \h^\infty \Psi^{-\infty}(\pi)$ and $E \in \h^{-2 N} \Psi_\delta^{-2 N}(\sigma)$.  In particular, with the notation $\nu_N$ as in Lemma \ref{lem:convergence-relative-character}, we obtain $\nu_N(D) \ll \h^\infty$ and $\nu_N(E) \ll \h^{-2 N}$.  Applying Lemma \ref{lem:convergence-relative-character} (part \eqref{item:convergence-relative-character-3}), we conclude that $\mathcal{H}(D \otimes E^\vee ) \ll \h^\infty$.
\end{proof}

\subsection{Construction of the test function $f$}\label{sec:construction-f}
By integrating by parts in the Fourier integral defining $(a'_{\h})^\vee$ and taking into account the smoothness and support properties of $a'$, we see that in $\tau$-coordinates $x = (x',x'')$,
\begin{equation}\label{eqn:a-h-vee-properties-for-test-vector-construction}
  \partial_x^\alpha (a'_{\h})^\vee (x)
  \ll
  \frac{\h^{-|\alpha|}
    (\h^{-1+\delta '})^{\dim(\mathfrak{g}) - \rank(\mathfrak{g})}
    (\h^{-1+\delta ''})^{\rank(\mathfrak{g})}}{
    \langle \h^{-1 + \delta '} x' \rangle^{N}
    \langle \h^{-1 + \delta ''} x'' \rangle^{N}
  }
\end{equation}
for all fixed $\alpha$ and $N$, where, as before,
\begin{equation*}
  \langle z \rangle := (1 + |z|^2)^{1/2}.
\end{equation*}
For now, the significance of the numerator in \eqref{eqn:a-h-vee-properties-for-test-vector-construction} is just that it is of size $\h^{-\O(1)}$.

We now truncate $(a'_{\h})^\vee$ to its essential support, as follows.  We fix $\nu > 0$ and $\eps > 0$ so that
\begin{equation}\label{eqn:nu-conditions}
  1 - \delta ' - \eps > 1/2 + \nu,
  \quad
  1 - \delta '' - \eps > \nu,
\end{equation}
as we may.  For $r > 0$, we set
\[\mathcal{D}(r) := \left\{x \in \mathfrak{g} :
    |x'| \leq r \h^{1 - \delta ' - \eps }, |x''| \leq r \h^{1 - \delta '' - \eps } \right\}.
\]
Let $\chi \in C_c^\infty(\mathfrak{g})$ be the nice cutoff (\S\ref{sec:nice-cutoffs}) implicit in the definition of $A'$, thus $A' = \Opp_{\h}(a:\pi,\chi)$.  We choose another nice cutoff $\tilde{\chi} \in C_c^\infty(\mathfrak{g})$ with
\begin{itemize}
\item $\tilde{\chi}(x) = 1$ for $x \in \mathcal{D}(1)$,
\item $\tilde{\chi}(x) = 0$ for $x \notin \mathcal{D}(2)$, and
\item $\partial^\alpha \tilde{\chi}(x) \ll \h^{-\O(1)}$ for all fixed multi-indices $\alpha$ and all $x \in \mathfrak{g}$.
\end{itemize}
Informally, $\tilde{\chi}$ is an envelope for the essential support of $(a'_{\h})^\vee$ as quantified by \eqref{eqn:a-h-vee-properties-for-test-vector-construction}.  Quantitatively, the difference $\chi (a'_{\h})^\vee - \tilde{\chi} (a'_{\h})^\vee$ is negligible in the sense that it is supported on a fixed compact set and has each fixed partial derivative of size $\O(\h^\infty)$.  The difference $\widetilde{\Opp}_{\h}(a':\chi) -\widetilde{\Opp}_{\h}(a':\tilde{\chi})$ (see \S\ref{sec:basic-oper-assignm} for notation) is then negligible in the analogous sense.  Setting
\[
  \tilde{A}' := \Opp_{\h}(a':\pi,\tilde{\chi})
\]
and using that $f \mapsto \pi(f)$ maps fixed bounded subsets of $C_c^\infty(G)$ to $\Psi^{-\infty}(\pi)$ (see \cite[\S3.6, Lem 2]{nelson-venkatesh-1}), we deduce that the difference $A' - \tilde{A}'$ is negligible in that it lies in $\h^\infty \Psi^{-\infty}(\pi)$.  By the proof of lemma \ref{lem:discard-A-double-prime}, we deduce that replacing $A'$ with $\tilde{A}'$ modifies the LHS of \eqref{eqn:Q-A-prime-B} by $\O(\h^\infty)$; in particular, that estimate remains valid for $\tilde{A}'$.

Since $a' \in \underline{S}^{-\infty}(\mathfrak{g}^\wedge)$, we may construct $f \in C_c^\infty(G)$ by setting
\[
  f := \widetilde{\Opp}_{\h}(a':\tilde{\chi}), \quad \text{ so that } \quad \pi(f) = \tilde{A}'.
\]
Strictly speaking, this defines $f$ as a smooth compactly-supported distribution on $G$; we identify it with an element of $C_c^\infty(G)$ by dividing by our choice of Haar measure $d g$ on $G$.

Since $\mathcal{D}(2)$ is concentrated near the origin in $\mathfrak{g}$, we see that $f$ is supported near the identity element of $G$.

\subsection{Passage to an individual vector $u$}\label{sec:pass-an-indiv}
Recall the fixed quantity $\kappa > 0$ introduced in \S\ref{sec:construction-setup}.  (This quantity is arbitrarily small but fixed.  It describes our target bounds.  We took $1/2 - \delta'$ and $1 - \delta ''$ small enough in terms of it.)

\begin{lemma}\label{lem:pass-to-individual-u}
  There exists a smooth unit vector $u \in \sigma$ so that
  \begin{equation}\label{eqn:sum-v-u-Opp-a-Opp-b-2}
    \sum _{v \in \mathcal{B}(\pi)}
    \mathcal{Q}(\tilde{A}' v \otimes B u)
    \gg
    \h^{\delta \dim(H)+\kappa/10}.
  \end{equation}
\end{lemma}
\begin{remark}
  The passage from \eqref{eqn:Q-A-prime-B} to \eqref{eqn:sum-v-u-Opp-a-Opp-b-2} serves primarily to improve the cosmetics of the arguments of \S\ref{sec:reduction-proof}.  It would be fine to retain the sum over $u$ throughout the whole argument.  We would then eventually require an estimate for the Hilbert--Schmidt norm of $B$, which is anyway the main input in the proof of Lemma \ref{lem:pass-to-individual-u}.  Thus while Lemma \ref{lem:pass-to-individual-u} is, strictly speaking, unnecessary, it allows us to simplify our presentation with negligible net cost.
\end{remark}
The proof appeals to results of \S\ref{sec:interl-regul-elem}, which are applicable in view of Lemma \ref{lem:stab-impl-reg}.

\begin{proof}[Proof of Lemma \ref{lem:pass-to-individual-u}]
  We observe first that the estimate \eqref{eqn:Q-A-prime-B} holds with $B$ replaced by the identity operator:
  \begin{equation}\label{eqn:sum-v-u-Opp-a-Opp-b-3}
    \sum _{v \in \mathcal{B}(\pi)}
    \sum _{u \in \mathcal{B}(\sigma)}
    \mathcal{Q} (\tilde{A}' v \otimes u)
    \asymp
    \h^{\delta \dim(H)}.
  \end{equation}
  To see this, we may reduce as in the proof of \eqref{lem:discard-A-double-prime} to verifying the same estimate but with $\tilde{A}'$ replaced by $A$.  We then apply the proof of \eqref{eqn:sum-v-u-Opp-a-Opp-b}, but using \eqref{eqn:relative-character-asymptotics-2} in place of \eqref{eqn:relative-character-asymptotics}.  The main term is given by $\int _{s \in H} |a(s \cdot \tau )|^2 \, d s \asymp \h^{\delta \dim(H)}$, while the remaining terms contribute much less.

  We observe next that, since $b$ is real-valued, the operator $B$ is self-adjoint (see \eqref{eqn:adjoint-opp-a}).  We now estimate the squared Hilbert--Schmidt norm $\|B\|_2^2$ of $B$, i.e., the trace of $B B^* = B^2$.  To do so, we apply the composition formula \eqref{eqn:composition-formula-2} to write $B^2 = \Opp_{\h}(q:\sigma) + \mathcal{E}$, where
  \begin{itemize}
  \item $q \in S_\delta^{-\infty}(\mathfrak{h}^\wedge)$ is supported on $\tau_H + \O(\h^\delta)$, and
  \item $\mathcal{E} \in \h^N \Psi_\delta^{-N}(\sigma)$.
  \end{itemize}
  By Theorem \ref{thm:trace-class-Psi-neg-N}, we have $\trace(\mathcal{E}) \ll \h^N$.  By the asymptotic form of the Kirillov formula (Theorem \ref{thm:asymptotic-kirillov}), we have
  \[\trace(\Opp_{\h}(q:\sigma))
    \ll \h^{-(\dim(\mathfrak{h}) - \rank(\mathfrak{h}))/2} \mu + \h^N,\] where $\mu$ denotes the symplectic volume of $\h \mathcal{O}_\sigma \cap \supp(b)$.  To estimate $\mu$, we recall the support conditions stated in \S\ref{sec:constr-basic-symb} and apply Lemma \ref{lem:bounds-symplectic-measure-ball-regular}.  We obtain
  \begin{equation*}
    \|B\|_2^2
    \ll
    \h^{2 (\delta -1/2) (\dim(\mathfrak{h}) -
      \rank(\mathfrak{h}))/2}
    \ll
    \h^{-\kappa/10},
  \end{equation*}
  using in the final step that $1/2-\delta$ is small enough in terms of $\kappa$.

  We may take for $\mathcal{B}(\sigma)$ a basis of eigenvectors $u$ for $B$, with eigenvalues $\lambda_u$ (cf.\ Lemma \ref{lem:convergence-relative-character}, part \eqref{item:convergence-relative-character-4}).  The LHS of \eqref{eqn:sum-v-u-Opp-a-Opp-b} then reads
  \begin{equation*}
    \sum _{u \in \mathcal{B}(\sigma)}
    |\lambda_u|^2
    \sum _{v \in \mathcal{B}(\pi)}
    \mathcal{Q}(\tilde{A}' v \otimes u).
  \end{equation*}
  We have
  \begin{equation}\label{eqn:bounds-for-lambda}
    \sum_{u \in \mathcal{B}(\sigma)}
    |\lambda_u|^2 = \|B\|_2^2 \ll \h^{-\kappa/10}.
  \end{equation}
  Recall from \eqref{eqn:H-vs-ell-normalization} that $\mathcal{Q}$ is definite.  For small enough fixed $\eps > 0$, we see by comparing \eqref{eqn:Q-A-prime-B} and \eqref{eqn:sum-v-u-Opp-a-Opp-b-3} that
  \begin{equation}\label{eqn:sum-restricted-to-lambda-u-at-least-eps}
    \sum _{\substack{
        u \in \mathcal{B}(\sigma) : \\
        |\lambda_u| \geq \eps 
      }
    }
    |\lambda_u|^2
    \sum _{v \in \mathcal{B}(\pi)}
    \mathcal{Q}(\tilde{A}' v \otimes u)
    \asymp \h^{\delta \dim(H)}.
  \end{equation}
  On the other hand, by \eqref{eqn:bounds-for-lambda}, we have $\# \{u \in \mathcal{B}(\sigma) : |\lambda_u| \geq \eps \} \ll \h^{-\kappa/10}$.  There is thus a unit eigenvector $u$ of $B$ for which \eqref{eqn:sum-v-u-Opp-a-Opp-b-2} holds (e.g., any eigenvector whose eigenvalue has maximal size).  We have $b \in \underline{S}^{-\infty}(\mathfrak{h}^\wedge)$, so $B \in \underline{\Psi}^{-\infty}(\sigma)$, and thus the vector $u = \lambda_u^{-1} B u$ is smooth (cf.\ \S\ref{sec:spaces-operators}).  The proof is now complete.
\end{proof}
As noted in \eqref{eq:A-B-defns-rel-char}, we have $B \in \Psi_\delta^0(\sigma)$, so the operator norm of $B$ is $\O(1)$.  We deduce that there is a smooth unit vector $u \in \sigma$ so that
\begin{equation}\label{eqn:desiderata-for-u}
  \sum _{v \in \mathcal{B}(\pi)}
  \mathcal{Q}(\tilde{A}' v \otimes u)
  \gg
  \h^{\delta \dim(H)+\kappa/10}.
\end{equation}
\begin{remark}\label{rmk:refine-via-microlocalization-of-Psi}
  The estimate \eqref{eqn:desiderata-for-u} suffices for our purposes, but is less precise than what is given by Lemma \ref{lem:pass-to-individual-u} in that it ``forgets'' the microlocalization of $u$ at $\tau_H$ imposed by $B$.  It may be possible to exploit that microlocalization property to refine the numerical exponent of Theorem \ref{thm:main}.  In the setting of \S\ref{sec:coadj-reform}, this would amount to exploiting more than just the central directions in $H_{\tau_H}$.
\end{remark}

\subsection{Proof of part \eqref {thm:f-item-1} of Theorem \ref{thm:construct-test-function}}\label{sec:proof-part-eqref}
We write $n$ for the natural number for which $(G,H)$ is a form of $(\U_{n+1},\U_n)$, so that $\dim H = [F:\mathbb{R}] n^2$ and $\dim_F H = n^2$.  Then $\h^{\dim(H)/2} = T^{-n^2/2}$.  The required bound
\begin{equation}
  \sum _{v \in \mathcal{B}(\pi)}
  \mathcal{Q}(\pi(f) v \otimes u)
  \gg T^{-n^2/2-\kappa}
  \tag{\ref{eqn:sum-v-H-pi-f-0-v-u}}
\end{equation}
follows from \eqref{eqn:desiderata-for-u} upon recalling that $\delta$ is close enough to $1/2$ in terms of $\kappa$.

\subsection{Proof of part \eqref{item:f-ell-one-norm-on-H} of Theorem \ref{thm:construct-test-function}}\label{sec:proof-part-eqref-1}
We recall \eqref{eqn:at-frakq-f-sharp-via-f0}:
\begin{equation}
  f^\sharp(g) := \int _{z \in Z}
  \omega_\pi(g) (f \ast f^*)(z g) \, d z,
  \tag{\ref{eqn:at-frakq-f-sharp-via-f0}}
\end{equation}
where $f^*(g) = \overline{f(g^{-1})}$ and $f \ast f^*$ denotes the convolution product.  Our task is to check that $\int_H |f^\sharp| \ll T^{n/2+\kappa}$.

By construction, $f$ is supported near the identity element of $G$ and satisfies the support condition that $f(\exp(x))$ vanishes unless $x ' \ll \h^{1 - \delta ' - \eps}$ and $x '' \ll \h^{1 - \delta '' - \eps}$.  These conditions imply in particular that $x \ll \h^{\nu}$, hence $\exp(x) = 1 + \O(\h^\nu)$ (recall from \eqref{eqn:nu-conditions} the defining properties of $\nu$).  By Lemma \ref{lem:approximate-x-prime-various}, we deduce further that $\exp(x) \tau = \tau + \O(\h^{1/2+\nu})$.  Thus every element $g$ of the support of $f$ satisfies
\begin{equation}\label{eqn:f-sharp-support}
  g = 1 + \O(\h^\nu),
  \quad
  g \tau = \tau + \O(\h^{1/2+\nu}).
\end{equation}
If $g$ satisfies \eqref{eqn:f-sharp-support}, then so does $g^{-1}$, and if two group elements satisfy \eqref{eqn:f-sharp-support}, then so does their product.  It follows that every element of the support of $f \ast f^*$ satisfies \eqref{eqn:f-sharp-support}.  The same is then true for every element of the support of $|f^\sharp| : \bar{G} \rightarrow \mathbb{R}_{\geq 0}$, where we now interpret the first estimate in \eqref{eqn:f-sharp-support} as taking place inside the adjoint group $\bar{G}$.

Recall from \S\ref{sec:prelim-reductive-groups} the notation $\dim, \dim_F, \rank, \rank_F$.  It follows from the basic Fourier estimate \eqref{eqn:a-h-vee-properties-for-test-vector-construction} that
\[\|f\|_{L^1(G)} \ll \|(a_{\h}')^\vee \|_{L^1(\mathfrak{g})} \ll 1\] and that
\[
  \|f\|_{L^\infty(G)} \ll \h^{(-1+\delta ') (\dim(\mathfrak{g}) - \rank(\mathfrak{g}))} \h^{(-1+\delta '') \rank(\mathfrak{g})} \ll T^{n(n+1)/2+\kappa},
\]
using here that $\dim_F(\mathfrak{g}) -\rank_F(\mathfrak{g}) = n(n+1)$ and that $(\delta ', \delta '')$ is close enough to $(1/2,1)$.  It follows that
\[
  \|f \ast f^*\|_{L^\infty(G)} \leq \|f\|_{L^1(G)} \|f\|_{L^\infty(G)} \ll T^{n(n+1)/2+\kappa}
\]
and hence also that
\begin{equation}\label{eqn:f-sharp-bound}
  \|f^\sharp\|_{\infty}
  \ll
  T^{n(n+1)/2 + \kappa},
\end{equation}
using for this last estimate that for a fixed compact subset $U$ of $G$ and all $g \in G$, we have $\vol \{z \in Z : g z \in U\} \ll 1$.

It follows from \eqref{eqn:f-sharp-support} and \eqref{eqn:f-sharp-bound} that
\[
  \int_H |f^\sharp| \ll T^{n(n+1)/2 +\kappa} \vol \{g \in H : |g \tau - \tau| \leq \h^{1/2} \},
\]
say.  Using the ``principal bundle'' consequence of stability (\S\ref{sec:cons-stab}), we see that
\[
  g \in H, \, |g \tau - \tau| \leq \h^{1/2} \quad \implies \quad |g - 1| \ll \h^{1/2},
\]
so the volume in question is $\ll \h^{(\dim H)/2} = T^{-n^2/2}$.  The required bound $\int_H |f^\sharp| \ll T^{n/2+\kappa}$ follows.

\subsection{Summary}
We have reduced the proof of our main local result (Theorem \ref{thm:construct-test-function}), hence that of our main global result (Theorem \ref{thm:main}), to establishing that the test function $f$ constructed above satisfies the conclusion of part \eqref{thm:f-item-3} of Theorem \ref{thm:construct-test-function}.  The proof is given in \S\ref{sec:proof-theor-refthm:c}.  The proof relies on Theorem \ref{thm:key-volume-bound}, whose proof is postponed to \S\ref{sec:volume-estimates}, which in turn relies on Theorem \ref{thm:stability-consequence-for-1-H}, whose proof is given in \S\ref{sec:some-invar-theory}.

\section{Bilinear forms estimates}\label{sec:bilinear-forms-estimates}

\subsection{Notation for distance functions and thickened centralizers}
Let $F$ be a local field of characteristic zero, let $G$ be an $F$-algebraic group, and let $\tau \in \mathfrak{g}^\wedge$.  We denote by $Z$ the center of $G$ and write $\bar{G} = G/Z$.  We define $D_\tau : \bar{G} \rightarrow \mathbb{R}_{\geq 0}$ by
\begin{equation*}
  D_\tau(g) := \max(\dist_F(g \tau,\tau), \dist_F(g^{-1} \tau, \tau)),
\end{equation*}
where we adopt the notation\index{Lie algebra!$\dist, \dist_F$}
\begin{equation*}
  \dist_F(\xi,\eta) := |\xi - \eta|_F,
  \quad
  \dist(\xi,\eta) := |\xi - \eta|.
\end{equation*}
We recall that $|.|_F$ denotes the normalized valuation, thus $|z|_F = |z|^{[F:F_0]}$ if $F_0$ is the prime subfield of $F$ (i.e., $F_0 = \mathbb{R}$ or $\mathbb{Q}_p$).  For a subset $S$ of $\mathfrak{g}^\wedge$, we set $\dist(\xi,S) := \inf_{\eta \in S} \dist(\xi,\eta)$ and define $\dist_F(\xi,S)$ similarly.

We have $D_\tau(g) = 0$ if and only if $g$ lies in the centralizer of $\tau$, so we may understand $D_\tau$ as quantifying the distance to that centralizer.  For $r \geq 0$, we define the ``thickened centralizer''
\begin{equation*}
  \mathcal{C}_\tau(r) = \left\{ g \in \bar{G} : D_\tau(g) \leq r \right\},
\end{equation*}
By construction, $\mathcal{C}_\tau(r)$ is symmetric.

\subsection{Statement of main bilinear forms estimate}

Let $F$ be a local field of characteristic zero.  Let $(G,H)$ be a unitary GGP pair over $F$.  Write, as usual, $Z$ (resp. $Z_H$) for the center of $G$ (resp. $H$), $\mathfrak{z}$ (resp. $\mathfrak{z}_H$) for the center of $\mathfrak{g}$ (resp. $\mathfrak{h}$), and $\bar{G} = G/Z$ for the adjoint group.  We assume that $Z_H$ and $H$ come equipped with Haar measures; volumes and integrals in what follows will be computed with respect to these measures.  We note that each of the centers $Z$ and $Z_H$ is one-dimensional over $F$.

Recall from \S\ref{sec:stability-defn} the definition of the stable subset $\mathfrak{g}^\wedge_{\stab} \subseteq \mathfrak{g}^\wedge$, as well as the properties of that subset summarized in \S\ref{sec:cons-stab}.  Recall from \S\ref{sec:quant-dist-subgr} the definition of $d_H : \bar{G} \rightarrow [0,1]$.

\begin{theorem}\label{thm:bilinear-forms-estimate}
  Let $\mathfrak{S}$ be a compact subset of $\mathfrak{g}^\wedge_{\stab}$.  There are bounded symmetric open neighborhoods $\mathcal{Z} \subseteq Z_H$ and $\mathcal{G} \subseteq \bar{G}$ of the respective identity elements, which may be chosen arbitrarily small, for which the following assertions hold.  Let $\Omega$ be a compact subset of $H$.  There exists $C \geq 0$ with the following property.  Let $u_1 ^0 , u_2 ^0 : H \rightarrow \mathbb{R}_{\geq 0}$ be nonnegative measurable functions.  Define the convolutions
  \begin{equation*}
    u_i(x) := \int _{z \in \mathcal{Z} } u_i ^0 (x z^{-1}) \, d z.
  \end{equation*}
  Let $\tau \in \mathfrak{S}$, $r \geq 0$ and $\gamma \in \bar{G} - H$.  Then the integral
  \begin{equation*}
    I := \int _{x, y \in \Omega } u_1(x) u_2(y) 1_{\mathcal{G} \cap \mathcal{C}_\tau(r)}  (x ^{-1} \gamma y) \,d x \, d y
  \end{equation*}
  satisfies the estimate
  \begin{equation}\label{eq:i-leq-c}
    I \leq C r^{\dim_F(H)} \min \left( 1, \frac{r}{ d_H(\gamma)} \right) \|u_1^0 \|_{L^2(\Omega \mathcal{Z} )} \|u_2^0 \|_{L^2(\Omega \mathcal{Z})}.
  \end{equation}
\end{theorem}

\subsection{Deduction of Theorem \ref{thm:construct-test-function}, part \eqref{thm:f-item-3}}\label{sec:proof-theor-refthm:c}
Let $(F,G,H,\mathfrak{S})$ be given and fixed, with $F$ archimedean.  Let $T \ggg 1$.  Fix $\kappa > 0$.  Let $f$ be as constructed in \S\ref{sec:construction-f}.  Our task is to verify part \eqref{thm:f-item-3} of Theorem \ref{thm:construct-test-function}.  We must show that for each fixed compact $\Omega \subseteq H$, there is a fixed compact $\Omega ' \subseteq H$ so that the estimate \eqref{eq:int-_x-y} holds.

We apply Theorem \ref{thm:bilinear-forms-estimate} to $(F,G,H,\mathfrak{S})$.  This yields $(\mathcal{Z},\mathcal{G})$, which may be taken fixed and bounded.  Let $\Omega \subseteq H$ be any fixed compact subset.  Theorem \ref{thm:bilinear-forms-estimate} yields some fixed $C$ for which \eqref{eq:i-leq-c} holds.  We will show that the required bound \eqref{eq:int-_x-y} holds for any fixed compact $\Omega '$ that contains $\Omega \mathcal{Z}$.

To that end, we recall from the discussion following \eqref{eqn:f-sharp-support} that $f ^\sharp (g) \neq 0$ only if
\begin{equation*}
  g = 1 + \O(\h^\nu), \quad  \dist(g \tau, \tau) \ll \h^{1/2+\nu} \quad \text{ and } \quad \dist(g^{-1} \tau, \tau) \ll \h^{1/2+\nu}.
\end{equation*}
In particular, recalling the relationship \eqref{eq:h-:=-t} between $T$ and $\h$, we then have
\begin{equation*}
  g \in \mathcal{G} \cap \mathcal{C}_\tau(T^{-1/2}).
\end{equation*}
We recall also the $L^\infty$-bound \eqref{eqn:f-sharp-bound}.  It follows that
\begin{equation*}
  f(g) \ll T^{n(n+1)/2 + \kappa} 1_{\mathcal{G} \cap \mathcal{C}_\tau(T^{-1/2})}(g).
\end{equation*}
Let $\Psi_1, \Psi_2$ be as in the hypotheses of \eqref{eq:int-_x-y}.  Let $u_i^0$ denote the function supported on $\Omega \mathcal{Z}$ and given there by $|\Psi_i|$.  For $x \in \Omega$, we have
\begin{equation*}
  u_i(x) = \int _{z \in \mathcal{Z} } |\Psi_i(x z)| \, d z
  = \vol(\mathcal{Z}) |\Psi_i(x)|,
\end{equation*}
using that $\Psi_i$ has a unitary central character.  Since $\mathcal{Z}$ is fixed and open, we have $\vol(\mathcal{Z}) \gg 1$.  It follows that for $x,y \in \Omega$,
\begin{equation*}
  \overline{\Psi_1(x)} \Psi_2(y) f^\sharp(x^{-1} \gamma y)
  \ll
  T^{n(n+1)/2 + \kappa}
  u_1(x) u_2(y)
  1_{\mathcal{G} \cap \mathcal{C}_\tau(T^{-1/2})}(x^{-1} \gamma y).
\end{equation*}
We now apply the conclusion \eqref{eq:i-leq-c} of Theorem \ref{thm:bilinear-forms-estimate}, taking the second alternative in the minimum.  We obtain
\begin{equation*}
  \int _{x, y \in \Omega }
  \left\lvert \overline{\Psi_1(x)} \Psi_2(y) f^\sharp(x^{-1} \gamma y) \right \rvert
  \, d x \, d y
  \ll
  \frac{T^{n/2 - 1/2 + \kappa}}{d_H(\gamma)}
  \|u_1^0\|_{L^2(\Omega \mathcal{Z})}
  \|u_2^0\|_{L^2(\Omega \mathcal{Z})}.
\end{equation*}
For any fixed compact neighborhood $\Omega '$ of $\Omega \mathcal{Z}$, we have $\|u_i\|^0_{L^2(\Omega \mathcal{Z})} \leq \|\Psi_i\|_{L^2(\Omega ')}$.  The required estimate \eqref{eq:int-_x-y} follows.

\subsection{Statement of volume bounds}
The key ingredient in the proof of Theorem \ref{thm:bilinear-forms-estimate} is the following.
\begin{theorem}\label{thm:key-volume-bound}
  Let $\mathfrak{S}$ be a compact subset of $\mathfrak{g}^\wedge_{\stab}$.  There exists $C \geq 0$ and neighborhoods $\mathcal{Z} \subseteq Z_H$ and $\mathcal{G} \subseteq \bar{G}$ of the respective identity elements so that for each $\tau \in \mathfrak{S}$, $g \in \mathcal{G}$ and $r \geq 0$, we have
  \begin{equation*}
    \vol  \left\{ z \in \mathcal{Z} : \dist(g z \tau, H \tau) \leq r \right\} \leq C \frac{r}{d_H(g)}.
  \end{equation*}
\end{theorem}
We note that the conclusion remains valid upon replacing $\mathcal{Z}$ and $\mathcal{G}$ with smaller neighborhoods, so we may assume in applications of Theorem \ref{thm:key-volume-bound} that these neighborhoods are small (in particular, bounded), open and symmetric.

\subsection{The near-identity case}
As a first step towards the proof of Theorem \ref{thm:bilinear-forms-estimate}, we treat what amounts to the special case in which $\Omega$ and $\gamma$ are taken sufficiently close to the identity element:

\begin{lemma}\label{prop:scratch-research:let-f-be}
  Let $F$ be a local field of characteristic zero.  Let $(G,H)$ be a unitary GGP pair over $F$, with the usual accompanying notation.  Let $\mathfrak{S}$ be a compact subset of $\mathfrak{g}^\wedge_{\stab}$.  There exists $C \geq 0$ and bounded symmetric open neighborhoods $\mathcal{Z} \subseteq Z_H$, $\mathcal{G} \subseteq \bar{G}$ and $\mathcal{H} \subseteq \mathcal{H} '' \subseteq H$ of the respective identity elements, which may be chosen arbitrarily small, so that the following assertions hold.  Let $u_1^0, u_2^0 : H \rightarrow \mathbb{R}_{\geq 0}$ be nonnegative measurable functions.  Define the convolutions
  \begin{equation*}
    u_i(y) := \int _{z \in \mathcal{Z} } u_i^0(y z^{-1}) \, d z.
  \end{equation*}
  Let $\tau \in \mathfrak{S}$, $g \in \mathcal{G}$, and $r \geq 0$.  Then the integral
  \begin{equation*}
    I := \int _{
      x, y \in \mathcal{H} 
    }
    u_1(x) u _2 (y)
    1_{\mathcal{C}_\tau(r)}(x^{-1} g y)
    \,d x \,d y
  \end{equation*}
  satisfies the estimate
  \begin{equation*}
    I \leq C r ^{\dim (H)} \min \left( 1, \frac{r}{d_H(g)} \right) \|u_1^0\|_{L^2(\mathcal{H} '')} \|u_2^0\|_{L^2(\mathcal{H} '')}.
  \end{equation*}
\end{lemma}
\begin{proof}
  By transfer (\S\ref{sec:some-axioms}), it is equivalent to prove the same estimate, but with quantification over $F$ through $C$ restricted to fixed quantities.  In other words, we assume that the quantities $F,G,H$ and $\mathfrak{S}$ are fixed, and aim to show that the claimed estimate holds for some $C \ll 1$.  The content of this assumption is that implied constants in asymptotic notation may depend upon such quantities.

  We apply Theorem \ref{thm:key-volume-bound} to $\mathfrak{S}$.  Let $C'$, $\mathcal{Z} '$ and $\mathcal{G} '$ be the results of that application; we may and shall assume that they are fixed and that $\mathcal{Z} '$ has volume at most one.  Then for all $g \in \mathcal{G} '$, we have
  \begin{equation}\label{eq:vol-left-z}
    \vol  \left\{ z \in \mathcal{Z}' : g z \in H \mathcal{C}_\tau(r) \right\} \ll \frac{r}{d_H(g)}.
  \end{equation}
  (We have used here that $g z \in H C_\tau(r) \implies \dist(g z \tau, H \tau) \leq r$.)  Let $\mathcal{Z} \subseteq Z_H$, $\mathcal{G} \subseteq \bar{G}$ and $\mathcal{H} \subseteq \mathcal{H} ' \subseteq \mathcal{H} '' \subseteq H$ be fixed bounded symmetric neighborhoods of the respective identity elements such that
  \begin{equation*}
    \mathcal{G} \mathcal{H}' \subseteq \mathcal{G} ',
    \quad
    \mathcal{Z} \mathcal{Z} \subseteq \mathcal{Z} ',
    \quad \mathcal{H} \mathcal{Z} \subseteq \mathcal{H} ',
    \quad
    \mathcal{H} ' \mathcal{Z} ' \subseteq \mathcal{H} ''.
  \end{equation*}
  We claim that the required estimate holds with such choices.

  Let $\tau \in \mathfrak{S}$, $g \in \mathcal{G}$ and $r \geq 0$.  Cauchy--Schwarz gives $I \leq \sqrt{I_1 I_2}$, where
  \begin{equation*}
    I_1 = \int _{
      x, y \in \mathcal{H}
    }
    u_1(x)^2
    1_{\mathcal{C}_\tau(r)}(x^{-1} g y)
    \, d x \, d y,
    \quad 
    I_2 = \int _{
      x, y \in \mathcal{H}
    }
    u_2(y)^2
    1_{\mathcal{C}_\tau(r)}(x^{-1} g y)
    \, d x \, d y.
  \end{equation*}
  We reduce to verifying that
  \begin{equation}\label{eq:i_i-ll-r}
    I_i \ll r ^{\dim_F(H)} \min \left( 1, \frac{r}{d_H(g)} \right) \|u_i^0\|_{L^2(\mathcal{H} '')}.
  \end{equation}

  Since the set $\mathcal{C}_\tau(r)$ is symmetric, we have $1_{\mathcal{C}_\tau(r)}(x^{-1} g y) = 1_{\mathcal{C}_\tau(r)}(y^{-1} g^{-1} x)$.  The case $i=2$ of \eqref{eq:i_i-ll-r} is thus obtained from the case $i=1$ by replacing $u_1^0$ with $u_2^0$ and $g$ with $g^{-1}$.  Since $d_H(g^{-1}) = d_H(g)$, it will suffice to treat the case $i=1$.  We have
  \begin{equation*}
    I_1 = \int _{x \in \mathcal{H} } u_1(x)^2
    \vol(\mathcal{H} \cap g^{-1} x \mathcal{C}_\tau(r))
    \, d x.
  \end{equation*}
  By the ``principal bundle'' consequence of stability (see \S\ref{sec:cons-stab}), we have
  \begin{equation*}
    \vol(\mathcal{H} \cap g^{-1} x \mathcal{C}_\tau(r)) \ll r^{\dim_F(H)} 1_{\mathcal{H} \mathcal{C}_\tau(r)}(g^{-1} x),
  \end{equation*}
  where the volume is computed with respect to the Haar measure on $H$.  Therefore
  \begin{equation*}
    I_1 \ll r^{\dim_F(H)} I_1', \quad I_1' := \int _{x \in \mathcal{H} } u_1(x)^2 1_{\mathcal{H} \mathcal{C}_\tau(r)}(g^{-1} x) \, d x.
  \end{equation*}
  The trivial bound
  \begin{equation*}
    I_1' \leq \|u_1\|_{L^2(\mathcal{H})}^2 \leq \vol(\mathcal{Z})^2 \|u_1^0\|^2_{L^2(\mathcal{H} ')}
    \leq \|u_1^0\|^2_{L^2(\mathcal{H} ')}
  \end{equation*}
  gives $I_1 \ll r^{\dim_F(H)} \|u_1^0\|_{L^2(\mathcal{H} '')}$, yielding part of the desired bound \eqref{eq:i_i-ll-r}.  To obtain the other part, we must verify that
  \begin{equation}\label{eq:i_1-ll-fracrdhg}
    I_1' \ll
    \frac{r}{d_H(g)}
    \|u_1^0\|^2_{L^2(\mathcal{H} ')}.
  \end{equation}
  To that end, we insert the definition of $u_1$ into the definition of $I_1'$, giving a triple integral over $x \in \mathcal{H}$ and $z_1, z_2 \in \mathcal{Z}$.  We then apply the change of variables $x \mapsto x z_1$ followed by $z_2 \mapsto z_2 z_1$.  These operations yield
  \begin{align*}
    I_1'
    &=
      \int _{x \in \mathcal{H} } \int _{z_1, z_2 \in \mathcal{Z} }
      u_1^0(x z_1^{-1})
      u_1^0(x z_2^{-1})
      1_{\mathcal{H} \mathcal{C}_\tau(r)} (g^{-1} x)
      \, d z_2 \, d z_1 \, d x
    \\
    &\leq  \int _{x \in \mathcal{H} '} u_1^0(x) \left( \int _{z_2 \in \mathcal{Z} '} u_1^0(x z_2^{-1}) \, d z_2 \right)
      \left( \int _{z_1 \in \mathcal{Z} } 1_{\mathcal{H} \mathcal{C}_\tau(r)} ( g^{-1} x z_1)
      \, d z_1
      \right)
      \, d x.
  \end{align*}
  We apply H\"{o}lder with exponents $(2,2,\infty)$ to obtain
  \begin{equation*}
    I_1' \leq \vol(\mathcal{Z} ') \|u_2^0\|_{L^2(\mathcal{H} '')}^2 V,
  \end{equation*}
  where
  \begin{equation*}
    V := \sup _{x \in \mathcal{H} '} \vol \left\{ z \in \mathcal{Z} : g^{-1} x z \in \mathcal{H} \mathcal{C}_\tau(r) \right\}.
  \end{equation*}
  For each $x \in \mathcal{H} '$, we have $g^{-1} x \in \mathcal{G} \mathcal{H} ' \subseteq \mathcal{G} '$.  We may thus apply \eqref{eq:vol-left-z} to see that
  \begin{equation*}
    V \ll \frac{r}{d_H(g^{-1} x)}.
  \end{equation*}
  Since $g$ (resp. $x$) lies in the fixed precompact subset $\mathcal{G} \subseteq \bar{G}$ (resp. $\mathcal{H}' \subseteq H$), we have by \eqref{eqn:d-H-h1gh2} that
  \begin{equation*}
    d_H(g^{-1} x) \asymp d_H(g^{-1}) = d_H(g).
  \end{equation*}
  The required bound \eqref{eq:i_1-ll-fracrdhg} follows.
\end{proof}

\subsection{Proof of Theorem \ref{thm:bilinear-forms-estimate}}
The basic idea is to split the domain into small pieces, reducing to Lemma \ref{prop:scratch-research:let-f-be}.

We may and shall assume that $(F,G,H,\mathfrak{S})$ are fixed.  We apply Lemma \ref{prop:scratch-research:let-f-be} to these.  Let $(C,\mathcal{Z},\mathcal{G}',\mathcal{H} , \mathcal{H} '')$ be the results.  After shrinking $\mathcal{H}$ if necessary, we may find a bounded symmetric open neighborhood $\mathcal{G} \subseteq \bar{G}$ for which
\begin{equation}\label{eq:mathc-mathc-mathc}
  \mathcal{H} \mathcal{G} \mathcal{H} \subseteq \mathcal{G}'.
\end{equation}
Let $\Omega$ be as indicated.  We may and shall assume that all of these quantities are fixed.  It will suffice then to verify that the required estimate \eqref{eq:i-leq-c} holds for some $C \ll 1$.  The integral to be estimated depends only upon the restrictions of the $u_i$ to the set $\Omega \mathcal{Z}$, so we may and shall assume that the $u_i$ vanish outside that set.  We thereby reduce to establishing the modified estimate obtained by replacing $L^2(\Omega \mathcal{Z})$ with $L^2(H)$, namely,
\begin{align*}
  &\int _{x, y \in \Omega } u_1(x) u_2(y) 1_{\mathcal{G} \cap \mathcal{C}_\tau(r)}  (x ^{-1} \gamma y) \,d x \, d y
  \\
  &\quad \ll r^{\dim_F(H)} \min \left( 1, \frac{r}{ d_H(\gamma)} \right) \|u_1^0 \|_{L^2(H)} \|u_2^0 \|_{L^2(H)}.
\end{align*}

We may cover the fixed compact set $\Omega$ by $\O(1)$ many translates $s \mathcal{H}$, with $s \in \Omega$, of the fixed neighborhood $\mathcal{H}$ of the identity.  We thereby reduce to verifying for each $s,t \in \Omega$ the corresponding bound for the ``localized'' integrals
\begin{equation*}
  I_{s,t} := \int _{x, y \in \mathcal{H} } u_1(s x) u_2(t y) 1_{\mathcal{G} \cap \mathcal{C}_\tau(r)}(x^{-1} g y)  \,d x \,d y,
  \quad
  g := s^{-1} \gamma t
\end{equation*}

Suppose given $x,y \in \mathcal{H}$ for which $x^{-1} g y \in \mathcal{G} \cap \mathcal{C}_\tau(r)$.  Then by \eqref{eq:mathc-mathc-mathc}, we have
\begin{equation}\label{eq:g-in-x}
  g \in x (\mathcal{G} \cap \mathcal{C}_\tau(r)) y^{-1} \subseteq \mathcal{G} '.
\end{equation}
We may thus apply Lemma \ref{prop:scratch-research:let-f-be} with $u_1^0$ and $u_2^0$ replaced by their left translates
\begin{equation*}
  u_1^0(s \cdot) := [x \mapsto u_1^0(s x)], \quad u_2^0(t \cdot) := [y \mapsto u_2^0(t y)] 
\end{equation*}
to see that
\begin{equation*}
  I_{s,t}
  \ll
  r^{\dim_F(H)}
  \min \left( 1, \frac{r }{ d_H(g) } \right) \|u_1^0(s \cdot)\|_{L^2(\mathcal{H}'')} \|u_2^0(t \cdot)\|_{L^2(\mathcal{H}'')}.
\end{equation*}
By \eqref{eq:g-in-x}, we see in particular that $g$ lies in a fixed compact subset of $\bar{G}$.  Since $s$ and $t$ are contained in the fixed compact subset $\Omega$ of $H$, we have by \eqref{eqn:d-H-h1gh2} that
\begin{equation*}
  d_H(g) \asymp d_H(\gamma).
\end{equation*}
We conclude by noting that
\begin{equation*}
  \|u_1^0(s \cdot)\|_{L^2(\mathcal{H}'')}
  \leq
  \|u_1^0\|_{L^2(H)},
  \quad 
  \|u_2^0(t \cdot)\|_{L^2(\mathcal{H}'')}
  \leq
  \|u_2^0\|_{L^2(H)}.
\end{equation*}

\section{Volume bounds}\label{sec:volume-estimates}
The purpose of this section is to prove Theorem \ref{thm:key-volume-bound}.

\subsection{Reformulation}
Before diving into the proof of Theorem \ref{thm:key-volume-bound}, it will be convenient to reformulate it using asymptotic notation and terminology:

\begin{theorem}[Reformulation of Theorem \ref{thm:key-volume-bound}]\label{thm:key-volume-bound-2}
  Let $(G,H)$ be a fixed unitary GGP pair over a fixed local field $F$ of characteristic zero.  Let $\tau$ belong to a fixed compact subset of $\mathfrak{g}^\wedge_{\stab}$.  Let $g$ be an element of $G$ with $g \simeq 1$.  Let $\mathcal{Z}$ be a subset of $Z_H$ that is contained in every fixed neighborhood of the identity.  Then for each $r \geq 0$,
  \begin{equation}\label{eqn:volume-estimate-with-H}
    \vol \left\{ z \in \mathcal{Z} : \dist_F(g z \tau, H \tau) \leq r \right\} \ll \frac{r}{d_H(g)}.
  \end{equation}
\end{theorem}

\begin{proof}[Verification of the equivalence between Theorems \ref{thm:key-volume-bound-2} and \ref{thm:key-volume-bound}]
  First of all, since the volume bound for a given set $\mathcal{Z}$ implies the same for any smaller set, Theorem \ref{thm:key-volume-bound} is obviously equivalent to the following:
  \begin{itemize}
  \item Let $(G,H)$ be a unitary GGP pair over a local field $F$ of characteristic zero.  Let $\mathfrak{S}$ be a compact subset of $\mathfrak{g}^\wedge_{\stab}$.  There exists $C \geq 0$ and neighborhoods $\tilde{\mathcal{Z}} \subseteq Z_H$ and $\mathcal{G} \subseteq \bar{G}$ of the respective identity elements so that for each $\tau \in \mathfrak{S}$, $g \in \mathcal{G}$, $r \geq 0$ and $\mathcal{Z} \subseteq \tilde{\mathcal{Z}}$, we have
    \begin{equation}\label{eq:vol-left-z-1}
      \vol  \left\{ z \in \mathcal{Z} : g z \in \dist(g z \tau, H \tau) \leq r \right\} \leq C \frac{r}{d_H(g)}.
    \end{equation}
  \end{itemize}
  By transfer (\S\ref{sec:some-axioms}), the above is equivalent to the same assertion, but with quantification over $G$ through $\mathcal{G}$ restricted to fixed quantities.  The conclusion is filtered with respect to $(C, \tilde{\mathcal{Z}}, \mathcal{G})$, so by idealization (\S\ref{sec:some-axioms}), we may commute quantification over those with that over $(\tau, g, r, \mathcal{Z})$.  We obtain the following further reformulation:
  \begin{itemize}
  \item Let $\mathfrak{S}$ be a fixed compact subset of $\mathfrak{g}^\wedge_{\stab}$.  Let $\tau \in \mathfrak{S}$.  Let $g \in \bar{G}$ and $r \geq 0$.  Let $\mathcal{Z} \subseteq Z_H$.  There is a fixed $C \geq 0$ and there are fixed neighborhoods $\tilde{\mathcal{Z}} \subseteq Z_H$ and $\mathcal{G} \subseteq \bar{G}$ of the respective identity elements so that if $g \in \mathcal{G}$ and $\mathcal{Z} \subseteq \tilde{\mathcal{Z}}$, then \eqref{eq:vol-left-z-1} holds.
  \end{itemize}
  The equivalence of this last assertion with Theorem \ref{thm:key-volume-bound-2} follows from elementary logic.  For instance, we use the equivalence of the following two statements:
  \begin{itemize}
  \item There is a $\tilde{\mathcal{Z}}$ so that if $\mathcal{Z} \subseteq \tilde{\mathcal{Z}}$, then $(\dotsb)$.
  \item If $\mathcal{Z} \subseteq \tilde{\mathcal{Z}}$ for all $\tilde{\mathcal{Z}}$, then $(\dotsb)$.
  \end{itemize}
\end{proof}

\subsection{Reduction to a size estimate}\label{sec:reduct-size-estim}
We may reduce the proof of Theorem \ref{thm:key-volume-bound-2} to that of the following.
\begin{theorem}\label{thm:key-z-g-bound-G-tau}
  Let $(G,H)$ be a fixed unitary GGP pair over a fixed local field $F$ of characteristic zero.  Let $\tau$ belong to a fixed compact subset of $\mathfrak{g}^\wedge_{\stab}$.  Let $g \in \bar{G}_\tau$ with $g \simeq 1$.  Let $z \in Z_H$ with $z \simeq 1$.  Then
  \begin{equation}\label{eqn:z-minus-1-Ad-g-minus-1-bound}
    |z - 1| \cdot |\Ad(g) - 1|
    \ll \dist(g z \tau, H \tau).
  \end{equation}
\end{theorem}
\begin{proof}[Proof of Theorem \ref{thm:key-volume-bound-2}, assuming Theorem \ref{thm:key-z-g-bound-G-tau}]
  Our hypotheses involve a subset $\mathcal{Z} \subseteq Z_H$ that is contained in every fixed neighborhood of the identity.  We will use in what follows that this condition on $\mathcal{Z}$ is equivalent to the implication $z \in \mathcal{Z} \implies z \simeq 1$.
  
  We first reduce to verifying that the required volume bound \eqref{eqn:volume-estimate-with-H} holds under the further hypothesis that
  \begin{equation}\label{eqn:hypothesis-g-tau-tau-r}
    \dist_F(g \tau , \tau) \ll r.
  \end{equation}
  Supposing for the moment that it does hold under that hypothesis, let us verify the general case of \eqref{eqn:volume-estimate-with-H}.  In doing so, we may assume that there exists $z_0 \in \mathcal{Z}$ and $h_0 \in H$ with $\dist_F(g z_0 \tau, h_0 \tau ) \leq r$, since otherwise the set whose volume we must estimate is empty.  Since $g \simeq 1$ and $z_0 \simeq 1$, we then have $h_0 \tau \simeq g z_0 \tau \simeq \tau$.  By stability, it follows that $h_0 \simeq 1$.  Setting $g' := h_0^{-1} g z_0 \in \bar{G}$, we see by construction that
  \[
    \dist_F(g' \tau, \tau) = \dist_F(h_0^{-1} g z_0 \tau, \tau) \asymp \dist_F(g z_0 \tau, h_0 \tau) \leq r,
  \]
  so $g'$ satisfies our hypothesis \eqref{eqn:hypothesis-g-tau-tau-r}.  We see from \eqref{eqn:d-H-h1gh2} that $d_H(g') \asymp d_H(g)$.  Suppose now that $z \in \mathcal{Z}$ satisfies $\dist_F(g z \tau, H \tau) \leq r$.  Then for some $r' \asymp r$, the element $z' := z_0^{-1} z$ satisfies $\dist_F(g' z' \tau, H \tau) \leq r'$.  Setting $\mathcal{Z} ' := \{z_0^{-1} z : z \in \mathcal{Z}\}$, we see that the hypotheses for \eqref{eqn:volume-estimate-with-H} and \eqref{eqn:hypothesis-g-tau-tau-r} are satisfied by $(g',\mathcal{Z} ', r')$, hence
  \begin{align*}
    \vol
    \left\{
    z \in \mathcal{Z} : \dist_F(g z \tau, H \tau) \leq r
    \right\}
    &\leq
      \vol
      \left\{
      z' \in \mathcal{Z}' : \dist_F(g' z' \tau, H \tau) \leq r'
      \right\} \\
    &\ll
      \frac{r '}{d_H(g')}
      \asymp \frac{r}{d_H(g)},
  \end{align*}
  as required.

  We now reduce further to the case that
  \begin{equation*}
    g \tau = \tau,
  \end{equation*}
  i.e., $g \in \bar{G}_\tau$.  Suppose that $g \in \bar{G}$ satisfies $g \simeq 1$ and \eqref{eqn:hypothesis-g-tau-tau-r}.  By Lemma \ref{lem:submersivity-orbit-map-regular-elt}, we may find $u \in \bar{G}$ so that
  \[
    u g \tau = \tau, \quad |\Ad(u) - 1|_F \ll r.
  \]
  Set
  \begin{equation*}
    g' := u g,
  \end{equation*}
  so that $g' \tau = \tau$.  For $z \in \mathcal{Z}$, we have
  \[
    \dist_F(g' \tau, g \tau) \ll |\Ad(u) - 1|_F \ll r,
  \]
  hence for some $r' \asymp r$, we have the implication
  \[
    \dist(g z \tau, H \tau) \leq r \quad \implies \quad \dist(g' z \tau, H \tau) \leq r'.
  \]
  We may assume that $r \lll d_H(g)$, because otherwise the required estimate \eqref{eqn:volume-estimate-with-H} follows from the fact that $\vol \mathcal{Z} \ll 1$.  Then $|\Ad(u) - 1|_F \lll d_H(g)$, so \eqref{eqn:d-H-Z-compare-u-g-g} implies that $d_H(g') \asymp d_H(g)$.  The estimate \eqref{eqn:volume-estimate-with-H} for $(g',\mathcal{Z},r')$ thus implies the same estimate for $(g,\mathcal{Z}, r)$.  This completes the proof of the claimed reduction.

  It remains to consider the case $g \in \bar{G}_\tau$.  It follows then from \eqref{eqn:d-H-Z-vs-dist-G-mod-Z} and \eqref{eqn:z-minus-1-Ad-g-minus-1-bound} that
  \[
    |z - 1|_F \cdot d_H(g) \ll |z - 1|_F \cdot |\Ad(g) - 1|_F \ll \dist_F(g z \tau, H \tau).
  \]
  Since $\dim_F \mathfrak{z}_H = 1$, we have for $r' \lll 1$ that
  \[\vol \{z \in \mathcal{Z} : |z - 1|_F \leq r'\} \ll r'.\]
  We may thus majorize the LHS of \eqref{eqn:volume-estimate-with-H} by $r / d_H(g)$, as required.  This completes the proof.
\end{proof}

The proof of Theorem \ref{thm:key-z-g-bound-G-tau} occupies the remainder of \S\ref{sec:volume-estimates}.

\subsection{A Lie-algebraic result}
Recall that
\begin{equation*}
  \mathfrak{g}_\xi := \{x \in \mathfrak{g} : [x,\xi] = 0\}
\end{equation*}
denotes the centralizer of $\xi \in \mathfrak{g}^\wedge$.  A key ingredient in the proof of Theorem \ref{thm:key-z-g-bound-G-tau} is the following Lie-algebraic result.
\begin{theorem}\label{thm:stability-consequence-for-1-H}
  Let $(G,H)$ be a unitary GGP pair over a field $F$ of characteristic zero.  Let $\tau \in \mathfrak{g}^*_{\stab}$, $x \in \mathfrak{g}_\tau - \mathfrak{z}$ and $0 \neq z \in \mathfrak{z}_H$.  Then
  \begin{equation}\label{eqn:x-1-H-tau-not-in-brak-h-tau}
    [x,[z,\tau]] \notin [\mathfrak{h},\tau].
  \end{equation}
\end{theorem}
The proof of Theorem \ref{thm:stability-consequence-for-1-H} occupies \S\ref{sec:some-invar-theory}.  When $F$ is a local field (so that the Pontryagin dual $\mathfrak{g}^\wedge$ is defined), the same conclusion holds for $\tau \in \mathfrak{g}^\wedge_{\stab}$ in view of the equivariant identifications noted in \S\ref{sec:stability-defn}.

The conclusion \eqref{eqn:x-1-H-tau-not-in-brak-h-tau} may be reformulated as follows:
\begin{equation*}
  [x,z] \notin \mathfrak{h} + \mathfrak{g}_\tau.
\end{equation*}
Indeed, since $x$ centralizes $\tau$, we have $[x,[z,\tau]] = [[x,z], \tau]$, so \eqref{eqn:x-1-H-tau-not-in-brak-h-tau} asserts that $[x,z]$ is not congruent modulo $\mathfrak{h}$ to an element of $\mathfrak{g}_\tau$.  We will pass between the two formulations freely.  We note also that, since $\mathfrak{z}_H$ is a one-dimensional vector space over $F$, it is equivalent to ask that the conclusion holds for some given nonzero element $z$.

\subsection{Application of the implicit function theorem}\label{sec:appl-impl-funct}
As in \S\ref{sec:orth-compl-tau-coord}, we fix an ``orthogonal complement'' assignment $\flat$ on the set of subspaces of $\mathfrak{g}^\wedge$.  We fix a basis element $0 \neq z \in \mathfrak{z}_H$.

For $\tau \in \mathfrak{g}^\wedge_{\stab}$, we denote by
\begin{equation*}
  \rho_\tau : \bar{G}_\tau \rightarrow \mathfrak{g}^\wedge
\end{equation*}
the map given by the composition
\begin{equation*}
  \bar{G}_\tau \rightarrow [\mathfrak{g},\tau] / [\mathfrak{h},\tau] \cong [\mathfrak{h},\tau]^{\flat} \subseteq \mathfrak{g}^\wedge,
\end{equation*}
where the first arrow is given by
\begin{equation*}
  g \mapsto [\Ad(g) z, \tau] \pmod{[\mathfrak{h},\tau]}.
\end{equation*}

\begin{lemma}\label{lem:appl-impl-funct}
  Let $\tau$ belong to a fixed compact subset of $\mathfrak{g}^\wedge_{\stab}$.  There is a fixed neighborhood $\mathcal{G} \subseteq \bar{G}$ of the identity element so that for all $g \in \mathcal{G} \cap \bar{G}_\tau$, we have
  \begin{equation}\label{eqn:compare-dist-G-mod-Z-via-g-tau-2}
    |\Ad(g) - 1|
    \asymp
    |\rho_\tau(g)|.
  \end{equation}
\end{lemma}
\begin{proof}
  We observe first that for each $\tau \in \mathfrak{g}^\wedge_{\stab}$, the map $\rho_\tau$ has injective derivative at the identity.  Indeed, it is equivalent to check that the linear map
  \begin{equation*}
    \mathfrak{g}_\tau / \mathfrak{z} \rightarrow  [\mathfrak{g},\tau] / [\mathfrak{h},\tau]
  \end{equation*}
  \begin{equation}\label{eq:x-mapsto-x-tau-1_h-pmodmathfrakh-tau-}
    x \mapsto [[x,z],\tau] \pmod{[\mathfrak{h},\tau]}
  \end{equation}
  has trivial kernel, i.e., that for all $0 \neq x \in \mathfrak{g}_\tau / \mathfrak{z}$, we have $[[x,z],\tau] \notin [\mathfrak{h},\tau]$.  This is precisely the conclusion of Theorem \ref{thm:stability-consequence-for-1-H} in view of the identity $[x,[z,\tau]] = [[x,z],\tau]$.

  We now apply Lemma \ref{lem:FTOC-families-nonstandard} with $(S,V,V') = (\mathfrak{g}^\wedge_{\stab}, \End(\mathfrak{g}), \mathfrak{g}^\wedge)$, $X = \{(g,s) : s \in S, g \in \bar{G}_s\}$, $X' = \mathfrak{g}^\wedge \times S$ and $f : X \rightarrow X'$ given by $f(g,s) := (\rho_s(g), s)$, with $X$ embedded in $V$ via the adjoint representation.  The submersivity of $X \rightarrow S$ follows from the fact that ``stable'' implies ``regular'' (Lemma \ref{lem:stab-impl-reg}), the definition of ``regular'' (\S\ref{sec:prelim-reductive-groups}), and the rank theorem.  Let $\mathfrak{S} \subseteq \mathfrak{g}^\wedge_{\stab}$ be a fixed compact set that contains $\tau$, and take $E := \mathfrak{S} \times \{1\}$.  Lemma \ref{lem:FTOC-families-nonstandard} yields a fixed neighborhood $U$ of $E$ such that, in particular, \eqref{eqn:compare-dist-G-mod-Z-via-g-tau-2} holds whenever $g \in U$.  Since $\mathfrak{S}$ is compact, we may find a fixed neighborhood $\mathcal{G} \subseteq \bar{G}$ of the identity such that for $X \cap (\bar{G} \times \mathfrak{S}) \subseteq U$.  The required conclusion holds with this choice of $\mathcal{G}$.
\end{proof}

\subsection{Uniform Lie-algebraic transversality}
For the following lemma, let $\mathcal{Y}$ denote the manifold consisting of pairs $(Y_1,Y_2)$ of $F$-subspaces $Y_1, Y_2$ of $\mathfrak{g}$ with $\dim_F(Y_1) = n+1$ and $\dim_F(Y_2) = n^2 + 1$.  Thus $\mathcal{Y}$ is a product of two Grassmannians.  In particular, $\mathcal{Y}$ is compact.  Let $\mathcal{Y}_0 \subseteq \mathcal{Y}$ denote the open submanifold consisting of pairs $(Y_1,Y_2)$ for which $Y_1 \cap Y_2 = \{0\}$.
\begin{lemma}\label{eqn:uniform-transversality-closure}
  Let $\mathfrak{S}$ be a compact subset of $\mathfrak{g}^\wedge_{\stab}$.  There is a neighborhood $\mathcal{G} \subseteq \bar{G}$ of the identity with the following properties.
  \begin{enumerate}[(i)]
  \item For all $\tau \in \mathfrak{S}$ and $1 \neq g \in \mathcal{G} \cap \bar{G}_\tau$, we have
    \begin{equation}\label{eq:adg-mathfr-nots}
      \Ad(g) \mathfrak{z}_H \not\subseteq \mathfrak{h} + \mathfrak{g}_\tau.
    \end{equation}
    In particular, the space
    \begin{equation}\label{eq:wg-:=-mathfrakh}
      W(g) := \mathfrak{h} + \Ad(g) \mathfrak{z}_H.
    \end{equation}
    is $(n^2+1)$-dimensional and intersects $\mathfrak{g}_\tau$ trivially, so the set of pairs
    \begin{equation*}
      \mathcal{X} := \{(\mathfrak{g}_\tau, W(g)) : \tau \in \mathfrak{S}, 1 \neq g \in \mathcal{G} \}
    \end{equation*}
    is a subset of $\mathcal{Y}_0$.
  \item Let $\bar{\mathcal{X}}$ denote the closure in $\mathcal{X}$ in the compact manifold $\mathcal{Y}$.  Then
    \begin{equation}\label{eqn:bar-X-inside-Y-0}
      \bar{\mathcal{X}} \subseteq \mathcal{Y}_0.
    \end{equation}
  \end{enumerate}
\end{lemma}
\begin{proof}
  By transfer (\S\ref{sec:some-axioms}), we may assume that $(F,G,H,\mathfrak{S})$ are fixed.  By Lemma \ref{lem:appl-impl-funct}, there is a neighborhood of the identity $\mathcal{G} \subseteq \bar{G}$ so that for all $\tau \in \mathfrak{S}$ and $g \in \mathcal{G} \cap \bar{G}_\tau$, the estimate \eqref{eqn:compare-dist-G-mod-Z-via-g-tau-2} holds.  By shrinking $\mathcal{G}$ if necessary, we may and shall assume that it is compact.

  We fix $0 \neq z \in \mathfrak{z}_H$, and define $\rho_\tau$ as in \S\ref{sec:appl-impl-funct}.  For $\tau \in \mathfrak{S}$ and $1 \neq g \in \mathcal{G} \cap \bar{G}_\tau$, the estimate \eqref{eqn:compare-dist-G-mod-Z-via-g-tau-2} implies in particular that $\rho_\tau(g) \neq 0$, i.e., $[\Ad(g) z, \tau] \notin [\mathfrak{h},\tau]$, or equivalently, $\Ad(g) z \notin \mathfrak{h} + \mathfrak{g}_\tau$.  Thus \eqref{eq:adg-mathfr-nots} holds.

  It remains to verify \eqref{eqn:bar-X-inside-Y-0}.  We appeal to the sequential criterion for closedness.  Let $(\tau_j, g_j)$ be a sequence of pairs with $\tau_j \in \mathfrak{S}$ and $1 \neq g_j \in \mathcal{G} \cap \bar{G}_{\tau_j}$ such that the sequence of pairs $(g_{\tau_j}, W(g_j))$ has a limit in $\mathcal{Y}$.  We must check that that this limit lies in $\mathcal{Y}_0$.  Since $\mathfrak{S}$ and $\mathcal{G}$ are compact, we may pass to a subsequence of the sequences $\tau_j$ and $g_j$ to reduce to the case that these sequence tend to respective limits $\tau \in \mathfrak{S}$ and $g \in \mathcal{G}$.  If $g \neq 1$, then the pair $(\mathfrak{g}_{\tau_j}, W(g_j)) \in \mathcal{X}$ tends to $(\mathfrak{g}_\tau, W(g)) \in \mathcal{X} \subseteq \mathcal{Y}_0$.  It remains to consider the case $g = 1$.  Let $S$ denote the unit sphere in $\mathfrak{g}/\mathfrak{z}$ with respect to some fixed norm.  For $j$ large enough, we may write $g_j = \exp(t_j x_j)$, where $t_j \in F^\times$ and $x_j \in S$.  Thus $t_j \rightarrow 0$.  We may assume that $x_j$ tends to some limit $x \in S$.  Since $x_j \in \mathfrak{g}_{\tau_j}/\mathfrak{z}$, we have $x \in \mathfrak{g}_{\tau}/\mathfrak{z}$.  Since $x \neq 0$, we know from Theorem \ref{thm:stability-consequence-for-1-H} that
  \begin{equation}\label{eq:x-1_h-tau} [[x,z],\tau] \notin [\mathfrak{h},\tau].
  \end{equation}
  In particular, $[x,z] \notin \mathfrak{h}$.  Since $z$ spans the one-dimensional $F$-vector space $\mathfrak{z}_H \subseteq \mathfrak{h}$, we have
  \begin{equation*}
    W(g_j) = \mathfrak{h} + F \Ad(g) z = \mathfrak{h} + F y_j,
  \end{equation*}
  where
  \begin{equation*}
    y_j := t_j^{-1} (\Ad(g_j) z - z).
  \end{equation*}
  By the Taylor expansion $\Ad(g_j) = \exp(\ad(t_j x_j)) = 1 + \ad(t_j x_j) + \dotsb$, the limit of $y_j$ exists and is given by
  \begin{equation*}
    \lim_{j \rightarrow \infty} y_j = y :=  [x,z].
  \end{equation*}
  By \eqref{eq:x-1_h-tau}, we have $y \notin \mathfrak{h}$.  It follows that the limit of $W(g_j)$ equals $\mathfrak{h} + [x, \mathfrak{z}_H]$.  By another application of \eqref{eq:x-1_h-tau}, we see that this last space intersects $\mathfrak{g}_\tau$ trivially.  Thus the limit of $(g_{\tau_j}, W(g_j))$ in $\mathcal{Y}$ belongs to $\mathcal{Y}_0$, as required.
\end{proof}

\begin{lemma}\label{cor:nice-complement-W}
  Let $\tau$ belong to some fixed compact subset of $\mathfrak{g}^\wedge_{\stab}$.  Let $g \in \bar{G}_\tau$ with $g \simeq 1$.  There is a subspace $W$ of $\mathfrak{g}$ with the following properties.
  \begin{enumerate}[(i)]
  \item $W \supseteq \mathfrak{h} \cup \Ad(g) \mathfrak{z}_H$.
  \item $W$ is complementary to $\mathfrak{g}_\tau$, i.e., $W \cap \mathfrak{g}_\tau = \{0\}$ and $\mathfrak{g} = W \oplus \mathfrak{g}_\tau$.
  \item The pair $(\mathfrak{g}_\tau, W)$ lies in some fixed compact collection of pairs of complementary subspaces of $\mathfrak{g}$, regarded as a subset of a product of Grassmannians.
  \end{enumerate}
\end{lemma}
\begin{proof}
  Let $\mathfrak{S}$ be a fixed compact subset of $\mathfrak{g}^\wedge_{\stab}$ that contains $\tau$.  We apply Lemma \ref{eqn:uniform-transversality-closure}.  It yields a neighborhood $\mathcal{G}$ of the identity, which we may assume is fixed.  Since $g \simeq 1$, we then have $g \in \mathcal{G}$.  Set $W_0 := W(g)$ as in \eqref{eq:wg-:=-mathfrakh}, so that $W_0$ is a subspace of $\mathfrak{g}$ that contains $\mathfrak{h} \cup \Ad(g) \mathfrak{z}_H$ and has trivial intersection with $\mathfrak{g}_\tau$.  Fix an ``orthogonal complement'' assignment $\flat$ on the set of subspaces of $\mathfrak{g}$, as in \S\ref{sec:orth-compl-tau-coord}.  Set $W_1 := (W_0 \oplus \mathfrak{g}_\tau)^{\flat}$ and $W := W_0 \oplus W_1$.  Then $W$ is a subspace of $\mathfrak{g}$ complementary to $\mathfrak{g}_\tau$.

  Clearly (i) and (ii) are satisfied.  The remaining assertion, (iii), is a consequence of \eqref{eqn:bar-X-inside-Y-0}.  Indeed, let $\mathcal{Z}$ denote the image of $\bar{\mathcal{X}}$ under the map $(U, V_0) \mapsto (U, V)$, where $V$ is defined in terms of $V_0$ by first defining $V_1 := (V_0 \oplus U)^{\flat}$ and then taking $V := V_0 \oplus V_1$.  Then $\mathcal{Z}$ is a fixed collection of complementary subspaces of $\mathfrak{g}$.  It is compact, since it is the continuous image of the compact set $\bar{\mathcal{X}}$.  It contains $(\mathfrak{g}_\tau, W)$, by construction.
\end{proof}

\subsection{Exponential coordinates on a coadjoint orbit}\label{sec:expon-coord-coadj}
Let $\tau \in \mathfrak{g}^\wedge_{\reg}$.  Recall that the tangent space to the coadjoint orbit $\mathcal{O}_\tau$ at $\tau$ may be described as
\[
  T_\tau(\mathcal{O}_\tau) = [\mathfrak{g},\tau] =\mathfrak{g}_\tau^\perp.
\]
We have $[\mathfrak{g},\tau] \cong \mathfrak{g}/\mathfrak{g}_\tau$, so the choice of subspace $W$ of $\mathfrak{g}$ complementary to $\mathfrak{g}_\tau$ (e.g., $W = \mathfrak{g}_\tau^{\flat}$ as defined in \S\ref{sec:orth-compl-tau-coord}) induces a map
\[
  \exp_{(\tau,W)} : \mathfrak{g}_\tau^\perp \rightarrow \mathcal{O}_\tau
\]
as follows.  Each element of $\mathfrak{g}_\tau^\perp$ may be written uniquely in the form $[x,\tau]$ for some $x$ in our chosen complement $W$.  We then set
\[
  \exp_{(\tau,W)} [x,\tau] := \exp(x) \tau \quad (x \in W).
\]

We observe that the derivative of such a map at the origin is the identity map:
\begin{equation}\label{eq:d-exp_tau-w_0}
  (d \exp_{(\tau,W)})_0 = \id : \mathfrak{g}_\tau^\perp \rightarrow T_\tau(\mathcal{O}_\tau) = \mathfrak{g}_\tau^\perp.
\end{equation}
Indeed, $\exp(x) \tau := \Ad(\exp(x)) \tau = \exp(\ad(x)) \tau = \tau + [x, \tau] + \dotsb$.

Informally, the following estimate says that $\exp_{(\tau,W)}$ approximately preserves distances near the origin provided that $W$ is ``separated enough'' from $\mathfrak{g}_\tau$.  (Conversely, one can check that such an assumption on $W$ is necessary.)
\begin{lemma}\label{lem:exp-tau-W}
  Let $\tau$ belong to a fixed compact subset of $\mathfrak{g}^\wedge_{\reg}$.  Let $W$ be a subspace of $\mathfrak{g}$ complementary to $\mathfrak{g}_\tau$.  Assume that $(\mathfrak{g}_\tau, W)$ lies in some fixed compact collection of pairs of complementary subspaces of $\mathfrak{g}$.  Then for all $\xi, \eta \in \mathfrak{g}_\tau^{\perp}$ with $\xi, \eta \lll 1$, we have
  \begin{equation}\label{eqn:xi-eta-exp-tau-W}
    |\xi - \eta| \asymp |\exp_{(\tau,W)} \xi - \exp_{(\tau,W)} \eta|.
  \end{equation}
\end{lemma}
\begin{proof}
  This follows readily from Lemma \ref{lem:FTOC-families-nonstandard} and the observation \eqref{eq:d-exp_tau-w_0}.  We record the details for completeness.  By hypothesis, we can find a fixed compact collection $S$ consisting of pairs $(\tau', W')$, with $\tau ' \in \mathfrak{g}^\wedge_{\reg}$ and $W'$ a subspace of $\mathfrak{g}$ complementary to $\mathfrak{g}_\tau$, so that $(\tau,W) \in S$.  Take
  \begin{equation*}
    X := \left\{ (x,\tau',W') : (\tau ', W') \in S, x \in \mathfrak{g}_{\tau'}^{\perp} \right\},
  \end{equation*}
  \begin{equation*}
    Y :=  \mathfrak{g}^\wedge \times S.
  \end{equation*}
  The natural map $X \rightarrow S$ is submersive (i.e., has everywhere surjective derivative) by the rank theorem, while $Y \rightarrow S$ is obviously submersive.  The map $f : X \rightarrow Y$ given by
  \begin{equation*}
    f(x, \tau ', W') := (\exp_{(\tau ', W')}(x), \tau ', W')
  \end{equation*}
  is analytic.  We are thus in the setting of Lemma \ref{lem:FTOC-families-nonstandard}.  The fibral maps $f_s : X_s \rightarrow Y_s$ are the maps $\exp_{(\tau ', W')}$ of interest.  Let $E \subseteq X$ be the fixed compact set consisting of all tuples $(0, \tau ', W')$ with $(\tau ', W') \in S$.  We have noted in \eqref{eq:d-exp_tau-w_0} that $(d f_s)_0$ is the identity map for all $s \in S$.  By Lemma \ref{lem:FTOC-families-nonstandard}, we can find a fixed neighborhood $E \subseteq U \subseteq X$ so that the estimate \eqref{eqn:xi-eta-exp-tau-W} holds whenever $(\xi, \tau, W)$ and $(\eta, \tau, W)$ lie in $U$.  Since $S$ is compact, we deduce in particular that \eqref{eqn:xi-eta-exp-tau-W} holds when $\xi, \eta \lll 1$.
\end{proof}

\subsection{Proof of Theorem \ref{thm:key-z-g-bound-G-tau}}
We recall the setting of Theorem \ref{thm:key-z-g-bound-G-tau}: $\tau$ belongs to some fixed compact subset of $\mathfrak{g}^\wedge_{\stab}$, $g \in \bar{G}_\tau$ with $g \simeq 1$, and $z \in Z_H$ with $z \simeq 1$.

It is enough to show for each $h \in H$ that
\begin{equation}\label{eq:leftlvert-z-1}
  \left\lvert   z - 1 \right\rvert \cdot \lvert \Ad(g) - 1 \rvert \ll \left\lvert g z \tau - h \tau  \right\rvert.
\end{equation}
Since $z \simeq 1$ and $g \simeq 1$, the required estimate is trivial unless $h \tau \simeq g z \tau$.  Since $g z \tau \simeq \tau$, it follows that $h \tau \simeq \tau$, hence by the ``principal bundle'' consequence of stability (see \S\ref{sec:cons-stab}) that $h \simeq 1$.  In particular, we may write $z = \exp(\log(z))$ and $h = \exp(\log(h))$ for some $\log(z) \in \mathfrak{z}_H$ and $\log(h) \in \mathfrak{h}$ with $\log(z), \log(h) \simeq 0$.  Then
\[
  g z \tau = g z g^{-1} \tau = \exp(\Ad(g) \log(z)) \tau, \quad h \tau = \exp(\log(h)) \tau.
\]

Let $W$ be the subspace of $\mathfrak{g}$ produced by Lemma \ref{cor:nice-complement-W}.  The hypotheses of Lemma \ref{lem:exp-tau-W} are then satisfied by $\tau$ and $W$.  Since $g,z,h \simeq 1$ and $\Ad(g) \log(z), \log(h) \in W$, we deduce that
\begin{equation*}
  \lvert g z \tau - h \tau \rvert \asymp \lvert [\Ad(g) \log(z),\tau] - [\log(h), \tau] \rvert.
\end{equation*}
Let $\pi_\tau : \mathfrak{g}^\wedge \rightarrow [\mathfrak{h},\tau]^{\flat}$ denote the composition of the natural projection $\mathfrak{g}^\wedge \rightarrow \mathfrak{g}^\wedge / [\mathfrak{h},\tau]$ with the natural isomorphism of the latter onto $[\mathfrak{h},\tau]^{\flat}$.  By the minimum property \eqref{eq:minimum-norm-for-orth-compl} for norms with respect to orthogonal complements, we have
\begin{equation*}
 \Abs{ \pi_\tau(\xi) } = \min _{x \in \mathfrak{h} } \Abs{ \xi - [x,\tau] }.
\end{equation*}
In particular,
\begin{equation*}
  \Abs{ [\Ad(g) \log(z),\tau] - [\log(h), \tau] }
  \geq
  \Abs{ \pi_\tau([\Ad(g) \log(z),\tau] ) }.
\end{equation*}
Fix a basis element $0 \neq z_0 \in \mathfrak{z}_H$.  (We apologize for any notational confusion resulting from that $z \in Z_H$ but $z_0 \in \mathfrak{z}_H$.)  By linearity, we have
\begin{equation*}
  \Abs{ \pi_\tau([\Ad(g) \log(z),\tau] }
  =
  \frac{\Abs{ \log(z) }}{\Abs{ z_0 }}
  \Abs{ \pi_\tau([\Ad(g) z_0,\tau] ) }.
\end{equation*}
By the assumption $z \simeq 1$, we have
\begin{equation*}
  \Abs{ \log(z) } \asymp \Abs{ z - 1 }.
\end{equation*}
We define $\rho_\tau$ (\S\ref{sec:appl-impl-funct}) using the element $z_0$, so that
\begin{equation*}
  \pi_\tau([\Ad(g) z_0,\tau])
  = \rho_\tau(g).
\end{equation*}
Recall the key estimate \eqref{eqn:compare-dist-G-mod-Z-via-g-tau-2}:
\begin{equation*}
  \Abs{ \rho_\tau(g) } \asymp  \Abs{ \Ad(g) - 1 }.
\end{equation*}
Combining the above estimates, inequalities and identities, we obtain the required estimate \eqref{eq:leftlvert-z-1}.

\section{Lie-algebraic considerations}\label{sec:some-invar-theory}

In this section, we prove Theorem \ref{thm:stability-consequence-for-1-H}, which is all that remains.  We retain the notation and setting of that theorem; in particular, $(G,H)$ is a unitary GGP pair over a field $F$ of characteristic zero, $Z \leq G$ and $Z_H \leq H$ denote the centers, fraktur letters denote Lie algebras, and ``stable,'' denoted by a subscripted $\stab$, refers to Definition \ref{defn:stable-elements}.

Theorem \ref{thm:stability-consequence-for-1-H} asserts, for each stable $\tau$, that a certain system of linear equations over $F$ admits no solution.  It is equivalent to ask that the same system has no solutions over any extension field of $F$, e.g., an algebraic closure.  For this reason, we may and shall assume that $F$ is algebraically closed.  This assumption will be used sparingly, but it seemed simpler to impose it at the outset rather than at each place where it is used.

\subsection{Reduction to a stability characterization}\label{sec:reduct-stab-char}
We may reduce the proof of Theorem \ref{thm:stability-consequence-for-1-H} to that of the following.
\begin{theorem}\label{thm:stab-characterization-1}
  Let $(G,H)$ be a unitary GGP pair over an algebraically closed field $F$ of characteristic zero.  Let $\tau \in \mathfrak{g}^*$ and $0 \neq z \in \mathfrak{z}_H$.  The following are equivalent.
  \begin{enumerate}[(i)]
  \item $\tau$ is stable.
  \item Both $\tau \in \mathfrak{g}^*$ and its restriction $\tau_H \in \mathfrak{h}^*$ are regular, and the bilinear form
    \begin{equation}\label{eqn:pairing-relevant-for-stability}
      \mathfrak{g}_\tau   /  \mathfrak{z}
      \otimes \mathfrak{h}_{\tau_H} \rightarrow F      
    \end{equation}
    \begin{equation*}
      (x,y) \mapsto \langle [z, [x,y]], \tau  \rangle
    \end{equation*}
    is nondegenerate.
  \end{enumerate}
\end{theorem}

\begin{proof}[Proof of Theorem
  \ref{thm:stability-consequence-for-1-H},
  assuming Theorem \ref{thm:stab-characterization-1}]
  We first verify (using that $\tau \in \mathfrak{g}^*_{\stab}$) the identity
  \begin{equation}\label{eqn:frak-h-tau-equals-perp-cap-perp}
    [\mathfrak{h},\tau]
    = \{\xi \in \mathfrak{g}_\tau^\perp :
    \xi|_{\mathfrak{h}_{\tau_H}} =0 \}
    =: \mathfrak{g}_\tau^\perp \cap \mathfrak{h}_{\tau_H}^{\perp},
  \end{equation}
  where as usual $\mathfrak{g}_\tau^\perp := \{\xi \in \mathfrak{g}^* : \langle x, \xi \rangle = 0 \text{ for all } x \in \mathfrak{g}_\tau \}$.  This identity is implicit in \cite[\S16]{nelson-venkatesh-1}; we collect the proof here for convenience.  We recall from \S\ref{sec:cons-stab} that
  \begin{itemize}
  \item the map $\mathfrak{g}^*_{\stab} \rightarrow \{\text{stable } (\lambda,\mu) \in [\mathfrak{g}^*] \times [\mathfrak{h}^*]\}$ is a principal $H$-bundle over its image, and that
  \item the fiber of that map over $(\lambda,\mu)$, if nonempty, is the $H$-torsor $\mathcal{O}^{\lambda,\mu}$ defined as in \eqref{eqn:O-lambda-mu}.
  \end{itemize}
  It follows that
  \begin{itemize}
  \item $[\mathfrak{h},\tau]$ is the tangent space $T_\tau(H \cdot \tau)$ at $\tau$ to the subvariety $H \cdot \tau$ of $\mathfrak{g}^*$, and
  \item $H \cdot \tau = \{\xi \in \mathfrak{g}^* : [\xi] = [\tau], [\xi_H] = [\tau_H] \}$.
  \end{itemize}
  Since $\tau$ is stable, we see from the ``principal bundle'' assertion that the map $\mathfrak{g}^* \rightarrow [\mathfrak{g}^*] \times [\mathfrak{h}^*]$ has surjective derivative at $\tau$.  Thus $T_\tau(H \cdot \tau)$ identifies with the kernel of that derivative.  Since $\tau$ and $\tau_H$ are regular (Lemma \ref{lem:stab-impl-reg}), the maps $\mathfrak{g}^* \rightarrow [\mathfrak{g}^*]$ and $\mathfrak{h}^* \rightarrow [\mathfrak{h}^*]$ have surjective derivatives at $\tau$ and $\tau_H$, respectively (Theorem \ref{thm:kostant-regular-elements}).  The kernels of those derivatives are thus $T_\tau(G \cdot \tau) = \mathfrak{g}_\tau^\perp$ and $T_{\tau_H}(H \cdot \tau_H) = \mathfrak{h}_{\tau_H}^\perp$.  The required identity follows.

  We now deduce \eqref{eqn:x-1-H-tau-not-in-brak-h-tau}.  Let $x \in \mathfrak{g}_\tau - \mathfrak{z}$ and $0 \neq z \in \mathfrak{z}_H$.  Suppose that $[x,[z,\tau]]$ lies in $[\mathfrak{h},\tau]$.  By \eqref{eqn:frak-h-tau-equals-perp-cap-perp}, $[x,[z,\tau]]$ is in particular orthogonal to $\mathfrak{h}_{\tau_H}$, and so for any $y \in \mathfrak{h}_{\tau_H}$, we have
  \[
    0 = \langle y, [x,[z,\tau]] \rangle = -\langle [x,y], [z,\tau] \rangle = \langle [z,[x,y]], \tau \rangle.
  \]
  Thus the image of $x$ in $\mathfrak{g}_\tau / \mathfrak{z}$ is a nonzero element of the kernel of the pairing \eqref{eqn:pairing-relevant-for-stability}.  But since $\tau$ is stable, we know by Theorem \ref{thm:stab-characterization-1} that that pairing is nondegenerate.  We thereby obtain the required contradiction.
\end{proof}

\subsection{Reduction to a determinantal identity}\label{sec:reduct-determ-ident}
Retain the setting of Theorem \ref{thm:stab-characterization-1}.  Since $F$ is algebraically closed, we may assume (see \S\ref{sec:unitary-groups-ggp}) that
\begin{equation}\label{eqn:everything-is-general-linear}
  (G,H) = (\GL_{n+1}({F}), \GL_n({F})),
  \quad 
  (\mathfrak{g},\mathfrak{h}) = (\glLie_{n+1}({F}), \glLie_n({F})),
\end{equation}
with $\GL_n$ included in $\GL_{n+1}$ in the usual way as the upper-left $n \times n$ block.  We denote by $\langle , \rangle$ the trace pairing on $\glLie_{n+1}(F)$.  Using the trace pairing, we may $G$-equivariantly identify $\mathfrak{g}^* = \mathfrak{g}$ and $H$-equivariantly identify $\mathfrak{h}^* = \mathfrak{h}$.  The restriction map $\mathfrak{g}^* \rightarrow \mathfrak{h}^*$, $\xi \mapsto \xi_H$ then identifies with the map $\mathfrak{g} \rightarrow \mathfrak{h}, x \mapsto x_H$ given by
\begin{equation}\label{eqn:x-mapsto-x-H-std}
  x_H :=
  \begin{pmatrix}
    a & 0 \\
    0 & 0
  \end{pmatrix}
  \text{ if }
  x =
  \begin{pmatrix}
    a & b \\
    c & d
  \end{pmatrix},
\end{equation}
where $a,b,c,d$ have respective dimensions $n \times n, n \times 1, 1 \times n, 1 \times 1$.

We define what it means for an element of $\mathfrak{g}$ to be stable exactly as we did for $\mathfrak{g}^*$.  Since the identification $\mathfrak{g} = \mathfrak{g}^*$ is equivariant under $G$, hence under $H$, this definition is compatible with our earlier definition.  Our task is then equivalent to showing that \eqref{eqn:x-1-H-tau-not-in-brak-h-tau} holds for stable $\tau \in \mathfrak{g}$.  The advantage of switching from $\mathfrak{g}^*$ to $\mathfrak{g}$ is that matrix multiplication equips $\mathfrak{g}$ with the structure of an associative algebra.

We write $1 \in \mathfrak{g}$ for the identity matrix and $1_H \in \mathfrak{h}$ for its image under the map $\mathfrak{g} \rightarrow \mathfrak{h}$ described in \eqref{eqn:x-mapsto-x-H-std}.  Explicitly,
\[
  1_H = \diag(1,\dotsc,1,0).
\]
We note that $1_H$ defines a basis element for the center $\mathfrak{z}_H$ of $\mathfrak{h}$.

We introduce the abbreviation
\[
  \tau_H^j := 1_H (\tau_H)^j.
\]
The factor $1_H$ is significant only when $j=0$, in which case it ensures that $\tau_H^0 = 1_H$ (i.e., that the zeroth power is taken in the algebra $\mathfrak{h}$ rather than in $\mathfrak{g}$).  We will never refer to the other potential interpretation $(\tau^j)_H$ of the LHS.

For $\tau \in \mathfrak{g}$, let $\Delta(\tau)$ denote the determinant of the $n \times n$ matrix $D(\tau)$ with $(i,j)$ entry ($1 \leq i, j \leq n$) given by
\begin{equation*}
  D(\tau)_{i j}
  :=
  \langle [1_H, [\tau^i, \tau_H^{j-1}]], \tau  \rangle.
\end{equation*}

Let $\Delta_0(\tau)$ denote the following normalization of the resultant of the characteristic polynomials of $\tau \in \mathfrak{g}$ and $\tau_H \in \mathfrak{h}$:
\begin{equation}\label{eqn:Delta-0-defn}
  \Delta_0(\tau) = \prod_{i=1}^{n+1}
  \prod_{j=1}^n
  (\lambda_i - \mu_j),
\end{equation}
where $\{\lambda_1,\dotsc,\lambda_{n+1}\}$ and $\{\mu_1,\dotsc,\mu_n,0\}$ denote the multisets of eigenvalues of $\tau, \tau_H \in \mathfrak{g}$.

\begin{theorem}\label{thm:det-id}
  We have $\Delta(\tau) = 2^n (-1)^{n(n-1)/2} \Delta_0(\tau)$ for all $\tau \in \mathfrak{g}$.
\end{theorem}

\begin{proof}[Proof of Theorem
  \ref{thm:stab-characterization-1},
  assuming Theorem \ref{thm:det-id}]
  Let $\tau \in \mathfrak{g}$.  It is clear from the stability characterization noted in Lemma \ref{lem:stability-summary} that $\Delta_0(\tau) \neq 0$ if and only if $\tau$ is stable.  We conclude via the following lemma.
\end{proof}

\begin{lemma}\label{lem:equivalence-Delta-degen}
  We have $\Delta(\tau) \neq 0$ if and only if condition (ii) of Theorem \ref{thm:stab-characterization-1} holds.
\end{lemma}
\begin{proof}
  Suppose first that $\Delta(\tau) \neq 0$.  If $\tau$ is irregular, then its minimal polynomial has degree strictly less than $n+1$, so we may find coefficients $c_i \in F$ (not all zero) so that $x := \sum_{i=1}^n c_i \tau^i$ belongs to $\mathfrak{z}$.  The row vector $(c_1,\dotsc,c_n)$ then lies in the left kernel of the matrix $D(\tau)$.  Similarly, if $\tau_H$ is irregular, then we may find coefficients $c_j \in F$ (not all zero) so that $\sum_{j=1}^n c_j \tau_H^{j-1} = 0$, in which case the column vector $(c_1,\dotsc,c_n)^t$ lies in the right kernel of $D(\tau)$.  In either case, we obtain the contradiction that $\Delta(\tau) = 0$.  Thus $\tau$ and $\tau_H$ are both regular.  In that case,
  \[\{\tau^i : 1 \leq i \leq n \} \quad \text{ and } \quad \{\tau_H^{j-1} : 1 \leq j \leq n\}\] give bases for $\mathfrak{g}_\tau / \mathfrak{z}$ and $\mathfrak{h}_{\tau_H}$, respectively.  Since $z$ and $1_H$ are nonzero multiples of one another, we deduce that the nondegeneracy of the indicated bilinear form is equivalent to the nonvanishing of $\Delta(\tau)$.
  
  Conversely, if $\Delta(\tau) = 0$, then we see from the previous sentence that, even if $\tau$ and $\tau_H$ are regular, the indicated bilinear form is degenerate. Thus condition (ii) of Theorem \ref{thm:stab-characterization-1} fails.
\end{proof}

\begin{remark}\label{rmk:reduct-determ-ident-1}
  Theorem \ref{thm:det-id} yields the determinantal characterization
  \[\mathfrak{g}_{\stab} =
    \left\{\tau \in \mathfrak{g} : \Delta(\tau) \neq 0 \right\}
  \]
  of the stable subset.  Rallis--Schiffmann \cite[Prop 6.1]{2007arXiv0705.2168R} obtained a different determinantal characterization of a set closely related to $\mathfrak{g}_{\stab}$.  Their arguments give
  \begin{equation*}
    \mathfrak{g}_{\stab} = \left\{ \tau \in \mathfrak{g} :
      \Delta_{\RS}(\tau) \neq 0 \right\},
  \end{equation*}
  where $\Delta_{\RS}(\tau)$ denotes the determinant of the matrix $D_{\RS}(\tau)$ with $(i,j)$ entry
  \[
    D_{\RS}(\tau)_{i j} = \langle e_{n+1}^*, \tau^{i + j - 2} e_{n+1} \rangle \quad (1 \leq i, j \leq n),
  \]
  with $e_1,\dotsc,e_{n+1}$ and $e_1^*,\dotsc,e_{n+1}^*$ the standard basis and dual basis for $F^{n+1}$.  (In the notation used below, $D_{\RS}(\tau)_{i j} = q \tau^{i+j-2} p$.)  Wei Zhang has pointed out to us that one can use the identity \eqref{eqn:D-formula-1} stated below and some elementary row/column operations to show that $\Delta_{\RS}(\tau) = \pm 2^n \Delta(\tau)$.  This observation yields a shorter overall proof of Theorem \ref{thm:stab-characterization-1}, the result of primary interest.  One can show in turn that $\Delta_{\RS}(\tau) = \pm \Delta_0(\tau)$ by explicit calculation in the case that $\tau_H$ is regular semisimple, yielding an alternative proof of Theorem \ref{thm:det-id}.  We suspect that it may be possible to give further proofs using the many standard determinantal interpretations of resultants (Sylvester matrix, B{\'e}zout matrix, etc.).  We retain below our original treatment, which we emphasize is one among many possibilities.  Finally, it may comfort the skeptical reader to note that the statement of Theorem \ref{thm:det-id} was discovered via computer algebra and has been thoroughly tested.
\end{remark}

\subsection{Verification of the determinantal identity}\label{sec:verif-determ-ident}
In this section, we prove Theorem \ref{thm:det-id}.

\subsubsection{Roadmap}

Before diving into the details, we record an overall roadmap.  Recall the Vandermonde determinantal identity:
\[
  P(x) := \det \left( ( x_i^{j-1} )_{i,j=1}^n\right) = P_0(x) := \prod_{1 \leq i < j \leq n} (x_j - x_i).
\]

One of the standard proofs consists of three steps:
\begin{enumerate}
\item \emph{Divisibility}.  Viewing both sides as polynomials in the variables $x_1,\dotsc,x_n$, show that $P_0$ divides $P$.  The key observation is that $P(x)$ vanishes whenever some $x_i = x_j$ ($i \neq j$).
\item \emph{Homogeneity}.  Both sides are polynomials of the same degree, so the divisibility forces $P(x) = c P_0(x)$ for some constant $c$.
\item \emph{Evaluation of the constant}.  Comparing the coefficients of some well-chosen monomial, deduce that $c = 1$.  One can formulate this step in other ways, e.g., as a comparison of asymptotics as $x_{i+1}/x_{i} \rightarrow \infty$.
\end{enumerate}
The proof given below of Theorem \ref{thm:det-id} follows the same overall strategy, with Lemmas \ref{lem:Delta-0-divides-Delta}, \ref{lem:Delta-constant-multiple-Delta0} and \ref{lem:verif-determ-ident} fulfilling the three indicated steps.

\subsubsection{Notation and setup}

Turning to details, we retain the setup of \S\ref{sec:reduct-determ-ident}, but replacing \eqref{eqn:everything-is-general-linear} with
\[
  G = \GL(V), \quad H = \GL(V_H), \quad \mathfrak{g} = \glLie(V) = \End(V), \quad \mathfrak{h} = \glLie(V_H) = \End(V_H),
\]
where $V$ is an $(n+1)$-dimensional vector space equipped with a decomposition $V = V_H \oplus F e$ into an $n$-dimensional subspace $V_H$ and a one-dimensional subspace $F e$ spanned by some nonzero vector $e \in V$.  Thus \eqref{eqn:everything-is-general-linear} corresponds to the case
\begin{equation}\label{eqn:V-V-H-e-standard}
  V = F^{n+1}, \quad V_H = F^n
  = \{(\ast, \dotsc, \ast,0)\} \subseteq V,
  \quad
  e = (0,\dotsc,0,1).
\end{equation}
The reason for introducing this abstraction is that it will be convenient for us to work with multiple choices of coordinates throughout the proof.

We denote as usual by $[\mathfrak{g}]$ and $[\mathfrak{h}]$ the geometric invariant theory quotients.  The maps $\mathfrak{g} \rightarrow [\mathfrak{g}]$ and $\mathfrak{h} \rightarrow [\mathfrak{h}]$ may then be regarded as sending a matrix to the coefficients of its characteristic polynomial (see, e.g., \cite[\S13.4.1]{nelson-venkatesh-1}).  Let $R$ denote the coordinate ring of $[\mathfrak{g}] \times [\mathfrak{h}]$.  It is a free polynomial ring in $2 n + 1$ variables.  We may identify $R$ with
\[
  F[\lambda_1, \dotsc, \lambda_{n+1}, \mu_1, \dotsc, \mu_n]^{S(n+1) \times S(n)},
\]
where the $\lambda_i$ and $\mu_j$ are algebraically independent variables, corresponding to the generalized eigenvalues of a matrix, while $S(n+1)$ and $S(n)$ are symmetric groups acting by permutations on $\{\lambda_i\}$ and $\{\mu_j\}$, respectively.  Let
\[\iota : \mathfrak{g} \rightarrow
  [\mathfrak{g}] \times [\mathfrak{h}]
\]
denote the map assigning to $\tau \in \mathfrak{g}$ the invariants of the pair $(\tau, \tau_H)$.  The map $\iota$ is surjective (see for instance \cite[\S14.3, Lemma 1]{nelson-venkatesh-1}), so we may use the pullback $\iota^*$ to identify $R$ with a space of functions $\mathfrak{g} \rightarrow F$.  For instance, it is clear from \eqref{eqn:Delta-0-defn} that $\Delta_0 \in R$.

\subsubsection{$\Delta$ is an invariant polynomial}

\begin{lemma}\label{lem:Delta-in-R}
  We have $\Delta \in R$.
\end{lemma}
\begin{proof}
  We work with the standard coordinates \eqref{eqn:V-V-H-e-standard}.  We regard $F^{n+1}$ as a space of column vectors.  We set $p := e$ (a column vector) and $q := e^t$ (a row vector).  Then, multiplying matrices and vectors in the usual way, we have
  \begin{equation}\label{eqn:frak-h-via-p-q}
    \mathfrak{h} = \{x \in \mathfrak{g} : x p = 0, q x = 0\},
  \end{equation}
  \begin{equation}\label{eqn:1-H-1-minus-p-q}
    1_H = 1 - p q.
  \end{equation}

  We observe that for each $j \in \mathbb{Z}_{\geq 0}$, we have
  \begin{equation}\label{eqn:tau-q-tau-j-p-in-R}
    [\tau \mapsto  q \tau^j p] \in R.
  \end{equation}
  This fact was observed (implicitly) in \cite[\S14]{nelson-venkatesh-1} to follow from the identity of formal power series
  \begin{equation*}
    \sum_{j \in \mathbb{Z}_{\geq 0}}
    t^j q \tau^j p
    =
    q (1 - t \tau)^{-1} p
    =
    \frac{\det(1 - t \tau_H)}{\det(1- t \tau)},
  \end{equation*}
  which is itself a consequence of Cramer's rule.

  We thereby reduce to verifying that each matrix entry $D(\tau)_{i j}$ is a polynomial combination of the quantities $q \tau^j p$.  To that end, we apply \eqref{eqn:1-H-1-minus-p-q} and the fact that $1$ commutes with everything to see that for any $x \in \mathfrak{g}$, we have
  \begin{align}\label{eq:1-H-z-tau-comm-brack}
    \langle [1_H, x], \tau  \rangle
    &=
      - \langle [ p q, x], \tau  \rangle
    \\ \nonumber
    &=
      - \langle p q, [x,\tau]  \rangle
    \\
    &=\nonumber
      q [\tau,x] p.
  \end{align}
  We then specialize this observation to $x = [\tau^i, \tau_H^{j-1}]$, so that $D_{i j}(\tau) = \langle [1_H, x], \tau \rangle$, and expand out $\tau_H = 1_H \tau 1_H$ and $1_H = 1 - p q$ and all commutators.
\end{proof}

The following result is not necessary for our purposes, but seems worth recording.
\begin{lemma}
  The ring $R$ consists of all $H$-invariant polynomial functions $\mathfrak{g} \rightarrow F$, i.e., $\iota$ defines the geometric invariant theory quotient of $\mathfrak{g}$ by $H$.
\end{lemma}
\begin{proof}
  Let $S$ denote the ring of $H$-invariant polynomial functions $\mathfrak{g} \rightarrow F$.  We must check that $R = S$.  Clearly $R \subseteq S$.  Conversely, it follows from \cite[Lem 3.1]{MR3245011} that $S$ is generated by
  \begin{itemize}
  \item the coefficients of the characteristic polynomial of $\tau$, and
  \item the quantities $q \tau^j p$ with $j \in \{1,\dotsc,n\}$ as in \eqref{eqn:tau-q-tau-j-p-in-R}.
  \end{itemize}
  By \eqref{eqn:tau-q-tau-j-p-in-R}, we conclude that $S \subseteq R$.
\end{proof}

\subsubsection{Divisibility}
We next verify ``one divisibility'' in the required identity relating $\Delta$ and $\Delta_0$.
\begin{lemma}\label{lem:Delta-0-divides-Delta}
  $\Delta_0$ divides $\Delta$ in the ring $R$.
\end{lemma}
\begin{proof}
  The ideal $(\Delta_0)$ is generated by a squarefree element of the unique factorization domain $R$, hence is radical, so by the Nullstellensatz, it is enough to verify that \[\Delta_0(\tau) = 0 \implies \Delta(\tau) = 0.\]

  Suppose that $\Delta_0(\tau) = 0$, i.e., that $\tau \in \mathfrak{g}$ is not stable.  We must check that $\Delta(\tau) = 0$.  In view of Lemma \ref{lem:equivalence-Delta-degen}, it is enough to check that condition (ii) of Theorem \ref{thm:stab-characterization-1} fails.  Suppose otherwise, for the sake of contradiction, that it holds.  In particular, $\tau$ and $\tau_H$ are regular.

  Since $\tau$ and $\tau_H$ are not stable, we see from \cite[\S14.2, Lem]{nelson-venkatesh-1} that either
  \begin{itemize}
  \item $\tau$ and $\tau_H$ admit a common eigenvector, or
  \item their transposes admit a common eigenvector.
  \end{itemize}
  We consider below the first case.  The second case may be treated similarly by applying the same argument to the transposes.

  Since $\tau_H$ is regular, we may decompose $V_H$ into generalized eigenspaces $W_1,\dotsc,W_m$ for $\tau_H$, with distinct eigenvalues $\mu_1,\dotsc,\mu_m$, and with the action of $\tau_H$ on $W_j$ given with respect to some basis by a standard Jordan block, given in the case $\dim(W_j)=3$ (say) by
  \begin{equation}\label{eqn:tau-H-jordan-form}
    \tau_H|_{W_j}
    \sim
    \begin{pmatrix}
      \mu_j & 1 & 0 \\
      & \mu_j & 1 \\
      & & \mu_j
    \end{pmatrix}.
  \end{equation}
  The operator $\tau_H$ has exactly $m$ eigenspaces, each one-dimensional, one for each summand $W_j$.

  By relabeling, we may suppose that $\tau$ has an eigenvector in $W_1$.  Let $e_1,\dotsc,e_{\dim(W_1)}$ be a basis of $W_1$ with respect to which $\tau_H|_{W_1}$ is described by a matrix as in \eqref{eqn:tau-H-jordan-form}.  Then $e_1$ is an eigenvector of $\tau$.  We extend to a basis $e_1,\dotsc,e_n$ of $V_H$ and then further to a basis $e_1,\dotsc,e_{n+1}$ of $V$.
  
  The centralizer $\mathfrak{h}_{\tau_H}$ is the direct sum over $j \in \{1, \dotsc, m\}$ of the $\glLie(W_j)$-centralizer of $\tau_H|_{W_j}$, the latter of which consists in the case $\dim(W_j) = 3$ (say) of all matrices of the form
  \[
    \begin{pmatrix}
      a & b & c \\
      0 & a & b \\
      0 & 0 & a
    \end{pmatrix}.
  \]
  There is thus a unique (nonzero) element $y \in \mathfrak{h}_{\tau_H}$ for which
  \begin{itemize}
  \item $y|_{W_j} = 0$ for $j \neq 1$, and
  \item $y|_{W_1} e_i = 0$ for $1 \leq i < \dim(W_1)$ and $y|_{W_1} e_{\dim(W_1)} = e_1$.
  \end{itemize}
  For instance, in the case $\dim(W_1) = 3$ (say),
  \begin{equation*}
    y|_{W_1}
    \sim
    \begin{pmatrix}
      0 & 0 & 1 \\
      0 & 0 & 0 \\
      0 & 0 & 0
    \end{pmatrix}.
  \end{equation*}

  Let $x$ be any element of $\mathfrak{g}_\tau$.  We claim that
  \begin{equation}\label{eqn:key-claim-stability-failure-implies-degeneracy}
    [1_H, [x,y]]_{i j} = 0 \text{ unless } (i,j) = (1,n+1).
  \end{equation}
  Indeed, since $\tau$ is regular, we know that the element $x \in \mathfrak{g}_\tau$ is a linear combination of powers of $\tau$, and so $e_1$ is an eigenvector of $x$.  Thus $x_{i,1} \neq 0$ only if $i = 1$.  By the construction of $y$ and some matrix multiplication, we deduce that $[x,y]_{i j} \neq 0$ only if $i = 1$.  For any $z \in \mathfrak{g}$, we have
  \[
    [1_H,z]_{i j} = \begin{cases}
      0 & \text{ if } i \leq n, j \leq n, \\
      0  & \text{ if } i=j=n+1, \\
      z_{i j} &  \text{ if } i \leq n, j = n+1, \\
      - z_{i j} & \text{ if } i = n+1, j \leq n.
    \end{cases}
  \]
  Applying this observation to $z = [x,y]$ yields the claim.

  We deduce from \eqref{eqn:key-claim-stability-failure-implies-degeneracy} that $\langle [1_H, [x,y]], \tau \rangle$ is a multiple of $\tau_{n+1,1}$, which vanishes in view of our assumption that $e_1$ is a $\tau$-eigenvector.  Since $y$ is nonzero and this last property holds for all $x \in \mathfrak{g}_\tau$, we deduce that the bilinear form in condition (ii) of Theorem \ref{thm:stab-characterization-1} is degenerate, contrary to hypothesis.  This gives the required contradiction and completes the proof of the lemma.
\end{proof}

\subsubsection{Homogeneity}
We may upgrade lemma \ref{lem:Delta-0-divides-Delta} as follows:
\begin{lemma}\label{lem:Delta-constant-multiple-Delta0}
  There exists $c \in F$ so that $\Delta(\tau) = c \Delta_0(\tau)$ for all $\tau \in \mathfrak{g}$.
\end{lemma}
\begin{proof}
  We consider the behavior of $\Delta$ and $\Delta_0$ under homothety.  For $t \in F^\times$, we have
  \begin{equation*}
    D(t \tau)_{i j} =
    t^{i +j} D(\tau)_{i j}.
  \end{equation*}
  Since $1 + 2 + \dotsb + n = n(n+1)/2$, we obtain
  \begin{equation}\label{eqn:homothety-Delta}
    \Delta(t \tau) = t^{n(n+1)} \Delta(\tau).
  \end{equation}
  On the other hand, if $\lambda_1,\dotsc,\lambda_{n+1}$ and $\mu_1,\dotsc,\mu_n$ are the eigenvalues of $\tau$ and $\tau_H$, then
  \begin{equation}\label{eqn:homothety-Delta0}
    \Delta_0(t \tau) =
    \prod_{i=1}^{n+1} \prod_{j=1}^n (t \lambda_i
    - t \mu_j)
    =
    t^{n(n+1)} \Delta_0(\tau).
  \end{equation}
  Thus $\Delta$ and $\Delta_0$ are homogeneous elements of $R$ of the same degree.  We have seen in Lemma \ref{lem:Delta-0-divides-Delta} that $\Delta_0$ divides $\Delta$.  Since $\Delta_0$ is not identically zero, the conclusion follows.
\end{proof}

\begin{remark}
  To complete the proofs of Theorems \ref{thm:stability-consequence-for-1-H} and \ref{thm:stab-characterization-1}, it suffices at this point to show that $\Delta$ is not identically zero: lemma \ref{lem:Delta-constant-multiple-Delta0} then forces $\Delta$ to be a nonzero constant multiple of $\Delta_0$.
\end{remark}

\subsubsection{Evaluation of the constant}

The following calculation shows that the constant $c$ in Lemma \ref{lem:Delta-constant-multiple-Delta0} equals $2^n (-1)^{n(n-1)/2}$, thereby completing the proof of Theorem \ref{thm:det-id}.
\begin{lemma}\label{lem:verif-determ-ident}
  There exists an element $\tau \in \mathfrak{g}$ for which
  \begin{equation}\label{eqn:det-tau-2-n}
    \Delta(\tau) = 2^n (-1)^{n(n-1)/2}
  \end{equation}
  and
  \begin{equation}\label{eqn:delta-0-tau-neg-1-n}
    \Delta_0(\tau) = 1.
  \end{equation}
\end{lemma}
\begin{proof}
  We will exhibit $\tau$ explicitly.  To do so, it is convenient to choose our coordinates carefully.

  First, some notation.  For $v \in V, m \in \End(V)$ and $u \in V^*$, we may canonically define
  \[m v \in V, \quad u m \in V^*, \quad u m v \in F, \quad u v \in F, \quad v u \in \End(F)
  \]
  via ``matrix multiplication,'' regarding $V$ (resp. $V^*$) as a space of column (resp. row) vectors.

  Recall that the $(n+1)$-dimensional space $V$ is the sum of the $n$-dimensional space $V_H$ and the line $F e$.  Set $p := e \in V$, and let $q \in V^*$ denote the unique element with $q|_{V_H} = 0$ and $q p = 1$.  Then $\mathfrak{h}$ is described as in \eqref{eqn:frak-h-via-p-q}, and $1_H = 1 - p q$.

  We choose an isomorphism $V \cong F^{n+1}$ so that, with the respect to the corresponding basis $e_1,\dotsc,e_{n+1}$ of $V$ and dual basis $e_1^*,\dotsc,e_{n+1}^*$ of $V^*$, we have
  \[
    p = e_1, \quad q = e_1^* + e_{n+1}^*.
  \]
  
  We remark that the resulting embedding $\glLie_{n}(F) \cong \mathfrak{h} \hookrightarrow \mathfrak{g} \cong \glLie_{n+1}(F)$ is not customary.  For instance, when $n=2$, it is may given by
  \begin{equation*}
    \begin{pmatrix}
      a & b \\
      c & d
    \end{pmatrix}
    \mapsto
    \begin{pmatrix}
      0 & -c & -d \\
      0 & a & b \\
      0 & c & d
    \end{pmatrix}.
  \end{equation*}

  In these coordinates, we take for $\tau$ the standard lower-triangular nilpotent Jordan block, given for (say) $n=3$ by
  \[
    \tau = \begin{pmatrix}
      0 & 0 & 0 & 0 \\
      1 & 0 & 0 & 0 \\
      0  & 1 & 0 & 0 \\
      0 & 0 & 1 & 0
    \end{pmatrix}.
  \]
  We note that both $p$ and $q$ are $\tau$-cyclic, so the element $\tau \in \mathfrak{g}$ is stable by the stability characterization recorded in \cite[\S14.3, Lem 3]{nelson-venkatesh-1} (see also \cite[Thm 6.1]{2007arXiv0705.2168R}).  We write $e_{i j} := e_{i,j} := e_i e_j^* \in \mathfrak{g}$ for the elementary matrices with respect to the chosen coordinates.  The matrix $p q$ is given by
  \[p q = e_{1, 1} + e_{1, n+1},\] and so
  \[1_H = 1 - p q = \sum_{i=2}^{n+1} e_{i,i} -e_{1,n+1}.\] We deduce that
  \[\tau_H = \sum_{i=2}^n e_{i+1,i} - e_{1,n} - e_{2,n+1}\]
  and then, by induction on $\ell \geq 0$, that
  \begin{equation}\label{eqn:formula-one-H-tau-H-ell}
    \tau_H^\ell =
    \sum_{i,j: i \geq 2, i -j = \ell} e_{i,j}
    -
    \sum_{i,j: j - i = n - \ell} e_{i,j}.
  \end{equation}
  For example, when $n=3$, the elements $\tau_H^\ell$ are given for $\ell=0,1,2$ by the matrices
  \begin{equation}\label{eqn:1-H-nu-H-ell}
    \begin{pmatrix}
      0 & 0 & 0 & -1 \\
      0 & 1 & 0 & 0 \\
      0 & 0 & 1 & 0 \\
      0 & 0 & 0 & 1 \\
    \end{pmatrix}, \quad
    \begin{pmatrix}
      0 & 0 & -1 & 0 \\
      0 & 0 & 0 & -1 \\
      0 & 1 & 0 & 0 \\
      0 & 0 & 1 & 0 \\
    \end{pmatrix}, \quad
    \begin{pmatrix}
      0 & -1 & 0 & 0 \\
      0 & 0 & -1 & 0 \\
      0 & 0 & 0 & -1 \\
      0 & 1 & 0 & 0 \\
    \end{pmatrix}.
  \end{equation}

  The identity \eqref{eqn:delta-0-tau-neg-1-n} is now simple to verify.  Since $\tau$ is nilpotent, every eigenvalue $\lambda_i$ of $\tau$ is zero.  By computing the characteristic polynomial of $\tau_H$ (e.g., the second matrix in \eqref{eqn:1-H-nu-H-ell}), we see that the eigenvalues $\{\mu_1,\dotsc,\mu_n\}$ of $\tau_H$ on $V_H$ are the solutions $\mu$ to $\mu^n = -1$.  The product of $-\mu$ over such $\mu$ is $1$, whence \eqref{eqn:delta-0-tau-neg-1-n}.

  It remains to verify \eqref{eqn:det-tau-2-n}.  Let $k, \ell \in \{1, \dotsc, n\}$.  We use \eqref{eq:1-H-z-tau-comm-brack} to write
  \[
    D(\tau)_{k,l} = q [\tau, [\tau^k, \tau_H^{\ell-1}]] p.
  \]
  Expanding both commutators, we obtain four terms.  Two terms vanish because $\tau_H^{l-1} p =0$ and $q \tau_H^{l-1} = 0$.  We obtain
  \begin{equation}\label{eqn:D-formula-1}
    D(\tau)_{k,l}
    =
    - q \tau \tau_H^{\ell-1}\tau^k p
    - q \tau^k \tau_H^{\ell-1}  \tau p.
  \end{equation}
  Using \eqref{eqn:formula-one-H-tau-H-ell} and the definition of $\tau$, we see that
  \begin{equation}\label{eqn:D-formula-3}
    q \tau \tau_H^{\ell-1}
    =
    1_{l < n}
    e_{n-l+1}^*
    -
    1_{l=n}
    e_{n+1}^*,
    \quad
    \tau^k p = e_{k+1},
  \end{equation}
  and
  \begin{equation}\label{eqn:D-formula-4}
    q \tau^k
    =
    e_{n-k+1}^*,
    \quad 
    \tau_H^{\ell-1}
    \tau 
    p
    =
    e_{\ell+1}
    -
    1_{\ell = n} e_1.
  \end{equation}
  Combining \eqref{eqn:D-formula-1}, \eqref{eqn:D-formula-3} and \eqref{eqn:D-formula-4}, we arrive at
  \begin{equation*}
    D(\tau)_{k,l}
    = \begin{cases}
      -2 & \text{ if } k + \ell = n, \\
      2 & \text{ if } (k,\ell) = (n,n), \\
      0 & \text{ otherwise}.
    \end{cases}
  \end{equation*}
  For instance, in the case $n=3$, we obtain
  \[
    D(\tau) = \begin{pmatrix}
      0 & -2 & 0 \\
      -2 & 0 & 0  \\
      0 & 0 & 2
    \end{pmatrix}.
  \]
  Thus only one permutation contributes to the sum formula for the determinant of $D(\tau)$.  The required identity \eqref{eqn:det-tau-2-n} follows after computing the sign of that permutation, which we leave to the reader.
\end{proof}

\printindex


\begin{thebibliography}{10}

\bibitem[Ag]{2019arXiv190309638A} Keshav {Aggarwal}.
  \newblock {A new subconvex bound for $\rm GL(3)$ $L$-functions in the $t$-aspect}.
  \newblock {\em arXiv e-prints}, page arXiv:1903.09638, March 2019.

\bibitem[An]{MR0282982} A.~N. Andrianov.
  \newblock Spherical functions for {${\rm GL}_{n}$} over local fields, and the summation of {H}ecke series.
  \newblock {\em Mat. Sb. (N.S.)}, 83 (125):429--451, 1970.

\bibitem[BM]{MR2154720} Ehud~Moshe Baruch and Zhengyu Mao.
  \newblock Bessel identities in the {W}aldspurger correspondence over the real numbers.
  \newblock {\em Israel J. Math.}, 145:1--81, 2005.

\bibitem[BR]{MR2726097} Joseph Bernstein and Andre Reznikov.
  \newblock Subconvexity bounds for triple {$L$}-functions and representation theory.
  \newblock {\em Ann. of Math. (2)}, 172(3):1679--1718, 2010.

\bibitem[BP1]{2015arXiv150601452B} R.~{Beuzart-Plessis}.
  \newblock {A local trace formula for the Gan-Gross-Prasad conjecture for unitary groups: the archimedean case}.
  \newblock {\em ArXiv e-prints}, June 2015.

\bibitem[BP2]{2016arXiv160206538B} R.~{Beuzart-Plessis}.
  \newblock {Comparison of local spherical characters and the Ichino-Ikeda conjecture for unitary groups}.
  \newblock {\em ArXiv e-prints}, February 2016.

\bibitem[BP3]{2018arXiv181200047B} Rapha{\"e}l {Beuzart-Plessis}.
  \newblock {Plancherel formula for $\mathrm{GL}_n(F)\backslash \mathrm{GL}_n(E)$ and applications to the Ichino-Ikeda and formal degree conjectures for unitary groups}.
  \newblock {\em arXiv e-prints}, page arXiv:1812.00047, November 2018.

\bibitem[BCZ]{2020arXiv200705601B} Rapha{\"e}l {Beuzart-Plessis}, Pierre-Henri {Chaudouard}, and Micha{\l} {Zydor}.
  \newblock {The global Gan-Gross-Prasad conjecture for unitary groups: the endoscopic case}.
  \newblock {\em arXiv e-prints}, page arXiv:2007.05601, July 2020.

\bibitem[BLZZ]{2019arXiv191207169B} Rapha{\"e}l {Beuzart-Plessis}, Yifeng {Liu}, Wei {Zhang}, and Xinwen {Zhu}.
  \newblock {Isolation of cuspidal spectrum, with application to the Gan--Gross--Prasad conjecture}.
  \newblock {\em arXiv e-prints}, page arXiv:1912.07169, December 2019.


\bibitem[BB1]{MR4203038} Valentin Blomer and Jack Buttcane.
  \newblock On the subconvexity problem for {$L$}-functions on {$\rm GL(3)$}.
  \newblock {\em Ann. Sci. \'{E}c. Norm. Sup\'{e}r. (4)}, 53(6):1441--1500, 2020.

\bibitem[Blo]{MR3127065} Valentin Blomer.
  \newblock Applications of the {K}uznetsov formula on {$GL(3)$}.
  \newblock {\em Invent. Math.}, 194(3):673--729, 2013.

\bibitem[BB2]{MR4039487} Valentin Blomer and Jack Buttcane.
  \newblock Subconvexity for {$L$}-functions of non-spherical cusp forms on {${\rm GL}(3)$}.
  \newblock {\em Acta Arith.}, 192(1):31--62, 2020.

\bibitem[BKY]{MR3127809} Valentin Blomer, Rizwanur Khan, and Matthew Young.
  \newblock Distribution of mass of holomorphic cusp forms.
  \newblock {\em Duke Math. J.}, 162(14):2609--2644, 2013.

\bibitem[BM1]{MR3384442} Valentin Blomer and P{\'e}ter Maga.
  \newblock The sup-norm problem for {PGL}(4).
  \newblock {\em Int. Math. Res. Not. IMRN}, (14):5311--5332, 2015.

\bibitem[BM2]{2014arXiv1405.6691B} Valentin Blomer and P\'{e}ter Maga.
  \newblock Subconvexity for sup-norms of cusp forms on {$\rm {PGL}(n)$}.
  \newblock {\em Selecta Math. (N.S.)}, 22(3):1269--1287, 2016.

\bibitem[Bu1]{MR3130656} Jack Buttcane.
  \newblock On sums of {$SL(3,\Bbb{Z})$} {K}loosterman sums.
  \newblock {\em Ramanujan J.}, 32(3):371--419, 2013.

\bibitem[Bu2]{MR3461048} Jack Buttcane.
  \newblock The spectral {K}uznetsov formula on {$SL(3)$}.
  \newblock {\em Trans. Amer. Math. Soc.}, 368(9):6683--6714, 2016.

\bibitem[Co]{cocalc}
SageMath, Inc.
\newblock {\em CoCalc Collaborative Calculation Online}, 2020.
\newblock {\tt https://cocalc.com/}.

\bibitem[DFI]{MR1923476} W.~Duke, J.~B. Friedlander, and H.~Iwaniec.
  \newblock The subconvexity problem for {A}rtin {$L$}-functions.
  \newblock {\em Invent. Math.}, 149(3):489--577, 2002.

\bibitem[Ha]{MR3159075} R.~Neal Harris.
  \newblock The refined {G}ross-{P}rasad conjecture for unitary groups.
  \newblock {\em Int. Math. Res. Not. IMRN}, (2):303--389, 2014.

\bibitem[HN]{RHPNtwists} Roman Holowinsky and Paul~D. Nelson.
  \newblock Subconvex bounds on gl(3) via degeneration to frequency zero.
  \newblock {\em Mathematische Annalen}, 01 2018.

\bibitem[H\"or]{MR2304165} Lars H\"ormander.
  \newblock {\em The analysis of linear partial differential operators. {III}}.
  \newblock Classics in Mathematics. Springer, Berlin, 2007.
  \newblock Pseudo-differential operators, Reprint of the 1994 edition.

\bibitem[Hr]{MR546178} Karel Hrbacek.
  \newblock Nonstandard set theory.
  \newblock {\em Amer. Math. Monthly}, 86(8):659--677, 1979.

\bibitem[II]{MR2585578} Atsushi Ichino and Tamutsu Ikeda.
  \newblock On the periods of automorphic forms on special orthogonal groups and the {G}ross-{P}rasad conjecture.
  \newblock {\em Geom. Funct. Anal.}, 19(5):1378--1425, 2010.

\bibitem[IS1]{iwan-sar} Henryk Iwaniec and Peter Sarnak.
  \newblock {$L^\infty$} norms of eigenfunctions of arithmetic surfaces.
  \newblock {\em Ann. of Math. (2)}, 141(2):301--320, 1995.

\bibitem[IS2]{MR1826269} Henryk Iwaniec and Peter Sarnak.
  \newblock Perspectives on the analytic theory of {$L$}-functions.
  \newblock {\em Geom. Funct. Anal.}, (Special Volume, Part II):705--741, 2000.
  \newblock GAFA 2000 (Tel Aviv, 1999).

\bibitem[Iw]{Iwaniec1992}
Henryk Iwaniec.
\newblock The spectral growth of automorphic l-functions.
\newblock {\em Journal für die reine und angewandte Mathematik}, 428:139--160,
  1992.

\bibitem[Ka]{MR723332} Toru Kawai.
  \newblock Nonstandard analysis by axiomatic method.
  \newblock In {\em Southeast {A}sian conference on logic ({S}ingapore, 1981)}, volume 111 of {\em Stud. Logic Found. Math.}, pages 55--76. North-Holland, Amsterdam, 1983.

\bibitem[Kh]{MR3334086} Rizwanur Khan.
  \newblock On the subconvexity problem for {$\rm{GL}(3)\times\rm{GL}(2)$} {$L$}-functions.
  \newblock {\em Forum Math.}, 27(2):897--913, 2015.

\bibitem[Ki]{MR1701415} A.~A. Kirillov.
  \newblock Merits and demerits of the orbit method.
  \newblock {\em Bull. Amer. Math. Soc. (N.S.)}, 36(4):433--488, 1999.

\bibitem[Kn]{MR855239} Anthony~W. Knapp.
  \newblock {\em Representation theory of semisimple groups}, volume~36 of {\em Princeton Mathematical Series}.
  \newblock Princeton University Press, Princeton, NJ, 1986.
  \newblock An overview based on examples.

\bibitem[Ko]{MR0158024} Bertram Kostant.
  \newblock Lie group representations on polynomial rings.
  \newblock {\em Amer. J. Math.}, 85:327--404, 1963.

\bibitem[KMS]{2020arXiv200607819K} Sumit {Kumar}, Kummari {Mallesham}, and Saurabh~Kumar {Singh}.
  \newblock {Sub-convexity bound for $GL(3) \times GL(2)$ $L$-functions: $GL(3)$-spectral aspect}.
  \newblock {\em arXiv e-prints}, page arXiv:2006.07819, June 2020.

\bibitem[Ku]{Kurtzke} John~F. Kurtzke, Jr.
  \newblock Centralizers of irregular elements in reductive algebraic groups.
  \newblock {\em Pacific J. Math.}, 104(1):133--154, 1983.

\bibitem[Li]{MR2753605} Xiaoqing Li.
  \newblock Bounds for {${\rm GL}(3)\times {\rm GL}(2)$} {$L$}-functions and {${\rm GL}(3)$} {$L$}-functions.
  \newblock {\em Ann. of Math. (2)}, 173(1):301--336, 2011.

\bibitem[LMS]{2019arXiv191209473L} Yongxiao {Lin}, Philippe {Michel}, and Will {Sawin}.
  \newblock {Algebraic twists of $\mathrm{GL}_3\times \mathrm{GL}_2$ $L$-functions}.
  \newblock {\em arXiv e-prints}, page arXiv:1912.09473, December 2019.

\bibitem[LRS]{MR1703764}
  Wenzhi Luo, Ze\'{e}v Rudnick, and Peter Sarnak.
  \newblock On the generalized {R}amanujan conjecture for {${\rm GL}(n)$}.
  \newblock In {\em Automorphic forms, automorphic representations, and
    arithmetic ({F}ort {W}orth, {TX}, 1996)}, volume~66 of {\em Proc. Sympos.
    Pure Math.}, pages 301--310. Amer. Math. Soc., Providence, RI, 1999.


\bibitem[Lu]{2020arXiv200913947L} Zhilin {Luo}.
  \newblock {A Local Trace Formula for the Local Gan-Gross-Prasad Conjecture for Special Orthogonal Groups}.
  \newblock {\em arXiv e-prints}, page arXiv:2009.13947, September 2020.

\bibitem[Mi]{MR2331346} Philippe Michel.
  \newblock Analytic number theory and families of automorphic {$L$}-functions.
  \newblock In {\em Automorphic forms and applications}, volume~12 of {\em IAS/Park City Math. Ser.}, pages 181--295. Amer. Math. Soc., Providence, RI, 2007.

\bibitem[MV1]{MichelVenkateshICM} Philippe Michel and Akshay Venkatesh.
  \newblock Equidistribution, {$L$}-functions and ergodic theory: on some problems of {Y}u.\ {L}innik.
  \newblock In {\em International {C}ongress of {M}athematicians. {V}ol. {II}}, pages 421--457. Eur. Math. Soc., Z\"urich, 2006.
  
\bibitem[MV2]{michel-2009} Philippe Michel and Akshay Venkatesh.
  \newblock The subconvexity problem for {${\rm GL}_2$}.
  \newblock {\em Publ. Math. Inst. Hautes \'Etudes Sci.}, (111):171--271, 2010.

\bibitem[Mu1]{MR3369905} Ritabrata Munshi.
  \newblock The circle method and bounds for {$L$}-functions---{III}: {$t$}-aspect subconvexity for {$GL(3)$} {$L$}-functions.
  \newblock {\em J. Amer. Math. Soc.}, 28(4):913--938, 2015.

\bibitem[Mu2]{MR3418527} Ritabrata Munshi.
  \newblock The circle method and bounds for {$L$}-functions---{IV}: subconvexity for twists of {$\rm GL(3)$} {$L$}-functions.
  \newblock {\em Ann. of Math. (2)}, 182(2):617--672, 2015.

\bibitem[Mu3]{2018arXiv181000539M} Ritabrata {Munshi}.
  \newblock {Subconvexity for $GL(3)\times GL(2)$ $L$-functions in $t$-aspect}.
  \newblock {\em arXiv e-prints}, page arXiv:1810.00539, October 2018.

\bibitem[Ne1]{MR469763} Edward Nelson.
  \newblock Internal set theory: a new approach to nonstandard analysis.
  \newblock {\em Bull. Amer. Math. Soc.}, 83(6):1165--1198, 1977.

\bibitem[Ne2]{2021arXiv210915230N}
Paul~D. {Nelson}.
\newblock {Bounds for standard {$L$}-functions}.
\newblock {\em arXiv e-prints}, page arXiv:2109.15230, September 2021.

\bibitem[NV]{nelson-venkatesh-1} Paul~D. {Nelson} and A.~{Venkatesh}.
  \newblock {The orbit method and analysis of automorphic forms}.
  \newblock {\em Acta Math.}, 226(1):1--209, 2021.

\bibitem[RS]{2007arXiv0705.2168R} S.~{Rallis} and G.~{Schiffmann}.
  \newblock {Multiplicity one Conjectures}.
  \newblock {\em ArXiv e-prints}, May 2007.

\bibitem[Ri]{MR1064995} Marc~A. Rieffel.
  \newblock Lie group convolution algebras as deformation quantizations of linear {P}oisson structures.
  \newblock {\em Amer. J. Math.}, 112(4):657--685, 1990.

\bibitem[Ro1]{MR587333} W.~Rossmann.
  \newblock Limit characters of reductive {L}ie groups.
  \newblock {\em Invent. Math.}, 61(1):53--66, 1980.

\bibitem[Ro2]{MR650379} W.~Rossmann.
  \newblock Tempered representations and orbits.
  \newblock {\em Duke Math. J.}, 49(1):231--247, 1982.

\bibitem[Ro3]{MR508985} Wulf Rossmann.
  \newblock Kirillov's character formula for reductive {L}ie groups.
  \newblock {\em Invent. Math.}, 48(3):207--220, 1978.

\bibitem[SV]{SV} Yiannis Sakellaridis and Akshay Venkatesh.
  \newblock Periods and harmonic analysis on spherical varieties.
  \newblock {\em Ast\'{e}risque}, (396):viii+360, 2017.

\bibitem[Sa1]{MR780071} Peter Sarnak.
  \newblock Fourth moments of {G}r\"ossencharakteren zeta functions.
  \newblock {\em Comm. Pure Appl. Math.}, 38(2):167--178, 1985.

\bibitem[Sa2]{sarnak-morawetz}
  Peter Sarnak.
  \newblock Letter to {C}athleen {M}orawetz.
  \newblock August, 2004, 2004.

\bibitem[Se]{MR2179691}
Jean-Pierre Serre.
\newblock {\em Lie algebras and {L}ie groups}, volume 1500 of {\em Lecture
  Notes in Mathematics}.
\newblock Springer-Verlag, Berlin, 2006.
\newblock 1964 lectures given at Harvard University, Corrected fifth printing
  of the second (1992) edition.
  
\bibitem[Sh]{2020arXiv201010153S} Prahlad {Sharma}.
  \newblock {Subconvexity for $GL(3)\times GL(2)$ $L$-functions in $GL(3)$ spectral aspect}.
  \newblock {\em arXiv e-prints}, page arXiv:2010.10153, October 2020.

\bibitem[SV]{SV-AQUE} Lior Silberman and Akshay Venkatesh.
  \newblock Quantum unique ergodicity for locally symmetric spaces {I}{I}.

\bibitem[SZ]{MR2874638} Binyong Sun and Chen-Bo Zhu.
  \newblock Multiplicity one theorems: the {A}rchimedean case.
  \newblock {\em Ann. of Math. (2)}, 175(1):23--44, 2012.

\bibitem[Ve]{venkatesh-2005} Akshay Venkatesh.
  \newblock Sparse equidistribution problems, period bounds and subconvexity.
  \newblock {\em Ann. of Math. (2)}, 172(2):989--1094, 2010.

\bibitem[Wa]{Wald4} Jean-Loup Waldspurger.
  \newblock Une formule int\'egrale reli\'ee \`a la conjecture locale de {G}ross-{P}rasad, 2e partie: extension aux repr\'esentations temp\'er\'ees.
  \newblock {\em Ast\'erisque}, (346):171--312, 2012.
  \newblock Sur les conjectures de Gross et Prasad. I.

\bibitem[We]{MR0427267}
Andr\'{e} Weil.
\newblock {\em Basic number theory}.
\newblock Die Grundlehren der mathematischen Wissenschaften, Band 144.
  Springer-Verlag, New York-Berlin, third edition, 1974.
  
\bibitem[Zh1]{MR3245011} Wei Zhang.
  \newblock Fourier transform and the global {G}an-{G}ross-{P}rasad conjecture for unitary groups.
  \newblock {\em Ann. of Math. (2)}, 180(3):971--1049, 2014.


\bibitem[Zh2]{MR3164988} Wei Zhang.
  \newblock Automorphic period and the central value of {R}ankin-{S}elberg {L}-function.
  \newblock {\em J. Amer. Math. Soc.}, 27(2):541--612, 2014.

\end{thebibliography}

\def\cprime{$'$} \def\cprime{$'$} \def\cprime{$'$} \def\cprime{$'$}

\end{document}